\documentclass[letterpaper]{amsart}
\usepackage[T1]{fontenc}
\usepackage{amssymb}
\usepackage{amsmath}
\usepackage{amsthm}
\usepackage[dvipsnames]{xcolor}
\usepackage[shortlabels]{enumitem}
\usepackage{braket,comment,soul,cancel}
\usepackage{hyperref}
\usepackage{pdfpages,faktor}
\usepackage{graphicx,mathabx,bbm}
\usepackage{caption}
\usepackage{subcaption}
\usepackage{mathrsfs} 
\usepackage{tikz-cd} 
\usepackage{bm}
\usepackage{enumitem}
\usepackage{mathtools}
\usepackage{pgfplots}
\usepackage{import}
\usepackage{xifthen}
\usepackage{transparent}
\usepackage{xargs} 
\usepackage{mathabx}
\usepackage{array,tabularx}
\newcolumntype{C}[1]{>{\centering\arraybackslash\hspace{0pt}}p{#1}}
\newcolumntype{Y}{>{\centering\arraybackslash}X}

\usepackage[colorinlistoftodos,prependcaption,textsize=tiny]{todonotes}
\newcommandx{\unsure}[2][1=]{\todo[linecolor=red,backgroundcolor=red!25,bordercolor=red,#1]{#2}}
\newcommandx{\change}[2][1=]{\todo[linecolor=blue,backgroundcolor=blue!25,bordercolor=blue,#1]{#2}}
\newcommandx{\info}[2][1=]{\todo[linecolor=OliveGreen,backgroundcolor=OliveGreen!25,bordercolor=OliveGreen,#1]{#2}}
\newcommandx{\improvement}[2][1=]{\todo[linecolor=Plum,backgroundcolor=Plum!25,bordercolor=Plum,#1]{#2}}

\makeatletter
\DeclareFontFamily{U}{tipa}{}
\DeclareFontShape{U}{tipa}{m}{n}{<->tipa10}{}
\newcommand{\arc@char}{{\usefont{U}{tipa}{m}{n}\symbol{62}}}%

\newcommand{\arc}[1]{\mathpalette\arc@arc{#1}}

\newcommand{\arc@arc}[2]{%
  \sbox0{$\m@th#1#2$}%
  \vbox{
    \hbox{\resizebox{\wd0}{\height}{\arc@char}}
    \nointerlineskip
    \box0
  }%
}
\makeatother

\definecolor{mainred}{RGB}{185,40,40}
\definecolor{extgreen}{RGB}{10,120,55}
\definecolor{intblue}{RGB}{20,75,160}
\definecolor{cardfill}{RGB}{255,250,246}
\definecolor{lightgray}{RGB}{248,248,248}
\newcommand{\ext}[1]{\textcolor{black}{#1}}
\newcommand{\inn}[1]{\textcolor{black}{#1}}
\newcommand{\mainthm}[1]{\textcolor{mainred}{\bfseries #1}}

\usetikzlibrary{arrows.meta,positioning,calc}
\definecolor{myblue}{RGB}{22,52,120}
\definecolor{lightblue}{RGB}{245,248,255}
\definecolor{mygreen}{RGB}{31,94,54}
\definecolor{lightgreen}{RGB}{242,248,243}
\definecolor{mypurple}{RGB}{103,66,148}
\definecolor{lightpurple}{RGB}{248,244,252}
\definecolor{mygold}{RGB}{184,127,20}
\definecolor{lightgold}{RGB}{255,249,236}
\definecolor{myred}{RGB}{150,45,52}
\definecolor{lightred}{RGB}{253,244,245}
\definecolor{myteal}{RGB}{17,180,190}
\definecolor{lightteal}{RGB}{214,239,239}

\numberwithin{equation}{section}
\theoremstyle{plain}
\newtheorem{theorem}{Theorem}[section]

\newtheorem{prop}[theorem]{Proposition}
\newtheorem{lem}[theorem]{Lemma}
\newtheorem{cor}[theorem]{Corollary}

\newtheorem*{question*}{Question}

\theoremstyle{definition}
\newtheorem{defn}[theorem]{Definition}
\newtheorem{remark}[theorem]{Remark}

\theoremstyle{plain}
\newtheorem{thmx}{Theorem}

\newcommand{\Hyp}{\mathbb{H}}
\newcommand{\R}{\mathbb{R}}

\newcommand{\C}{\mathbb{C}}

\newcommand{\Proj}{\mathbb{P}}

\newcommand{\Z}{\mathbb{Z}}

\newcommand{\N}{\mathbb{N}}
\newcommand{\D}{\mathbb{D}}

\newcommand{\cA}{\mathcal{A}}
\newcommand{\cB}{\mathcal{B}}
\newcommand{\cC}{\mathcal{C}}
\newcommand{\cE}{\mathcal{E}}
\newcommand{\cF}{\mathcal{F}}
\newcommand{\cG}{\mathcal{G}}
\newcommand{\cH}{\mathcal{H}}

\newcommand{\cP}{\mathcal{P}}

\newcommand{\cK}{\mathcal{K}}
\newcommand{\cL}{\mathcal{L}}
\newcommand{\cD}{\mathcal{D}}
\newcommand{\cT}{\mathcal{T}}

\newcommand{\cW}{\mathcal{W}}

\newcommand{\id}{\mathrm{id}}
\newcommand{\qs}{\mathrm{qs}}
\newcommand{\qc}{\mathrm{qc}}
\newcommand{\homeo}{\mathrm{h}}
\newcommand{\ah}{\mathrm{ah}}
\newcommand{\gf}{\mathrm{gf}}
\newcommand{\cc}{\mathrm{cc}}
\newcommand{\twoended}{two-ended}
\newcommand{\mhf}{MHF}
\newcommand{\rigid}{rigid}
\newcommand{\ranktwo}{rank-two}
\newcommand{\disk}{disk}
\newcommand{\ball}{ball}
\newcommand{\nonel}{non-elementary}
\newcommand{\para}{parabolic}

\renewcommand{\epsilon}{\varepsilon}

\DeclareMathOperator{\dist}{dist}

\DeclareMathOperator{\dep}{dep}
\DeclareMathOperator{\conf}{conf}

\DeclareMathOperator{\QS}{QS}

\DeclareMathOperator{\PSL}{PSL}

\DeclareMathOperator{\Isom}{Isom}
\DeclareMathOperator{\Conf}{Conf}

\DeclareMathOperator{\Int}{Int}
\DeclareMathOperator{\chull}{Cvx\, Hull}
\DeclareMathOperator{\diam}{diam}

\DeclareMathOperator{\val}{val}

\DeclareMathOperator{\Teich}{Teich}
\DeclareMathOperator{\Homeo}{Homeo}
\DeclareMathOperator{\HQ}{HQ}

\DeclareMathOperator{\fil}{fill}

\DeclareMathOperator{\core}{core}

\DeclareMathOperator{\height}{ht}

\def\emty{\emptyset}
\def\cbar{\widehat{\C}}
\def\stab{\mbox{\rm Stab}}

\def\ep{\varepsilon}
\def\g{\gamma}
\def\la{\lambda}
\def\La{\Lambda}

\renewcommand{\SS}{\mathbb S}
\renewcommand{\tilde}{\widetilde}

\renewcommand{\hat}{\widehat}

\newcommand\ben{\begin{enumerate}}
\newcommand\een{\end{enumerate}}
\newcommand\bit{\begin{itemize}}
\newcommand\eit{\end{itemize}}

\makeatletter
\newsavebox{\@brx}
\newcommand{\llangle}[1][]{\savebox{\@brx}{\(\m@th{#1\langle}\)}%
  \mathopen{\copy\@brx\kern-0.5\wd\@brx\usebox{\@brx}}}
\newcommand{\rrangle}[1][]{\savebox{\@brx}{\(\m@th{#1\rangle}\)}%
  \mathclose{\copy\@brx\kern-0.5\wd\@brx\usebox{\@brx}}}
\makeatother

\makeatletter
\@ifpackageloaded{lmodern}{%
  \UndeclareTextCommand{\k}{T1}%
  \DeclareTextCommand{\k}{T1}[1]%
   {\hmode@bgroup\ooalign{\null#1\crcr\hidewidth\char12\kern-.2ex}\egroup}%
  \DeclareTextComposite{\k}{T1}{a}{161}%
  \DeclareTextCompositeCommand{\k}{T1}{o}{\textogonekcentered{o}}%
  \DeclareTextCompositeCommand{\k}{T1}{O}{\textogonekcentered{O}}%
}{}
\makeatother

\pgfplotsset{compat=1.18} 

\numberwithin{figure}{section}

\title[QS universality, QI classification and Topological rigidity]{Quasisymmetric universality, quasi-isometric classification and topological rigidity of Kleinian groups}

\begin{author}[P.~Ha\"issinsky]{Peter Ha\"issinsky}
\address{Université d'Aix-Marseille, CNRS, Institut de Mathématiques de Marseille (I2M) - UMR 7373}
\email{phaissin@math.cnrs.fr}
\thanks{P.H. was partially supported by ANR-22-CE40-0004 GoFR}
\end{author}

\begin{author}[Y.~Luo]{Yusheng Luo}
\address{Department of Mathematics, Cornell University, 212 Garden Ave, Ithaca, NY 14853, USA}
\email{yusheng.s.luo@gmail.com, yl3769@cornell.edu}
\thanks{Y.L. was partially supported by NSF Grant DMS-2349929}
\end{author}

\begin{document}
\maketitle
\vspace*{-0.5cm}{\centerline{\small\it This paper is dedicated to the memory of Misha Kapovich}}

\begin{abstract}
We study the quasisymmetric classification of limit sets of Kleinian groups and obtain a characterization of those geometrically finite limit sets that are quasisymmetrically universal.
This allows us to obtain a quasi-isometric classification of finitely generated Kleinian groups. 
We also obtain a classification of virtually topologically rigid geometrically finite Kleinian groups.
\end{abstract}

\setcounter{tocdepth}{1}
\tableofcontents

\section{Introduction}
Our primary concern in this work is the classification of homeomorphic compact metric spaces up to quasisymmetry. A homeomorphism $h: X \rightarrow Y$ between two metric spaces is {\em quasisymmetric} if there exists a homeomorphism $\eta:[0, \infty) \rightarrow [0, \infty)$ so that for any three distinct points $x, y, z \in X$, we have
$$
\frac{d_Y(h(x), h(y))}{d_Y(h(x), h(z))} \leq \eta \left(\frac{d_X(x,y)}{d_X(x,z)}\right).
$$
Limit sets of geometrically finite Kleinian groups form a particularly interesting class, as they share many qualitative geometric features, and exhibit enough flexibility to support a rich and subtle classification theory. 
They also serve as the initial step for understanding the boundaries of relatively hyperbolic groups, and fit into a broader framework that includes Julia sets of rational maps and other limit sets in conformal dynamics.
We give a complete answer to the quasisymmetric classification problem within this class, and find a sharp set of quasi-isometry invariants of finitely generated Kleinian groups in terms of their boundaries.

The guiding question we consider is:
\textit{Given a topological type of limit set, which quasisymmetric equivalence classes can arise?}
    
Two sharply different phenomena emerge, see Figure~\ref{fig:sum}.
\begin{itemize}
    \item In the {\em rigid regime}, such as {\em Schottky sets}, i.e., compact subsets of $\widehat{\mathbb C}$ whose complements are unions of at least three pairwise disjoint open round disks, quasisymmetric equivalence is governed by a strong algebraic rigidity: two geometrically finite Schottky limit sets are quasisymmetrically equivalent if and only if the corresponding groups are commensurable (up to conjugacy within $\PSL_2(\C)$) \cite{Fri06, BKM09};
    \item In the {\em universality regime}, such as topological circles, any two geometrically finite limit sets are quasicircles and hence are quasisymmetrically equivalent \cite{AB60, Ber60}.
\end{itemize}

Our main theorems (Theorem~\ref{thm:quasiconformalUniv} and Theorem~\ref{thm:quasiconformalUnivGF}) give a complete classification of limit sets in the universality regime. We show that the only obstructions to promoting topological equivalence of limit sets to quasisymmetric equivalence are carpet subsets and the coarse order structures at rank-two cut points; see Theorem~\ref{thm:qsunivgeneral} and Theorem~\ref{thm:qsunivgeneralGF}.

This also provides us with a quasi-isometric classification of finitely generated Kleinian groups with non-spherical boundary (Theorem~\ref{thm:qiclassification}): the quasi-isometry class is determined by the topology of the boundary together with the conformal gauges of carpet subsets of its boundary and coarse order structures at rank-two cut points.  In particular, if the boundary of a minimally parabolic representative is carpet-free and without rank-two cut points, then the quasi-isometry class is determined solely by the topology of the boundary.

A key ingredient in establishing this quasisymmetric universality phenomenon is the classification of the virtually topologically rigid geometrically finite Kleinian groups (Theorem~\ref{thm:topoRigid}), i.e., those groups that have finite index in the homeomorphism group of their limit set \cite{KK00}. 
We show this rigidity occurs exactly when the limit set is homeomorphic to a carpet-free Schottky set. 
It follows that every homeomorphism between two such limit sets is quasisymmetric, and every homeomorphism between two geometrically finite carpet-free Schottky limit sets is M\"obius; see Figure~\ref{fig:sch}.

We now turn to a detailed statement of results.

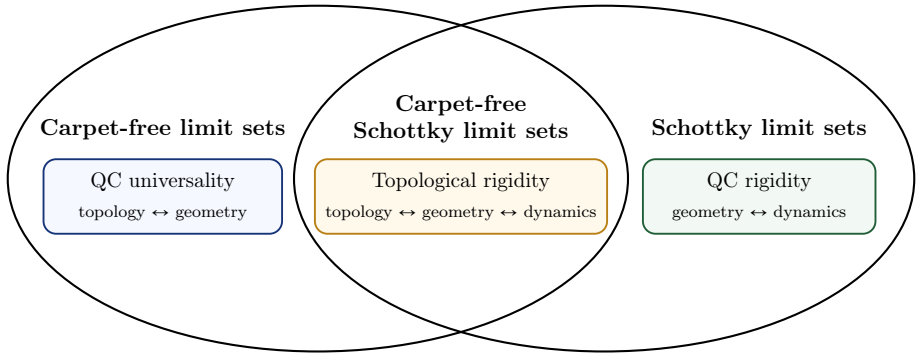
\begin{figure}[!ht]
\captionsetup{width=0.96\linewidth}
  \centering
  \resizebox{0.95\textwidth}{!}{%
\begin{tikzpicture}[every node/.style={align=center}]

\filldraw[
    fill=lightgreen,
    fill opacity=0,
    draw=black,
    line width=1.0pt
] (2.4,0) ellipse [x radius=5.2, y radius = 2.9];

\filldraw[
    fill=lightblue,
    fill opacity=0,
    draw=black,
    line width=1.0pt
] (-2.4,0) ellipse [x radius=5.2, y radius=2.9];


\node[
    font=\bfseries\large,
    text width=4.8cm
] at (-5,0.82) {Carpet-free limit sets};

\node[
    font=\bfseries\large,
    text width=4.4cm
] at (0,1) {Carpet-free \\ Schottky limit sets};

\node[
    font=\bfseries\large,
    text width=4.6cm
] at (5,0.82) {Schottky limit sets};

\node[
    draw=myblue,
    rounded corners=6pt,
    line width=0.9pt,
    fill=lightblue,
    fill opacity=0.95,
    text opacity=1,
    minimum width=4.0cm,
    minimum height=1.25cm,
    inner sep=5pt,
    font=\normalsize
] at (-5,-0.3)
{QC universality\\[2pt] {\footnotesize topology $\leftrightarrow$ geometry}};

\node[
    draw=mygold,
    rounded corners=6pt,
    line width=0.9pt,
    fill=lightgold,
    fill opacity=0.95,
    text opacity=1,
    minimum width=4.35cm,
    minimum height=1.25cm,
    inner sep=5pt,
    font=\normalsize
] at (0,-0.3)
{Topological rigidity\\[2pt]
{\footnotesize topology $\leftrightarrow$ geometry $\leftrightarrow$ dynamics}};

\node[
    draw=mygreen,
    rounded corners=6pt,
    line width=0.9pt,
    fill=lightgreen,
    fill opacity=0.95,
    text opacity=1,
    minimum width=3.9cm,
    minimum height=1.25cm,
    inner sep=5pt,
    font=\normalsize
] at (5,-0.3)
{QC rigidity\\[2pt] {\footnotesize geometry $\leftrightarrow$ dynamics}};

\end{tikzpicture}%
}
  \caption{A schematic summary of the dichotomy of quasiconformal universality and quasiconformal rigidity, and their connection to topological rigidity.}
  \label{fig:sum}
\end{figure}

\begin{figure}[!ht]
\captionsetup{width=0.96\linewidth}
  \centering
  \includegraphics[width=0.48\textwidth]{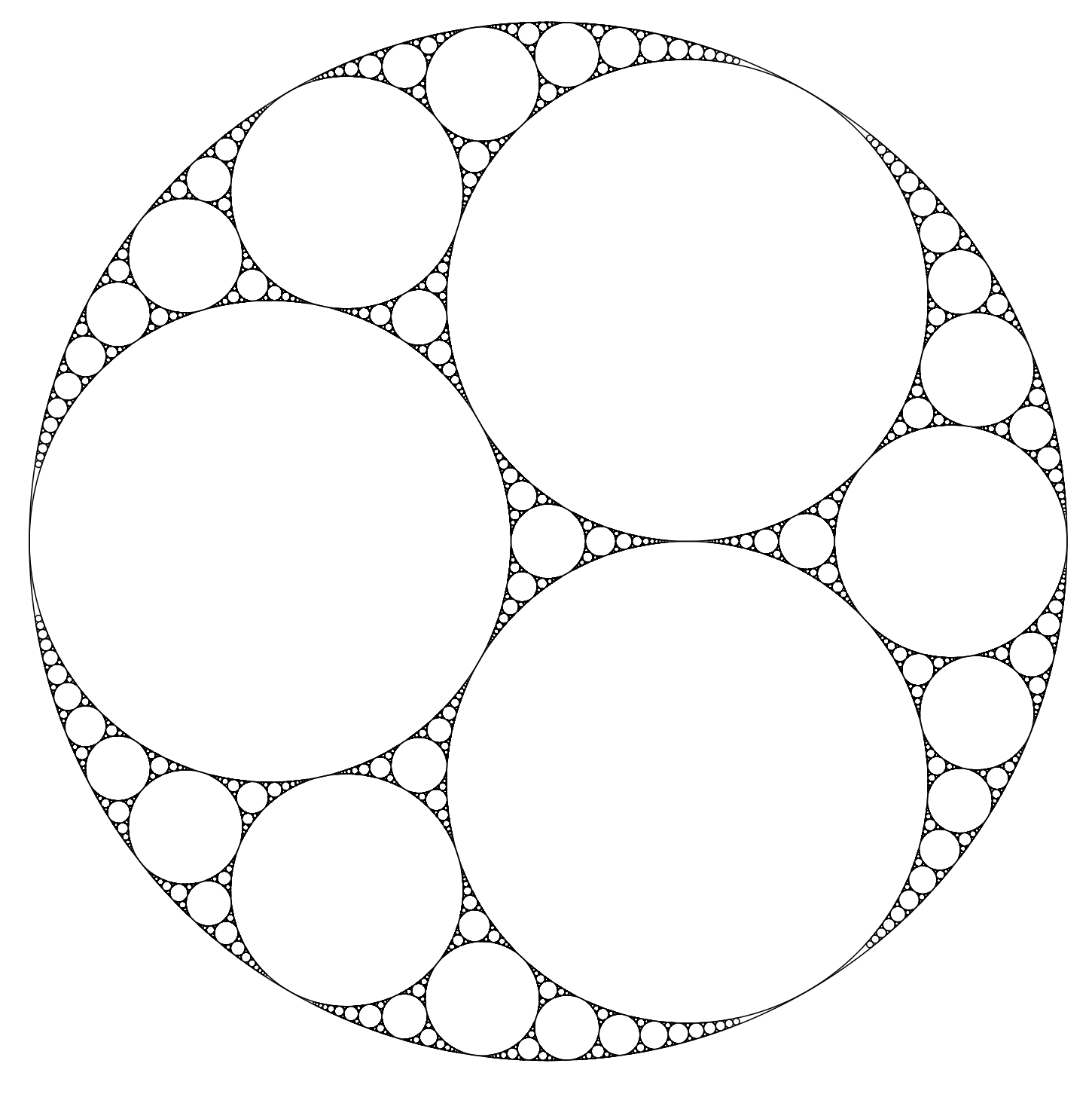}
  \includegraphics[width=0.48\textwidth]{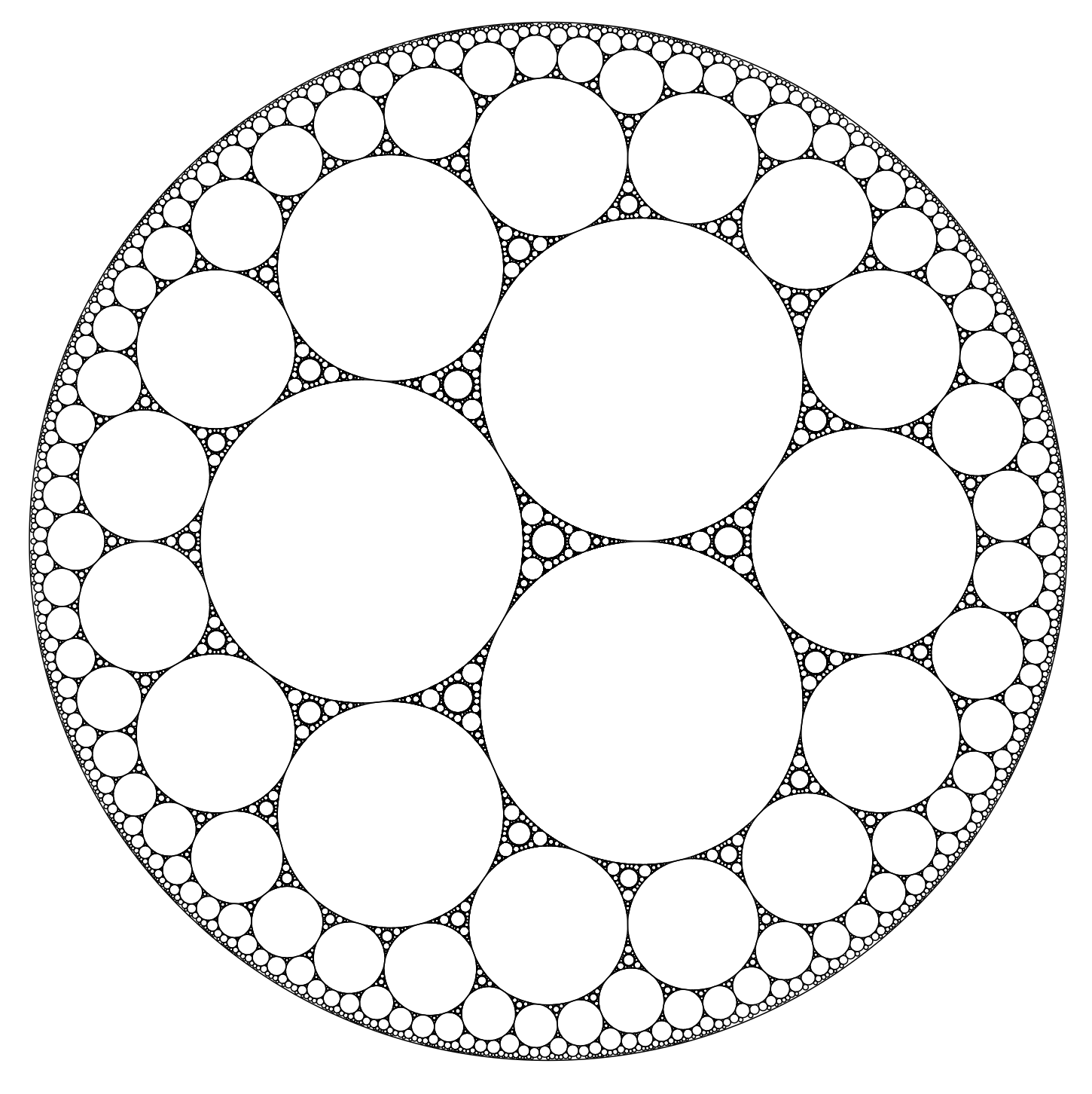}

  \includegraphics[width=0.48\textwidth]{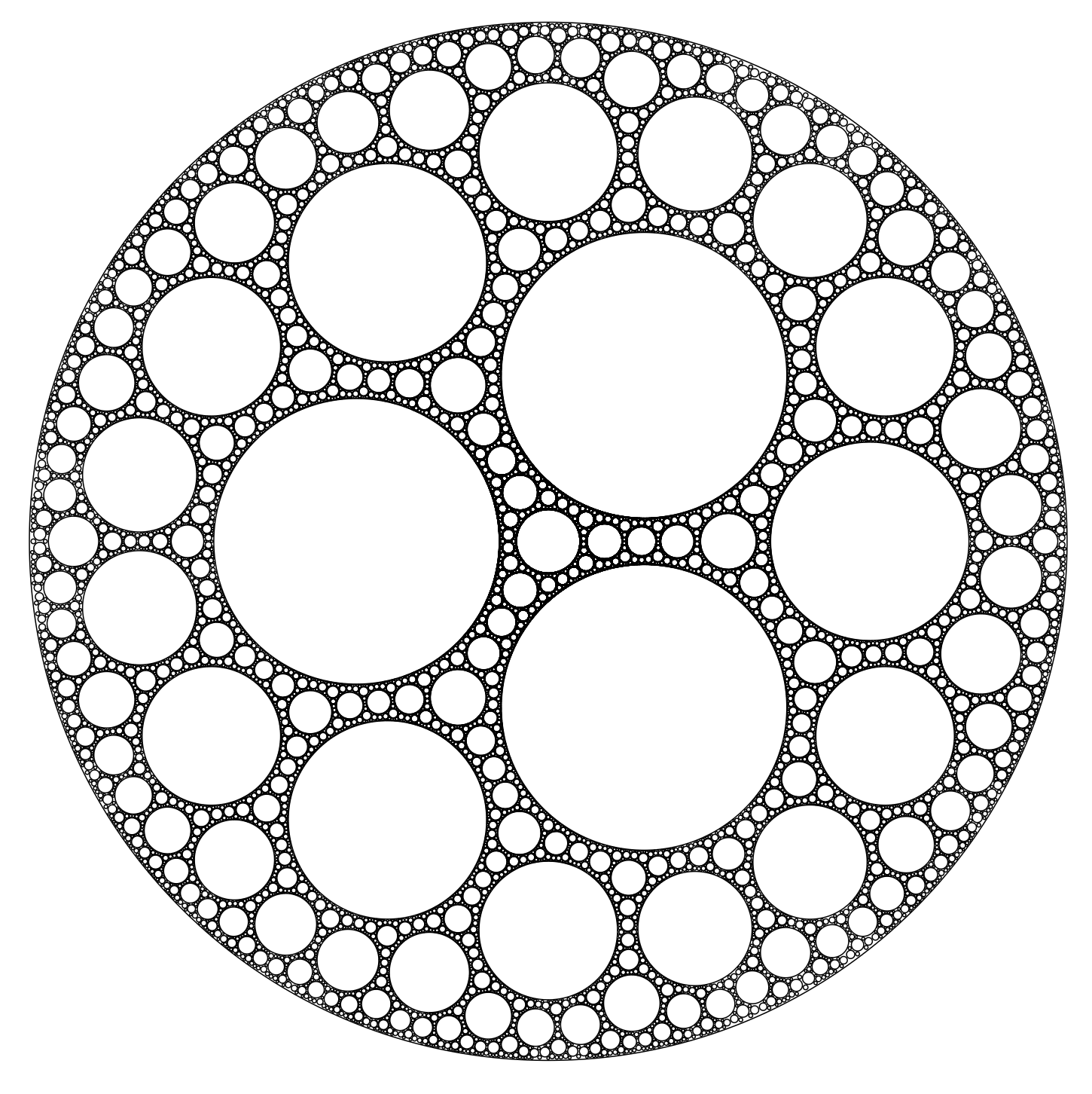}
  \includegraphics[width=0.48\textwidth]{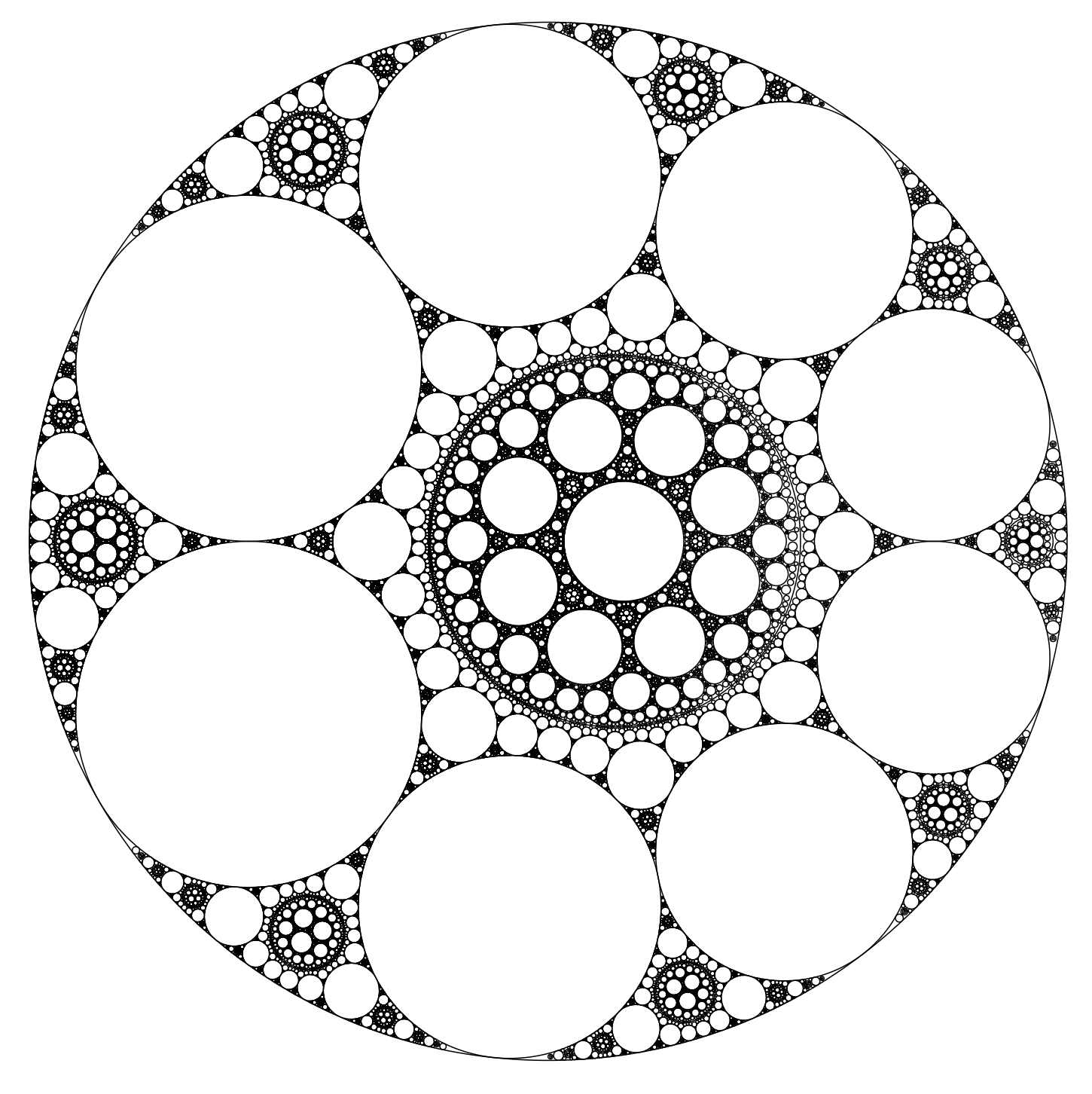}
  \caption{Four examples of Schottky limit sets. The top two are carpet-free, while the bottom two contain carpets. The first example is the Apollonian circle packing and has connected nerve. In contrast, the nerve of the second has infinitely many components. 
  The third is a carpet, whose nerve consists of infinitely many isolated vertices. The last example contains M\"obius copies of both the second and third, and thus it contains carpets but is not itself a carpet. }
  \label{fig:sch}
\end{figure}

\subsection{Quasisymmetric universality}
Let $\mathfrak{S}$ be a collection of closed subsets of $\widehat\C$. 
We denote by
$$
[X]_{\homeo}, \;\; [X]_{\ah}, \;\; [X]_{\qs}, \;\; [X]_{\qc},
$$
the equivalence classes of $X$ in $\mathfrak{S}$ under homeomorphisms, ambient homeomorphisms in $\widehat\C$, quasisymmetric homeomorphisms, and ambient quasiconformal homeomorphisms in $\widehat\C$, respectively.

Inspired by the recent work of McMullen \cite{McM26}, we say a set $X \in \mathfrak{S}$ is 
\begin{itemize}
    \item {\em quasisymmetrically universal} in $\mathfrak{S}$ if $[X]_{\homeo} = [X]_{\qs}$, i.e., any $Y \in \mathfrak{S}$ homeomorphic to $X$ is quasisymmetrically homeomorphic to $X$; and
    \item {\em quasiconformally universal} in $\mathfrak{S}$ if $[X]_{\ah} = [X]_{\qc}$, i.e., any $Y\in \mathfrak{S}$ ambiently homeomorphic to $X$ is quasiconformally homeomorphic to $X$ in $\widehat\C$.
\end{itemize}

A {\em carpet} is by definition a compact set homeomorphic to the standard Sierpi\'nski carpet, which is topologically characterized as a compact, locally connected, 1-dimensional continuum that has no local cut-points \cite{Why58}. 
We say $X$ is {\em carpet-free} if no subset of $X$ is homeomorphic to a carpet.

\medskip

We first focus on the collection of convex cocompact limit sets that are precisely  all the conformally autosimilar subsets in $\widehat\C$ \cite{BH18}.
\begin{thmx}\label{thm:quasiconformalUniv}
    Consider the collection
    $$
    \mathfrak{X}_{\cc} = \{\Lambda(G): G \text{ is a convex cocompact Kleinian group}\}.
    $$ 
    Then the following are equivalent.
    \begin{enumerate}
        \item\label{eqn:thm:quasiconformalUniv:0} $\Lambda(G)$ is quasisymmetrically universal in $\mathfrak{X}_{\cc}$;
        \item\label{eqn:thm:quasiconformalUniv:1} $\Lambda(G)$ is quasiconformally universal in $\mathfrak{X}_{\cc}$;
        \item\label{eqn:thm:quasiconformalUniv:2} either $\Lambda(G)$ is carpet-free, or $\Lambda(G) = \widehat\C$;
        \item\label{eqn:thm:quasiconformalUniv:3} either $\dim_{\conf}(\Lambda(G)) \leq 1$ or $\dim_{\conf}(\Lambda(G)) = 2$, 
        where $\dim_{\conf}(\La(G))$ stands for the conformal dimension of $\La(G)$.
    \end{enumerate}
   Moreover, if $\Lambda(G)$ is not quasisymmetrically universal, then $[\Lambda(G)]_\homeo$ (or $[\Lambda(G)]_\ah$) contains infinitely many quasisymmetric (or quasiconformal) classes.
\end{thmx}

Background on quasiconformal and quasisymmetric maps includes \cite{Ahl66, LV73, heinonen:analysis}. We also refer to \cite{MT10} regarding conformal dimension.

\subsection*{Relative quasisymmetric universality}
Our methods can also be used to study quasisymmetric and quasiconformal universality for geometrically finite Kleinian groups with appropriate modifications. 
The existence of parabolic elements introduces two important subtleties in our analysis.

First, note that any quasisymmetry between geometrically finite limit sets must preserve parabolic points and their ranks unless they are homeomorphic 
to the unit circle or to the Riemann sphere; see Theorem~\ref{thm:qsparab} and Figure~\ref{fig:ccvsgf}, see also \cite{GSWW26}.  
Denote the parabolic locus 
$$
\mathfrak{P}(G) = (\mathfrak{P}_1(G), \mathfrak{P}_2(G)),
$$
where $\mathfrak{P}_j(G)$ is the collection of rank $j$ parabolic fixed points.
This leads us to consider {\em type-preserving homeomorphisms}
$$
h: (\Lambda(G), \mathfrak{P}(G)) \rightarrow (\Lambda(G'), \mathfrak{P}(G'))
$$ 
where $h$ sends rank $1$ (resp. rank $2$) parabolic fixed points of $G$ onto those of $G'$ (we discuss in \S\ref{subsec:parahomeo} when type-preservation can be relaxed).
Thus, we consider the collection of pairs
$$
\mathfrak{X}_{\gf}= \{(\Lambda(G), \mathfrak{P}(G)): G \text{ is geometrically finite}, \mathfrak{P}(G)= (\mathfrak{P}_1(G), \mathfrak{P}_2(G))\}.
$$
We say $(\Lambda(G), \mathfrak{P}(G)) \in \mathfrak{X}_{\gf}$ is 
\begin{itemize}
    \item {\em relatively quasisymmetrically universal} in $\mathfrak{X}_{\gf}$ if 
    $$
    [(\Lambda(G), \mathfrak{P}(G))]_\homeo = [(\Lambda(G), \mathfrak{P}(G))]_\qs;
    $$
    \item {\em relatively quasiconformally universal} in $\mathfrak{X}_{\gf}$  if 
    $$
    [(\Lambda(G), \mathfrak{P}(G))]_\ah = [(\Lambda(G), \mathfrak{P}(G))]_\qc.
    $$
\end{itemize}

\begin{figure}[ht]
\captionsetup{width=0.96\linewidth}
  \centering  
  \includegraphics[width=0.43\textwidth]{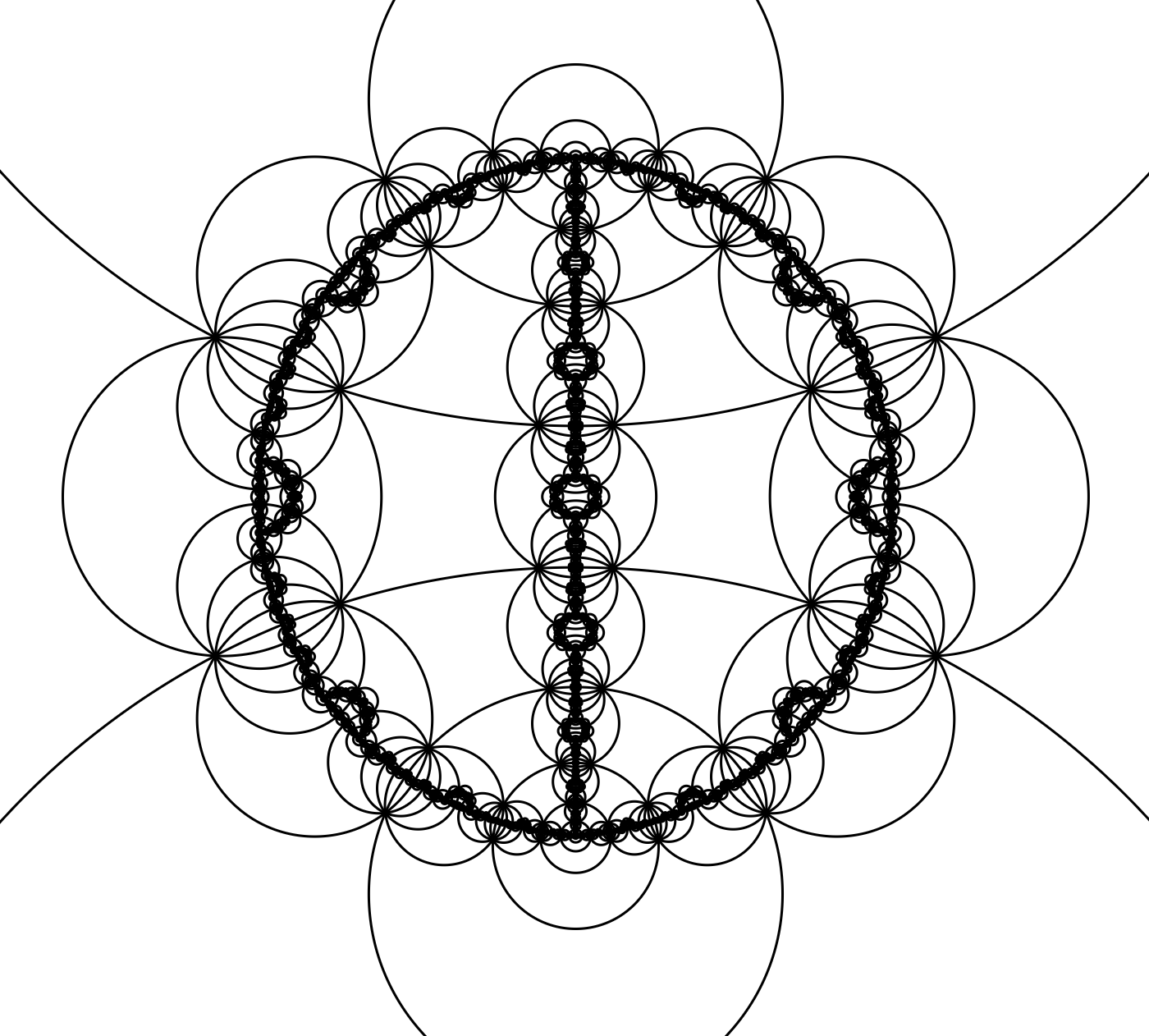}
  \includegraphics[width=0.43\textwidth]{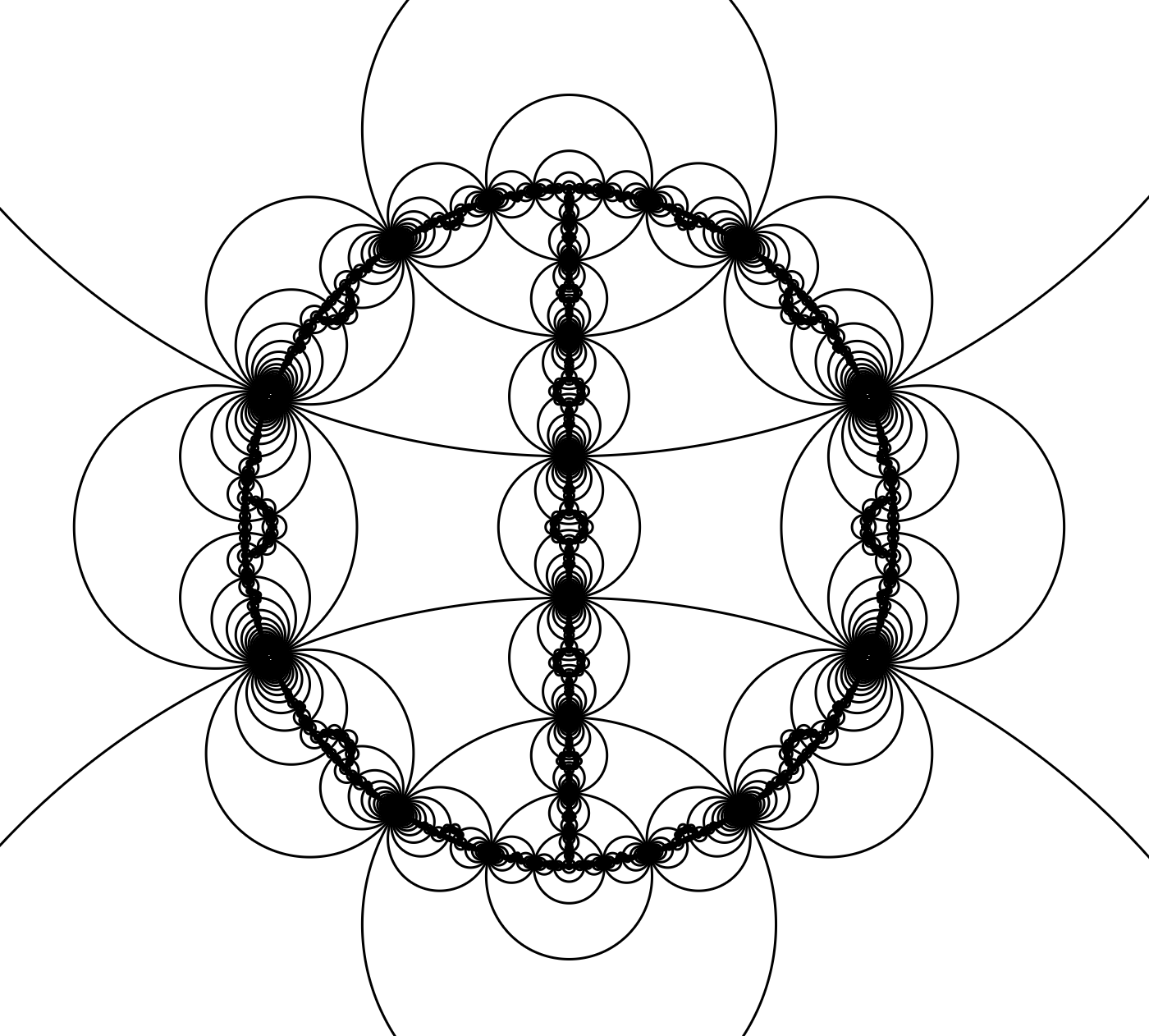}
  \caption{Left: A carpet-free convex cocompact limit set. Right: A geometrically finite limit set, which is homeomorphic but {\bf not} quasiconformally homeomorphic to the left limit set. Both limit sets are relatively quasisymmetrically and quasiconformally universal in $\mathfrak{X}_{\gf}$.}
  \label{fig:ccvsgf}
\end{figure}

The second subtlety created by parabolic elements is the existence of cut points. 
Let $a$ be a rank-two cut point (see Figure~\ref{fig:ranktwoSch}, right). 
Let $\mathcal{C}_a \subseteq \Lambda(G)$ be the component that contains $a$.
The natural linear order on the components of $\mathcal{C}^*_a := \mathcal{C}_a\setminus \{a\}$  given by its embedding in $\cbar\setminus\{a\}$
gives an identification $\pi_0(\mathcal{C}^*_a) \cong \Z$.

We say a homeomorphism is {\em coarsely order-respecting} at a rank-two cut point $a$ if the induced bijection $h_*: \pi_0(\mathcal{C}^*_a) \rightarrow \pi_0(\mathcal{C}^*_{h(a)})$ has the form $h_*(n) = \epsilon n + c + O(1)$, where $\epsilon = \pm 1$. 
Observe that any ambient homeomorphism is coarsely order-respecting, and any quasisymmetry between limit sets is coarsely order-respecting at any rank-two cut point (Proposition~\ref{prop:inducedbilipschitz}).
We say a rank-two cut point $a$ is {\em homogeneous} if all components of $\mathcal{C}^*_a$ are homeomorphic. Note that any homeomorphism at a homogeneous rank-two cut point $a$ can be modified to be coarsely order-respecting at $a$. This no longer holds for non-homogeneous cut points (Theorem~\ref{thm:nonqscutpoint}).

\begin{thmx}\label{thm:quasiconformalUnivGF}
    Let $G$ be geometrically finite with parabolic locus $\mathfrak{P}(G) = (\mathfrak{P}_1(G), \mathfrak{P}_2(G))$.
    \begin{enumerate}
        \item  The following are equivalent.
        \begin{enumerate}[label=(\theenumi\alph*)] 
        \item\label{thm:quasiconformalUnivGF:eq1a} $(\Lambda(G), \mathfrak{P}(G))$ is relatively quasiconformally universal in $\mathfrak{X}_{\gf}$;  
        \item\label{thm:quasiconformalUnivGF:eq1b} either $\Lambda(G)$ is carpet-free, or $\Lambda(G)= \widehat\C$.
        \end{enumerate}
        \item Similarly, the following are equivalent.
        \begin{enumerate}[label=(\theenumi\alph*)] 
            \item\label{thm:quasiconformalUnivGF:eq2a} $(\Lambda(G), \mathfrak{P}(G))$ is relatively quasisymmetrically universal in $\mathfrak{X}_{\gf}$;
            \item\label{thm:quasiconformalUnivGF:eq2b} either $\Lambda(G)$ is carpet-free and all rank-two cut points are homogeneous, or $\Lambda(G)= \widehat\C$.
        \end{enumerate}
    \end{enumerate}

    Moreover, if $(\Lambda(G), \mathfrak{P}(G))$ is not relatively quasisymmetrically (or quasiconformally) universal, then $[(\Lambda(G), \mathfrak{P}(G))]_\homeo$ (or $[(\Lambda(G), \mathfrak{P}(G))]_\ah$) contains infinitely many quasisymmetric (or quasiconformal) classes.

\end{thmx}

We emphasize that the homeomorphisms considered in the theorems are a priori {\bf not} conjugacies: 
by \cite[Theorem 3.3]{tukia:ihes85}, topological conjugacies between geometrically finite limit sets are automatically quasisymmetric. Thus, understanding how far Tukia's theorem can be extended beyond the equivariant setting is another motivation for this work.
The related question concerning the geometry of complementary components of limit sets is addressed in Theorem~\ref{thm:geoOmega}.
\begin{figure}[ht]
\captionsetup{width=0.96\linewidth}
  \centering
  \includegraphics[width=0.45\textwidth]{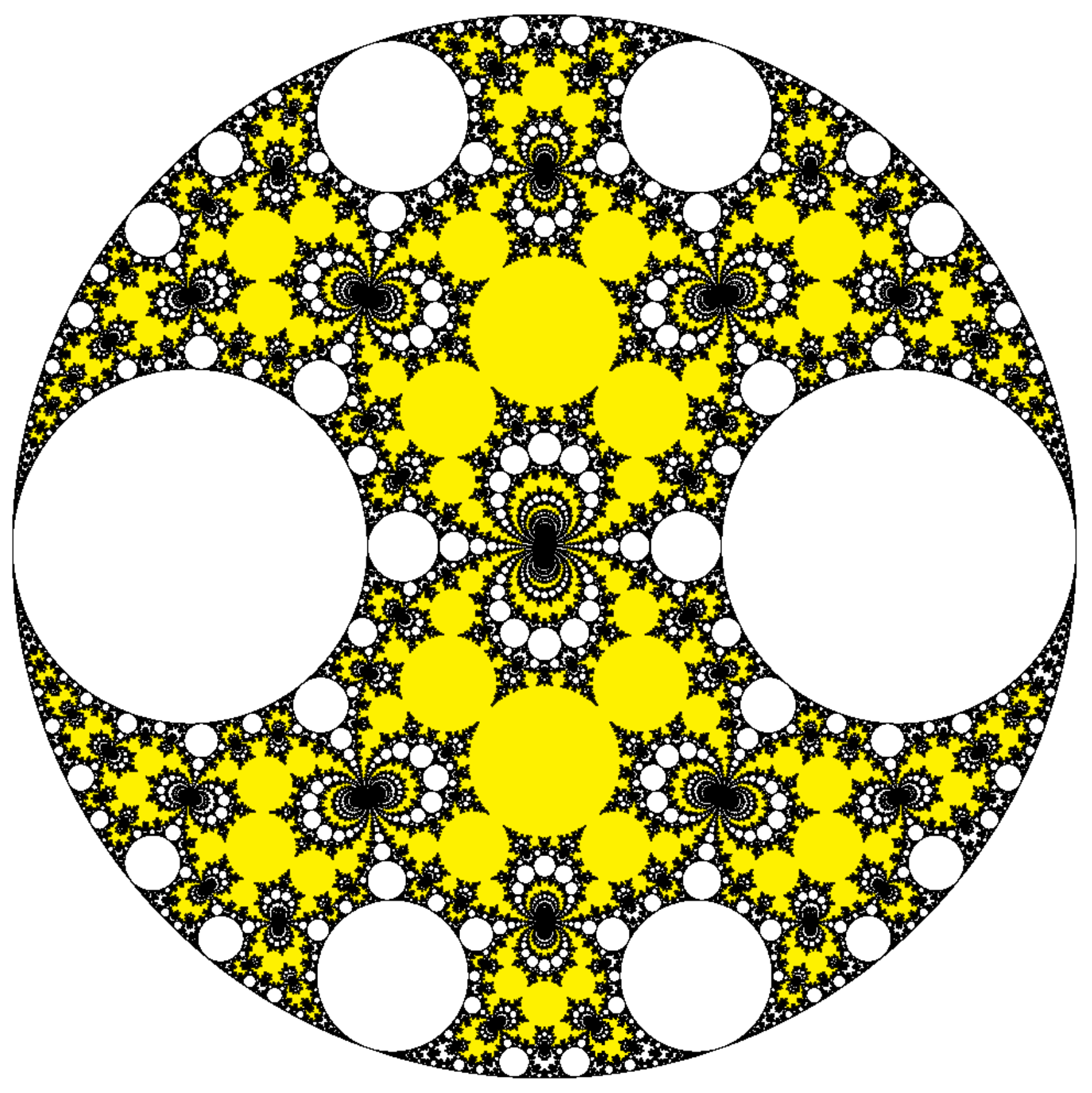}
  \includegraphics[width=0.45\textwidth]{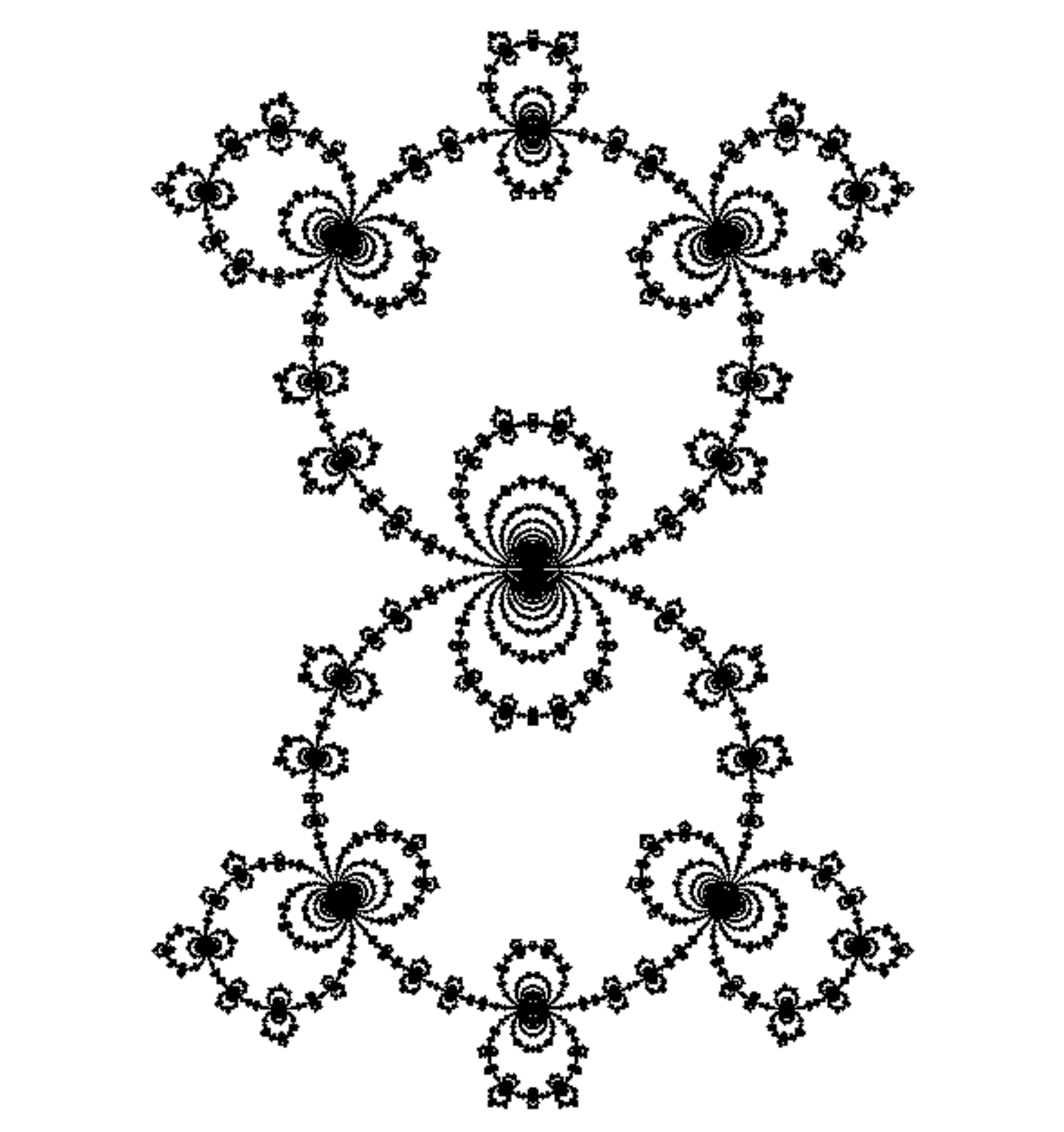}
  \caption{Left: a carpet-free Schottky limit set whose nerve has infinitely many components. Right: the corresponding de-parabolized Kleinian group, whose limit set contains homogeneous rank-two cut points forming double Hawaiian earrings. Thus, the limit set is relatively quasisymmetrically and quasiconformally universal in $\mathfrak{X}_{\gf}$.
  The left limit set is obtained from the right by taking the quotient along an acylindrical lamination.}
  \label{fig:ranktwoSch}
\end{figure}

\subsection*{General limit sets and carpet obstruction}
Theorem~\ref{thm:quasiconformalUniv} is a prototypical example illustrating the general philosophy that embedded carpets are the only obstruction to quasisymmetric universality for convex cocompact limit sets. 
This theme of potential carpet obstructions is enhanced by the Kapovich-Kleiner conjecture, which asserts that any word hyperbolic group with carpet boundary contains a 
finite-index subgroup isomorphic to a Kleinian group \cite{KK00}. See also \cite{ph:unifplanar}, and \cite[Theorem 1.2]{ph:abc} that shows that a word hyperbolic group whose boundary is homeomorphic to a carpet-free limit set of a convex Kleinian group is itself virtually Kleinian. 

We show that embedded carpets are indeed the only obstructions for general limit sets of convex cocompact Kleinian groups.
\begin{theorem}\label{thm:qsunivgeneral}
Let $h:\La(G)\to \La(G')$ be a homeomorphism between limit sets of two convex cocompact Kleinian groups. If the restriction of $h$
to every embedded carpet in $\La(G)$ is quasisymmetric, then $\La(G)$ and $\La(G')$ are quasisymmetric. 
\end{theorem}

The above theorem can be modified appropriately for geometrically finite Kleinian groups.
For a geometrically finite Kleinian group $G$, a subgroup $H\le G$ is called a {\em carpet-quotient subgroup} 
if its minimally parabolic representatives are carpet groups, see \S~\ref{subsec:rankonedeparalam}.
A subset $A \subseteq \Lambda(G)$ is called a {\em dynamical carpet-quotient} if it is the limit set of a carpet-quotient subgroup.
Note that $A$ is homeomorphic to the quotient of the Sierpi\'nski carpet by some lamination on the peripheral disks.

\begin{theorem}\label{thm:qsunivgeneralGF}
Let $h:(\La(G), \mathfrak{P}(G))\to (\La(G'), \mathfrak{P}(G'))$ be a homeomorphism between limit sets of two geometrically finite Kleinian groups. If $h$ is quasisymmetric on every dynamical carpet-quotient and is coarsely order-respecting at each rank-two cut point, then $(\La(G), \mathfrak{P}(G))$ and $(\La(G'), \mathfrak{P}(G'))$ are quasisymmetric. 
\end{theorem}
\begin{remark} 
Note that Theorem~\ref{thm:qsunivgeneral} (or Theorem~\ref{thm:qsunivgeneralGF}) does {\bf not} say nor imply that $h$ is quasisymmetric. It only guarantees the existence of a quasisymmetric map between $\La(G)$ and $\La(G')$ (or $(\La(G), \mathfrak{P}(G))$ and $(\La(G'), \mathfrak{P}(G'))$); cf. Theorem~\ref{thm:qsveGeneral} below. 
\end{remark}

\subsection*{A broader picture of the universality for conformal dynamics}   
Quasisymmetric and quasiconformal universality can be formulated more broadly in conformal dynamics, by enlarging the collection $\mathfrak{S}$ to include Julia sets of rational maps and limit sets of other conformal dynamical systems, such as algebraic correspondences.
A general conjecture in this direction was proposed in \cite[Conjecture 1.14]{LMM26}, together with supporting evidence; see also \cite{McM26}.
Theorem~\ref{thm:quasiconformalUniv} and Theorem~\ref{thm:quasiconformalUnivGF} give a complete affirmative answer to this conjecture for Kleinian groups.

The existence of carpet subsets also provides an obstruction for universality in this broader setting. 
This is illustrated by \cite{BLM16}, where it is shown that no carpet Julia set of a hyperbolic rational map is quasiconformally equivalent to a limit set (see \cite[Corollary 1.3]{BLM16}). Moreover, quasisymmetries between such carpet sets are forced to arise from the underlying dynamics.
See \cite{LN24, LMM26} for further discussions of quasiconformal homeomorphic limit sets and Julia sets.

On the other hand, there are infinitely many non-homeomorphic limit sets in $\mathfrak{X}_{\cc}$ that are homeomorphic to Julia sets of hyperbolic rational maps (cf. \cite{LLM22}). The broader conjecture predicts that all of them are quasiconformally homeomorphic to such Julia sets.

\subsection{Quasi-isometric classification} 
Let $X, Y$ be two metric spaces. A $(K,C)$-quasi-isometry is a map $h: X\rightarrow Y$ such that
$$
\frac{1}{K}d_X(x,y) - C \leq d_Y(h(x),h(y)) \leq Kd_X(x,y) + C,
$$
and for each $y \in Y$, there exists $x \in X$ with $d_Y(y, h(x)) \leq C$.
Our universality results have consequences on the quasi-isometric classification of Kleinian groups.

Let $G$ be a finitely generated torsion-free Kleinian group.
By \cite[Prop.\,2.13]{haiss:lec}, $G$ is isomorphic to a minimally parabolic geometrically finite Kleinian group $G_{min}$ (see also \cite[Theorem 19.6]{Kap09}).
Thus, $G$ is hyperbolic relative to a (unique) finite collection of conjugacy classes of maximal rank-two free abelian subgroups $\Proj$.
One way to construct its Bowditch boundary is to glue horoballs to $\Proj$ and to their cosets following Groves and Manning \cite{groves:manning:dehn} so that it transforms the Cayley graph into a metric hyperbolic space. Its boundary at infinity 
defines the Bowditch boundary of the group pair $(G,\Proj)$ that we may equip with the conformal gauge induced by a visual distance, i.e., the collection of metrics that are quasisymmetric to a visual distance. 
By \cite[Prop.\,3.21]{haiss:lec}, the Bowditch boundary $(\partial_B (G,\Proj), \partial_{\mathrm{par}}(G,\Proj))$ 
equipped with its parabolic locus is quasisymmetric to the limit set $(\Lambda(G_{min}), \mathfrak{P}(G_{min}))$. 
Note that the natural homeomorphisms between limit sets of different minimally parabolic models are coarsely order-respecting, so the coarse order structures at rank two cut points are well-defined.

\begin{thmx}[Quasi-isometric classification for finitely generated Kleinian groups]\label{thm:qiclassification}
Let $G$ be a finitely generated Kleinian group with $\partial_B (G,\Proj) \neq S^2$. Then its quasi-isometric class is determined by
\begin{itemize}
    \item the topology of $(\partial_B (G,\Proj), \partial_{\mathrm{par}}(G,\Proj))$; 
    \item the conformal gauges of every embedded carpet in $\partial_B (G,\Proj)$;
    \item the coarse order structures at rank-two cut points.
\end{itemize}
More precisely, let $G, G'$ be two finitely generated Kleinian groups, with $\partial_B (G,\Proj) \neq S^2$.
    Then the following are equivalent.
    \begin{enumerate}
        \item\label{thm:qiclassification:eq1} $G, G'$ are quasi-isometric;
        \item\label{thm:qiclassification:eq2} there is a homeomorphism 
        $$
        h: (\partial_B (G,\Proj), \partial_{\mathrm{par}}(G,\Proj)) \to (\partial_B (G',\Proj'), \partial_{\mathrm{par}}(G',\Proj'))
        $$ 
    such that $h$ is quasisymmetric on every embedded carpet in $\partial_B (G,\Proj)$ and is coarsely order-respecting at each rank-two cut point.
    \end{enumerate}
\end{thmx}

We remark that the quasi-isometric classification for finitely generated Kleinian group with $\partial_B (G,\Proj) = S^2$ follows from the work of Sullivan, Cannon and Cooper and Schwartz \cite{DS3, cannon:cooper, Sch95}; see \S~\ref{sec:lattice}.
Theorem~\ref{thm:qiclassification} also refines the quasi-isometric classification of \cite{paulin:determined} for convex cocompact Kleinian groups.

We also remark the first condition in Theorem~\ref{thm:qiclassification} can be replaced by the topology of $\partial_B (G,\Proj)$ and the location of parabolic components, cf. the discussion in \S \ref{subsec:parahomeo}.

\begin{cor}[Quasi-isometric classification, convex cocompact]\label{cor:qirigidity}  Let $G$, $G'$  be two convex cocompact Kleinian groups. 
Then $G$ and $G'$ are quasi-isometric if and only if there is a homeomorphism $h:\La(G)\to \La(G')$ such that the restriction of $h$ to every embedded carpet is quasisymmetric.

In particular, suppose the limit sets are carpet-free. Then $G$ and $G'$ are quasi-isometric if and only if their limit sets are homeomorphic.
\end{cor}
The classification for books of I-bundle groups was previously treated in \cite{Mal10, CM17} using the techniques introduced in \cite{BN08}.

\subsection*{Relative quasi-isometric classification}
We also obtain a refined relative quasi-isometric classification for geometrically finite Kleinian groups, where we take into account the type-preserving property as follows. 
A {\em quasi-isometry relative to their parabolic subgroups} between two geometrically finite groups $(G,\Proj_G)$ and $(H,\Proj_H)$ is
given by a quasi-isometry $\Phi:G\to H$ such that, for any $g\in G$ and $P\in\Proj_G$, there are $h\in H$ and $Q\in \Proj_H$ such that $\Phi(gP)$ is at bounded Hausdorff distance from $hQ$ and, for any $h\in H$ and $Q\in \Proj_H$, there are $g\in G$ and $P\in\Proj_G$ such that $\Phi(gP)$ is at bounded Hausdorff distance from $hQ$. 

\begin{theorem}[Relative quasi-isometric classification]\label{thm:qirigiditygeneralGF}  Let $G$, $G'$  be two geometrically finite Kleinian groups, with $\La(G), \La(G') \neq \widehat\C$.
Then $G$ and $G'$ are quasi-isometric relative to their parabolic subgroups if and only if there is a homeomorphism $h:(\La(G), \mathfrak{P}(G))\to (\La(G'), \mathfrak{P}(G'))$ such that $h$ is quasisymmetric on every dynamical carpet-quotient and is coarsely order-respecting at each rank-two cut point.

In particular, suppose the limit sets are carpet-free and has no non-homogeneous rank-two cut points. Then $G$ and $G'$ are quasi-isometric relative to their parabolic subgroups if and only if $(\Lambda(G), \mathfrak{P}(G))$ and $(\Lambda(G'), \mathfrak{P}(G'))$ are homeomorphic.
\end{theorem}

\begin{remark}
By \cite{paulin:determined}, the quasi-isometry class of a hyperbolic group is determined by the conformal gauge of its boundary.
This is no longer true for relatively hyperbolic groups, even with the type-preserving condition; see \S~\ref{sec:lattice} for more details.
This phenomenon also implies that Theorem~\ref{thm:qirigiditygeneralGF} is false for $\Lambda(G) = \widehat\C$, so that the relative quasisymmetric equivalence is strictly weaker than 
the relative quasi-isometric equivalence.
\end{remark}

\subsection{Topological rigidity}
We now discuss a key ingredient in the proof of the universality phenomenon, which constitutes another central theme of the paper.

Let $X$ be a closed subset of $\widehat\C$. We denote by $\Homeo(X)$  the group of homeomorphisms of $X$. 
According to \cite[Definition 2]{KK00}, a {\nonel} Kleinian group $G$ is said to be {\em topologically rigid} if every homeomorphism $f: \Lambda(G) \longrightarrow \Lambda(G)$ is induced by an element of $G$, i.e., $\Homeo(\Lambda(G)) = G$.
A {\nonel} Kleinian group $G$ is said to be {\em virtually topologically rigid} if $G$ is a finite index subgroup of $\Homeo(\Lambda(G))$.

Recall that $X$ is a {\em Schottky set} if its complement $\widehat\C \setminus X$ is a union of at least three disjoint round open disks (see \cite{BKM09}); see Figure~\ref{fig:sch}, \ref{fig:ranktwoSch}, \ref{fig:SchottkyMore}.
\begin{thmx}\label{thm:topoRigid}
    Let $G$ be a {\nonel} geometrically finite Kleinian group.
    Then the following are equivalent.
    \begin{enumerate}
        \item\label{eqn:thm:topoRigid:1} $G$ is virtually topologically rigid;
        \item\label{eqn:thm:topoRigid:2} $\Lambda(G)$ is homeomorphic to a carpet-free Schottky set.
    \end{enumerate}
    Moreover, if $G$ is not virtually topologically rigid, then the group $\Homeo^+(\Lambda(G)):= \Homeo^+(\widehat\C, \Lambda(G))/\sim$ is uncountable, where $f\sim g$ if $f|_{\Lambda(G)} = g|_{\Lambda(G)}$. 
\end{thmx}
Topologically rigid hyperbolic groups were first constructed in \cite{KK00}. These examples have boundary of topological dimension $2$ and this is the minimal possible dimension \cite[Corollary\,17]{KK00}.
Although Theorem~\ref{thm:topoRigid} follows formally from our universality result together with the rigidity of Schottky sets from \cite{BKM09}, we prove Theorem~\ref{thm:topoRigid} first and use it as an ingredient in the proof of the universality results.

\begin{figure}[ht]
\captionsetup{width=0.96\linewidth}
  \centering
  \includegraphics[width=0.45\textwidth]{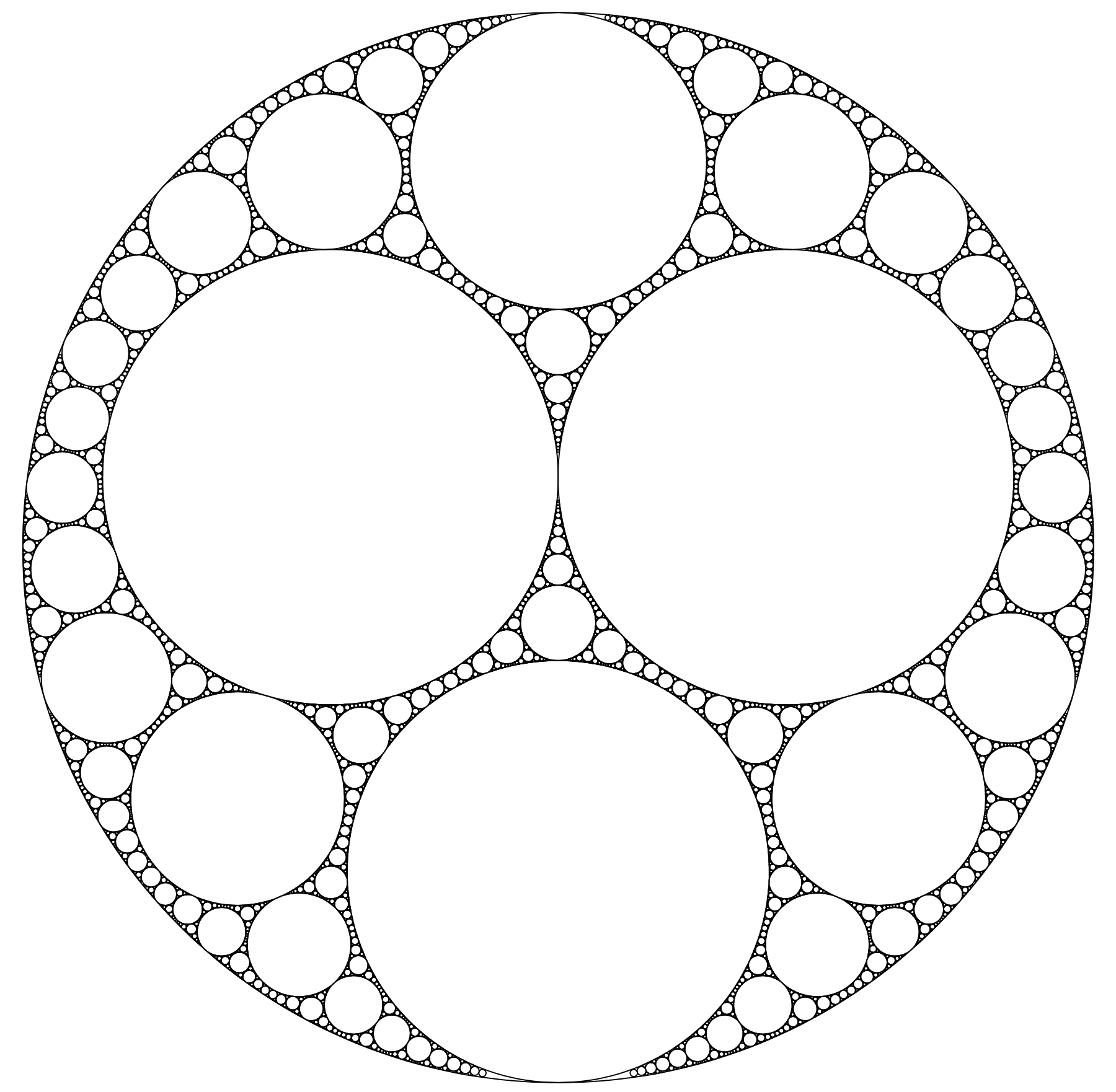}
  \includegraphics[width=0.45\textwidth]{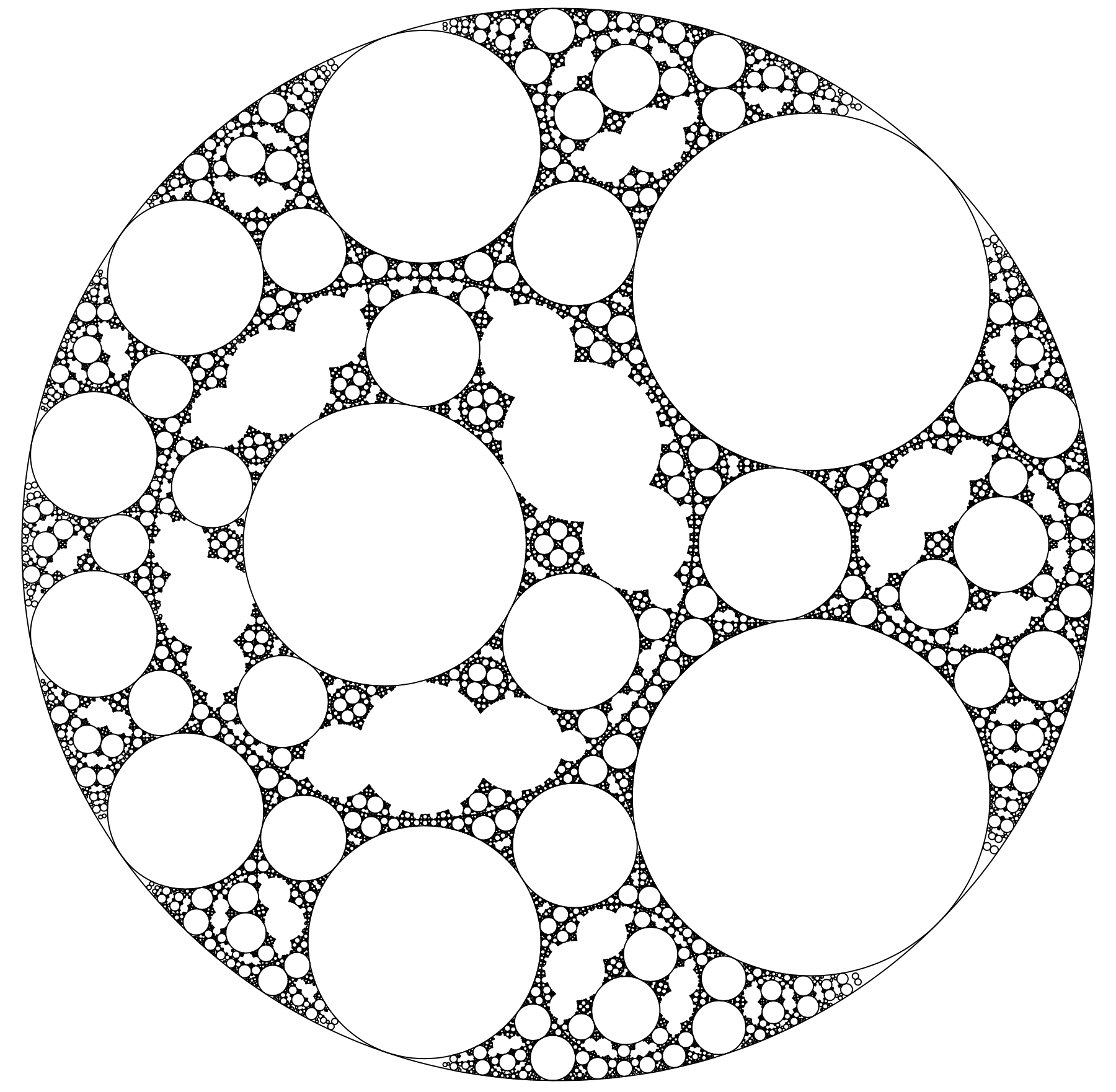}
  \caption{Two additional examples of limit sets homeomorphic to carpet-free Schottky sets. The nerve of the left example has exactly two components: one contains the outermost circle, and the other contains the two prominent horizontal circles. The nerve of the right example has infinitely many components. Note some peripheral disks in the right example are not round.}
  \label{fig:SchottkyMore}
\end{figure}

The following corollary follows from Theorem~\ref{thm:topoRigid} and Corollary~\ref{cor:toprig:finitext},  which completes the classification of topological rigidity for geometrically finite Kleinian groups, cf. \cite[Theorem 3.8]{LLMM23} and \cite[Corollary H]{LZ23}.
\begin{cor}\label{cor:mobius}
    Let $\Lambda$ be a geometrically finite carpet-free Schottky limit set. Then $$\Homeo(\Lambda) = \Conf(\Lambda).$$ 
\end{cor}

A homeomorphism between limit sets is {\em virtually equivariant} if it is a conjugacy between two finite index subgroups. 
In that case, it is also type-preserving, so it is automatically quasisymmetric according to \cite[Theorem 3.3]{tukia:ihes85}. 
If $G$ is a geometrically finite Kleinian group whose limit set is homeomorphic to a Schottky set, then $G$ is quasiconformally conjugate to some Kleinian group $G'$ with Schottky limit set (see \cite{McM90} and Proposition~\ref{prop:acylindricalSchottky}). Thus, we have the following immediate corollary.
\begin{cor}\label{cor:qsve}
    Let $\Lambda(G), \Lambda(G')$ be two geometrically finite limit sets that are homeomorphic to carpet-free Schottky sets.
    If $h: \Lambda(G) \rightarrow \Lambda(G')$ is a homeomorphism, then $h$ is quasisymmetric, virtually equivariant and extends to a quasiconformal map on $\widehat\C$.

    Suppose in addition $\Lambda(G), \Lambda(G')$ are actual carpet-free Schottky sets. Then any homeomorphism between them is the restriction of a M\"obius map (up to a reflection). Moreover, $G$ and $G'$ are commensurable (up to conjugacy within $\PSL_2(\C)$) if and only if their limit sets are homeomorphic.
\end{cor}

\subsection*{General Schottky limit sets}
Similar to Theorem~\ref{thm:qsunivgeneral} and Theorem~\ref{thm:qsunivgeneralGF}, we show the carpets are the only obstruction for quasisymmetries.
\begin{theorem}[Homeomorphisms between Schottky limit sets]\label{thm:qsveGeneral}
    Let $\Lambda, \Lambda'$ be two geometrically finite limit sets that are homeomorphic to Schottky sets and let $h: \Lambda \rightarrow \Lambda'$ be a homeomorphism. Then the following are equivalent.
    \begin{enumerate}
        \item\label{thm:qsveGeneral:eq1} $h$ is quasisymmetric;
        \item\label{thm:qsveGeneral:eq2} $h$ extends to a quasiconformal map on $\widehat\C$;
        \item\label{thm:qsveGeneral:eq3} $h$ is virtually equivariant;
        \item\label{thm:qsveGeneral:eq4} $h$ restricts to a quasisymmetry on each dynamical carpet-quotient.
    \end{enumerate} 
\end{theorem}
The notable difference between Theorem~\ref{thm:qsveGeneral} and Theorems~\ref{thm:qsunivgeneral} and \ref{thm:qsunivgeneralGF} is
that $h$ is quasisymmetric in the Schottky limit set setting.

\subsection{Structure of the paper and outline of the proofs} Since our results concern properties of limit sets, we may replace Kleinian groups by finite-index subgroups. In particular, by Selberg's lemma, we may pass to torsion-free Kleinian groups and invoke $3$-manifold topology.
We also assume throughout that Kleinian groups are non-elementary, since the elementary cases are trivial.

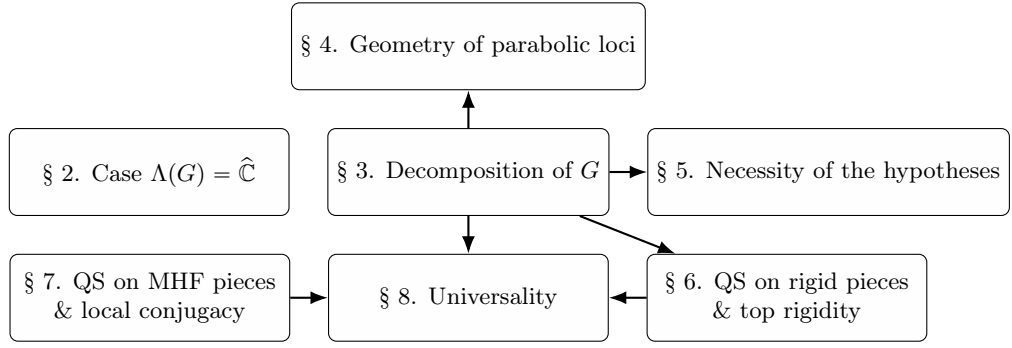
\begin{figure}[ht]
\centering
\begin{tikzpicture}[
    >=Latex,
    node distance=0.5cm and 0.5cm,
    box/.style={
        draw,
        rounded corners=3pt,
        fill=lightgray!10,
        minimum width=3.7cm,
        minimum height=1.15cm,
        align=center,
        font=\small
    },
    arr/.style={->, thick}
]

\node[box] (A) {\S~\ref{sec:lattice}. Case $\Lambda(G)=\widehat{\C}$};
\node[box, right=of A] (B) {\S~\ref{sec:prelim}. Decomposition of $G$};
\node[box, right=of B] (C) {\S~\ref{sec:tfou}. Necessity of the hypotheses};

\node[box, above=0.5cm of B] (G) {\S~\ref{sec:geompara}. Geometry of parabolic loci};
\node[box, below=0.5cm of B] (F) {\S~\ref{sec:uni}. Universality};
\node[box, left=of F] (E) {\S~\ref{sec:qsfuchsian}. QS on {\mhf} pieces\\ \& local conjugacy};
\node[box, right=of F] (D) {\S~\ref{sec:vtr}. QS on rigid pieces\\ \& top rigidity};

\draw[arr] (B) -- (G);
\draw[arr] (B) -- (C);
\draw[arr] (B) -- (F);
\draw[arr] (B) -- (D);
\draw[arr] (D) -- (F);
\draw[arr] (E) -- (F);

\end{tikzpicture}
\caption{Organization of the paper.}
\end{figure}

The case $\Lambda(G) = \widehat\C$ is straightforward and is treated separately in \S~\ref{sec:lattice}. Henceforth, we focus on groups with non-spherical limit sets.
In \S~\ref{sec:geompara}, we analyze the role of parabolic loci in the geometry of the limit set, and we justify our relative modification of the universality.
The proof of the necessity of the hypotheses in our main theorems on universality, quasi-isometric classification and topological rigidity is also comparatively simple, and is discussed in \S~\ref{sec:tfou}.

The main ingredient in proving the sufficiency of the hypotheses in our universality results, Theorems~\ref{thm:quasiconformalUniv} and~\ref{thm:quasiconformalUnivGF}, where topology implies geometry, comes from low-dimensional topology: geometrically finite groups admit splittings over elementary subgroups.
These splittings can be read off from the topology of the limit sets, and they endow the limit set with additional structures that we can exploit. In particular, these splittings are compatible with any homeomorphism, allowing us to impose a richer structure to each piece before patching them back together.

More precisely, the maximal free product decomposition records the decomposition of the limit set into its nontrivial connected components and the remaining totally disconnected part. The JSJ decomposition then further splits the nontrivial connected components along their cut points and cut pairs. In this way, we obtain simpler pieces that come from parabolic subgroups, loxodromic subgroups, maximal hanging Fuchsian groups and {\em rigid} groups (i.e, fundamental groups of pared acylindrical manifolds); see \S~\ref{sec:prelim}.

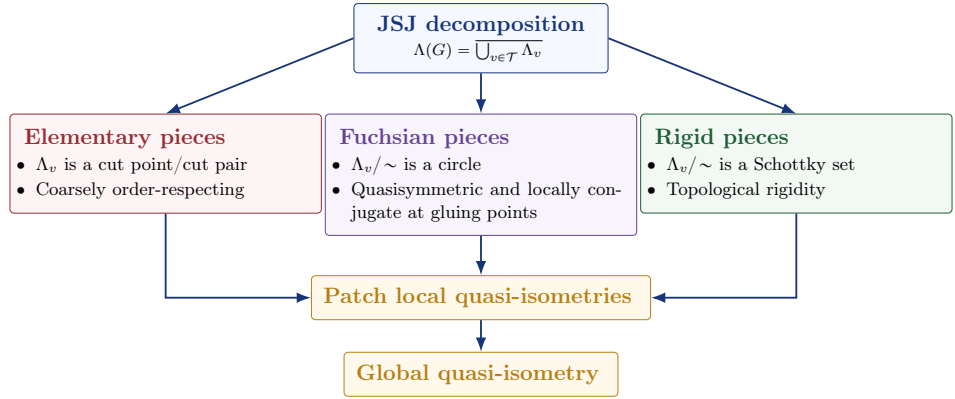
\begin{figure}[ht]
\centering
\begin{center}
\resizebox{\textwidth}{!}{
\begin{tikzpicture}[
  >=Latex,
  box/.style={draw=myblue, fill=lightblue, rounded corners=2pt,
    align=center, inner sep=6pt},
  rigidbox/.style={draw=mygreen, fill=lightgreen, rounded corners=2pt,
    align=left, inner sep=7pt},
  parabox/.style={draw=myred, fill=lightred, rounded corners=2pt,
    align=left, inner sep=7pt},
  fuchsbox/.style={draw=mypurple, fill=lightpurple, rounded corners=2pt,
    align=left, inner sep=7pt},
  goldbox/.style={draw=mygold, fill=lightgold, rounded corners=2pt,
    align=center, inner sep=6pt},
  every node/.style={font=\small}
]

\node[box, minimum width=4.4cm] (jsj) {
  {\color{myblue}\bfseries\large JSJ decomposition}\\[1pt]
  {\footnotesize $\Lambda(G) = \overline{\bigcup_{v\in \mathcal{T}} \Lambda_v}$ }
};

\node[rigidbox, below right=0.7cm and 0.55cm of jsj, minimum width=5.25cm] (rigid) {
  {\color{mygreen}\bfseries\large Rigid pieces}\\[3pt]
  \begin{minipage}{4.9cm}
  \begin{itemize}[leftmargin=2mm]\setlength\itemsep{0.18em}
    \item $\Lambda_v/\!\sim$ is a Schottky set
    \item Topological rigidity
  \end{itemize}
  \end{minipage}
};

\node[fuchsbox, below =0.7cm of jsj, minimum width=5.25cm] (fuchs) {
  {\color{mypurple}\bfseries\large Fuchsian pieces}\\[3pt]
  \begin{minipage}{4.9cm}
  \begin{itemize}[leftmargin=2mm]\setlength\itemsep{0.18em}
    \item $\Lambda_v/\!\sim$ is a circle
    \item Quasisymmetric and locally conjugate at gluing points
  \end{itemize}
  \end{minipage}
};


\node[parabox, below left=0.7cm and 0.55cm of jsj, minimum width=5.25cm] (para) {
  {\color{myred}\bfseries\large Elementary pieces}\\[3pt]
  \begin{minipage}{4.9cm}
  \begin{itemize}[leftmargin=2mm]\setlength\itemsep{0.18em}
    \item $\Lambda_v$ is a cut point/cut pair
    \item Coarsely order-respecting
  \end{itemize}
  \end{minipage}
};

\node[goldbox, below=3.5cm of jsj, minimum width=5.2cm] (patch) {
  {\color{mygold}\bfseries\large Patch local quasi-isometries}
};

\node[goldbox, below=0.55cm of patch, minimum width=4.4cm] (global) {
  {\color{mygold}\bfseries\large Global quasi-isometry}
};

\draw[->, myblue, line width=0.9pt]
  (jsj.east) -- (rigid.north);

\draw[->, myblue, line width=0.9pt]
  (jsj.south) -- (fuchs.north);

\draw[->, myblue, line width=0.9pt]
  (jsj.west) -- (para.north);

\draw[->, myblue, line width=0.9pt]
  (para.south) |- (patch.west);

\draw[->, myblue, line width=0.9pt]
  (rigid.south) |- (patch.east);

\draw[->, myblue, line width=0.9pt]
  (fuchs.south) -- (patch.north);

\draw[->, myblue, line width=0.9pt]
  (patch) -- (global);

\end{tikzpicture}
}
\end{center}
\caption{The road map to quasi-isometries}
\end{figure}

By \cite{PW02}, for the quasisymmetric universality it suffices to consider connected limit sets.
The rigid pieces are handled by Theorem~\ref{thm:topoRigid}, which implies that any homeomorphism on such pieces is already quasisymmetric. The planarity is essential in the proof. A Schottky limit set $\Lambda(G)$ can be obtained as the quotient $\Lambda(\widehat G)/\sim_{\cL}$ of a minimally parabolic geometrically finite Kleinian group $\widehat G$ 
by  some acylindrical lamination $\cL$. Any homeomorphism of $\Lambda(G)$ lifts to an $\cL$-preserving homeomorphism of $\Lambda(\widehat G)$. The key idea is to analyze the interaction of $\cL$ and the lamination coming from splitting the group over elementary subgroups. Planarity imposes a rich structure on the intersection pattern for the two laminations, which we exploit to prove the desired rigidity for $\cL$-preserving homeomorphisms; see \S~\ref{sec:vtr}.

Quasisymmetries between maximal hanging Fuchsian limit sets could be constructed by more standard methods, but a genuine difficulty arises: two quasisymmetric maps that coincide at finitely many points need not patch into a global quasisymmetry. Thus these quasisymmetries must be constructed with additional control. Theorem \ref{thm:colors} addresses this issue, by using the group dynamics to construct a quasisymmetry that is locally equivariant on definite neighborhoods of the gluing points. The key input is to consider the Bowen-Series map associated to the group: we show that it admits a quasisymmetric piecewise-linear model in which the local expansion at each break-point can be prescribed; see \S~\ref{sec:qsfuchsian}.

Quasisymmetric maps on the pieces induce quasi-isometries between the corresponding vertex groups. Our gluing result shows that these quasi-isometries can be assembled into a quasi-isometry between the two groups. This proves the sufficiency for the quasi-isometric classification, Theorem~\ref{thm:qiclassification}, as well as the quasisymmetric universality. The quasiconformal universality is then obtained by showing that, whenever the resulting quasisymmetry extends to a homeomorphism of the sphere, it in fact admits a quasiconformal extension, see \S~\ref{sec:uni}.

For the reader’s convenience, the following diagram records how the main theorems are broken down into the auxiliary results proved in the main text.

\begin{center}
\begin{tikzpicture}[
  font=\small,
  card/.style={draw, rounded corners=3pt, fill=lightgray!10, text width=0.48\textwidth, inner sep=6pt, align=left},
  card2/.style={draw, rounded corners=3pt, fill=lightgray!10, text width=0.96\textwidth, inner sep=6pt, align=left},
  title/.style={font=\bfseries\large, align=center},
  legend/.style={draw, rounded corners=2pt, fill=lightgray, text width=0.9\textwidth, inner sep=4pt, align=center}
]


\node[card] (A) at (0,0) {
\mainthm{Theorem~\ref{thm:quasiconformalUniv}}\quad Universality\\[2pt]
\begin{tabular}{@{}p{2.7cm}p{12.6cm}@{}}
\eqref{eqn:thm:quasiconformalUniv:2} $\iff$ \eqref{eqn:thm:quasiconformalUniv:3} & \inn{Prop.~\ref{prop:confdim}}.\\
\eqref{eqn:thm:quasiconformalUniv:0}, \eqref{eqn:thm:quasiconformalUniv:1} $\implies$ \eqref{eqn:thm:quasiconformalUniv:2} & \inn{Cor.~\ref{cor:nonqs}}.\\
\eqref{eqn:thm:quasiconformalUniv:2}$\implies$ \eqref{eqn:thm:quasiconformalUniv:0} & \inn{Thm.~\ref{thm:lattice}} \& \inn{Cor.~\ref{cor:relativeqsuniversalGF}}. \\
\eqref{eqn:thm:quasiconformalUniv:2}$\implies$  \eqref{eqn:thm:quasiconformalUniv:1} & \inn{Thm.~\ref{thm:lattice}} \& \inn{Cor.~\ref{cor:relativeqcuniversalGF}}.\\
Moreover part & \inn{Thm.~\ref{thm:nonqs}}.
\end{tabular}
};

\node[card, right=1mm of A] (B) {
\mainthm{Theorem~\ref{thm:quasiconformalUnivGF}}\quad Relative universality\\[2pt]
\begin{tabular}{@{}p{2.7cm}p{12.6cm}@{}}
\eqref{thm:quasiconformalUnivGF:eq1a} $\implies$ \eqref{thm:quasiconformalUnivGF:eq1b} & \inn{Cor.~\ref{cor:nonqs}}.\\
\eqref{thm:quasiconformalUnivGF:eq2a} $\implies$ \eqref{thm:quasiconformalUnivGF:eq2b} & \inn{Cor.~\ref{cor:nonqs}} \& \inn{Cor.~\ref{cor:nonqscutpoint}}.\\
\eqref{thm:quasiconformalUnivGF:eq1b} $\implies$ \eqref{thm:quasiconformalUnivGF:eq1a} & \inn{Thm.~\ref{thm:lattice}} \& \inn{Cor.~\ref{cor:relativeqcuniversalGF}}.\\
\eqref{thm:quasiconformalUnivGF:eq2b} $\implies$ \eqref{thm:quasiconformalUnivGF:eq2a} & \inn{Thm.~\ref{thm:lattice}} \& \inn{Cor.~\ref{cor:relativeqsuniversalGF}}.\\
Moreover part & \inn{Thm.~\ref{thm:nonqs}} \& \inn{Thm.~\ref{thm:nonqscutpoint}}.
\end{tabular}
};

\node[card, below=1mm of A] (C) {
\mainthm{Theorem~\ref{thm:qiclassification}}\quad QI classification\\[2pt]
\begin{tabular}{@{}p{2.7cm}p{12.6cm}@{}}
\eqref{thm:qiclassification:eq1} $\implies$ \eqref{thm:qiclassification:eq2} & \inn{Prop.~\ref{prop:inducedbilipschitz}}.\\
\eqref{thm:qiclassification:eq2} $\implies$ \eqref{thm:qiclassification:eq1} & \inn{Thm.~\ref{thm:qsclassG}}.
\end{tabular}
};

\node[card, below=1mm of B] (D) {
\mainthm{Theorem~\ref{thm:topoRigid}}\quad Virtual topological rigidity\\[2pt]
\begin{tabular}{@{}p{2.7cm}p{12.6cm}@{}}
\eqref{eqn:thm:topoRigid:1} $\implies$ \eqref{eqn:thm:topoRigid:2} & \inn{Thm.~\ref{thm:uncountablehom}}.\\
\eqref{eqn:thm:topoRigid:2} $\implies$ \eqref{eqn:thm:topoRigid:1} & \inn{Thm.~\ref{thm:carpetfreeSchottkyimpliesrigid}}.
\end{tabular}
};

\node[card2] (E) at (3.3,-4.55){
\mainthm{Theorems~\ref{thm:qsunivgeneral}, ~\ref{thm:qsunivgeneralGF}, ~\ref{thm:qirigiditygeneralGF} and ~\ref{thm:qsveGeneral}}\quad \\[2pt]
\begin{tabular}{@{}p{3.2cm}p{12.6cm}@{}}
Thms.~\ref{thm:qsunivgeneral} and \ref{thm:qsunivgeneralGF} & Both follow from \inn{Thm.~\ref{thm:qsclassG}}.\\
Thm.~\ref{thm:qirigiditygeneralGF} & ``Only if'' from \ext{\cite[Thm.~3.18]{haiss:lec}}; ``if'' from \inn{Thm.~\ref{thm:qsclassG}}.\\
Thm.~\ref{thm:qsveGeneral} & \eqref{thm:qsveGeneral:eq1} $\iff$ \eqref{thm:qsveGeneral:eq2} $\iff$ \eqref{thm:qsveGeneral:eq3} follows from \ext{\cite{BKM09}}; \\
& \eqref{thm:qsveGeneral:eq1} $\implies$ \eqref{thm:qsveGeneral:eq4} is trivial; \\
& \eqref{thm:qsveGeneral:eq4} $\implies$ \eqref{thm:qsveGeneral:eq1} follows from \inn{Thm.~\ref{thm:carpetfreeSchottkyimpliesrigid}} \& \inn{Prop.~\ref{finitextension}}.
\end{tabular}
};

\end{tikzpicture}
\end{center}

\subsection*{Acknowledgments}
The authors thank Michel Boileau, Curt McMullen and Luisa Paoluzzi for many useful discussions. The limit sets in the paper are produced with C. McMullen's program \textbf{lim}.

\section{On the classification of lattices}\label{sec:lattice}
In this section, we consider the case when the limit set is the full sphere. More generally, we compare the notions of quasi-isometry, (relative) quasisymmetry and quasiconformal equivalence among 
lattices of the group of orientation preserving isometries $\Isom^+(\Hyp^n)$  of the hyperbolic space $\Hyp^n$, $n\ge 2$. 
It enables us to point out some subtleties that arise from the presence of parabolic subgroups. 

Let us remark that if the lattice is uniform,
then it acts cocompactly on $\Hyp^n$ and is word hyperbolic, whereas if it is non-uniform, then it is 
geometrically finite, and hyperbolic relative to virtually free abelian groups $\Z^{n-1}$.

\begin{theorem}\label{thm:lattice}
Let us fix $n\ge 2$. The following statements hold.
\ben
\item\label{thm:lattice:eqn:1} There is a unique quasisymmetry class of limit sets of lattices in $\Isom^+(\Hyp^n)$.
\item\label{thm:lattice:eqn:2} Two uniform lattices of $\Isom^+(\Hyp^n)$ are quasi-isometric.
\item\label{thm:lattice:eqn:3} Two non-uniform lattices $G$ and $G'$ of $\Isom^+(\Hyp^n)$ are relatively quasisymmetric, i.e., there is a 
quasisymmetric homeomorphism $h:(\SS^{n-1},\mathfrak{P}(G))\to (\SS^{n-1},\mathfrak{P}(G'))$.
\een    
\end{theorem}

Let us comment on this statement. 
Point \eqref{thm:lattice:eqn:2} is well-known and follows from the \v{S}varc-Milnor lemma since they act cocompactly on $\Hyp^n$. 
Point \eqref{thm:lattice:eqn:1} says that we cannot distinguish lattices 
from the conformal
class of their limit sets, nor can we distinguish uniform lattices from non-uniform lattices. One way
to avoid this issue is to consider relative quasisymmetric maps, imposing that the parabolic loci coincide. In that setting, uniform and non-uniform lattices are distinguished since the parabolic locus is empty for
the former and non-empty for the latter. 

Nonetheless, the last point \eqref{thm:lattice:eqn:3} shows that two non-uniform lattices are always
relatively quasisymmetric. If $n=2$, then such a group is virtually free, so they are all
quasi-isometric. But if $n\ge 3$, then Schwartz \cite{Sch95} asserts that two non-uniform lattices are quasi-isometric if and only if they are commensurable. This shows that there are infinitely many quasi-isometry classes, and none of them can be distinguished from their (relative) conformal structure at infinity. 
This contrasts with the quasi-isometric classification of Paulin that ensures that two word hyperbolic groups are quasi-isometric if and only if their boundaries share the same quasiconformal structure \cite{paulin:determined}.
Thus, for relatively hyperbolic groups, no characterization as simple can hold, even if their Bowditch boundary can be equipped with a canonical conformal gauge.
Thus, we record the following corollary.

\begin{cor}
    Let $G$ be a non-uniform lattice of $\PSL_2(\C)$. Then there exists a non-uniform lattice $G'$ such that
    \begin{itemize}
        \item $(\Lambda(G'), \mathfrak{P}(G'))$ is quasiconformally equivalent to $(\Lambda(G), \mathfrak{P}(G))$; but
        \item $G'$ is not quasi-isometric to $G$.
    \end{itemize}
\end{cor}
\begin{proof}
    Let $G'$ be any non-uniform lattice that is not commensurable to $G$.
    By \cite{Sch95}, $G, G'$ are not quasi-isometric.
    By Theorem~\ref{thm:lattice}, $(\Lambda(G'), \mathfrak{P}(G'))$ is quasisymmetric and equivalently quasiconformal to $(\Lambda(G), \mathfrak{P}(G))$.
\end{proof}

For the proof of Theorem \ref{thm:lattice}, let us remark that the limit set of a lattice is always $\SS^{n-1}$, so they all coincide, and \eqref{thm:lattice:eqn:1} follows. Since uniform lattices act cocompactly on $\Hyp^n$, the \v{S}varc-Milnor lemma implies they are all 
quasi-isometric to $\Hyp^n$, so they belong to the same quasi-isometry class,  justifying \eqref{thm:lattice:eqn:2}. The last item
is a direct application of Theorem \ref{thm:qcbetweencountabledensesubsets} below.

\begin{theorem}\label{thm:qcbetweencountabledensesubsets}
   Let $X, Y \subseteq \SS^{n}$, $n\ge 1$ be countable and dense subsets. Then there exists a bi-Lipschitz map $H:(\SS^n, X) \rightarrow (\SS^n, Y)$. In particular, $H$ is quasisymmetric.    
\end{theorem}

\begin{proof}[Proof of Theorem~\ref{thm:qcbetweencountabledensesubsets}]
    Let $X = \{x_i: i\in \N\}$ and $Y =\{y_i:i\in \N\}$ be the list of elements.
    Let $X_m = \{x_0,..., x_m\}$ and $Y_m = \{y_0,..., y_m\}$.

    We construct inductively a $\prod_{j=0}^m(1+\frac{1}{(j+1)^2})$-bi-Lipschitz map $H_m: \SS^n \rightarrow \SS^n$ so that 
    \begin{itemize}
        \item $H_m(X_m) \subseteq Y$ and $H_m^{-1}(Y_m) \subseteq X$; 
        \item $H_m|_{X_{m-1}} = H_{m-1}|_{X_{m-1}}$ and $H_m^{-1}|_{Y_{m-1}} = H_{m-1}^{-1}|_{Y_{m-1}}$.
    \end{itemize}
    Since $\prod_{j=0}^\infty(1+\frac{1}{(j+1)^2}) < \infty$, the proposition follows by taking the limit.

    To prove the existence of such a sequence, we take $H_{-1}(z) = z$ for the base case. Suppose $H_{m-1}$ is constructed. If $H_{m-1}(x_m) \in Y$ and $y_m \in H_{m-1}(X)$, then we set $H_m = H_{m-1}$.
    Otherwise, without loss of generality, assume $H_{m-1}(x_m) \notin Y$. We choose $r$ sufficiently small so that 
    $$
    (X_{m-1} \cup H_{m-1}^{-1}(Y_{m-1})) \cap B(x_m,r) = \emptyset.
    $$
    Since $Y$ is dense, we can choose $y_l \in Y \setminus(H_{m-1}(X_{m-1}) \cup Y_{m-1})$ with $H_{m-1}^{-1}(y_l)$ sufficiently close to $x_m$ so that there is a $(1+\frac{1}{(m+1)^2})^{1/2}$-bi-Lipschitz map $\Phi_m$ with 
    $$
    \Phi_m(z) = z \text{ if } z \in \SS^n \setminus B(x_m,r) \text{ and } \Phi_m(x_m) = H_{m-1}^{-1}(y_l).
    $$
    Similarly, if $y_m \in H_{m-1}\circ \Phi_m(X)$, then set $H_m = H_{m-1}\circ \Phi_m$. Otherwise, there exist sufficiently small $s$ and $x_k \in X \setminus(X_m \cup \Phi_m^{-1}(H_{m-1}^{-1}(Y_{m-1})))$ and a $(1+\frac{1}{(m+1)^2})^{1/2}$-bi-Lipschitz map $\Psi_m$ so that 
    $$
    (X_m \cup \Phi_m^{-1}(H_{m-1}^{-1}(Y_{m-1}))) \cap B(x_k,s) = \emptyset,
    $$
    and $\Psi_m(z) = z$ if $z \in \SS^n \setminus B(x_k,s)$ and $\Psi_m(x_k) = \Phi_m^{-1}\circ H_{m-1}^{-1}(y_m)$.
    Set $H_m = H_{m-1} \circ \Phi_m \circ \Psi_m$. Then $H_m$ satisfies all properties. This concludes the construction of the sequence and the theorem follows.
\end{proof}

\section{Pared manifolds, laminations and decompositions}\label{sec:prelim}
In this section, we discuss some background on pared 3-manifold and the connection to geometrically finite Kleinian groups and their limit sets. We also discuss the JSJ and incompressible decomposition of geometrically finite Kleinian groups.

Since the statements of the main results are unchanged after passing to finite-index subgroups, we may assume that the Kleinian groups are torsion-free and work with 3-manifolds rather than orbifolds.

\subsection{Pared 3-manifolds}
Let $N$ be a compact orientable three-manifold with boundary. The manifold $N$ is {\it hyperbolizable} if its interior can be equipped with a complete Riemannian metric of constant negative curvature. We assume throughout this work that all the three-manifolds that will be considered are hyperbolizable.

An embedded surface $(S,\partial S)\to (N,\partial N)$ is {\it incompressible} if the inclusion $i:S\to N$ gives rise to an injective morphism $i_*:\pi_1(S,x)\to\pi_1(N,x)$. The {\it double} of  a manifold $N$ with boundary is the union of $N$ and of a copy of itself glued along its boundary. An embedded surface $S$ in $N$ is {\it boundary incompressible} if its double is incompressible in the double of $N$. A surface $S$ in $N$ is {\it non-peripheral} if it is {\it properly embedded}, i.e., $S\cap\partial N=\partial S$, and if the inclusion $i:S\to N$
is not homotopic relative to $\partial S$ to a map $f:S\to N$ such that $f(S)\subset\partial N$. A surface $S$ is {\it essential}
if it is properly embedded, two-sided, incompressible, boundary incompressible, non-peripheral and does not bound a $3$-ball. An essential disk is also called a {\it compression disk}. 
Note that $\pi_1(N)$ is one-ended if and only if there are no compression disks, i.e., the boundary is incompressible.

\medskip

Given a finite union $S$ of essential surfaces in $N$, we cut $N$ along $S$ and obtain a compact $3$-manifold, denoted by $N\setminus\!\setminus S$. Let $U(S)$ be a small neighborhood of $S$ homeomorphic to $S \times (-1, 1)$. Then each connected component of $N\setminus\!\setminus S$ is homeomorphic to a component of $N\setminus U(S)$.

\medskip

A compact {\it pared manifold} $(N,P)$ is given by a hyperbolizable compact three-manifold $N$ homeomorphic to neither a solid torus nor a thickened torus together with a {\it paring} $P\subset\partial N$ which is a finite
collection of pairwise disjoint incompressible cylinders $P_a$ and tori $P_t$ satisfying:
\begin{enumerate}[-]
\item  every non-cyclic abelian subgroup of $\pi_1(N)$ is conjugate to a subgroup of the fundamental group of a component of $P$ and
\item   any incompressible cylinder $(C,\partial C)\subset (N,P)$ can be homotoped relatively to its boundary into $P$.
\end{enumerate}

If $(N,P)$ is a pared manifold, then we write $\partial_0N =\partial N\setminus P$.

If $N$ is a hyperbolizable three-manifold that is non-homeomorphic to a solid torus nor a thickened torus, 
then  $(N,P_t)$ is a pared manifold, where $P_t$ denotes the union of the incompressible tori 
in $\partial N$.

An {\it acylindrical compact manifold} is a hyperbolizable manifold with no compression disk nor essential cylinders. We also say that $(N,P)$ is {\it acylindrical} 
if up to homotopy, there is no essential disk or cylinder in $N$ disjoint from $P$.

\subsubsection{Geometrically finite Kleinian groups and their associated pared 3-manifolds}
A Kleinian group is a discrete subgroup of $\PSL_2(\C)$. It acts by orientation preserving isometries on $\Hyp^3$ and by M\"obius transformations on $\cbar$.
The latter action defines a dynamical partition $\cbar=\Omega(G)\sqcup \La(G)$ where $\Omega(G)$ is the maximal open subset where the action is properly discontinuous : we call $\Omega(G)$ the ordinary set or the domain of discontinuity. As soon as $G$ is non elementary, i.e., $\Lambda(G)$ has at least three points, then we may equip every component of $\Omega(G)$ with its Poincar\'e metric. 

By Selberg's lemma, a Kleinian group admits a torsion-free finite index subgroup. The {\it Kleinian manifold} of a torsion-free Kleinian group $G$ 
corresponds to $\overline{M}_G = (\Hyp^3\sqcup\Omega(G))/G$.

\medskip

Recall that a torsion free Kleinian group $G$ (and the corresponding hyperbolic 3-manifold $M= \Hyp^3/G$) is called geometrically finite if there exists a finite-sided fundamental polyhedron for the action of $G$ on $\mathbb{H}^3$. Equivalently, the unit neighborhood of the \emph{convex core} of $M$, {defined by}
$$\core(M):=\chull(\Lambda(G))/G,$$
has finite volume, where $\chull(\Lambda(G))$ is the convex hull of the limit set $\Lambda(G)$ of $G$ in $\mathbb{H}^3$ \cite{marden:gf, beardon:maskit}. The connected components of its limit set are locally connected \cite{anderson:maskit}. If $\core(M)$ is compact,
then $G$ is {\it convex cocompact}. This implies that $G$ is word hyperbolic in the sense of Gromov.
We say that a Kleinian group is {\it minimally parabolic} if all its maximal parabolic subgroups are isomorphic to $\Z\times \Z$.

Suppose that $G$ is a geometrically finite Kleinian group. Let $N := \core_\epsilon(M)$ be the unit closed neighborhood of the convex core of $M=\Hyp^3/G$, 
to which we have removed the $\epsilon$-thin cuspidal neighborhoods for all cusps, where $\epsilon>0$ is smaller than the 3-dimensional  Margulis constant. Let $P\subseteq\partial\core_\epsilon(M)$ be the union of boundaries of all cuspidal neighborhoods, then $(N,P)$ is a pared 3-manifold.
We call $(N, P)$ the associated pared $3$-manifold for $M$ and for the Kleinian group $G$.

Note that if $G$ is convex cocompact, then $P$ is empty. If $G$ is minimally parabolic, then $P$ is the union of the torus components of $\partial N$.
 
\medskip

If $N$ is a hyperbolizable compact manifold, then we say that $G$ {\it uniformizes $N$} if $\Hyp^3/G$ is homeomorphic to the interior of $N$. This yields a faithful representation $\rho:\,\pi_1(N)\to G\,\le\PSL_2(\C)$. 
Thurston's  hyperbolisation theorem implies that if $N$ is a hyperbolizable manifold equipped with a paring $P$, then there exists a geometrically finite  
Kleinian group $G$ uniformizing $N$ such that $\overline{M}_G$ is homeomorphic to $N\setminus P$.

\subsubsection{Acylindrical geometrically finite Kleinian group}
We say that the geometrically finite hyperbolic $3$-manifold $M$ (and the corresponding Kleinian group $G$) is acylindrical if the associated pared 3-manifold $(N, P)$ is. Note that the connected components of $\partial_0N$ are in one-to-one correspondence with the boundary components of the Kleinian manifold $\overline{M}_G = (\Hyp^3 \cup \Omega(G))/G$.

Note that the quasiconformal deformation space $\mathcal{QC}(G)$ can be identified with the Teichm\"uller space $\Teich(\partial \overline{M}_G)$ when every boundary component is incompressible. 

We have the following proposition relating the notion of acylindricity and the topology of the limit set, see \cite[Corollary 4.3]{McM90}, \cite[Proposition~8.4]{LZ23}.
\begin{prop}\label{prop:acylindricalSchottky}
    Let $G$ be a geometrically finite Kleinian group. 
    Then $G$ is acylindrical if and only if its limit set $\Lambda(G)$ is homeomorphic to a Schottky set. 
    
    Suppose $G$ is acylindrical. Then there exists a geometrically finite Kleinian group $G'$ quasiconformally conjugate to $G$ so that $\Lambda(G')$ is a Schottky set. In particular, $\Lambda(G)$ is quasiconformally homeomorphic to a Schottky set.
\end{prop}

\subsubsection{Geodesic laminations} Let $G$ be a non-elementary Kleinian group with non-empty domain of discontinuity that we equip with the Poincar\'e metric on each of its components.
A {\it lamination} $\cL$ is a non-empty collection of pairwise disjoint geodesics in $\Omega(G)$ such that 
their union $|\cL|$, {\it the support} of $\cL$, is closed in $\Omega(G)$. The closure $\overline{\cL}$ consists of the collection of the closures in $\cbar$ of each geodesic. 

A typical example is obtained by lifting to $\Omega(G)$ a collection of pairwise disjoint geodesics
in the conformal boundary $\Omega(G)/G$. In this case, one obtains a $G$-invariant lamination.

\subsection{De-parabolizations and laminations}\label{subsec:rankonedeparalam}
Let us recall the following theorem of Tukia \cite[Theorem 2]{tukia:floyd}.

\begin{theorem}[Tukia]\label{thm:tukia}
Let $G$ be a geometrically finite Kleinian group. Then there is another geometrically finite Kleinian group $\widehat{G}$ 
without parabolic elements of rank one, a continuous map $\pi:\cbar\to\cbar$ and an isomorphism  $\Phi: \widehat{G}\to G$ such that
$\Phi(g)\pi = \pi g$ for all $g\in \widehat{G}$ and $\pi^{-1}(\{x\})$ is a point for all $x\in\cbar$ 
except if $x $ is a parabolic fixed point of rank 1 in which case $\pi^{-1}(\{x\})$ is a closed arc $J$. The
interior of $J$ lies in $\Omega(\widehat{G})$,  but its endpoints are the fixed points of a loxodromic $h\in\widehat{G}$  such that $\Phi(h)$ is parabolic with the fixed
point $x$.
Furthermore,  $\pi$ can be
extended to a continuous equivariant map  $F:\overline{\Hyp^3}\to \overline{\Hyp^3}$. 
\end{theorem}

\medskip

We call $\widehat G$ the {\em (rank-one) de-parabolized Kleinian group} associated to $G$, as each rank-one parabolic element of $G$ is turned into a loxodromic element in $\widehat G$. It is minimally parabolic by construction. For simplicity, we will just speak of the {\em de-parabolization} of the group $G$. 

Topologically, if $G$ is a Kleinian group uniformizing a pared $3$-manifold $(N, P)$, where $P = P_t \sqcup P_a$, then the de-parabolized group $\widehat G$ 
of $G$ uniformizes $(N,P_t)$. Note that $\widehat G$ is convex cocompact if $P_t = \emptyset$, i.e., $\partial N$ has no torus component, 
but it contains rank two cusps if $P_t \neq \emptyset$.

\subsubsection{Quotient laminations}
Let $G$ be a torsion free geometrically finite Kleinian group with de-parabolization $\widehat{G}$ that uniformizes a manifold $N$ with parings $P_t \sqcup P_a$ and $P_t$ 
respectively. We assume that $P_a\ne \emty$. Then the annuli of $P_a$ define conjugacy classes of loxodromic elements of $\widehat G$, and a collection of
 pairwise disjoint 
closed hyperbolic geodesics $\mathcal{C}$ on the conformal boundary $\Omega({\widehat G})/\widehat{G}$. The {\it quotient lamination} is the lift $\mathcal{L}$ 
of $\mathcal{C}$ in $\Omega({\widehat G})$.

If $\Delta$ is a component of $\Omega({\widehat G})$, we let  $\mathcal{L}_\Delta$ be the restriction of $\mathcal{L}$ to $\Delta$.
We call a component of $\Delta \setminus |\mathcal{L}|$ a {\em gap} of the lamination $\mathcal{L}$.

We remark that the geodesic lamination $\mathcal{L}$ induces an equivalence relation $\sim_{\mathcal{L}}$ on $\Lambda(\widehat G)$, by collapsing the closure of each leaf of $\mathcal{L}$ to a single point.

Theorem \ref{thm:tukia} implies the following by replacing each arc $J$ by the hyperbolic geodesic in $\Omega({\widehat G})$ in its isotopy class relative to its endpoints.

\begin{prop}
    Let $G$ be a torsion free geometrically finite Kleinian group, and let $\widehat G$ be a de-parabolized Kleinian group associated to $G$, with corresponding geodesic lamination $\mathcal{L}$.
    Then the limit set $\Lambda(G)$ is homeomorphic to $\Lambda(\widehat G) / \sim_{\mathcal{L}}$.
\end{prop}

\subsubsection{Lifting homeomorphisms}
Let $G$ be a Kleinian group with limit set $\Lambda(G)$.
We define the group of homeomorphisms of $\Lambda(G)$ by
\begin{align}\label{eqn:homeogroup}
    \Homeo^+(\Lambda(G)):=\{h:\widehat\C \longrightarrow\widehat \C: &\ h \text{ is an orientation preserving homeomorphism } \nonumber\\ &\text{with } h(\Lambda(G)) = \Lambda(G)\}/\sim,
\end{align}
where $h_1 \sim h_2$ if $h_1|_{\Lambda(G)} = h_2|_{\Lambda(G)}$.
By a slight abuse of notation, we let $h \in \Homeo^+(\Lambda(G))$ denote a homeomorphism, though it strictly represents an equivalence class of homeomorphisms of
the sphere.
One may also regard an element of $\Homeo^+(\Lambda(G))$ as a homeomorphism of $\Lambda(G)$ that admits an orientation-preserving ambient extension. 

Now suppose $G$ is geometrically finite with a corresponding de-parabolized Kleinian group $\widehat G$ and geodesic lamination $\mathcal{L}$.
Consider the group of orientation preserving homeomorphisms of $\Lambda(\widehat{G})$ that preserves $\mathcal{L}$:
\begin{align}
\Homeo^+(\Lambda(\widehat{G}), \mathcal{L}) = \{h:\,\widehat\C \rightarrow \widehat\C: &\ h \text{ is an orientation preserving homeomorphism }\nonumber\\ &\text{with }h(\Lambda(\widehat{G})) = \Lambda(\widehat{G}), h(|\mathcal{L}|) = |\mathcal{L}|\}/\sim
\end{align}
where $h_1 \sim h_2$ if $h_1|_{\Lambda(\widehat{G})} = h_2|_{\Lambda(\widehat{G})}$.

Let $\pi: \Lambda(\widehat{G})  \longrightarrow \Lambda(G)$ be the projection map. Then we have a homomorphism $p: \Homeo^+(\Lambda(\widehat{G}), \mathcal{L}) \longrightarrow \Homeo^+(\Lambda(G))$ so that $p(h) \circ \pi = \pi \circ h$.
\begin{prop}\label{prop:homeomisom}
    Let $G$ be a geometrically finite Kleinian group, and let $\widehat G$ be a de-parabolized Kleinian group associated to $G$, with corresponding geodesic lamination $\mathcal{L}$.
    Then the homomorphism $p: \Homeo^+(\Lambda(\widehat{G}), \mathcal{L}) \longrightarrow \Homeo^+(\Lambda(G))$ is an isomorphism.
\end{prop}
\begin{proof}  It is clear that the homomorphism is injective. A proof of the surjectivity follows from \cite[\S 4, Step 5]{ph:lp:gw:schottky}. Here is a sketch. For each rank one parabolic point for $G$, one considers equivariantly two disjoint open horoballs. Their complement is a $G$-invariant compact subset $Y$; let $X=\pi^{-1}(Y)$ so that
$\pi:X\to Y$ is one-to-one except over the parabolic points where it is $2-1$. An orientation preserving homeomorphism of $\La(G)$ has a representative that preserves $Y$.
We may lift $h$ to a homeomorphism on $X$ (each access of a rank-one parabolic point provides an access to one of the fixed points of the corresponding loxodromic), 
and then extend it continuously to its complement since its components form a null sequence.   \end{proof}

\subsection{JSJ decomposition}\label{subsec:JSJ}
We start by introducing the characteristic cylinder decomposition which defines  the JSJ decomposition by Johannson-Jaco-Shalen \cite{johannson:jsj,jaco:shalen:seifert}
for hyperbolizable pared manifolds, see also \cite[Theorem 3.8]{bonahon:3-var} and \cite[Theorem 4.10]{haiss:lec} for the present form.
We describe how it yields a decomposition of the limit set of a Kleinian group uniformizing the manifold in \S~\ref{subsubsec:decomplimit}. 
In  \S~\ref{subsec:separatinggeodandlam}, we further 
discuss some properties of the corresponding decomposition of the limit set by introducing some canonical laminations on the ordinary set.
Then, we briefly summarize some additional properties of the JSJ tree for convex cocompact Kleinian groups in \S~\ref{subsubsec:ccp}.

\begin{theorem}[Characteristic cylinder decomposition]	\label{thm:charac annulus}
Let $(N,P)$ be a hyperbolizable pared three-manifold with incompressible boundary. 
Then, up to isotopy, there is a unique compact $2$-dimensional submanifold $(A,\partial A)\subset (N,\partial N\setminus P)$ of $N$ such that:
\begin{enumerate}[(i)]
\item Every component of $A$ is an essential cylinder.
\item Each component of $N\setminus\!\setminus A$, is either pared acylindrical or a pared $I$-bundle or a solid torus or a thickened torus.
\item Property (ii) fails when any component of $A$ is removed.
\end{enumerate}
\end{theorem}

Given a component $W$ of $N\setminus\!\setminus A$,
we set $P_W=(P\cap W)\cup (\partial W\setminus (W\cap\partial N))$. When $W$ is not a solid or thickened torus, item (ii) of Theorem \ref{thm:charac annulus} says that  $(W,P_W)$
is a pared manifold.
A {\em pared $I$-bundle} is a pared manifold $(N,P)$ such that $N$ is homeomorphic to a product $F\times I$ over a compact surface $F$ by a homeomorphism that maps $P$ into $\partial F\times I$.

If a cylinder is not on the boundary of a solid torus, it is convenient to add another parallel copy of the cylinder so that they bound a solid torus in $N$.
Let $B$ be the union of $A$ and these parallel cylinders. After such a modification, 
a component that is pared acylindrical, a pared $I$-bundle or a thickened torus is only adjacent to solid tori. We call the resulting surface $B$ the {\it balanced cylinder decomposition of $(N,P)$} that is uniquely defined (up to isotopy) by properties (i), (ii), (ii') and (iii') where
{\it \bit
\item[ (ii')]  Every component of $B$ lies in the closure of a component of $N\setminus B$ whose compactification is a solid torus. 
\item[(iii')]  Property (ii) or (ii') fails when any component of $B$ is removed.
\eit
}

\medskip

The (balanced) JSJ decomposition yields a cofinite action of $\pi_1(N)$ on a bipartite simplicial tree $\cT$, where
\bit
\item the vertices are partitioned into five sets
$$
V(\mathcal{T}) = V_{1,p}(\mathcal{T}) \sqcup V_{1,l}(\mathcal{T}) \sqcup V_2(\mathcal{T}) \sqcup V_3(\mathcal{T}) \sqcup V_4(\mathcal{T}), 
$$
where $V_{1,p}$ corresponds to (the universal coverings of) solid tori that contain a component of $P$,  
$V_{1,l}$ to solid tori disjoint from $P$, $V_2$ to pared $I$-bundles, $V_3$ to the remaining acylindrical pieces and $V_4$ to thickened tori;   
\item an edge joins a vertex from $v\in V_{1,p}(\cT)\cup V_{1,l}(\cT)$ to a vertex $w\in V_2(\cT)\cup V_3(\cT) \cup V_4(\cT)$ if their intersection $v\cap w$  corresponds to a cylinder from $B$.
Moreover, two distinct type IV vertices have distance at least $2$ in $\cT$.
\eit

We interpret below the JSJ decomposition in terms of the topology of the  limit set of a Kleinian group uniformizing $(N,P)$ following \cite{Bow98, Bow01, HH23}.
The relationships between the topology of $N$ and the limit set are detailed in \cite[\S\,4.4]{haiss:lec}.

\subsubsection{Decomposition of the limit set}\label{subsubsec:decomplimit}
Let $(N,P)$ be a pared manifold with incompressible boundary that is uniformized by a Kleinian group $G$ that provides us with a faithful representation $\rho:\pi_1(N)\to G$. 
It follows from the assumption that $\Lambda=\Lambda(G)$ is connected.

In \cite[\S 8]{HH23}, the authors identify the JSJ tree $\mathcal{T}$ with a canonical {\em exact cut pair/cut point tree}, which depends only on the topology of the limit set, and that captures
the JSJ decomposition. Let us recall some terminology.

\medskip

{\noindent\bf Topological terminology for limit sets.}
Let $x \in \Lambda$. We define its {\em valence}, denoted by $\val(x)$, as the number of ends of the locally compact space $\Lambda \setminus \{x\}$. 
A point $x$ is a {\it local cut point} if $\val(x) \geq 2$.
We say it is a {\em (global) cut point} if $\Lambda \setminus \{x\}$ is disconnected. 

Given another point $y\in\Lambda$, we let $N(x, y)$ be the number of connected components of $\Lambda \setminus \{x, y\}$.
A pair of local cut points is an  {\em exact cut pair} if they are not global cut points and if
\begin{equation}\label{eqn:exactcutpair}
    N(x,y) = \val(x) = \val(y).
\end{equation}

A subset of $\Lambda$ is {\it inseparable} if no cut point nor any other cut pair separate them.

A finite subset $S$ of a connected topological space $Y$ is called a {\it cyclic subset} if either $S$ is a cut pair or there is an ordering $S  =\{s_j,\ j\in\Z/n\Z\}$, $n\ge 3$, and continua $M_j\subset Y$, $j\in\Z/n\Z$, such that
\begin{enumerate}[- ]
\item $M_i \cap M_{i+1} = \{s_i\}$, $i\in\Z/n\Z$,
\item $M_i \cap M_j = \emptyset$ whenever $|i - j| > 1$,
\item $\bigcup M_i = Y$.
\end{enumerate}
An infinite subset in which all finite subsets of cardinality at least $2$ are cyclic is also called
{\it cyclic}. A maximal cyclic subset with at least 3 elements is called a {\it necklace}.

\medskip

We now describe the JSJ splitting from the point of view of the limit set of $G$. We first relate the paring with cut points: a point $c\in \Lambda$ is a cut point if and only 
if there is a component $P_i$ of $P$ and a solid torus or a thickened torus $W$ in $N\setminus B$ containing $P_i$ such that $c$ is the fixed point 
of a conjugate of $\rho(\pi_1(W))=\rho(\pi_1(P_i))$ \cite[Lemma 4.15]{haiss:lec}. The stabilizer of the cut point $c$ contains a primitive accidental parabolic element, i.e.,
a parabolic transformation that stabilizes a component $\Delta$ of the ordinary set together with a geodesic in $\Delta$.

\medskip

Now we explain how each vertex $v\in \cT$ is associated to a compact subset $\Lambda_v \subseteq \Lambda$.

\begin{enumerate}
    \item[($\mathrm{I}_p$)] Let $v$ be a type $\mathrm{I}_p$, or a {\em parabolic \twoended} vertex, i.e., a solid torus that contains a component of $P$. Then 
    \begin{itemize}
        \item $\Lambda_v = \{a_v\}$ consists of exactly one point, which is a global cut point; and 
        \item the stabilizer $G_v$ is generated by some parabolic element $g_v \in G$ that fixes $a_v$.
    \end{itemize} 
    \item[($\mathrm{I}_l$)] Let $v$ be a type $\mathrm{I}_l$, or a {\em loxodromic \twoended} vertex, i.e., a solid torus disjoint from $P$. Then 
    \begin{itemize}
        \item $\Lambda_v = \{a_v, b_v\}$ consists of exactly two points forming an  inseparable exact cut pair (see \eqref{eqn:exactcutpair}); and 
        \item the stabilizer $G_v$ is generated by some loxodromic element $g_v \in G$ that fixes $a_v, b_v$.
    \end{itemize}
   \item[(II)] Let $v$ be a type II, or a {\em maximally hanging Fuchsian (\mhf)} vertex, i.e., a pared $I$-bundle. Then 
    \begin{itemize}
        \item $\Lambda_v$ is a necklace Cantor set;  and
        \item the stabilizer $G_v$ is conjugate to a geometrically finite Fuchsian group $Q$ by
        a homeomorphism between their limit sets that preserves their cyclic order.
    \end{itemize} 

    \item[(III)] Let $v$ be a type III, or a {\em \rigid} vertex, i.e., an acylindrical pared submanifold. Then 
    \begin{itemize}
        \item $\Lambda_v$ is a maximal inseparable set with at least $3$ points; and
        \item the stabilizer $G_v$ is a geometrically finite Kleinian group that is not maximal hanging Fuchsian (see \cite[Proposition 8.7]{HH23}).
    \end{itemize}
    
    \item[(IV)] Let $v$ be a type IV, or a {\em \ranktwo} vertex, i.e., a thickened torus. Then 
    \begin{itemize}
        \item $\Lambda_v = \{a_v\}$ consists of exactly one point, which is a global cut point; and 
        \item the stabilizer $G_v$ is a rank-two parabolic subgroup that fixes $a_v$.
    \end{itemize}
\end{enumerate}

\subsubsection{Separating geodesics and laminations}\label{subsec:separatinggeodandlam}
Let $(N,P)$ be a pared manifold with incompressible boundary that is uniformized by a  geometrically finite Kleinian group $G$ with connected limit set.
In this subsection, we associate to each vertex $v \in V(\mathcal{T})$ of the JSJ-tree a canonical lamination $\cL_v$.

\medskip

We proceed as follows.
Let $B$ be the balanced cylinder decomposition of $(N,P)$. Let $\cC_{JSJ}$ denote the collection of boundary curves of the cylinders in $B$ that are not homotopic to a curve in $P$ in $\partial N$, i.e., those boundary curves that are represented by closed geodesics in $\Omega(G)/G$. 
We let $\cL_{JSJ}$ denote its lift to $\Omega(G)$. It is a $G$-invariant lamination such that every leaf
is a geodesic in $\Omega(G)$ stabilized by a loxodromic element and its limit set is contained in some vertex limit set $\Lambda_v$. 

For every vertex $v$ of the JSJ-tree, we let $\cL_v\subset \cL_{JSJ}$ consist of those geodesic leaves in $\Omega(G)$ that
have both endpoints in $\Lambda_v$. Two such laminations have leaves in common if they correspond to pieces of $N\setminus\!\setminus B$ that share a cylinder in their boundaries. Let us describe their properties in terms of the vertex types.

\subsection*{Type $\mathrm{I}_p$}
Let $v$ be a  parabolic {\twoended} vertex with $\Lambda_v = \{a_v\}$. Thus $a_v$ is an accidental parabolic 
cut point and every leaf $l\in\cL_v$ defines a Jordan curve $l\cup\{a_v\}\subset\Omega(G)\cup\{a_v\}$.
Let $\val(v)$ be the number of connected components of $\Lambda \setminus \{a_v\}$. Then $\cL_v$ has $\val(v)-1$ leaves, and 
$a_v$ is on the boundary of $\val(v)+1$ components of $\Omega(G)$, two of which do not contain any leaves (see Figure~\ref{fig:I_p}).

\begin{figure}[ht]
\captionsetup{width=0.96\linewidth}
  \centering
  \includegraphics[width=0.6\textwidth]{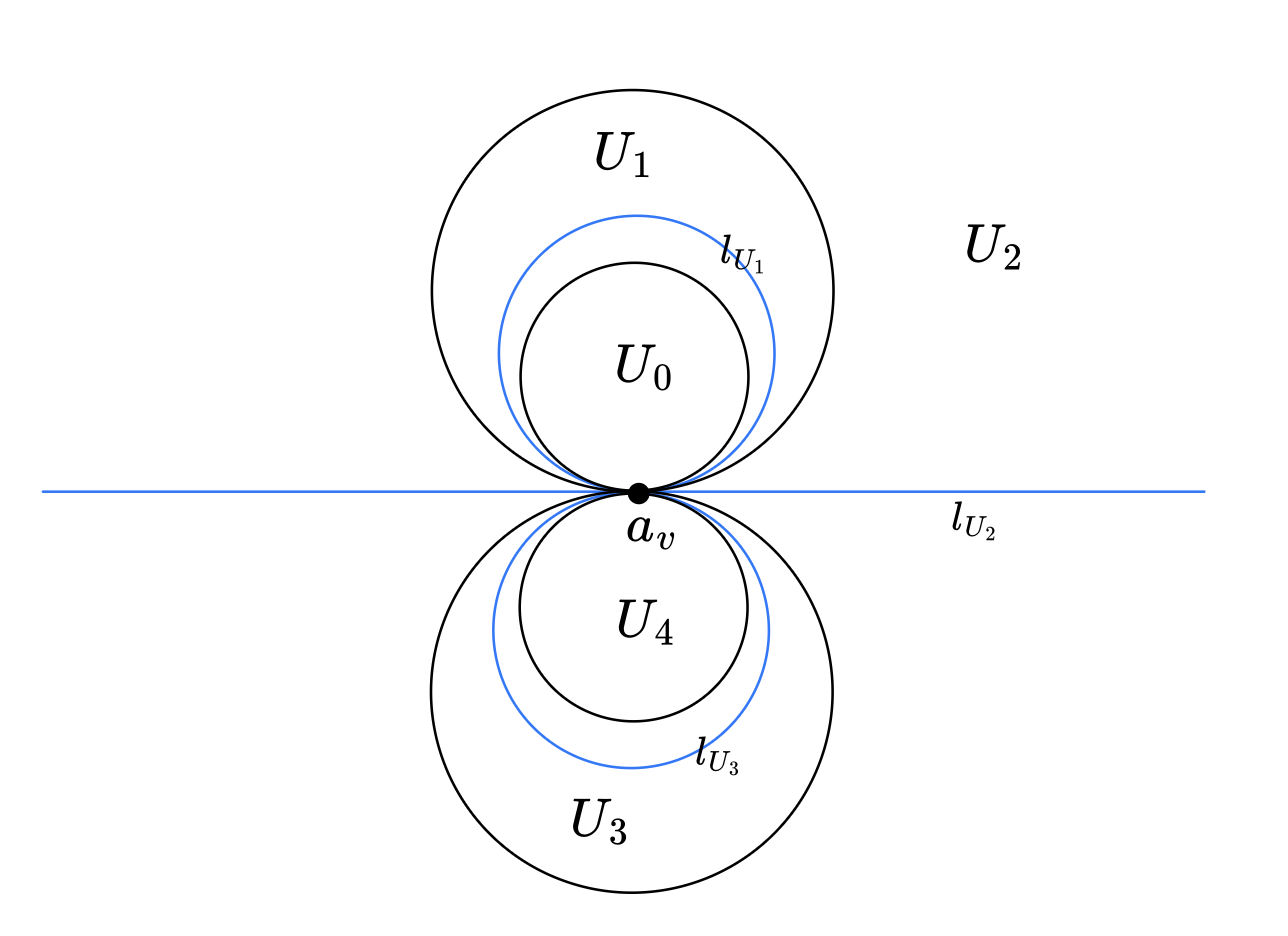}
  \caption{A schematic drawing of the lamination $\mathcal{L}_v$ for a type $\mathrm{I}_p$ vertex with finite valence $\val(v) = 4$. Each black circle represents a component of $\Lambda\setminus \{a_v\}$. The blue curves represent the leaves in $\mathcal{L}_v$.}
  \label{fig:I_p}
\end{figure}

\subsection*{Type $\mathrm{I}_l$}
Let $v$ be a loxodromic {\twoended} vertex with $\Lambda_v = \{a_v, b_v\}$. 
Then $\cL_v$ contains exactly  $N(a_v,b_v)$ leaves.

\subsection*{Type II and type III}
Let $v$ be a {\mhf} or a {\rigid} vertex. There exists a unique connected component $U_v$ of $\cbar\setminus \overline{|\cL_v|}$ such that $\Lambda_v\subset \overline{U_v}$.

\subsection*{Type IV}
Let $v$ be a {\ranktwo} vertex. Then $\cL_v$ is infinite and $\overline{|\cL_v|}$ is homeomorphic to a double Hawaiian earring (see Figure~\ref{fig:ranktwoSch}, right).

\medskip

The lamination $\cL_v$ naturally induces an equivalence relation $\sim_v$ on $\Lambda_v$.
Theorem~\ref{thm:charac annulus} (ii) enables us to understand the {\mhf} and {\rigid} vertices as follows.

\begin{prop}\label{prop:rigidvertexSchottky}
    \noindent\begin{enumerate}[leftmargin=8mm]
    \item Let $v \in \mathcal{T}$ be a {\mhf} vertex. Then $\Lambda_v/\sim_v$ is homeomorphic to a circle.
    \item Let $v \in \mathcal{T}$ be a {\rigid} vertex. Then $\Lambda_v/\sim_v$ is homeomorphic to a Schottky set.
\end{enumerate}
\end{prop}

\begin{proof} Theorem \ref{thm:charac annulus} (ii) ensures that each {\mhf} and {\rigid} vertex corresponds to a pared I-bundle and to a pared acylindrical manifold respectively.

\medskip

Thus, the first statement follows from the fact that the paring consists of the holes of the surface of the bundle, and 
therefore the equivalence relation consists in identifying points from the same gap.
    
The second statement follows from the fact that the limit set of a Kleinian group uniformizing a pared acylindrical manifold is homeomorphic to a Schottky set.
\end{proof}

\subsubsection{Convex cocompact Kleinian groups}\label{subsubsec:ccp}
If the uniformization of $N$ is a convex cocompact Kleinian group, then we obtain the following simplifications. As there are no global cut point nor any parabolic point, all the vertices are convex cocompact and their possible types are loxodromic {\twoended}, {\mhf} and {\rigid}, and the edges correspond to inseparable exact cut pairs.

\subsection{Incompressible decomposition}\label{subsec:icd}
In this subsection, we discuss a decomposition of a geometrically finite Kleinian group with disconnected limit sets.

\subsubsection{Compression disks and incompressible cores}
Let $(N, P)$ be a pared 3-manifold uniformized by a geometrically finite Kleinian group $G$. 
It follows from \cite[Theorem 3.7]{bonahon:3-var} that there is a (possibly disconnected) submanifold $\Sigma$ of $N$ so that $(\Sigma, P\cap \partial \Sigma)$ has incompressible boundary, and $N$ is obtained from $\Sigma$ by adding $1$-handles (see also \cite{CMT99}). The submanifold $\Sigma$, which is unique up to isotopy, is called the {\em incompressible core of $(N, P)$}. Its complement with a neighborhood of $\partial_0N$ contains every compression disk up to homotopy. 

We may find a finite and minimal collection $\cE$ of compression disks so that each component of $N\setminus\!\setminus \bigcup_{D\in\cE}D$
 is either a ball or a  pared $3$-manifold with incompressible boundary.
 We call such a decomposition {\it a core decomposition of $(N,P)$ along  compression disks}. 
 More precisely, 
 a collection $\mathcal{E}:=\{D_1,\dots,D_s\}$ of disjoint compression disks is called a
{\em core decomposition} if $\partial D_i \subseteq \partial_0 N$ and, after cutting $N$ along
$\{D_j\}$ to obtain
$$
N' = N \setminus\!\setminus \bigcup_{j=1}^s D_j,
$$
so that each $D_j$ gives rise to two boundary disks $D_j^+$ and $D_j^-$,
the following hold:
\begin{enumerate}
    \item\label{completecompressiondisk1} each component $X$ of $N'$ is either a topological $3$-ball with $P \cap \partial X = \emptyset$, or $(X, P\cap \partial X)$ has incompressible boundary;
    \item the collection $\mathcal{E}$ is minimal in the sense that no proper subset of $\mathcal{E}$ satisfies condition~\eqref{completecompressiondisk1}.
\end{enumerate}

 We remark that such a decomposition is generally not unique (cf. Proposition~\ref{prop:admiss}).

Let $\Sigma$ be the union of the components of $N'=N \setminus\!\setminus \bigcup_{j=1}^s D_j$ that has incompressible boundary.
Then $\Sigma$ is the incompressible core for $N$. Note that $\Sigma$ may be disconnected, and each connected component $X$ of $\Sigma$ is either 
a non-elementary pared-manifold or homeomorphic to a thickened torus or to a solid torus with $P \cap \partial X \neq \emptyset$.
Each component of $N' \setminus\!\setminus \Sigma$ is a topological $3$-ball.
Note that the boundary $\partial \Sigma$ consists of a part of $\partial N$ and of the compression disks. 

\subsubsection{Decomposition of the group and the limit set}
Let $\mathcal{E} = \{D_1,..., D_s\}$ be a core decomposition.
In this subsection, we describe how it gives a free splitting of $G$, which we call the {\em incompressible decomposition}.
This is the analogue of the JSJ decomposition along essential cylinders (see \S~\ref{subsec:JSJ}).

The decomposition can be better described in terms of the decomposition of the limit set.
We assume that $\partial D_i \subseteq \partial_0N$ which is identified with the conformal boundary $\Omega(G)/G$.
We assume that $\partial D_i, i=1,...,s$ are hyperbolic geodesics on the conformal boundary.

Let $\mathcal{D}$ be the collection of all lifts of $\{\partial D_1,..., \partial D_s\}$ in $\Omega(G)$.
Since we assume that $\partial D_i, i=1,...,s$ are hyperbolic geodesics, $\mathcal{D}$ is a collection of Jordan curves.

Let $|\mathcal{D}|:=\bigcup_{C\in \mathcal{D}} C$.
We call $\mathcal{D}$ the {\em circle lamination} associated to the incompressible decomposition.

Given three Jordan curves $C_1, C_2, C_3$, we say $C_2$ separates $C_1, C_3$ if $C_1, C_3$ are contained in two different components of $\widehat\C \setminus C_2$.
A maximal $\mathcal{S} \subseteq \mathcal{D}$ so that no Jordan curve in $\mathcal{S}$ separates two Jordan curves in $\mathcal{S}$ is called a {\em star}.
Note that each star $\mathcal{S}$ corresponds to a component of $N\setminus\!\setminus \bigcup_{j=1}^s D_j$.

We construct a simplicial tree $\mathcal{T}$ by associating a vertex to each Jordan curve in $\mathcal{D}$ and to each star of $\mathcal{D}$.
Let $v_\mathcal{S}, v_C$ be vertices corresponding to a star $\mathcal{S}$ and a Jordan curve $C \in \mathcal{D}$ respectively.
We connect them by an edge if $C \in \mathcal{S}$. We obtain a cofinite simplicial action of $G$ on $\mathcal{T}$ with no edge inversion and 
a function $p:\Lambda(G)\to \cT\cup\partial \cT$ defined as follows. If $z\in\Lambda(G)$ belongs to a star component, then we map it to the vertex it defines. Otherwise,
$z$ is contained in a nested sequence of Jordan domains $D_0 \supset D_1 ...$ bounded by Jordan curves of $\mathcal{D}$; note that the diameters of $D_i$ shrink to $0$, thus $\bigcap_n \overline{D_n} \cap \Lambda(G)$ is a singleton. We let $p(z)$ denote the end of the corresponding ray in $\mathcal{T}$. Let us define $G_v=\stab_G(v)$.

\subsubsection{Limit set $\Lambda_v$ and its stabilizer $G_v$}
Let $v$ be a vertex. We now associate to $v$ the limit set $\La_v$ of the stabilizer $G_v = \stab_G(v)$ for the simplicial action of $G$ on $\cT$.

Suppose that $v$ corresponds to a Jordan curve $C \in \mathcal{D}$. Then we have 
$$
\Lambda_v = \emptyset.
$$

Suppose that $v$ corresponds to a star $\mathcal{S}$, and $C \in \mathcal{S}$, then there is a unique component $D_C^-$ of $\widehat\C \setminus C$ that is disjoint from any Jordan curve in $\mathcal{S}$.
Let $K_v:= \widehat\C \setminus \bigcup_{C\in \mathcal{S}} D_C^-$.
Since $K_v$ is invariant under $G_v$, we have $\Lambda_v \subseteq \Lambda(G) \cap K_v$.
For the other inclusion, let $x \in \Lambda(G) \cap K_v$. By definition of the limit set, there is a sequence $C_n \in \mathcal{S}$ with $C_n \to x$. Then $x \in \Lambda_v$ as the action of $G_v$ on $\mathcal{S}$ is cofinite. 
Thus, we have
$$
\Lambda_v= \Lambda(G) \cap K_v \text{ and } G_v = \stab_G(\Lambda_v).
$$ 

The topology of these limit sets provides us with a partition of the vertex set into four categories
$$
V(\mathcal{T}) = V_D(\mathcal{T}) \sqcup V_B(\mathcal{T}) \sqcup V_{NE}(\mathcal{T}) \sqcup V_P(\mathcal{T})
$$
where
\begin{enumerate}
    \item[(D)] $V_D(\mathcal{T})$ consists of the vertices corresponding to the Jordan curves in $\mathcal{D}$ so that $G_v = \{\id\}$ and $\Lambda_v = \Lambda(G_v) = \emptyset$;
    \item[(B)] $V_B(\mathcal{T})$ consists of the vertices corresponding to the $3$-balls in $N\setminus\!\setminus\bigcup D_i$ so that $G_v = \{\id\}$, $\Lambda_v = \Lambda(G_v) = \emptyset$ and $K_v$ is finitely connected;
    \item[(NE)] $V_{NE}(\mathcal{T})$ consists of the vertices corresponding to the pared manifolds of $\Sigma$ so that $G_v$ is a non-elementary geometrically finite Kleinian group and $\Lambda_v = \Lambda(G_v)$ is connected;
    \item[(P)] $V_P(\mathcal{T})$ consists of the vertices corresponding to a thickened torus component of $\Sigma$ or a solid torus component $X$ of $\Sigma$ with $P \cap \partial X \neq \emptyset$ so that $G_v = \Z^2$ or $\Z$ and $\Lambda_v = \Lambda(G_v) = \{p\}$, where $p$ is the parabolic fixed point of any nontrivial element in $G_v$.
\end{enumerate}
We call them {\em \disk} (type D), {\em \ball} (type B), {\em \nonel} (type NE) and {\em \para} (type P) vertices respectively.

\subsection{Strong accessibility}\label{subsec:sa}
Let $G$ be a minimally parabolic geometrically finite Kleinian group. Then $G$ is strongly accessible by \cite{LT17}.
Therefore, we obtain a {\em hierarchy} of $G$ by finite iterations of JSJ or/and incompressible decompositions such that each terminal vertex group does not admit any further splitting. Thus, the terminal
vertex groups are either elementary, Fuchsian or carpet groups.
We call the total levels of iterations of JSJ or incompressible decomposition the {\em height of the hierarchy} of $G$, and is denoted by $\height(G)$.

Let $G$ be a geometrically finite Kleinian group.
Recall that its de-parabolization $\widehat G$ is a minimally parabolic geometrically finite Kleinian group as introduced in \S~\ref{subsec:rankonedeparalam}. We define the height of the hierarchy of $G$ by $\height(G) = \height(\widehat G)$.

\begin{prop}\label{prop:finitehierarchy}
Let $G$ be a geometrically finite Kleinian group with $\Lambda(G) \neq \widehat\C$.
    Then $\Lambda(G)$ contains a carpet if and only if the limit set of its de-parabolization $\hat{G}$ contains a carpet. In this case, $\Lambda(\hat{G})$ contains a carpet subgroup.

    In particular, $\Lambda(G)$ is carpet-free if and only if the terminal groups arising from the strong accessibility of its de-parabolization $\hat{G}$    are elementary or Fuchsian.
\end{prop}
\begin{proof}
    Suppose $\Lambda(G)$ contains a carpet (this happens in particular if there is a carpet-quotient). We may choose an embedded carpet whose peripheral circles do not intersect the boundary of the components of the domain of discontinuity.
    Then $\Lambda(\widehat G)$ contains a carpet as well by taking its preimage under $\La(\hat{G})\to \La(G)$. This carpet is contained in the  (carpet) limit set of a terminal vertex group coming from the strong accessibility of $\hat{G}$ and the fact carpets have no cut points or cut pairs. 

    Conversely, suppose $\widehat G$ contains a carpet.
    By \cite[Corollary 2.7]{LT17}, $\widehat G$ admits a hierarchy over elementary subgroups.
    Therefore, $\widehat G$ contains a subgroup $H$ with carpet limit set $\Lambda(H)$. By Moore's theorem \cite{Moo25}, $\Lambda(H)$ contains a carpet subset $K$ whose peripheral Jordan curves are buried, i.e., does not intersect the peripheral Jordan curves of $\Lambda(H)$. Then $K \subseteq \Lambda(\widehat G)$ projects to a carpet in the quotient $\Lambda(G)$, so $\Lambda(G)$ is not carpet-free.
\end{proof}

We record the following well-known proposition regarding carpet subsets and conformal dimensions.
\begin{prop}\label{prop:confdim}
    Let $G$ be a convex cocompact Kleinian group. Then 
    \begin{enumerate}
        \item $\Lambda(G)$ is carpet-free if and only if $\dim_{\conf}(\Lambda(G)) \leq 1$;
        \item $\Lambda(G) = \widehat\C$ if and only if $\dim_{\conf}(\Lambda(G)) = 2$.
    \end{enumerate}
\end{prop}
\begin{proof}
    The first part follows from \cite[Theorem~3.4.6]{MT10}, \cite[Corollary~1.2]{Mac10} and \cite[Corollary~1.2]{CM22}. The second part follows from \cite[Corollary~1.3]{BJ97}.
\end{proof}

\section{Geometry of limit sets and parabolic loci}\label{sec:geompara}
In this section, we discuss how parabolic points affect the geometry of the limit sets and the components of the ordinary sets.

In \S~\ref{sec:qsparabolic}, we show that, except in the circle and sphere cases, any quasisymmetry between limit sets must preserve the parabolic loci; see Theorem~\ref{thm:qsparab}. This explains why a relative modification of universality is necessary for geometrically finite Kleinian groups.
In \S~\ref{subsec:parahomeo}, we discuss when a parabolic point and/or its rank can be detected topologically.
Finally, in \S~\ref{subsec:geomofordinary}, we complete our discussion by analyzing the geometry of components of the ordinary set in terms of the parabolic locus.

\subsection{Quasisymmetric maps and parabolic points}\label{sec:qsparabolic}
We prove that a quasisymmetric map between the limit sets of two geometrically finite Kleinian groups preserve the parabolic points provided the limit set is not a circle nor the sphere.

\begin{theorem}\label{thm:qsparab} Let $G,G'$ be two {\nonel} geometrically finite Kleinian groups with limit sets different from a Jordan curve and a sphere. If $h:\La(G)\to \La(G')$ is quasisymmetric, then 
$h$ maps the parabolic locus of $G$ onto the parabolic locus of $G'$ and preserves their ranks.
\end{theorem}

To distinguish the geometry of the limit sets in neighborhoods of a point, we introduce the notion of relative size of collection of subsets. 
Let $X$ be a metric space and $\cC$ a  collection of non-empty subsets of $X$. Given $p\in X$ and $r>0$, set $$\sigma_{\cC}(p,r)= \frac{1}{r}\sup\{\diam\,K: K\subset B(p,r), K\in\cC\}$$
with the convention that $\sigma_\cC(p,r)=0$ if no such $K$ exists.
Set  $$\sigma_{\cC}(p)= \limsup_{r\to 0} \sigma_\cC(p,r)\,.$$

\begin{lem}[Quasisymmetry invariance]\label{prop:qminv}
Let $h:X\to Y$ be a quasisymmetric map and $\cC$ a collection of subsets in $X$. Then, given $p\in X$,
$$\sigma_\cC(p)>0\quad \hbox{if and only if}\quad \sigma_{h(\cC)}(h(p))>0\,.$$
\end{lem}

\begin{proof} Let us assume that  $\sigma_\cC(p)>0$. Let $\eta$ be the distortion function of $h$. 
It follows that if $K\subset B(p,r)$ with $r>0$ small enough then, by \cite[Prop.\,10.8]{heinonen:analysis}, 
we have
\[ \frac{1}{2\eta\left(\frac{\hbox{\rm diam}\,B(p,r)}{\diam\,K}\right)} \leq \frac{\diam  \,
h(K)}{\diam\,h(B(p,r))} \leq
\eta\left(2\frac{\diam\,K}{\diam\,B(p,r)}\right)\,.\]

Thus, if  $K\in\cC$, $K\subset B(p,r)$ and $\diam\,K\ge \tau r$, then $\diam\,K\ge  (\tau/2)\diam\,B(p,r) $, $h(K)\subset B(h(p),\diam\, h(B(p,r)))$ and
$$\sigma_{h(\cC)} (h(p),\diam\, h(B(p,r)) \ge \frac{ \diam\,h(K)}{\diam\, h(B(p,r))} \ge \frac{1}{2\eta\left(\frac{\diam\,B(p,r)}{\diam\,K}\right)} \ge \frac{1}{2\eta(2/\tau)}\,.$$
Hence, $\sigma_{\cC}(p)>0$ implies $\sigma_{h(\cC)}(h(p))>0$. We obtain the lemma by symmetry.\end{proof}

Note that a point in $\La(G)$ is either a conical point or a parabolic point.
We show that the relative size of orbits of subsets distinguishes conical limit points from parabolic points. 

\begin{lem}\label{lma:conpar_qs} Let $G$ be a  Kleinian group and $K\subset \La(G)$ with at least two points. Let us assume that $GK$ is a null-sequence.
\ben
\item If $p$ is a conical point then $\sigma_{GK}(p)>0$.
\item If $p$ is a bounded  parabolic point that belongs to finitely many elements of $GK$, then $\sigma_{GK}(p)=0$.
\een
\end{lem}

\begin{proof}
Let us first assume that $p$ is a conical point. It follows that there exist sequences  $(g_n)$ in $G$  and $(r_n)_n$ tending to $0$ such that $h_n(z):= g_n(p+r_nz)$ is a convergent
sequence of univalent maps defined on the unit disk. Let us fix $g\in G$ such that $g(K) \subset h_n(B(0,1/2))$ for $n$ large enough. Thus, we may find $r,\tau >0$ such that
$g(K)\subset B(g_n(p),r)$ and $\diam g(K)\ge \tau r$. The Koebe distortion theorem implies that $(B(g_n(p), r)\stackrel{g_n^{-1}}{\longrightarrow} g_n^{-1}(B(g_n(p), r))$ are uniformly
quasisymmetric. Thus, \cite[Prop.\,10.8]{heinonen:analysis} implies as above that $\sigma_{GK}(p)>0$.

We now assume that $p$ is parabolic, with stabilizer $H<G$.  Let $L\subset \La(G)\setminus\{p\}$ be a compact set
containing a fundamental domain of the action of $H$. Since $GK$ is a null sequence, there exists another compact set $L'$ disjoint from $\{p\}$ that contains every element
of $A \in GK$ which both intersects $L$ and is disjoint from $p$.  

By compactness of $L'$, there exists $\delta >0$ such that $d( x, p )\ge \delta$ for all $x\in L'$. Let us assume that $g(K)\subset B(p,r)$ for $r>0$ small enough. Since there are only finitely many elements of $GK$ containing $p$, by shrinking $r$ if necessary, we assume $p \notin g(K)$. 
Let $h\in H$ such that
$h^{-1}g(K)\subset L'$. Using the fact that $d(h(x),h(y)) \lesssim d(h(p),h(x))\cdot d(h(p),h(y))$ for all $h\in H$, it follows that  
$\diam\, g(K)  \lesssim r^2$ so that $\sigma_{GK}(p)=0$.
\end{proof}

\begin{proof}[Proof of Theorem \ref{thm:qsparab}]
Suppose $\La(G)$ is not a Cantor set, a Jordan curve or a sphere. If $\La(G)$ is homeomorphic to a Schottky set, then we let $K$ be the boundary of a component of $\Omega(G)$. Otherwise, $\La(G)$ admits a nontrivial JSJ or incompressible decomposition and we let $K = \Lambda_v$ be a nontrivial vertex limit set, i.e., $|\Lambda_v| \geq 2$. Then  $GK$ is a null-sequence and is mapped under the quasisymmetry to finitely many orbits. For any parabolic point $p$ that is not a rank-two cut point, there are only finitely many elements in $GK$ that contains $p$. Since rank-two cut points are preserved by any homeomorphism, the theorem follows from Lemmas \ref{prop:qminv} and \ref{lma:conpar_qs}.

It remains to treat the case when $\La(G)$ is a Cantor set. Even if there are no remarkable subsets as above, we also rely on the conformal elevator principle  to distinguish conical points from 
parabolic points. 
For any $\pi/2>r_0>0$, there exists a constant $C_0 = C_0(r_0)\ge 1$ such that for any  $x\in \La( G)$ and $r \in [r_0, \pi/2]$,
 one can find a subset  $A\subset \La( G)$ such that 
$$B(x,r/C) \cap \La(G) \subset A\subset B(x,r) \quad\hbox{and}\quad \dist(A, \La(G) \setminus  A)\ge r/C_0\,.$$

Let us now assume that $p$ is a conical point. Then we consider sequences  $(g_n)_n$ in $G$  and $(r_n)_n$ tending to $0$ such that $g_n(p+r_n\cdot)$ is a convergent
sequence of univalent maps defined on the unit disk. By applying the previous remark to $g_n(p)$, we get by the Koebe distortion theorem a sequence $(r_n')_n$ tending to zero
and sets $A_n$ such that 
$$B(p,r_n'/C) \cap \La(G) \subset A_n\subset B(p,r_n') \quad\hbox{and}\quad \dist(A_n, \La(G) \setminus  A_n)\ge r_n'/C\,.$$
This property is invariant under quasisymmetries by  \cite[Prop.\,10.8]{heinonen:analysis}.

Assume now that $p$ is parabolic and let $h\in\stab(p)$ be an infinite order element and $x\in \La(G)\setminus \{p\}$.  
Let us fix a sufficiently small radius $d(x,p)/2\ge r>0$  and a subset $A\subset \La(G)$ such that
$(B(p,r/C) \cap \La(G))\subset A\subset B(p,r)$ for some constant $C>1$. 
Since the positive orbit $\{h^n(x), n\ge 0\}$ tends to $p$, it gets
into $A$, but is not contained in $A$. Therefore, we may find $k\ge 0$ such that $h^{k+1}(x)\in A$ but $h^{k}(x)\notin A$. Therefore, as $p$ is the fixed point of the parabolic element $h$,
$$\dist(A,\La(G)\setminus A)\le d(h^k(x),h^{k+1}(x))\lesssim r^2\,.$$
This shows that the above property does not hold for parabolic points. Therefore, parabolic points are preserved under quasisymmetries.

We now deal with the invariance of their ranks. Let $p\in\La(G)$ be a parabolic point. If its rank is one then any weak tangent of $\La(G)$ at $p$ is contained in a line since it is contained in the complement of two tangent disks at $p$, whereas if it is two then any weak tangent is the plane since $\La(G)$ contains the inversion with pole
$p$ of a lattice of the plane. Let $\phi:\La(G)\to \La(G')$ be quasisymmetric and assume that $p=0=\phi(0)$ and $\phi(\infty)=\infty$.
For each sequence $(1/\nu_n)_n \in \Lambda(G)$ with $1/\nu_n \to 0$, there exists an unbounded positive sequence $(\la_n)_{n\ge 0}$ such that $|\la_n\phi(1/\nu_n)|=1$. This defines an equicontinuous family of maps $\la_n\phi(\cdot/\nu_n):\nu_n\La(G) \rightarrow \la_n\La(G')$ \cite[Prop.\,10.26]{heinonen:analysis},
so the weak tangent spaces are homeomorphic: the ranks are preserved.
\end{proof}

\subsection{Parabolic loci and homeomorphisms between limit sets}\label{subsec:parahomeo}
We pause to analyze our assumptions concerning the parabolic loci for limit sets different from the circle and the sphere. 
Recall that $\mathfrak{P}(G) = (\mathfrak{P}_1(G), \mathfrak{P}_2(G))$ where $\mathfrak{P}_j(G)$ is the collection of rank $j$ parabolic fixed points of $G$. A type-preserving homeomorphism is a map $h: (\Lambda(G), \mathfrak{P}(G)) \rightarrow (\Lambda(G'), \mathfrak{P}(G'))$ that sends rank $1$ (resp. rank $2$) parabolic points of $G$ to those of $G'$. 
We first note that a quasisymmetric map is automatically type-preserving as above by Theorem \ref{thm:qsparab}. What about topological homeomorphisms\,? We organize this discussion
from the splittings over elementary subgroups that provide us with four cases. 
\bit
\item A trivial component of a limit set can be either conical or parabolic of any rank and they cannot be distinguished topologically. For instance, all planar Cantor sets are ambiently homeomorphic \cite[Theorem 13.7]{moi78}. Therefore, neither parabolic points nor their ranks are topologically distinguishable. 
\item Only parabolic points can be cut points of components of the limit set. In that case, their ranks are distinguishable since the JSJ tree will have finitely many  incident edges for rank one points, and infinitely many for rank two points. Thus, a homeomorphism will preserve the parabolic cut points with their ranks. 
\item Limit sets of MHF subgroups contain only conical and rank one parabolic points. As above, they are not distinguishable (but their rank is determined); see Figure~\ref{fig:ccvsgf}.
\item Similarly, the residual limit set, i.e., the set of limit points that do not belong to the boundary of any component of the ordinary set \cite{Abi73}, contains conical points and possibly rank two parabolic points; they are not distinguishable (but their rank is determined). 
\eit

\subsection{Geometry of the ordinary set}\label{subsec:geomofordinary}
Let $G$ be an analytically finite Kleinian group, i.e., $\Omega(G)/G$ is a finite union of Riemann surfaces of finite type. The stabilizer of a component $\Omega$ of $\Omega(G)$ is then a finitely generated function group. We refine our knowledge on the possible geometry of $\Omega$ by adding the characterization of the uniformity of $\Omega$ in the multiply connected case.  We gather these results as follows (we work with the spherical metric). The definitions of John and uniform domains are provided below.

\begin{theorem}\label{thm:geoOmega}
Let $G$ be a finitely generated function group with an invariant component $\Omega$.

\ben
\item\label{item:john} The following are equivalent.
\ben
\item $\Omega$ is a John domain.
\item The group $G$ is geometrically finite and every maximal parabolic subgroup of $G$ admits an invariant disk in $\Omega$. 
\item There exists a quasiconformal map $\phi:\cbar\to \cbar$  that is conformal on $\Omega$ and that conjugates $G$ to a function group $G'$ such that $\phi(\Omega)$ is a circle domain.
\een
\item\label{item:unif} The following are equivalent.
\ben
\item $\Omega$ is a uniform domain.
\item $\Omega$ is an inner uniform domain.
\item $G$ is geometrically finite and every parabolic point is non-accidental and admits no pair of tangent disks in $\Omega$.
\een

\een
\end{theorem}

Since a proper simply connected uniform domain of the plane is a quasidisk \cite[Cor. 2.33]{MS79}, we have the following corollary if $\Omega$ is simply connected, cf. \cite{Mas70} and \cite[Corollary 1.2]{McM98}. 
\begin{cor}
    Let $G$ be a finitely generated function group with a simply connected invariant component $\Omega$. Then the following are equivalent.
    \ben
    \item $\Omega$ is a quasidisk.
    \item $\Omega$  is a John disk.
    \item $G$ is geometrically finite and has no accidental parabolic points.
    \een
\end{cor}

\medskip

Since $\La(G)$ is uniformly perfect \cite{Pom84}, the Poincar\'e metric and the so-called quasi-hyperbolic metric are bi-Lipschitz in $\Omega$. This enables us to define John domains and uniform domains as follows. For $x\in \Omega$, let $\delta_\Omega(x) =\dist(x,\partial\Omega)$ and define the distance ratio metric by  
$$j_\Omega(x,y) =\log \left( 1+ \frac{d(x,y)}{\min\{\delta_\Omega(x),\delta_\Omega(y)\}}\right)$$
for all $x,y\in\Omega$. Let us denote by $d_{\Omega}$ the Poincar\'e metric of $\Omega$. 

Following \cite[Corollary 1]{GO79}, $\Omega$ is uniform if there exist  $a,b >0$ such that
$$d_\Omega(x,y) \le a j_\Omega(x,y) +b \quad\hbox{for all } x,y\in\Omega.$$

Following \cite[Main corollary (5)]{herron:john}, $\Omega$ is John if there exist $x_0\in \Omega$ and $a,b >0$ such that
$$d_\Omega(x_0,y) \le a j_\Omega(x_0,y) +b \quad\hbox{for all } y\in\Omega.$$

\begin{proof}[Proof of Theorem \ref{thm:geoOmega}]
We start with (\ref{item:john}). The items (a) and (b) are equivalent according to \cite[Theorem 1.1]{McM98}. That (a) and (b) imply (c) follows from \cite{maskit:koebe} since the stabilizer $H$
of any nontrivial component of $\La(G)$ is a so-called factor group (i.e., the component of $\Omega(H)$ that contains $\Omega$ is simply connected, there are no accidental parabolic for $H$
and every $G$-parabolic point in $\La(H)$ is parabolic for $H$). For the converse, (c) implies that $G$ is geometrically finite with \cite{maskit:decfuncgps} and has no accidental parabolic points.

\medskip

For the item (\ref{item:unif}), the equivalence between (a) and (b) follows from  \cite{maskit:koebe} and from the fact that an inner uniform domain that has bounded turning is uniform. 
Let us remark that if $\La(G)$ is a Cantor set, then the equivalence of (a) and (c) follows from \cite[Theorem 3]{macmanus:cantor:qcircle} with Theorem \ref{thm:qsparab}.

We prove (a) implies (c). If $\Omega$ is uniform, then $G$ is geometrically finite and has no accidental parabolic points according to (\ref{item:john}). Since the uniformity of a domain is quantitatively invariant under M\"obius transformations \cite[Corollary 3]{GO79}, any Carath\'eodory limit of a sequence $(g_n(\Omega))_n$ will be uniform and connected, so that no parabolic point has two tangent disks in $\Omega$ \cite[Theorem 1.3]{McM98}.

For the converse, (c) implies (a), we proceed as follows. We know from the equivalence (\ref{item:john}) that $\Omega$ is a John domain with some center $x_0\in \Omega$.
As $G$ is geometrically finite, there exists $m>0$ such that any pair of points in $\La(G)$ can be $m$-separated by $G$, cf. \cite{beardon:maskit}. 
We consider a $G$-invariant family of pairwise disjoint disks $\cH=\{D_p, p\in\mathfrak{P}(G)\}$ in $\Omega$ attached to each parabolic point. We may assume that the diameter of each disk is smaller than $m/4$ and that it does not contain $x_0$. Let $\Omega'= \Omega\setminus \sqcup D_p$ and observe that the action of $G$ on $\Omega'$ is cocompact. Let $K\subset\Omega'$ be a 
compact Dirichlet domain centered at $x_0$. Enlarging the constants if necessary, we may assume that
$$d_\Omega(x,y) \le a j_\Omega(x,y) +b \quad\hbox{for all } x\in K,\ y\in\Omega.$$

Let us now consider $x,y\in \Omega$. We consider three cases. 

If $x,y$ belong to the same disk $D_p$, then there is an arc $\gamma$ joining $x$ and $y$  of length at most $2|x-y|$ such that $\delta_\Omega(z) \ge \min\{l(\gamma [x,z]), l(\gamma [y,z])\}/2$ for all $z\in\gamma$. It follows
that $d_{\Omega}(x,y)\le a_0j_{\Omega}(x,y)+b_0$ for some universal constants $a_0,b_0 >0$ \cite[Theorem 1]{GO79}. 

If $x\in \Omega'$, there exists $g\in G$ such that $g(x)\in K$. Therefore,  we have
$$d_\Omega(g(x),g(y)) \le a j_\Omega(g(x),g(y))+b.$$
Since M\"obius maps are isometries in the Poincar\'e metric and since $j_{\Omega}(g(x),g(y))\le \la j_\Omega(x,y) +\mu$ for some uniform constants
$\la,\mu$ \cite[Theorem 4]{GO79}, it follows that
$$d_\Omega(x,y) \le a' j_\Omega(x,y) +b' \quad\hbox{for all } x\in \Omega',\ y\in\Omega,$$
for some constants  $a',b'$ that only  depend on $a,b$  and the fundamental domain $K$ of the action of $G$.

We now assume that $x,y$ belong to two disjoint disks $D_p$ and $D_q$ in $\cH$ with $p\ne q$. By quasi-invariance, we may assume that $d(p,q)\ge m$ so that $d(x,y)\ge m/2$. 
By the above case, we have 
$$d_\Omega(x,y) \le  d_{\Omega}(x,x_0)+ d_{\Omega}(x_0,y) \le a( j_\Omega(x,x_0) + j_\Omega(x_0,y)) + 2b\,.$$
Since $\delta_\Omega(x_0)\gtrsim \max\{\delta_{\Omega}(x), \delta_\Omega(y)\}$ and  $d(x,x_0), d(x_0,y) \lesssim d(x,y)$ hold, we have
$$j_\Omega(x,x_0) + j_{\Omega}(x_0,y) \lesssim \max\left\{ \log\left( 1 +\frac{d(x,y)}{\delta_{\Omega}(x)}\right) , \log\left( 1 +\frac{d(x,y)}{\delta_{\Omega}(y)}\right) \right\} \lesssim j_\Omega(x,y)$$
where we have used the Bernoulli inequality $\log (1+as) \le a\log (1+s)$ for $a\ge 1$ and $s>0$ several times.
This concludes this case as well. \end{proof}

\section{Topological flexibility and obstructions to universality}\label{sec:tfou}
In this section, we show the necessity of the conditions for universality, quasi-isometric classification and topological rigidity in our main theorems.

In \S~\ref{sec:uncountableGroup}, we show that if the limit set is not a carpet-free Schottky set, then its homeomorphism group is uncountable; see Theorem~\ref{thm:uncountablehom}. 
This shows that both the carpet-free and the Schottky-set conditions are necessary for topological rigidity.

In \S~\ref{sec:obstruction}, we show that the existence of carpet subset is an obstruction to both quasisymmetric and quasiconformal universality, and the existence of non-homogeneous rank-two cut points is an obstruction to quasisymmetric universality; see Theorem~\ref{thm:nonqs} and Theorem~\ref{thm:nonqscutpoint}, together with their immediate corollaries. Thus, carpet-freeness and the homogeneity of rank-two cut points are necessary conditions for universality.
\subsection{Uncountable homeomorphism groups}\label{sec:uncountableGroup}
Recall that we have $\Homeo^+(\Lambda(G)) := \Homeo^+(\widehat\C, \Lambda(G))/\sim$ where $f\sim g$ if $f|_{\Lambda(G)} = g|_{\Lambda(G)}$, see \eqref{eqn:homeogroup}.
\begin{theorem}\label{thm:uncountablehom}
    Let $G$ be a {\nonel} geometrically finite Kleinian group whose limit $\Lambda(G)$ is not homeomorphic to a carpet-free Schottky limit set.
    Then $\Homeo^+(\Lambda(G))$ is uncountable. In particular, $G$ is not virtually topologically rigid.
\end{theorem}

\subsubsection{$G$-homeomorphisms for carpet limit sets}
In this subsection, we prepare ourselves with the following definition and proposition, which are of independent interest.
\begin{defn}\label{defn:Ghomeo}
    Let $G$ be a geometrically finite Kleinian group with a Sierpi\'nski carpet limit set $\Lambda(G)$.
    We say a homeomorphism $h:(\widehat\C, \Lambda(G)) \longrightarrow (\widehat\C, \Lambda(G))$ is a {\em $G$-homeomorphism} if for each peripheral disk $D$ of $\Lambda(G)$, there exists $g_D \in G$ so that $h|_D = g_D|_D$.
\end{defn}
\begin{prop}\label{prop:uncountableGhom}
    Let $G$ be a geometrically finite Kleinian group with a Sierpi\'nski carpet limit set $\Lambda(G)$. Then the group of $G$-homeomorphisms is uncountable.
\end{prop}
\begin{proof}
    The proof is a modification of \cite[Theorem~1.3]{Nta21}. We sketch the proof here, and refer the reader to \cite{Nta21} for more details.

    We first introduce a definition adapted from \cite{Why58}, see \cite[Definition~2.1]{Nta21}.
    \begin{defn}
    Let $\varepsilon > 0$ and let $Y$ be a closed subset of $\widehat\C$. 
    A partition $\{Y_\alpha\}$ of $Y$ is called an \emph{$\varepsilon$-subdivision} if the collection $\{Y_\alpha\}$ consists of finitely many closed Jordan regions, also called \emph{$2$-cells}, with disjoint interiors and diameter less than $\varepsilon$ (where we use the spherical metric on $\widehat\C$).
    
    A partition $\{S_\alpha\}$ of a Sierpi\'nski carpet $S \subset \widehat\C$ into Sierpi\'nski carpets is called an \emph{$\varepsilon$-subdivision rel.\ ($Q_0,\ldots,Q_N$)} if there exist peripheral disks $Q_0,\ldots,Q_N$ of $S$ and an $\varepsilon$-subdivision $\{Y_\alpha\}$ of the closed domain
    \[
        Y := \mathbb{C} \setminus \bigcup_{i=0}^{N} Q_i,
    \]
    such that the boundaries of the $2$-cells $Y_\alpha$ are contained in 
    \[
        S \setminus \bigcup_{i=N+1}^{\infty} \partial Q_i,
    \]
    where $Q_i$, $i \ge N+1$, are the remaining peripheral disks of $S$; we let $S_\alpha = S \cap Y_\alpha$. Note that each $S_{\alpha}$ is homeomorphic to a Sierpi\'nski carpet.
    \end{defn}
    We need the following lemma, which is proved in \cite[Lemma~2.2]{Nta21}.
    \begin{lem}\label{lem:epsilondivision}
    Let $S, S' \subseteq \widehat\C$ be two Sierpi\'nski carpets. Let $\varepsilon > 0$ and let $D_0,..., D_M$ and $D'_0,..., D'_M$ be a collection of peripheral disks of $S$ and $S'$ respectively, so that both collections contain all peripheral disks of diameter $\geq \varepsilon$.
    Given any orientation-preserving homeomorphism $g_i: D_i \rightarrow D_i'$,
    there exist two $\varepsilon$-subdivisions (with the same set of indices) $\{S_{\alpha}\}$ rel.\ ($D_0,..., D_M$) and $\{S'_{\alpha}\}$ rel.\ ($D'_0,..., D'_M$)  and a homeomorphism $h: \widehat\C \longrightarrow \widehat\C$ so that
    \begin{itemize}
        \item $h|_{D_i} = g_i$;
        \item $h(S_\alpha) = S'_{\alpha}$.
    \end{itemize}
    \end{lem}

    Choose a sequence $\varepsilon_n \to 0$ as $n \to \infty$. Let $D_0,..., D_N$ be the list of peripheral disks of $\La(G)$ of diameter $\geq \epsilon_1$.
    We append to these collections some peripheral disks $D_{N+1},..., D_M$ and choose $g_i \in G$ so that
    \begin{enumerate}
        \item $D_i':= g_i(D_i), i = 0,..., M$ is a list of $M+1$ distinct peripheral disks;
        \item $\{D_0',..., D_M'\}$ contains all peripheral disks of diameter $\geq \epsilon_1$.
    \end{enumerate}
    By Lemma~\ref{lem:epsilondivision}, there exist two $\varepsilon_1$-subdivisions (with the same set of indices) $\{S_{\alpha}\}$ rel.\ ($D_0,..., D_M$) and $\{S'_{\alpha}\}$ rel.\ ($D'_0,..., D'_M$)  and a homeomorphism $h_1: \widehat\C \longrightarrow \widehat\C$ so that
    \begin{itemize}
        \item $h_1|_{D_i} = g_i$;
        \item $h_1(S_\alpha) = S'_{\alpha}$.
    \end{itemize}
    For each $\alpha$, there is a unique peripheral disk $D_{0,\alpha}$ (resp. $D'_{0,\alpha}$) of $S_\alpha$ (resp. $S'_\alpha$) that separates it from the other carpets in the subdivision{\tiny , denoted by $D_{0,\alpha}$ (or $D'_{0,\alpha}$)}. We let $D_{1,\alpha},..., D_{K,\alpha}$ be the list of peripheral disks of $S_{\alpha}$ that contain all of those with diameter $\geq \varepsilon_2$, and choose $g_{i, \alpha} \in G$ so that
    \begin{enumerate}
        \item $D_{i, \alpha}':= g_{i,\alpha}(D_{i, \alpha}), i = 1,..., K$ is a list of $K$ distinct peripheral disks of $S'_{\alpha}$, and $D_{i, \alpha}' \neq D'_{0,\alpha}$ for $i = 1,..., K$;
        \item $\{D_{1, \alpha}',..., D_{K,\alpha}'\}$ contains all peripheral disks of diameter $\geq \epsilon_2$ in $S'_\alpha$.
    \end{enumerate}
    Note that the above is possible as the orbit of every peripheral disk is dense in $\Lambda(G)$.
    We apply Lemma~\ref{lem:epsilondivision} again to obtain two $\varepsilon_2$-subdivisions (with the same set of indices) $\{S_{\alpha, \beta}\}_{\beta}$ rel.\ ($D_{0, \alpha},..., D_{K,\alpha}$) and $\{S'_{\alpha, \beta}\}_{\beta}$ rel.\ ($D_{0, \alpha}',..., D_{K,\alpha}'$) and a homeomorphism 
    $h_{2, \alpha}: \widehat\C \longrightarrow \widehat\C$ so that
    \begin{itemize}
        \item $h_{2, \alpha}|_{\partial D_{0, \alpha}} = h_1$;
        \item $h_{2, \alpha}|_{D_{i, \alpha}} = g_{i, \alpha}$ for $i = 1,..., K$;
        \item $h_{2, \alpha}(S_{\alpha, \beta}) = S'_{\alpha, \beta}$ for each $\beta$.
    \end{itemize}
    We repeat this procedure for each of the carpets $S_\alpha$, and glue $h_{2, \alpha}$ to $h_1$ along $\partial D_{0, \alpha}$ to obtain a homeomorphism $h_2: \widehat\C \longrightarrow \widehat\C$. We iterate ths construction for each carpet in the $\varepsilon_2$-subdivisions, and inductively, we obtain a sequence of homeomorphisms $h_n: \widehat\C \longrightarrow \widehat\C, n\geq 1$. The sequence $(h_n)_n$ is uniformly Cauchy as $\epsilon_n \to 0$ and similarly for its inverses. Thus, $(h_n)_n$ converges to a homeomorphism $h: \widehat\C \longrightarrow \widehat\C$. Note that $h$ gives a bijection of the set of peripheral disks, thus $h(\Lambda(G)) = \Lambda(G)$. For any peripheral disk $D$, $h|_D = h_n|_D$ for all sufficiently large $n$, thus there exists $g_D \in G$ so that $h|_D = g_D$. Thus, $h$ is a $G$-homeomorphism.

    From the construction, we see that at any stage, we have at least two (in fact, infinitely many) choices, which gives different homeomorphisms in the limit. Therefore, we conclude that the group of $G$-homeomorphisms is uncountable.
\end{proof}
\subsubsection{Proof of Theorem~\ref{thm:uncountablehom}}
    We consider two cases.

\vspace{0.3cm}
\noindent{\bf Case (1):} Suppose that $\Lambda(G)$ is not homeomorphic to a Schottky set.
Then $\Lambda(G)$ is either disconnected, equals to $\widehat\C$, or contains a global cut point or an inseparable exact cut pair.

Suppose that $\Lambda(G)$ is disconnected. If each component of $\Lambda(G)$ is a point, then $\Lambda$ is a Cantor set and $\Homeo^+(\Lambda(G))$ is uncountable. 
Thus, we can assume we have a nontrivial component $\Lambda_v$ of $\Lambda(G)$ and let $G_v$ be the stabilizer of $\Lambda_v$. 
Choose a compression disk $D$ in $(N, P)$ that cut off the boundary incompressible piece associated to $\Lambda_v$, and let $C \subseteq \Delta$ be a lift of $\partial D$, where $\Delta$ is a component of $\Omega(G_v)$.
Let $U$ be the Jordan domain bounded by $C$ that is disjoint from $\Lambda_v$.
Let $a \in \stab_{G_v}(\Delta)$ be a loxodromic element. Then $\{a^n U, n \in \N\}$ are pairwise disjoint.
Let $W \subseteq \Delta$ be a Jordan domain that contains $U, aU$ such that $W \cap \Lambda(G) = (U \cup aU) \cap \Lambda(G)$ and $\{a^{2n} W, n \in \N\}$ are pairwise disjoint.
We can define a homeomorphism $\eta: W \rightarrow W$ so that $\eta|_{\partial W} = \id$, $\eta|_{U} = a$, $\eta|_{aU} = a^{-1}$.
With the convention $\eta^0 = \id$, for each sequence $\sigma: \N \longrightarrow \{0,1\}$, we construct a homeomorphism on $\widehat\C$ so that
$$
h_\sigma(z) = \begin{cases}
    a^{2j}\circ \eta^{\sigma(j)}\circ a^{-2j}(z) &\text{if } z\in  a^{2j}W, j \in \N\\
    z &\text{otherwise }
\end{cases}.
$$
Since $h_{\sigma} \neq h_{\sigma'}$ if $\sigma \neq \sigma'$, $\Homeo^+(\Lambda(G))$ is uncountable.

Now suppose $\Lambda(G)$ is connected. 
If $\Lambda(G)$ is a Jordan curve or equals to $\widehat\C$, then $\Homeo(\Lambda(G))$ is uncountable.
Thus, we can assume we have nontrivial JSJ decomposition.
Let us assume we have an inseparable exact cut pair $a, b\in \Lambda(G)$ associated to a loxodromic {\twoended} vertex $v$ in the JSJ decomposition (see \S~\ref{subsec:JSJ}). The argument for a (global) parabolic cut point in the JSJ decomposition is similar.
Let $w$ be adjacent to $v$ with the corresponding limit set $\Lambda_w \subseteq \Lambda(G)$ and stabilizer $G_w$.
Let $v_i, i\in \N$ be the list of adjacent vertices to $w$ that is in the $G_w$ orbit of $v$.
For each $v_i$, we have a corresponding cut pair $a_i, b_i \in \Lambda(G)$.
Note that there are finitely many components $\Delta_{i,1},...\Delta_{i,k}$ of $\Omega(G)$ that has $a_i, b_i$ on their boundary, where we assume the boundary of $\Delta_{i,1}, \Delta_{i,2}$ intersects $\Lambda_w$.
Let $U_i$ be the Jordan domain bounded by the two geodesics connecting $a_i, b_i$ in $\Delta_{i,1}$ and $\Delta_{i,2}$ that is disjoint from $\Lambda_w$.
Let $g_i$ be a generator of $G_{v_i}$, the stabilizer of $a_i, b_i$.
Then for each sequence $\sigma: \N \longrightarrow \{0,1\}$, we can construct a homeomorphism on $\widehat\C$ so that for $z \in \Lambda(G)$, we have
$$
h_\sigma(z) = \begin{cases}
    g_i^{\sigma(i)}(z) &\text{if } z\in U_{i}\\
    z & \text{otherwise}
\end{cases}
$$
Thus, $\Homeo^+(\Lambda(G))$ is uncountable.

\vspace{0.3cm}
\noindent{\bf Case (2):} Suppose that $\Lambda(G)$ is homeomorphic to a Schottky set and contains a Sierpi\'nski carpet as a subset. 
    Let $\widehat G$ be a de-parabolization of $G$ and let $\mathcal{L}$ be the corresponding lamination so that 
    $\Lambda(G) = \Lambda(\widehat G)/\sim_\mathcal{L}$.
    
    Let $(N, P = P_t \cup P_a)$ be the corresponding pared 3-manifold for $G$. Then $(N, P = P_t)$ is the corresponding pared 3-manifold for $\widehat G$.

    By Proposition~\ref{prop:finitehierarchy}, after a finite iteration of the JSJ and/or incompressible decomposition, we have a vertex group $G_v$ whose limit set $\Lambda_v$ is a Sierpi\'nski carpet. 
    By Proposition~\ref{prop:uncountableGhom}, the group of $G_v$-homeomorphisms of $\Lambda_v$ is uncountable. Since each $G_v$-homeomorphism is in $\Homeo(\Lambda(\widehat G), \mathcal{L})$, we conclude that $\Homeo(\Lambda(\widehat G), \mathcal{L})$ is uncountable. Therefore, $\Homeo^+(\Lambda(G))$ is uncountable too.
\qed

\subsection{Obstructions for universality}\label{sec:obstruction}
In this subsection, we show there are two obstructions for (relative) universality: the existence of carpet subsets and the existence of non-homogeneous rank-two cut points.

\subsubsection{Carpet obstruction for universality}
\begin{theorem}\label{thm:nonqs}
    Let $G$ be a geometrically finite Kleinian group. Suppose that $\Lambda(G)$ contains a carpet and $\Lambda(G) \neq \widehat\C$. Then there exists a geometrically finite Kleinian group $G'$ so that
    \begin{itemize}
        \item $(\Lambda(G), \mathfrak{P}(G))$ is ambiently homeomorphic to $(\Lambda(G'), \mathfrak{P}(G'))$;
        \item $(\Lambda(G), \mathfrak{P}(G))$ is not quasisymmetrically homeomorphic to $(\Lambda(G'), \mathfrak{P}(G'))$.
    \end{itemize}

    Moreover, the homeomorphism class of $(\Lambda(G), \mathfrak{P}(G))$ contains infinitely many quasisymmetric classes in $\mathfrak{X}_{\gf}$.
\end{theorem}

As immediate corollaries, we have
\begin{cor}\label{cor:nonqs}
    Let $G$ be a geometrically finite Kleinian group. 
    Suppose that $\Lambda(G)$ contains a carpet and $\Lambda(G) \neq \widehat\C$. Then $(\Lambda(G), \mathfrak{P}(G))$ is neither relatively quasisymmetrically universal nor relatively quasiconformally universal in $\mathfrak{X}_{\gf}$.

    In particular, if $G$ is in addition convex cocompact, then $\Lambda(G)$ is neither quasisymmetrically universal nor quasiconformally universal in $\mathfrak{X}_{\cc}$.
\end{cor}

The lack of quasisymmetric/quasiconformal universality comes from the presence of Sierpi\'nski carpets. Indeed, Bourdon and Kleiner construct 
many examples of non quasi-isometric convex-cocompact Kleinian groups with Sierpi\'nski carpet limit sets \cite{BK15}. 
We provide a different approach of the latter that will also be used for the general case.

\begin{lem}\label{lem:nonQScarpet} Let $(M,P)$ be an acylindrical pared 3-manifold uniformized by a minimally parabolic geometrically finite Kleinian group.
Let $\partial_0 M = \partial M \setminus P$. 
Then there exists an acylindrical pared $3$-manifold $(M', P')$ uniformized by a minimally parabolic, geometrically finite Kleinian group with the following properties.
\ben
\item There is a continuous involution $f:M'\to M'$ such that $\partial_0 M'/(f)$ is homeomorphic to $\partial_0 M$.
\item The fundamental groups of $M$ and $M'$ are not quasi-isometric.
\een
\end{lem}
\begin{proof} 
Let us consider a minimally parabolic geometrically finite Kleinian group $G$ with Schottky limit set and that uniformizes $(M, P)$. It follows that the limit set is a round Sierpi\'nski carpet. 
Let $S_1,\ldots , S_n$ be the connected components of $\partial_0 M$. By assumption, these are closed surfaces of genus at least $2$. Pick $x\in M$, and consider curves $c_1,\ldots, c_n$
that join $x$ to each surface $S_j$. Let $X= \{x\} \cup (\bigcup_j (c_j\cup S_j)) \subset M$ that is a connected set such that $\pi_1(X,x)$ embeds in $\pi_1(M,x)$. 

Let $g\in \pi_1(M)\setminus \pi_1(X)$. Since $\pi_1(X)$ is finitely generated and $\pi_1(M)$ is LERF \cite[Theorem 9.2]{Ago13}, there is a finite index subgroup  $G''$ of $G= \pi_1(M)$ that contains 
$\pi_1(X)$ but not $g$. 
This yields a covering $p:M''\to M$ of finite degree $d\ge 2$. 
Let $S_1'',..., S_m''$ be the components of $\partial_0 M''$. Since $\pi_1(X)\subset \pi_1(M'')$, after relabeling the indices if necessary, we may assume $S_i''$ is homeomorphic to $S_i$ for all $i = 1,..., n < m$.
Let us note that the limit set of $G''$ coincides with that of $G$, so it is a round
Sierpi\'nski carpet and $M''=\chull(\Lambda(G''))/G''$ has totally geodesic boundary.  

We consider its copy $(-M'')$ with opposite orientation, that comes with a homeomorphism $f:M''\to (-M'')$.
Let $M'= M'' \sqcup (-M'')/ (f|_{\cup_{n < j \le m} S_j''})$ be the double of $M''$ glued along the surfaces in $\partial_0 M''\setminus ( \cup_{1\le j\le n}S_j'' )$. 
The components of $\partial_0 M'$ correspond to $\{S_j'',f(S_j''),\ 1\le j\le n\}$. This concludes the first point.

The manifold $M'$ is uniformized by the index $2$ Kleinian subgroup $G'$ of the group generated by $\langle G'', r_{n+1}\ldots, r_m\rangle$ where $r_j$ is the conformal reflection with respect to the boundary round circle of a component $\Delta_j$ of $\Omega({G''})$ such that $\Delta_j/\stab_{G''}(\Delta_j)$ corresponds to $S_j''$. Therefore, $\La({G'})$ is also a round Sierpi\'nski carpet. 
 
 If $\pi_1(M')$ and $\pi_1(M'')$ were quasi-isometric, then there would be a quasisymmetric homeomorphism $h:\Lambda({G'}) \to \Lambda({G''}) = \Lambda(G)$. It follows from \cite{BKM09} that $h$ would be a 
 M\"obius map. 
 This induces an isomorphism between $\Conf(\Lambda({G'}))$ and $\Conf(\Lambda({G''}))$. On the other hand, $\Conf(\Lambda({G''}))$ is an infinite index subgroup of $\Conf(\Lambda({G'}))$. Since any geometrically finite Kleinian group with carpet limit set is co-Hopfian (see \cite[Theorem 1.1]{OP98} or \cite[Theorem 1]{PW99}, cf. \cite{DS08}), this is a contradiction.
 It follows that $\pi_1(M)$ and $\pi_1(M')$ are not quasi-isometric.
 \end{proof}

 \begin{proof}[Proof of Theorem~\ref{thm:nonqs}]
    Let $G$ be a geometrically finite Kleinian group that is uniformizing a pared 3-manifold $(N, P)$. Since $\Lambda(G)$ contains a carpet and $\Lambda(G) \neq \widehat\C$, there is a subgroup $H\le G$ with de-parabolized Kleinian group $\widehat H$ whose limit set is a Sierpi\'nski carpet by Proposition~\ref{prop:finitehierarchy}.
    Let $(M, P_M^a \sqcup P_M^t)$ be the corresponding pared 3-manifold for $H$, where $M \subseteq N$ and $P_M^a, P_M^t$ are the union of annuli and tori components of $P \cap M$.
    Then $(M, P_M^t)$ is the pared 3-manifold for $\widehat H$ and is acylindrical.
    We may assume that the JSJ decomposition of $(N, P)$ has a component $(M, P_M^a \sqcup P_M^t)$.

    Let $(M', P_{M'}^t)$ be the pared $3$-manifold given by Lemma~\ref{lem:nonQScarpet}. We construct a new manifold $(N', P')$ by 
    \begin{itemize}
        \item first gluing to $M'$ along $\partial_0 M'$ twice the components of $N\setminus\setminus M$ through the degree $2$ map $\partial_0 M'\to \partial_0 M$;
        \item then adding twice the paring $P_M^a$ to $\partial_0 M'$ through the degree $2$ map $\partial_0 M'\to \partial_0 M$.
    \end{itemize}
    Let $G'$ be the geometrically finite Kleinian group uniformizing $(N', P')$.

    The JSJ decompositions of $N$ and $N'$ yield a topological model of their limit sets as a tree of compacta where each compact piece corresponds to the limit set of a vertex of the JSJ \cite[Theorem 4.2]{HPW16}. By construction, we have just replaced copies of the limit set of the conjugacy class of the representation of $\pi_1(M)$ by the class of $\pi_1(M')$ that are homeomorphic. This implies that the pairs $(\Lambda(G), \mathfrak{P}(G))$ and $(\Lambda(G'), \mathfrak{P}(G'))$ are ambiently homeomorphic.
    If they were quasisymmetrically homeomorphic, then, after conjugation, the corresponding subpairs $(\Lambda(H), \mathfrak{P}(H))$ and $(\Lambda(H'), \mathfrak{P}(H'))$ associated to $M, M'$ would also be quasisymmetrically homeomorphic. 
    Note that $\Lambda(H)$ and $\Lambda(H')$ are homeomorphic to Schottky sets. Thus by \cite[Theorem 6.6]{haiss:lec}, $\pi_1(M)$ and $\pi_1(M')$ are commensurable. But this is a contradiction to Lemma~\ref{lem:nonQScarpet}.

    For the moreover part, we can iterate the construction of the new group in our proof and obtain infinitely many quasisymmetric classes of such limit sets.
\end{proof}

\subsubsection{Cut point obstruction for universality}
Recall that a rank-two cut point $a \in \Lambda(G)$ is {\em homogeneous} if all components $\mathcal{C}^*_{a} = \mathcal{C}_{a}\setminus \{a\}$ are homeomorphic, where $\mathcal{C}_{a} \subseteq \Lambda(G)$ is the connected component that contains $a$.
\begin{theorem}\label{thm:nonqscutpoint}
    Let $G$ be a geometrically finite Kleinian group. Suppose that $\Lambda(G)$ contains a non-homogeneous rank-two cut point. Then there exists a geometrically finite Kleinian group $G'$ so that
    \begin{itemize}
        \item $(\Lambda(G), \mathfrak{P}(G))$ is homeomorphic to $(\Lambda(G'), \mathfrak{P}(G'))$;
        \item $(\Lambda(G), \mathfrak{P}(G))$ is not quasisymmetrically homeomorphic to $(\Lambda(G'), \mathfrak{P}(G'))$.
    \end{itemize}

    Moreover, the homeomorphism class of $(\Lambda(G), \mathfrak{P}(G))$ contains infinitely many quasisymmetric classes in $\mathfrak{X}_{\gf}$.
\end{theorem}

The limit sets in Theorem~\ref{thm:nonqscutpoint} are {\bf not} ambiently homeomorphic. Thus, the existence of non-homogeneous rank-two cut point only provides obstruction to quasisymmetric universality.
\begin{cor}\label{cor:nonqscutpoint}
    Let $G$ be a geometrically finite Kleinian group. 
    Suppose that $\Lambda(G)$ contains a non-homogeneous rank-two cut point. Then $(\Lambda(G), \mathfrak{P}(G))$ is not relatively quasisymmetrically universal in $\mathfrak{X}_{\gf}$.
\end{cor}

The lack of quasisymmetric universality comes from coarsely incompatible order structure at rank-two cut points.

Let $G$ be a geometrically finite Kleinian group and assume that $a\in\La(G)$ is a rank-two parabolic fixed point that is also a cut point of the component $\mathcal{C}_{a}\subseteq \La(G)$ it belongs to. Each component of $\mathcal{C}^*_{a} = \mathcal{C}_{a}\setminus \{a\}$ is two-ended and each end accumulates at $a$. 
Moreover, if $L\in \pi_0(\cC^*_a)$ then two components of $\C\setminus L$ separate $\cC^*_a\setminus\{L\}$ into two infinite collections. Thus, we may label
$\cC^*_a =\{L_n\}_n$ so that, for each $n\in\Z$ and all $k,m >0$, $L_n$ separates $L_{n-k}$ from $L_{n+m}$; see Figure~\ref{fig:ranktwoSch}, right.

\begin{prop}\label{prop:inducedbilipschitz}
    Let $h: (\Lambda(G), \mathfrak{P}(G)) \rightarrow (\Lambda(G'), \mathfrak{P}(G'))$ be quasisymmetric. Then for each rank-two cut point $a$, $h$ is coarsely order-respecting at $a$, i.e., the induced map $h_*: \pi_0(\mathcal{C}^*_{a}) \cong \Z \rightarrow \pi_0(\mathcal{C}^*_{h(a)}) \cong \Z$ has the form $h_*(n) = \epsilon n + c + O(1)$, where $\epsilon = \pm 1$.
\end{prop}
\begin{proof}
    First, we show that $h_*$ is bi-Lipschitz.
    Let us normalize so that $a = h(a) = \infty$. Suppose that both $h, h^{-1}$ are $\eta$-quasisymmetric. Let us fix the notation and label the components of $\Lambda(G) \setminus \{\infty\}$ and $\Lambda(G') \setminus \{\infty\}$ by $L_n$ and $L'_n$, $n \in \Z$, as above. We have $h(L_n) = L'_{h_*(n)}$.
    By $G$-invariance, there exists $M$ so that $d(L_n, L_{n+1}) \leq M$ for all $n$.
    Pick two points $x_n \in L_n, x_{n+1} \in L_{n+1}$ with $d(x_n, x_{n+1}) = d(L_n,L_{n+1})$,
and let $\gamma$ be the line segment connecting $h(x_n), h(x_{n+1})$. Then $\gamma$ intersects $L_s'$ for each $s$ between $h_*(n)$ and $h_*(n+1)$, pick one such intersection point and denote it by $y_s$. Then $d(h(x_n), y_s) < d(h(x_n), h(x_{n+1}))$.
    Thus $d(x_n, h^{-1}(y_s)) \leq \eta(1) d(x_n, x_{n+1}) \le \eta(1)M$. Since there are only finitely many numbers $k$ with $d(x_n, L_k) \leq \eta(1)M$, we conclude that $|h_*(n)-h_*(n+1)|$ is uniformly bounded. Thus $h_*$ is Lipschitz. Similarly, $h_*^{-1}$ is Lipschitz. Thus $h_*$ is bi-Lipschitz.
    
    The lemma now follows from the classification of bi-Lipschitz bijections on $\Z$, see \cite[Theorem 1]{BS15}.
\end{proof}

\begin{proof}[Proof of Theorem~\ref{thm:nonqscutpoint}]
    Let $(N, P)$ be the corresponding pared 3-manifold for $G$, and let $T$ be the thickened torus in the JSJ decomposition for the non-homogeneous rank-two cut point $a$.
    We first treat the case where $(N, P)$ is obtained by gluing $k$ pared $3$-manifolds $(M_0, P_0),..., (M_{k-1}, P_{k-1})$ to $T$ along $k$ cyclically ordered parallel annuli on the outer boundary of $T$, i.e., each piece is attached to $T$ along a single annulus. 
    Then the components $\mathcal{C}^*_{a}$ are linearly ordered by $$... L_0, L_1,..., L_{k-1}, L_0, L_1,..., L_{k-1},...,$$ where $L_i$ is the component associated to $M_i$. Since we assume $a$ is non-homogeneous, there is a nontrivial partition $\{0,1,..., k-1\} = \mathcal{I} \sqcup \mathcal{J}$ so that $L_0$ is homeomorphic to $L_i, i \in \mathcal{I}$ but not homeomorphic to $L_j, j \in \mathcal{J}$.

    Let $(N', P')$ be the pared manifold obtained by gluing 
    $$
    (M_0, P_0), (M_1, P_1), ..., (M_{k-1}, P_{k-1}), (M_0, P_0)
    $$ 
    along $k+1$ cyclically ordered parallel annuli on the outer boundary of $T$.
    Let $G'$ be the corresponding geometrically finite Kleinian group. Since the JSJ decomposition yields a topological model of the limit sets as a tree of compacta \cite[Theorem 4.2]{HPW16}, one sees that $(\Lambda(G), \mathfrak{P}(G))$ is homeomorphic to $(\Lambda(G'), \mathfrak{P}(G'))$.
    Suppose they were quasisymmetrically homeomorphic. We may assume that the quasisymmetric map $h$ sends the rank-two cut point $a$ to $a'$. Then the quasisymmetric homeomorphism is coarsely order-respecting at $a$ by Proposition~\ref{prop:inducedbilipschitz}. Since $h$ is a homeomorphism, by construction, 
    $$
    h_*(\{n: n \mod k = i, i \in \mathcal{I}\}) = \{n: n \mod k+1 = i, i \in \mathcal{I}\cup \{k\}\}.
    $$
    This is a contradiction, as $\{n: n \mod k = i, i \in \mathcal{I}\}$ has strictly lower density $0 < |\mathcal{I}|/k < (|\mathcal{I}|+1)/(k+1) < 1$, and the theorem follows.

    The case in which a piece is attached to $T$ along several annuli can be handled by applying the following modification. Suppose that $(M_0,P_0)$ is glued to $T$ along $l$ distinct annular components. We consider $l$ copies $(N_1,P_1),\ldots , (N_l,P_l)$ of $(N,P)$ and mark the 
     $l$ copies $T^1,..., T^l$ of $T$. Let $(M_0',P_0')$ be a further copy of $(M_0,P_0)$ with its $l$ gluing annuli $A_1',\ldots, A_l'$. 
     We now construct a new manifold $(N',P')$ by gluing $(M_0',P_0')$ to $(N_j,P_j)$ by identifying $A_j'$ to a new disjoint and parallel annulus in $T^j$, in such a way that each resulting torus $T^j$ contains now $k+1$ gluing parallel annuli.  The theorem now follows by the same argument as in the previous case. 

    For the moreover part, we can iterate the construction of the new group in our proof and obtain infinitely many quasisymmetric classes of such limit sets.
\end{proof}

\section{Virtual topological rigidity}\label{sec:vtr}
In this section, we prove that a geometrically finite Kleinian group with carpet-free Schottky limit set is virtually topologically rigid.

Let $G$ be a geometrically finite Kleinian group. 
Recall that a subgroup $H \le G$ is a carpet-quotient subgroup if its de-parabolization has carpet limit set, and a subset $A \subseteq \Lambda(G)$ is a dynamical carpet-quotient if $A = \Lambda(H)$ for some carpet-quotient subgroup $H$.
Let 
$$
\HQ^+(\Lambda(G)) \le \Homeo^+(\Lambda(G))
$$ 
be the subgroup of homeomorphisms $f: \Lambda(G) \rightarrow \Lambda(G)$ that are quasisymmetric on each dynamical carpet-quotient.
Note that if $\Lambda(G)$ is carpet-free, then $\HQ^+(\Lambda(G)) = \Homeo^+(\Lambda(G))$. Also note that if $\Lambda(G)$ is a Schottky set, then any homeomorphism on $\Lambda(G)$ extends to an ambient homeomorphism.

\begin{theorem}\label{thm:carpetfreeSchottkyimpliesrigid}
    Let $G$ be a geometrically finite Kleinian group with $\Lambda(G)$ homeomorphic to a Schottky set.
    Then $G$ has finite index in $\HQ^+(\Lambda(G))$.
    
    In particular, if $G$ is a geometrically finite Kleinian group with $\Lambda(G)$ homeomorphic to a carpet-free Schottky set, then $G$ is virtually topologically rigid.
\end{theorem}
The proof is by induction on the height of the hierarchical decomposition of the group discussed in \S~\ref{subsec:sa}. We start by presenting a simple mechanism to show finite index in \S~\ref{subsec:fiandfe}. We prove two key rigidity results in \S~\ref{sec:ccfrigidity} and \S~\ref{sec:bcrigidity}, which serve as essential inputs for the induction steps associated to the JSJ decomposition and the incompressible decomposition, respectively. The final induction argument is presented in \S~\ref{sec:proofThmRigid}.

\subsection{Finite index and finite extensions}\label{subsec:fiandfe}
We first record the following proposition, which gives a simple mechanism to show finite index of $G$ in $H$.
\begin{prop}\label{prop:finiteindex}
Let $G \le H$ be two groups, and let $X$ be a set  so that
\begin{enumerate}
    \item\label{prop:finiteindex:eqn1} both $G$ and $H$ act on $X$;
    \item\label{prop:finiteindex:eqn2} $X/G$ is finite;
    \item\label{prop:finiteindex:eqn3} for each $x\in X$, $\stab_G (x)$ has finite index in  $\stab_{H} (x)$.
\end{enumerate}
Then $G$ has finite index in $H$.
\end{prop}
\begin{proof}
By our assumption \eqref{prop:finiteindex:eqn2}, let $\{x_j,\ 0\le j \le k\}$ be a finite set of representatives of $X/G$.
For each $j\in \{0, ..., k\}$, let $T_j$ be a transversal of $\stab_G( x_j) \backslash \stab_H (x_j)$.
Note that $|T_j|$ is finite by our assumption \eqref{prop:finiteindex:eqn3}.

Let $H_j=\{h\in H,\ h(x_0)=x_j\}$, and let $J=\{j\in \{0,\ldots, k\},\ H_j\ne \emty\}$, i.e., $J$ contains the indices so that $x_j$ is in the orbit of $x_0$ under $H$.
For each $j\in J$, let us fix an element $h_j\in H_j$.

Let $h\in H$. Since $H$ acts on $X$ by our assumption \eqref{prop:finiteindex:eqn1}, $h(x_0) \in X$. Thus, there is some $j\in J$ such that $h(x_0)\in Gx_j$. Therefore, we may find $g_1\in G$ such that $g_1h\in H_j$.
Therefore, $g_1h h_j^{-1} \in \stab_H (x_j)$. Thus, by assumption, there is some $t\in T_j$ and $g_0\in \stab_G(x_j)$ such that
$g_1hh_j^{-1}= g_0 t$ and $h= g_1^{-1}g_0 t h_j$.
It follows that we may find $g\in G$ and $t\in T_j$ such that $h= g th_j$. Therefore, $$H = \bigcup_{j\in J}  \left[ \bigcup_{t\in T_j} G th_j \right]$$
so that $G$ has finite index in $H$. 
\end{proof} 

The following proposition shows that finite extensions of Kleinian groups are uniformly quasi-M\"obius,  i.e., the metric cross-ratios of four points  are controlled \cite{vam}.
\begin{prop}\label{finitextension} Let $G$ be a geometrically finite Kleinian group. Let $H$ be a subgroup of homeomorphism group of $\Lambda(G)$ that contains $G$ as a finite index subgroup.
Then $H$ is a group of uniform quasi-M\"obius maps. \end{prop}

\begin{proof}
Let us prove that the natural action of $H$ on $\Lambda(G)$ is a convergence action. Let us consider a sequence of distinct elements $(h_n)$. Since $G$ has finite index in $H$,
we may assume that, up to a subsequence,  there exists $h\in H$ and $g_n\in G$ such that $h_n=hg_n$. Since $G$ is a convergence group, there exists $a,b\in \Lambda(G)$
such that $(g_n)$ tends uniformly to $b$ on compact subsets of $\Lambda(G)\setminus\{a\}$. Therefore, $(hg_n)$ is also a convergence subsequence, that converges to
$h(b)$ on compact subsets of $\Lambda(G)\setminus\{a\}$. It follows that $H$ is a convergence group with $\Lambda(H)=\Lambda(G)$.

Since $G$ is geometrically finite, it follows that $H$ is too (every limit point is either conical or bounded parabolic since this is the case for $G$, cf. \cite{beardon:maskit}). 
Therefore, $H$ is a finitely generated hyperbolic group  relative to the stabilisers of its parabolic points. Let us consider
a Cayley graph for $H$. We glue horoballs following Groves and Manning to obtain a cusped space $X$ and a cusp-uniform action of $H$ (and $G$) \cite{groves:manning:dehn}. 
 According to \cite[Prop.\,3.21]{haiss:lec}, there is a $G$-equivariant quasi-isometry between this cusped space $X$ and the convex  hull of $\Lambda(G)$ in $\Hyp^3$.
 
Note that the action of $H$ on $\partial X$ equipped with a visual distance is uniformly quasi-M\"obius. Moreover, 
the aforementioned quasi-isometry extends as a quasi-M\"obius boundary map between $\partial X$ and $\Lambda(G)$, so conjugates the action of $H$ on $\partial X$ to 
a uniform quasi-M\"obius action on $\Lambda(G)$. 
\end{proof}

\begin{cor}\label{cor:toprig:finitext}
 Let $G$ be a geometrically finite Kleinian group with Schottky limit set. If the group $H$ of homeomorphisms of $\Lambda(G)$ is a finite extension of $G$, then $H$ is Kleinian and
 $\Lambda(H)=\Lambda(G)$. In particular, $H$ is a topologically rigid Kleinian group and any homeomorphism $h\in H$ that fixes three points is the identity.
 \end{cor}

\begin{proof} According to the above proposition, $H$ acts as a group of uniform quasi-M\"obius maps. The result now follows from \cite[Theorem 6.6]{haiss:lec}.
\end{proof}

\subsection{Convex cocompact Fuchsian groups with acylindrical laminations}\label{sec:ccfrigidity}
In this subsection, we prove a topological rigidity theorem (Theorem~\ref{thm:toprigidityconvexcocompact}) for non-elementary convex cocompact Fuchsian groups equipped with an acylindrical lamination. This result provides a key ingredient for the inductive steps in the proof of Theorem~\ref{thm:carpetfreeSchottkyimpliesrigid}.

The convex cocompact pieces arise from the JSJ decomposition of the de-parabolized group associated to the original Kleinian group (see \S~\ref{subsec:rankonedeparalam} and \S~\ref{subsec:JSJ}). The lamination constructed on the de-parabolized group naturally restricts to laminations on these convex cocompact pieces.

On each such piece, the induced lamination may contain leaves associated to proper arcs in addition to simple closed curves. We discuss the structure and role of these proper arcs in \S~\ref{subsec:lamination}.

\subsubsection{Lamination for convex cocompact Fuchsian groups}\label{subsec:lamination}
Let $G$ be a non-elementary convex cocompact Fuchsian group. 
Then the convex core 
$$
S:= \core(M) = \chull(\Lambda(G))/G
$$ 
is a compact hyperbolic surface with potentially non-empty totally geodesic boundary,
where the convex hull $\chull(\Lambda(G))$ is considered in $\Hyp^3$.

Note that $\core(M)$ is a closed hyperbolic surface if and only if $G$ is cocompact.
Thus, the associated pared manifold $(N,P)$ is
$$
N = S \times [-1,1], \;\; P = \emptyset.
$$

\subsection*{Mixed multicurves}
Let $\Sigma$ be a compact surface.
Recall that a {\em proper arc} on $\Sigma$ is a map $\alpha:[0,1] \longrightarrow \Sigma$ such that $\alpha^{-1}(\partial \Sigma) = \{0,1\}$ and is an embedding on its interior. We identify an arc with its image. 
The homotopy class of a proper arc is taken to be the homotopy class within the class of proper arcs.
For our purposes, we assume the homotopy is {\em relative to the boundary}, i.e., the endpoints stay fixed throughout the homotopy.
An arc $\alpha$ is {\em essential} if it not homotopic into a boundary component of $\Sigma$.

Similarly, we say a simple closed curve $\gamma$ is {\em nontrivial} if it is not homotopic to a point, and we say $\gamma$ is {\em essential} if it is not homotopic to a point nor to a boundary component of $\Sigma$.

Suppose that $X$ is a compact hyperbolic surface with totally geodesic boundary.
Then each homotopy class of an essential proper arc or of an essential loop has a geodesic representative. 
The geodesic representative is unique and is simple.
Homotopy classes of essential proper arcs and essential simple closed curves are said to be disjoint if their geodesic representatives do not intersect.

A {\em mixed multicurve} $\mathcal{C}$ on $X$ is a finite collection of pairwise disjoint homotopy classes of essential proper arcs and/or essential simple closed curves.
We call the union of geodesic representatives of a mixed multicurve a {\em geodesic mixed multicurve}. 
For convenience, we will also denote by $\mathcal{C}$ a geodesic mixed multicurve.

\subsection*{Acylindrical pairs of a mixed multicurve}
Recall that $G$ is a non-elementary convex cocompact Fuchsian group. 
Let $\D^-$ and $\D^+$ be the open unit disk and the exterior of the closed unit disk.
Denote by $\chull_{\D^-}(\Lambda(G)) \subseteq \D^-$ and $\chull_{\D^+}(\Lambda(G))\subseteq \D^+$ the convex hulls of the limit set in $\D^-$ and in $\D^+$ with respect to the Poincar\'e metric.
Let
$$
S^\pm:= \chull_{\D^\pm}(\Lambda(G))/G.
$$
Then $S^\pm$ are compact hyperbolic surfaces with totally geodesic boundary, and are canonically isometric to $S$.
Note that there is a canonical orientation on each $S^\pm$ coming from the orientation of $\D^\pm$ and the natural identification between $S^+$ and $S^-$ reverses the orientation.

There is a natural identification of $S^\pm$ with $S \times \{\pm 1\}$ in $N = S \times [-1, 1]$.
A pair of geodesic mixed multicurves $\mathcal{C}^\pm$ on $S^\pm$ is called {\em acylindrical} if there is no embedded cylinder in $N = S \times [-1,1]$ whose boundary components $\gamma^\pm$ are nontrivial simple closed curves in $S^\pm \setminus \mathcal{C}^\pm$. Note that we allow the possibility that $\gamma^\pm$ is homotopic to $\partial S^\pm$.
    Let $\pi^\pm: S^\pm \rightarrow S$ be the canonical identification. Then we have the following criterion for acylindricity.
\begin{prop}
    A pair $\mathcal{C}^\pm$ on $S^\pm$ is acylindrical if and only if each component of $S \setminus (\pi^+(\mathcal{C}^+) \cup \pi^-(\mathcal{C}^-))$ is simply connected.
\end{prop}
\begin{proof}
    Suppose that $S \setminus (\pi^+(\mathcal{C}^+) \cup \pi^-(\mathcal{C}^-))$ is not simply connected. Then there exists a nontrivial simple closed curve $\gamma \subseteq S$ disjoint from $\pi^+(\mathcal{C}^+) \cup \pi^-(\mathcal{C}^-)$. Then $\gamma \times [-1,1]$ is an embedded cylinder, so the pair $\mathcal{C}^\pm$ is cylindrical.

    Conversely, suppose there is an embedded cylinder. We may assume the boundaries $\gamma^\pm$ are geodesic representatives in $S^\pm$. 
    Since they bound a cylinder in $N$, they are homotopic and so $\pi^+(\gamma^+) = \pi^-(\gamma^-)$ holds. 
    Since $\gamma^\pm$ is disjoint from $\mathcal{C}^\pm$, $\pi^+(\gamma^+)$ is a nontrivial simple closed curve in $S \setminus (\pi^+(\mathcal{C}^+) \cup \pi^-(\mathcal{C}^-))$. Therefore, there is a component of $S \setminus (\pi^+(\mathcal{C}^+) \cup \pi^-(\mathcal{C}^-))$ that is not simply connected.
\end{proof}
It follows that the pair $\mathcal{C}^\pm$ is acylindrical if and only if the closure of each component of $S \setminus (\pi^+(\mathcal{C}^+) \cup \pi^-(\mathcal{C}^-))$ is a compact hyperbolic polygon with totally geodesic boundary (see Figure~\ref{fig:lam}).

\subsection*{Lifting a pair of mixed multicurves}\label{subsec:acylindricallamfuchsian}
By lifting to the universal cover, we obtain a pair of geodesic laminations $\mathcal{L}^\pm \subseteq \D^\pm$. 
We call the pair $\mathcal{L}^\pm$ {\em acylindrical} if the corresponding pair of mixed multicurves is acylindrical.
We also denote the lift of the totally geodesic boundary $\partial S^\pm$ in $\D^\pm$ by $\mathcal{L}^\pm_{\partial} \subseteq \D^\pm$ (see Figure~\ref{fig:lam}).
The following lemma follows immediately from the construction.
\begin{lem}\label{lem:descriptionlamination}
    \noindent\begin{enumerate}[leftmargin=8mm]
    \item A leaf $l \in \mathcal{L}^\pm_{\partial} \subseteq \D^\pm$ is a complete geodesic connecting two boundary points of a component of $\Omega(G)\subset\mathbb{S}^1$. Moreover, let $r$ be the reflection along $\mathbb{S}^1$. Then $r(l) \in \mathcal{L}^\mp_{\partial} \subseteq \D^\mp$ and $\overline{l \cup r(l)}$ form a round circle in $\widehat\C$.
    \item A leaf $l\in \mathcal{L}^\pm$ is either 
    \begin{itemize}
        \item (lift of simple closed curves) a complete geodesic in $\D^\pm$ connecting two points of $\La(G)\subset \mathbb{S}^1$; or
        \item (lift of proper arcs) a geodesic segment in $\D^\pm$ joining two leaves $l_1, l_2 \in \mathcal{L}^\pm_{\partial}$.
    \end{itemize}
    Moreover, for any two leaves $l_1, l_2 \in \mathcal{L}^\pm_{\partial}$, there are finitely many leaves in $\mathcal{L}^\pm$ joining them.
\end{enumerate}
\end{lem}
A leaf $l\in \mathcal{L}^\pm$ is {\em loop type} or {\em arc type} if it is the lift of a simple closed curve or a proper arc.

\begin{figure}[ht]
\captionsetup{width=0.96\linewidth}
  \centering
  \includegraphics[width=0.46\textwidth]{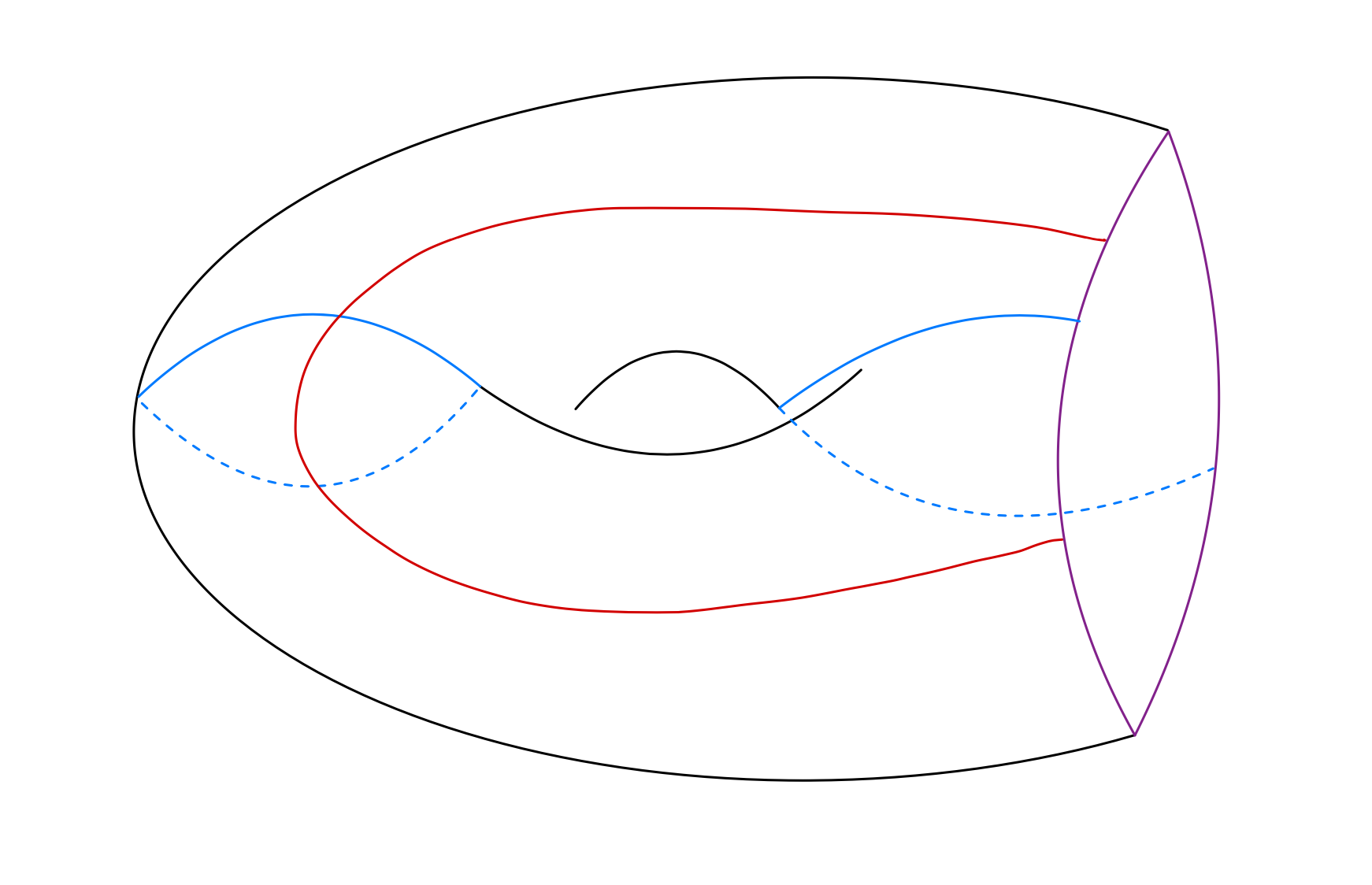}
  \includegraphics[width=0.46\textwidth]{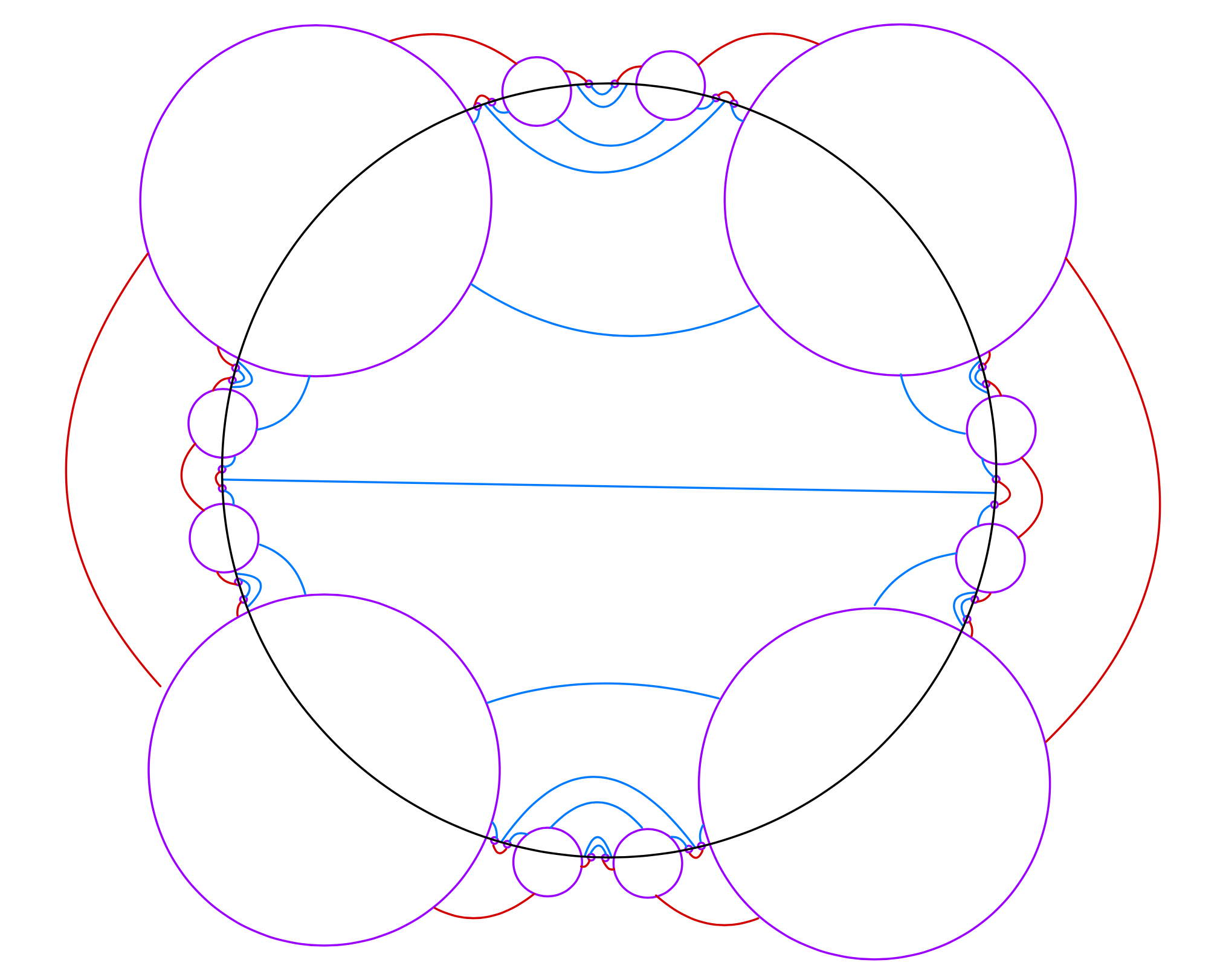}
  \caption{Left: a pair of mixed multicurves, shown in red and blue. For simplicity, both are drawn on a single surface. The pair is acylindrical.
  Right: the associated geodesic laminations. The purple leaves arise from lifts of the boundary components $\partial S^\pm$ and form $\mathcal{L}_\partial = \mathcal{L}_\partial^+ \cup \mathcal{L}_\partial^-$. The red and blue leaves correspond to $\mathcal{L}^+$ and $\mathcal{L}^-$ respectively.}
  \label{fig:lam}
\end{figure}

\subsection*{Lamination-preserving homeomorphisms}
Let $G$ be a convex cocompact Fuchsian group. Let $\mathcal{L}^\pm$ and $\mathcal{L}^\pm_{\partial}$ be the pair of geodesic laminations for a pair of mixed multicurves on $\partial S^\pm$.
We let 
$$
\mathcal{L}: = \mathcal{L}^+ \cup \mathcal{L}^- \text{ and } \mathcal{L}_{\partial}:= \mathcal{L}^+_{\partial} \cup \mathcal{L}^-_{\partial}.
$$
We denote the support of the lamination by $|\mathcal{L}|:= \bigcup_{l\in \mathcal{L}} l$ and $|\mathcal{L}_{\partial}|:= \bigcup_{l\in\mathcal{L}_{\partial}} l$.
Let $$
\partial \mathcal{L}:= \{\partial l = \{x, y\}: l \in \mathcal{L}\} \text{ and }\partial \mathcal{L}_\partial:= \{\partial l = \{x, y\}: l \in \mathcal{L}_\partial\}
$$ 
denote the collection of pairs of endpoints with their corresponding supports
$$
|\partial \mathcal{L}|:= \bigcup_{l\in \mathcal{L}} \partial l \text{ and } |\partial \mathcal{L}_\partial|:= \bigcup_{l\in \mathcal{L}_\partial} \partial l.
$$
Note that $|\partial \mathcal{L}_\partial| \subseteq \Lambda(G)$ but $|\partial \mathcal{L}|$ may contain points on $|\mathcal{L}_{\partial}|$ by Lemma~\ref{lem:descriptionlamination}.
For the lamination $\mathcal{L}^\pm_{\partial}$, we call the component of $\D^\pm \setminus |\mathcal{L}^\pm_{\partial}|$ that projects to $S^\pm$ under the universal covering map the {\em central gap}, and denote them by $U^\pm$.
Note that $U^\pm = \D^\pm$ if $G$ is cocompact.

We consider the group of homeomorphisms that preserve the laminations $\cL\cup \cL_{\partial}$:
\begin{align}\label{eqn:homeoconvexcocompactlamination}
\Homeo^+(\Lambda(G), \mathcal{L}, \mathcal{L}_\partial) = \{h:\widehat\C \rightarrow \widehat\C: &h \text{ is an orientation preserving homeomorphism }\nonumber\\ &\text{with }h(\Lambda(G)) = \Lambda(G), h(|\mathcal{L}|) = |\mathcal{L}|, \nonumber\\ &h(|\mathcal{L}_{\partial}|) = |\mathcal{L}_{\partial}|\}/\sim
\end{align}
where $h_1 \sim h_2$ if $h_1|_{\Lambda(G)\cup |\partial \mathcal{L}|} = h_2|_{\Lambda(G) \cup |\partial \mathcal{L}|}$.

\subsection*{Linked pairs of leaves}
Let $l \in \mathcal{L}$. We define the {\em envelope boundary} of $l$ as 
\begin{itemize}
    \item $\{a^1,a^2\}$ if $l$ connects two points $ a^1,a^2 \in \mathbb{S}^1$; and 
    \item $[a^1,b^1] \cup [a^2,b^2] \subseteq \mathbb{S}^1$ if $l$ joins $l^i \in \mathcal{L}_\partial,  i = 1,2$, with $\partial l^i = \{a^i, b^i\}$ and $(a^i, b^i) \subseteq \mathbb{S}^1 \setminus \Lambda(G)$.
\end{itemize}
Given two pairs of (possibly degenerate) closed intervals $\{I_1, J_1\}$ and $\{I_2, J_2\}$ in $\mathbb{S}^1$, we say they are {\em linked} if the intervals are pairwise disjoint and if every geodesic joining $I_1$ to $J_1$ intersects every geodesic joining $I_2$ to $J_2$ (see Figure~\ref{fig:link}).

Let $\eta : \D^+ \cup \D^- = \widehat\C \setminus \mathbb{S}^1 \longrightarrow \D$ be defined by
    $$
        \eta(z) = \begin{cases}
        z &\text{if }z \in \D^-\\
        \frac{1}{\bar z} & \text{if } z\in \D^+
        \end{cases}.
    $$
\begin{defn}\label{defn:link}
Given a pair of leaves $l_1, l_2 \in \mathcal{L} \cup \mathcal{L}_\partial$, we say they are {\em linked} if either
\begin{enumerate}
    \item\label{eqn:defn:link1} $l_1 \in \mathcal{L}$ and $l_2\in \mathcal{L}_\partial$ (or $l_1 \in \mathcal{L}_\partial$ and $l_2\in \mathcal{L}$) and $\overline{\eta(l_1)} \cap \overline{\eta(l_2)} \neq \emptyset$; or
    \item\label{eqn:defn:link2} $l_1, l_2 \in \mathcal{L}$ and the envelope boundaries of $l_1$ and $l_2$ are linked as pairs of closed intervals in $\mathbb{S}^1$.
\end{enumerate}
We define the collection of linked pairs by
\begin{align}
    \mathcal{X}_{link}&:=\{(l_1, l_2)\in (\mathcal{L} \cup \mathcal{L}_{\partial})^2: l_1, l_2 \text{ are linked}\}, \text{ and } \\
    \partial \mathcal{X}_{link} &:= \{(\partial l_1, \partial l_2): (l_1, l_2) \in \mathcal{X}_{link}\}.
\end{align}
\end{defn}

\begin{figure}[ht]
\captionsetup{width=0.96\linewidth}
  \centering
  \includegraphics[width=0.6\textwidth]{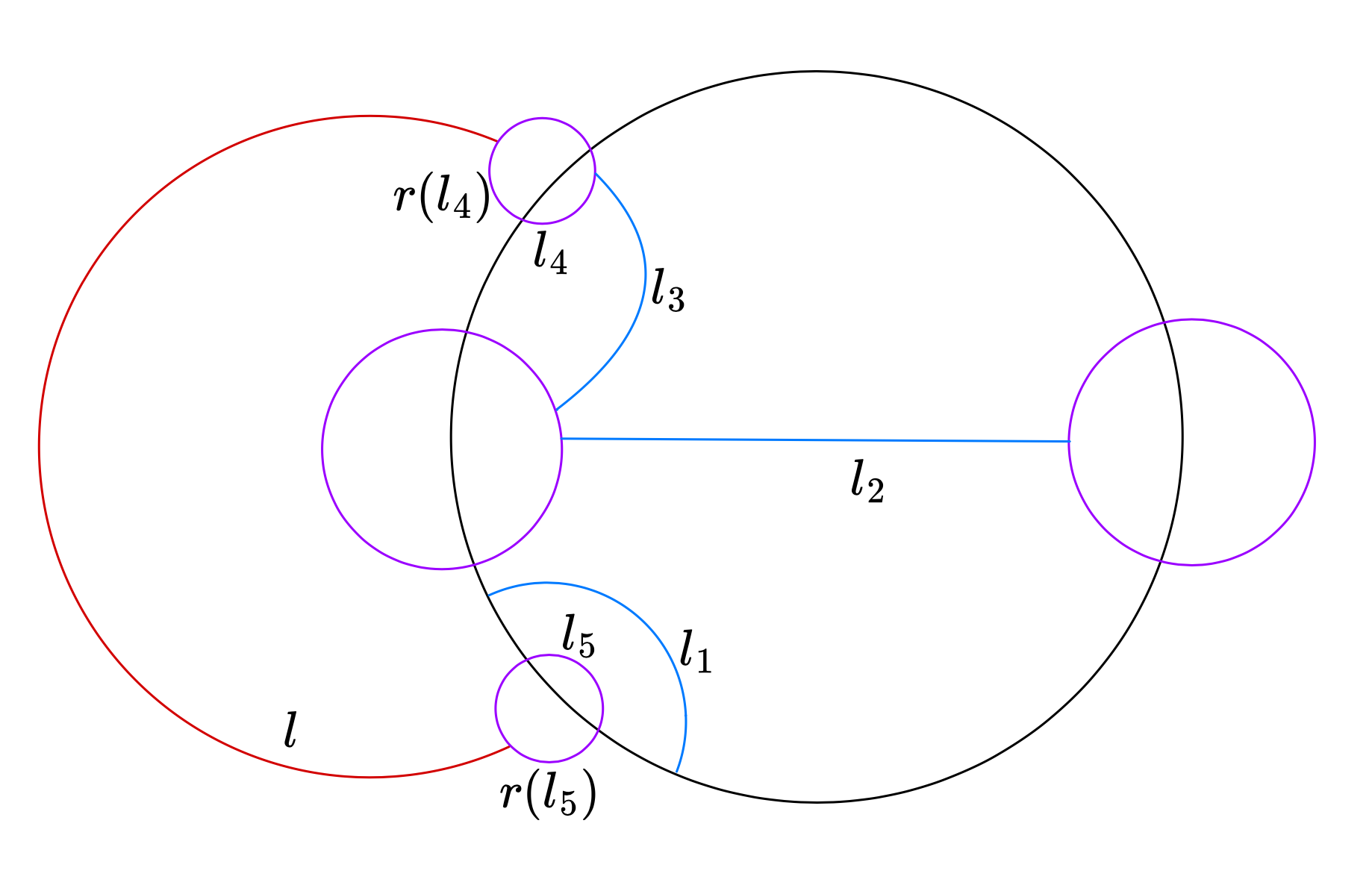}
  \caption{An illustration of the linking condition. The leaf $l$ is linked with $l_1, l_2, l_4, r(l_4), l_5, r(l_5)$, where $r$ is the reflection along $\mathbb{S}^1$. It is not linked with $l_3$, although $l$ intersects $r(l_3)$. Note that $l$ is linked with $l_4$ and $l_4$ is linked with $l_3$.}
  \label{fig:link}
\end{figure}

\begin{lem}\label{lem:actionLinked}
    Let $h \in \Homeo^+(\Lambda(G), \mathcal{L}, \mathcal{L}_\partial)$.
    If $(l_1, l_2) \in \mathcal{X}_{link}$, then $(h(l_1), h(l_2)) \in \mathcal{X}_{link}$. Thus, $\Homeo^+(\Lambda(G), \mathcal{L}, \mathcal{L}_\partial)$ acts on both $\mathcal{X}_{link}$ and $\partial \mathcal{X}_{link}$.
\end{lem}
\begin{proof}
    Suppose $(l_1, l_2)$ is in case~\eqref{eqn:defn:link1} of Definition~\ref{defn:link}. Assume that $l_1 \in \mathcal{L}$ and $l_2 \in \mathcal{L}_\partial$. 
    Thus, either $\partial l_1 \cap l_2 \neq \emptyset$ or $\partial l_1 \cap r(l_2) \neq \emptyset$, where $r$ is the reflection along $\mathbb{S}^1$.
    Hence, $\partial h(l_1) \cap h(l_2) \neq \emptyset$ or $\partial h(l_1) \cap h(r(l_2)) = \partial h(l_1) \cap r(h(l_2)) \neq \emptyset$. Therefore, $(h(l_1), h(l_2)) \in \mathcal{X}_{link}$.

    Suppose $(l_1, l_2)$ is in case~\eqref{eqn:defn:link2} of Definition~\ref{defn:link}.
    Since $h(|\mathcal{L}_\partial|) = |\mathcal{L}_\partial|$, $h$ preserves the cyclic order of the limit set $\Lambda(G)$.
    Therefore, $(h(l_1), h(l_2)) \in \mathcal{X}_{link}$.
\end{proof}

\begin{lem}\label{lem:finitepath}
    Suppose that the pair $\mathcal{L}^\pm$ is acylindrical.
    Then any two leaves $l, l'\in \mathcal{L} \cup \mathcal{L}_\partial$ are connected via a finite path $l^0=l, l^1,..., l^k = l'$ so that $(l^i, l^{i+1}) \in \mathcal{X}_{link}$.
\end{lem}
\begin{proof}
    Let $\mathcal{C}^\pm \subseteq S$ be the geodesic mixed multicurve associated to $\mathcal{L}^\pm$. Since the pair is acylindrical, the closure of each complementary component of $S \setminus (\mathcal{C}^+ \cup \mathcal{C}^-)$ is a compact hyperbolic polygon with totally geodesic boundary. Let $U = \eta(U^+) = \eta(U^-)$ be the central gap of the lamination $\eta(\mathcal{L}_{\partial})$. 
    Then any gap of $\eta(\mathcal{L}\cup \mathcal{L}_\partial)$ in $U$ is a compact hyperbolic polygon, and the union of these polygons gives a tiling of $U$.
    Therefore, we can find a finite path 
    \begin{equation}\label{eqn:finitepath}
    \widetilde{l}^0 = l, \widetilde{l}^1,..., \widetilde{l}^m = l',
    \end{equation}
    so that $\widetilde{l}^i \neq \widetilde{l}^j$ for $i \neq j$ and $\overline{\eta(\widetilde{l}^i)} \cap \overline{\eta(\widetilde{l}^{i+1})} \neq \emptyset$.
    
    Suppose we have two distinct leaves $l_1, l_2\in \mathcal{L} \cup \mathcal{L}_\partial$ with $\overline{\eta(l_1)} \cap \overline{\eta(l_2)} \neq \emptyset$. 
    We claim that 
    \begin{itemize}
        \item either $(l_1, l_2) \in \mathcal{X}_{link}$; or
        \item there exists $l_3 \in \mathcal{L} \cup \mathcal{L}_\partial$ with $(l_1, l_3), (l_3, l_2) \in \mathcal{X}_{link}$.
    \end{itemize}
    Indeed, if $l_1 \in \mathcal{L}$ and $l_2\in \mathcal{L}_\partial$ (or $l_1 \in \mathcal{L}_\partial$ and $l_2\in \mathcal{L}$), then $(l_1, l_2) \in \mathcal{X}_{link}$ by definition.
    Otherwise, $l_1, l_2 \in \mathcal{L}$. Since $\overline{\eta(l_1)} \cap \overline{\eta(l_2)} \neq \emptyset$, either $(l_1,l_2) \in \mathcal{X}_{link}$, or there exists a leaf $l_3 \in \mathcal{L}_\partial$ so that both $\eta(l_1)$ and $\eta(l_2)$ have an endpoint in $\eta(l_3)$. In the latter case, we have $(l_1, l_3), (l_3, l_2) \in \mathcal{X}_{link}$ (see Figure~\ref{fig:link}). This proves the claim.
    
    The lemma follows from the claim and the existence of the sequence in \eqref{eqn:finitepath}.
\end{proof}

\subsubsection{Topological rigidity for convex cocompact Fuchsian groups}
The main theorem of this section is the following.
We remark that when $G$ is cocompact, the theorem follows from \cite[Lemma 21]{KK00}. The proof for the convex cocompact case follows a similar strategy.
\begin{theorem}\label{thm:toprigidityconvexcocompact}
    Let $G$ be a convex cocompact Fuchsian group, and let $\mathcal{L}^\pm$ be a pair of geodesic laminations associated to a pair of mixed multicurves, which is acylindrical.
    Let $h \in H:= \Homeo^+(\Lambda(G), \mathcal{L}, \mathcal{L}_\partial)$ and let $(l_1, l_2) \in \mathcal{X}_{link}$.
    Suppose that 
    $$
    h|_{\partial l_1 \cup \partial l_2} = \id.
    $$
    Then $h = \id$.   
\end{theorem}
\begin{proof}
    Consider the sets
    $$
    A_{l_1}^\pm := \{l\in \mathcal{L}^\pm\cup \mathcal{L}^\pm_\partial: (l_1, l) \in \mathcal{X}_{link}\}, \text{ and } A_{l_1} = A_{l_1}^+ \cup A_{l_1}^-.
    $$ 
    By Lemma~\ref{lem:actionLinked}, the homeomorphism $h$ induces a bijection $h:A_{l_1}^\pm \rightarrow A_{l_1}^\pm$.
    We claim that $h|_{\partial l} = \id$ for all $l \in A_{l_1}$.
    \begin{proof}[Proof of the claim]
        We first note that $A_{l_1}^\pm$ is discrete and we have a linear order on each $A_{l_1}^\pm$. More precisely, this means that there exists a continuous map
        $$
        \phi^\pm: |A_{l_1}^\pm| = \bigcup_{l\in A_{l_1}^\pm} l \longrightarrow \Z
        $$
        so that the fiber of a point is a leaf, and $\phi^\pm(l) < \phi^\pm (l') < \phi^\pm(l'')$ if and only if $l'$ separates $l, l''$ in the central gap $U^\pm$.
        Since $h$ preserves the cyclic order of $\Lambda(G)\subset \mathbb{S}^1$, it preserves the linear order on $A_{l_1}^\pm$.
        We prove the claim on a case by case basis.

        \noindent {\bf Case (1)}: $l_1 \in \mathcal{L}^+$, and $l_2 \in \mathcal{L}^-$ (or similarly $l_1 \in \mathcal{L}^-$, and $l_2 \in \mathcal{L}^+$).
        Since $h$ preserves the linear order on $A_{l_1}^-$, and $h|_{\partial l_1 \cup \partial l_2} =\id$, we conclude that $h|_{\partial l} = \id$ for each $l \in A_{l_1}^-$.
        Note that $\mathcal{A}_{l_1}^+$ is either empty if $l_1$ is loop type; or consists of two leaves in $\mathcal{L}^+_\partial$ if $l_1$ is arc type. Since $h|_{\partial l_1} = \id$, we have $h|_{\partial l} = \id$ for each $l \in A_{l_1}^+$. The claim follows.

        \noindent {\bf Case (2)}: $l_1 \in \mathcal{L}^+$, and $l_2 \in \mathcal{L}_\partial$ (or similarly $l_1 \in \mathcal{L}^-$, and $l_2 \in \mathcal{L}_\partial$).
        Note that $l_1$ is arc type. Then each $A_{l_1}^\pm$ consists of finitely many leaves. Since $h$ preserves the linear order, $h|_{\partial l} =\id$ for each $l \in A_{l_1}^\pm$. The claim follows.
        
        \noindent {\bf Case (3)}: $l_1 \in \mathcal{L}_\partial$, and $l_2 \in \mathcal{L}^-$ (or similarly $l_1 \in \mathcal{L}_\partial$, and $l_2 \in \mathcal{L}^+$). The same argument as in Case (1) shows that $h|_{\partial l} = \id$ for each $l \in A_{l_1}^-$.

        Note that $A_{l_1}^+$ is either empty or infinite. We assume that $A_{l_1}^+$ is infinite, and we order the leaves by $L_n, n \in \Z$, with respect to the linear order.
        Let $l_1' \in \mathcal{L}_\partial$ so that $l_2$ connects $\overline{l_1\cup r(l_1)}$ and $\overline{l_1'\cup r(l_1')}$, where $r$ is the reflection along $\mathbb{S}^1$.
        Note that $h(\overline{l_1'\cup r(l_1')}) = \overline{l_1'\cup r(l_1')}$.
        Denote the two components of $\mathbb{S}^1 \setminus \partial (l_1 \cup \partial l_1')$ that intersects $\Lambda(G)$ by $I_1, I_2$, where the label is chosen so that $I_1 > I_2$ by the induced linear order on $L_n$.
        
        If there exists a leaf $L_j$ which connects $\overline{l_1\cup r(l_1)}$ and $\overline{l_1'\cup r(l_1')}$, then there are only finitely many such leaves by Lemma~\ref{lem:descriptionlamination}. Thus $h$ fixes their boundaries point-wise. Then the same argument as in Case (1) shows that $h|_{\partial l} = \id$ for each $l \in A_{l_1}^+$ and the claim follows.

        Otherwise, there exists $j \in \N$ so that the following holds.
        \begin{itemize}
            \item For all $n \geq j$, $L_n$ connects $l_1$ and $L_n' \in \mathcal{L}_\partial$ with $\partial L_n' \subseteq I_1$; and 
            \item for all $n < j$, $L_n$ connects $l_1$ and $L_n'\in \mathcal{L}_\partial$ with $\partial L_n' \subseteq I_2$
        \end{itemize}
        Then $h|_{\partial L_j} = \id$. The same argument as in Case (1) shows that $h|_{\partial l} = \id$ for each $l \in A_{l_1}^+$, and the claim follows.
    \end{proof}
    
    Let $l \in \mathcal{L} \cup \mathcal{L}_\partial$. By Lemma~\ref{lem:finitepath}, there exists a finite path $l^0=l_1, l^1,..., l^k = l$ so that $(l^i, l^{i+1}) \in \mathcal{X}_{link}$. Thus, we can inductively apply this argument and conclude that $h|_{\partial l} = \id$ for all $l \in \mathcal{L}\cup \mathcal{L}_\partial$. Since $|\partial \mathcal{L}|\cup |\partial \mathcal{L}_\partial|$ is dense in $\Lambda(G) \cup |\partial \mathcal{L}|$, we conclude that $h = \id$ in $\Homeo^+(\Lambda(G), \mathcal{L}, \mathcal{L}_\partial)$.
\end{proof}

\begin{figure}[ht]
\captionsetup{width=0.96\linewidth}
  \centering
  \includegraphics[width=0.45\textwidth]{Figures/SchottkyTwoComponents.png}
  \includegraphics[width=0.45\textwidth]{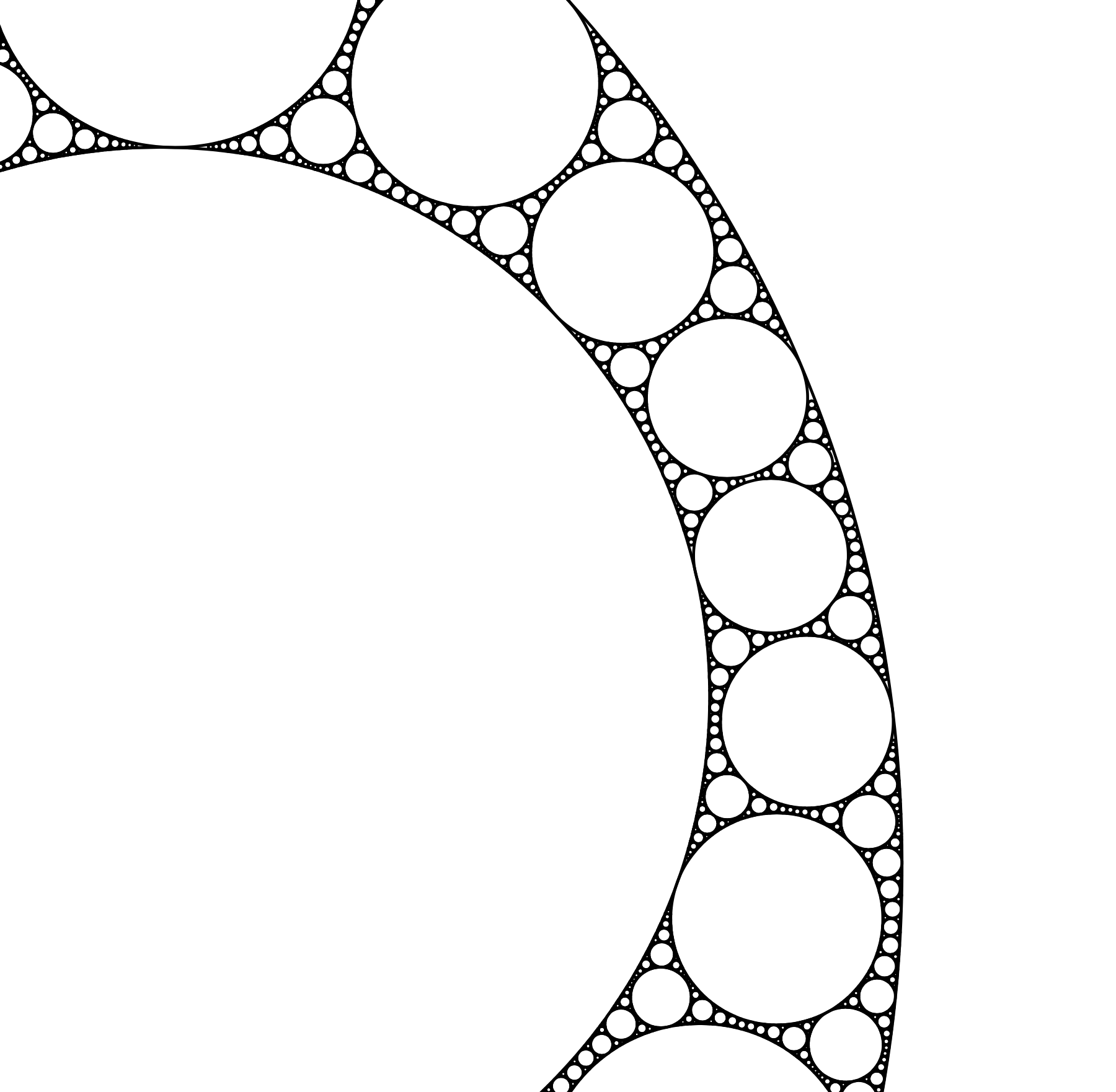}
  \caption{Left: A carpet-free Schottky set associated to $\Lambda(G)/{\sim_{\mathcal{L}}}$ for a cocompact Fuchsian group with an acylindrical lamination; its nerve has two components. Right: A zoom showing that the outer disk and the largest disk are not connected by any finite chain of touching disks.}
  \label{fig:SchottkyTwoComponents}
\end{figure}

As an immediate corollary, we have the following (see Figure~\ref{fig:SchottkyTwoComponents}).
\begin{cor}\label{cor:finiteindexfuchsian}
    Under the same assumption as in Theorem~\ref{thm:toprigidityconvexcocompact}, $G$ has finite index in $\Homeo^+(\Lambda(G), \mathcal{L}, \mathcal{L}_\partial)$.
\end{cor}
\begin{proof}
    By Lemma~\ref{lem:actionLinked}, both $G$ and $\Homeo^+(\Lambda(G), \mathcal{L}, \mathcal{L}_\partial)$ act on $\partial \mathcal{X}_{link}$. Note that $\partial \mathcal{X}_{link}/G$ is finite as there are only finitely many pairs of arcs/curves in $\mathcal{C}^+\cup \mathcal{C}^- \cup \partial S$, and each pair of distinct arcs/curves intersects at finitely many points. By Theorem~\ref{thm:toprigidityconvexcocompact}, for each $x \in \partial \mathcal{X}_{link}$, both $\stab_G(x), \stab_H(x)$ are finite.
    The corollary now follows from Proposition~\ref{prop:finiteindex}.
\end{proof}

\subsection{Boundary-compressible 3–manifolds with acylindrical laminations}\label{sec:bcrigidity}
In this subsection, we discuss the interaction between the incompressible decomposition (see \S~\ref{subsec:icd}) with an acylindrical lamination $\mathcal{L}$. We prove various topological rigidity theorems in this setting, which provide another ingredient for the inductive steps in the proof of Theorem~\ref{thm:carpetfreeSchottkyimpliesrigid}.

Let $G$ be a minimally parabolic geometrically finite Kleinian group, with corresponding pared 3-manifold $(N, P_t)$, where $P_t$ only consists of tori. Suppose that $(N, P_t)$ has compressible boundary, i.e., $G$ has disconnected limit set.
Let us complete $P_t$ into an acylindrical paring $P_t\sqcup P_a$ for $N$. Let $\cL \subseteq \Omega(G)$ denote the lamination induced by $P_a$.
The lifts of the boundaries of the compression disks of $(N,P_t)$ to $\Omega(G)$ together with the leaves in $\mathcal{L}$ form {\em ivy buds} (see \S~\ref{sec:defnsw} and Definition~\ref{defn:spiderweb}).
We prove topological rigidity results for ivy buds (see Proposition~\ref{prop:twopointrigiditygeneral} and its corollaries). 
They are then used to analyze vertex groups of different types of the incompressible decomposition in \S~\ref{subsec:swtypeII} and \S~\ref{subsec:swtypeIII}.

\subsubsection{Acylindrical laminations}\label{subsec:al}
Let $G$ be a minimally parabolic geometrically finite Kleinian group, with corresponding pared 3-manifold $(N, P_t)$.
Let $\mathcal{E}:=\{D_1,..., D_s\}$ be a core decomposition; see \S\,\ref{subsec:icd}.

Let $\mathcal{C}$ be an {\em acylindrical multicurve} on $\partial_0N = \partial N \setminus P_t$.
This means that if $P_{\mathcal{C}}$ is the union of pairwise disjoint neighborhoods of the curves in $\mathcal{C}$ homeomorphic to annuli in $\partial_0 N$, 
then the pared $3$-manifold $(N, P_t \sqcup P_{\mathcal{C}})$ is acylindrical.
Let 
$\mathcal{L} = \mathcal{L}_\mathcal{C}$ 
be the geodesic lamination associated to the lift of $\mathcal{C}$ in $\Omega(G)$.
Let $\mathcal{D}$ be the circle lamination, i.e., the lift of geodesic representatives of the boundary of core compression disks.

\subsection*{Correspondence between components}
Since $\mathcal{L}$ is acylindrical, the quotient $\Lambda(G)/\sim_{\mathcal{L}}$ is homeomorphic to a Schottky set.
It is useful to keep in mind the following correspondences between the quotient and the lamination.

\begin{equation}\label{eqn:correspodencecompcontact}
\begin{array}{rcl}
\text{complementary components of $\Lambda(G)/\sim_{\mathcal{L}}$}
& \longleftrightarrow &
\text{components of } \widehat{\C} \setminus \overline{|\mathcal{L}|}, \\[4pt]
\text{contact points between components}
& \longleftrightarrow &
\text{leaves in } \mathcal{L}.
\end{array}
\end{equation}

\begin{lem}\label{lem:threeleaves}
    Each Jordan curve $C\in \mathcal{D}$ intersects only finitely many and at least three leaves in $\mathcal{L}$.
\end{lem}
\begin{proof}
    Note that under the quotient map $\Omega(G) \rightarrow \Omega(G)/G$, $C$ is mapped to the boundary of a compression disk homeomorphically. Thus, there are finitely many intersections.

    If there is no leaf intersecting $C$, then $\Lambda(G)/\sim_{\mathcal{L}}$ is disconnected.
    If there is exactly one leaf intersecting $C$, then $\Lambda(G)/\sim_{\mathcal{L}}$ has a cut point.
    If there are exactly two leaves intersecting $C$, then there are two complementary components of $\Lambda(G)/\sim_{\mathcal{L}}$ intersecting at (at least) two points so that $\Lambda(G)/\sim_{\mathcal{L}}$ has a cut pair.
    None of the above is possible since $\Lambda(G)/\sim_{\mathcal{L}}$ is a Schottky set \cite[Proposition 5.3]{ph:lp:gw:schottky}.
    Therefore, $C$ intersects at least three leaves in $\mathcal{L}$.
\end{proof}

\subsection*{Cycles in the nerve}
Recall that the {\em nerve} of a Schottky set is an infinite planar graph whose vertices are complementary components and two vertices are connected by an edge if the corresponding components touch \cite[Def. 6.1]{ph:lp:gw:schottky}.
\begin{defn}
A {\em cycle} in a graph $\Gamma$ is a closed path consisting of finitely many vertices and edges. It is {\em simple} if the vertices and edges are all distinct.
A simple cycle is an {\em induced simple cycle} if it is an induced subgraph, i.e, it contains all edges in $\Gamma$ between its vertices.
\end{defn}
Note that given a simple cycle, there is a canonical way to decompose it into finitely many induced simple cycles.
The following correspondence between the Schottky set and the corresponding lamination is useful to record.
\begin{lem}\label{lem:cycle}
    Let $C\in \mathcal{D}$. It induces a cycle $\mathcal{C}= \mathcal{C}_C$ in the nerve of the Schottky set $\Lambda(G)/\sim_{\mathcal{L}}$, where 
\[
\begin{array}{rcl}
\text{vertices of the cycle}
& \longleftrightarrow &
\text{components of } C \setminus |\mathcal{L}|, \\[4pt]
\text{edges of the cycle}
& \longleftrightarrow &
\text{points in } |\mathcal{L}|\cap C.
\end{array}
\]
\end{lem}
\begin{proof}
    Denote the points in $|\mathcal{L}|\cap C$ by $a_1,..., a_k$ labeled cyclically and denote the components of $C \setminus |\mathcal{L}|$ by $I_j = (a_j, a_{j+1})$.
    Then each $I_j$ is contained in a unique component of $C \setminus |\mathcal{L}|$, and thus a complementary component $U_j$ of $\Lambda(G)/\sim_{\mathcal{L}}$ by \eqref{eqn:correspodencecompcontact}.
    Moreover, the closures of the two components $U_j, U_{j+1}$ touch.
    Therefore, $C$ induces a cycle in the nerve of the Schottky set $\Lambda(G)/\sim_{\mathcal{L}}$. 
\end{proof}

\subsubsection{Compatible core decomposition}\label{subsec:compatiblecompressiondisks}
Note that in general, the cycle in Lemma~\ref{lem:cycle} may not be simple nor induced.
In this subsection, we show that we can always select the compression disks so that the cycles are induced simple cycles.
\begin{defn}
    Let $(N,P)$ be a pared 3-manifold for a minimally parabolic geometrically finite Kleinian group. Let $\mathcal{L}$ be an acylindrical lamination.   
    We say that a collection of compression disks is {\em compatible} (with respect to the acylindrical lamination $\mathcal{L}$) if for each Jordan curve above the boundary of  one of its compression disks, the corresponding cycle $\mathcal{C}_C$ is an induced simple cycle.
\end{defn}
\begin{prop}\label{prop:admiss}
There exists a compatible core decomposition $\cE$ of compression disks.
\end{prop}

We will build our argument from Maskit's analysis of planar coverings of surfaces \cite{maskit:planar}.  
We localize our construction as follows. Let $S\subset \partial N$ be a boundary component, let $\tilde{N}$ be the universal covering of $N$ and let $\tilde{S}$ be
a component of $\partial\tilde{N}$ lying over $S$. Then $p:\tilde{S}\to S$ is a regular covering over a surface of finite type $S$ with
defining normal group $H < \pi_1(S)$, planar surface $\tilde{S}$ and deck transformation group $G$. 

Let us consider a multicurve $\cC_{\cL}\subset S$ such that $\cL= p^{-1}(\cC_{\cL})$ is a collection of arcs such that the connected components of $\tilde{S}\setminus |\cL|$ are open and simply
connected. 

\medskip

We will make use of the following construction. 

\begin{lem}[Crosscut surgery]\label{lem:crosscut}
Let $w\subset S$ be a simple loop in $H$ that  lifts  homeomorphically to a loop $W$ in $\tilde{S}$ under $p$. 
Let us consider a crosscut $C\subset \tilde{S}$ that joins 
two distinct points of $W$ that is disjoint from $G(W)$ otherwise, and let us consider a neighborhood $\tilde{S}'$ of $W\cup C$. If $C\cap |\cL|$ has at most one point, then $C$ can be deformed
in $\tilde{S}'$ so that $p(C)$ is an arc in $S$ and  there exist two simple loops $w_1$ and
$w_2$ in $H$ that are contained in $p(\tilde{S}')$ such that  $w\in \llangle  w_1, w_2\rrangle$.
 \end{lem}

\begin{proof} The crosscut $C$ decomposes $W$ into two arcs $W_+$ and $W_-$ parametrized so that $W=W_+\cdot W_-$ is their concatenation. 
It follows that $w=w_+\cdot w_-$ where $w_\pm=p(W_\pm)$. Let $D$ be the Jordan domain bounded by $W$ that contains $C$. 
The image $p(C)$ is a path that joins two points of $w$. We may deform $p(C)$ (and $C$) so that $p(C)$ has at most a single multiple point $o\in S$  
in its interior while keeping its endpoints fixed. 
Thus $p(C)$  can be written as a concatenation of two arcs $v_1', v_1''$ that joins $o$ to $w$ and $(n-1)$ simple closed curves $v_2,\ldots, v_n$ that either intersect exactly at $o$ or that coincide up to their orientation, i.e., $p(C)= v_1''\cdot v_2\cdot \ldots \cdot v_n \cdot v_1'$. We assume that $n$ is minimal.  If $n=1$, then $p(C)$ is an arc, so let us assume that  $n\ge 2$
and write  $v_1= v_1'\cdot  w_+\cdot v_1''$.  
If $C$ intersects $\cL$, then we choose the origin $o$  to be this intersection point. Note that  each curve $v_j$  intersects $|\cC_{\cL}|$ only at $o$, $1\le j\le n$.

Let us choose a lift $\tilde{o}\in C$ of $o$ so that $v_1'$ lifts to an arc in $C$ that joins $\tilde{o}$ to the origin of $W_+$. It follows that the lift $W_1$ of $v_1\cdot\ldots \cdot  v_n$ 
starting at $\tilde{o}$ coincides setwise with $W_+\cup C$. 
Let $W_2$ be the lifting of $v_2\cdot \ldots \cdot v_n \cdot v_1$ starting also at $\tilde{o}$. If the loop $W_2$ starts and ends on the same side of $D\setminus C$, then $v_1''\cdot v_2$ can 
be deformed in the neighborhood $S'=p(\tilde{S}')$ of $v_1''\cdot v_2$ 
close to $o$ to be disjoint from the curves $\cC_{\cL}\cup\{ v_1',v_3,\ldots,v_n\}$. 
This contradicts the minimality of $n$. It follows that $W_1$ and $W_2$ intersect transversely 
at $\tilde{o}$. From the planarity of $\tilde{S}$, it follows that $W_1$ and $W_2$ have a second point $\tilde{o}'$ in
common, where $p(\tilde{o}') = p(\tilde{o}) = o$. This implies that $C\cap |\cL|=\emty$ for otherwise $C\cap |\cL|$ would contain at least the two points
$\tilde{o}$ and $\tilde{o}'$.
We now write $W_1=  A_1\cdot B_1$ and $W_2= A_2\cdot B_2$ where $A_1$, $B_1^{-1}$, $A_2$ and $B_2^{-1}$ 
are arcs joining $\tilde{o}$ to $\tilde{o}'$ and where $B_1$ and $A_2$ do not contain any lift of $v_1$. 
Thus $A_2\cdot B_1$
is a  loop based at $\tilde{o}$ contained in $(\tilde{S}\setminus |\cL|)$.
Since the components of $\tilde{S}\setminus |\cL|$ are simply connected, 
it follows that $A_2\cdot B_1$ is homotopic to $\tilde{o}$. This also contradicts the minimality of $n$. Therefore, $C$ can be deformed in $\tilde{S}'$ so that $p(C)$ is an arc in $S$.

 To complete the proof, we deform the Jordan curves $W_+\cup C$ and $W_-\cup C$ into $D\setminus C$ so as to define two simple closed curves $W_1$ and $W_2$ in $\tilde{S}'$. 
 Their projections $w_1=p(W_1)$ and $w_2=p(W_2)$ are simple loops in $S$ that are in $H$ 
and that have been defined in such a way that $w\in  \llangle w_1, w_2 \rrangle$ in $H$.
\end{proof}

Let us now also consider a collection  $\cC_{\cD}$  of $q\ge 0$ simple curves in $H$ which lift as compatible closed curves in $\tilde{S}$. We set $\cD=p^{-1}(\cC_{\cD})$. 
Let $H_q =\llangle \cC_{\cD}\rrangle$ be the normal closure of $\cC_{\cD}$ in $H$. We let $\Gamma_{\cL}$ denote the nerve of the components of $\tilde{S}\setminus |\cL|$ and say that a collection of curves in $H$ is compatible (with respect to $\cL$) if the lifts of their boundary curves define induced simple cycles in $\Gamma_{\cL}$. 
The main step of the proof of Proposition \ref{prop:admiss} is the following.

\begin{prop}\label{prop:surg} Under the above notation, if $H\ne H_q$, there exists $w\in H\setminus H_q$ that is a simple loop whose lifts are compatible with $\cL$. \end{prop}

\begin{proof}  Let $w\subset S$ be a loop in $H\setminus H_q$ that we may assume to be disjoint from $\cC_{\cD}$ by \cite[Lemma 5]{maskit:planar}. It follows from
\cite[Theorem 2 and Lemma 7]{maskit:planar} that we may also assume that $w$ is a simple loop. 
We deform it locally so that it intersects minimally $\cC_{\cL}$.  
Let $W$ be a lift of $w$.  If $W$ has at least two components in some component $U$ of $\tilde{S}\setminus |\cL|$,
then we consider a crosscut of $W$ in $U\setminus \cD$ that joins two components  of $W\cap U$. By Lemma \ref{lem:crosscut}, we can decompose $w$ and find $w_1$ and $w_2$ such that 
$w\in \llangle w_1,w_2\rrangle$. Since $w\notin H_q$, we can assume that $w_1\notin H_q$. Since its lift is $W_1$, it has fewer components in $U$. As the surgery is local, it remains
disjoint from $\cC_{\cD}$. Since there are only finitely many such components to deal with, we obtain by induction a loop $w'\in H\setminus H_q$ disjoint from $\cC_{\cD}$ that defines a simple cycle
in the graph associated to $\tilde{S}\setminus |\cL|$, Figure~\ref{fig:phantomlam}, left. 

If one of its lifts $W'$ (hence all of them) does not define an induced cycle in the nerve, then we may find two components 
of $\tilde{S}\setminus |\cL|$ that are visited by $W'$ and that share a leaf 
of $\cL$ on their boundaries that is not intersected by $W'$. We may then consider a crosscut $C'$ of $W'$ that goes through this leaf. As above, 
by Lemma \ref{lem:crosscut}, we can decompose $w'$ and find $w_1'$ and $w_2'$ such that 
$w'\in \llangle w_1',w_2'\rrangle$. Since $w'\notin H_q$, we can assume that $w_1'\notin H_q$. The  combinatorial length of its lifts $\Gamma_{\cL}$
has decreased (compared to $w$). As the surgery is local, it remains
disjoint from $\cC_{\cD}$ and simple. 
Since there are only finitely many possibilities for obstructing the fact it defines an induced cycle, we may iterate finitely many times this procedure
so as to obtain in the end a loop $W''\subset \tilde{S''}$, so that  $w'' = p(W'')$ is a simple induced loop in $H\setminus H_q$ disjoint from $\cC_{\cD}$, Figure~\ref{fig:phantomlam}, right.
\end{proof}
 
 We now prove Proposition \ref{prop:admiss}.
 
 \begin{proof}[Proof of Prop. \ref{prop:admiss}]
 We proceed by induction. Since $\partial N$ is compressible, it contains  a simple loop that is not homotopically trivial in $\partial N$ but in $N$ by
 \cite[Theorem 2 and Lemma 7]{maskit:planar}. By Proposition \ref{prop:surg}, we may assume that
 it is a simple induced loop. By Dehn's lemma, this loop is the boundary 
 of a compression disk in $N$.
 
 As long as we did not select a full collection of loops that normally generate $\pi_1(\Omega(G))$, we may iterate this construction. It stops by Ahlfors' finiteness theorem since
 $\Omega(G)/G$ has finite type. 
 \end{proof}

\begin{figure}[ht]
\captionsetup{width=0.96\linewidth}
  \centering
  \includegraphics[width=0.45\textwidth]{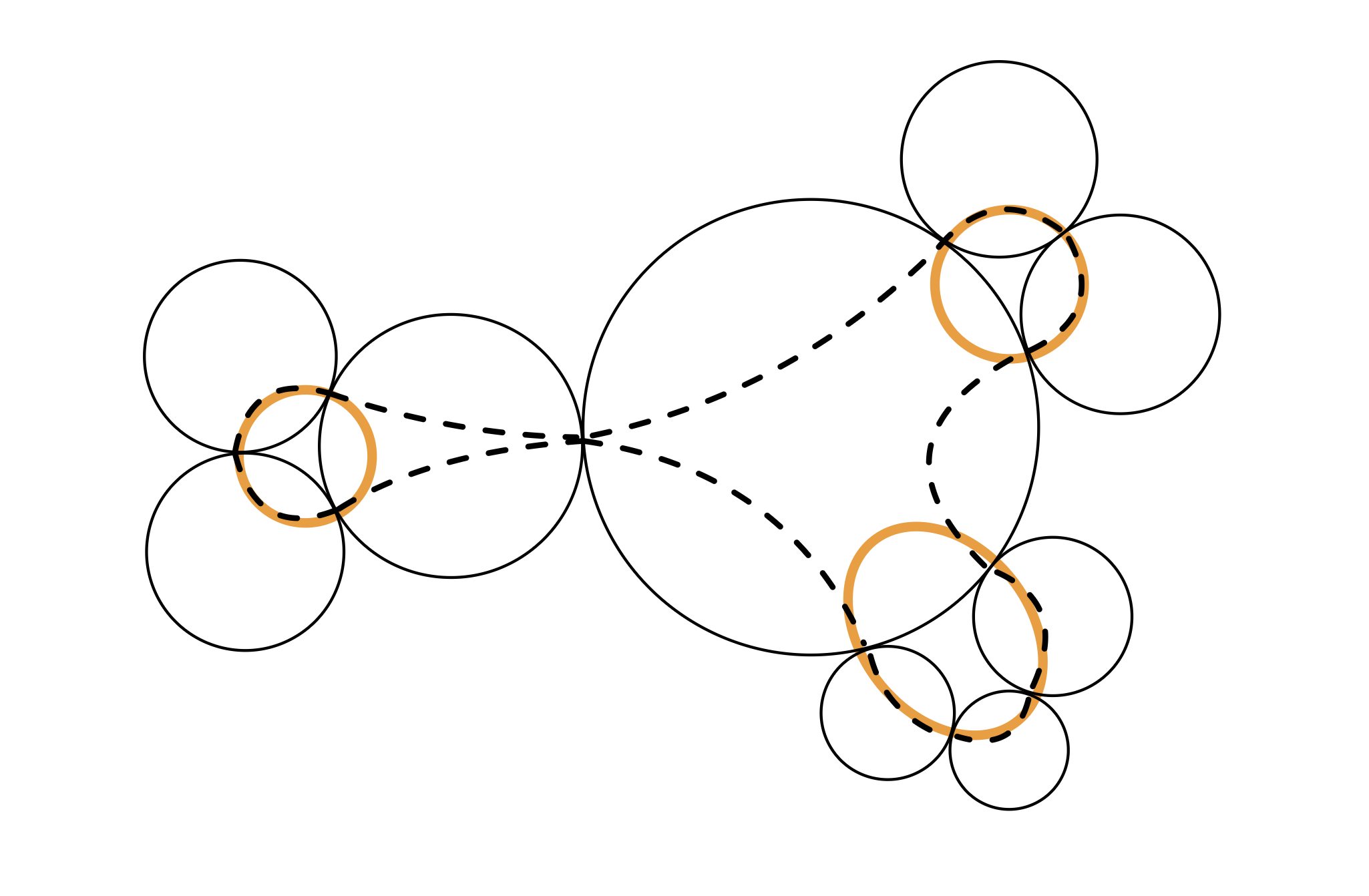}
  \includegraphics[width=0.45\textwidth]{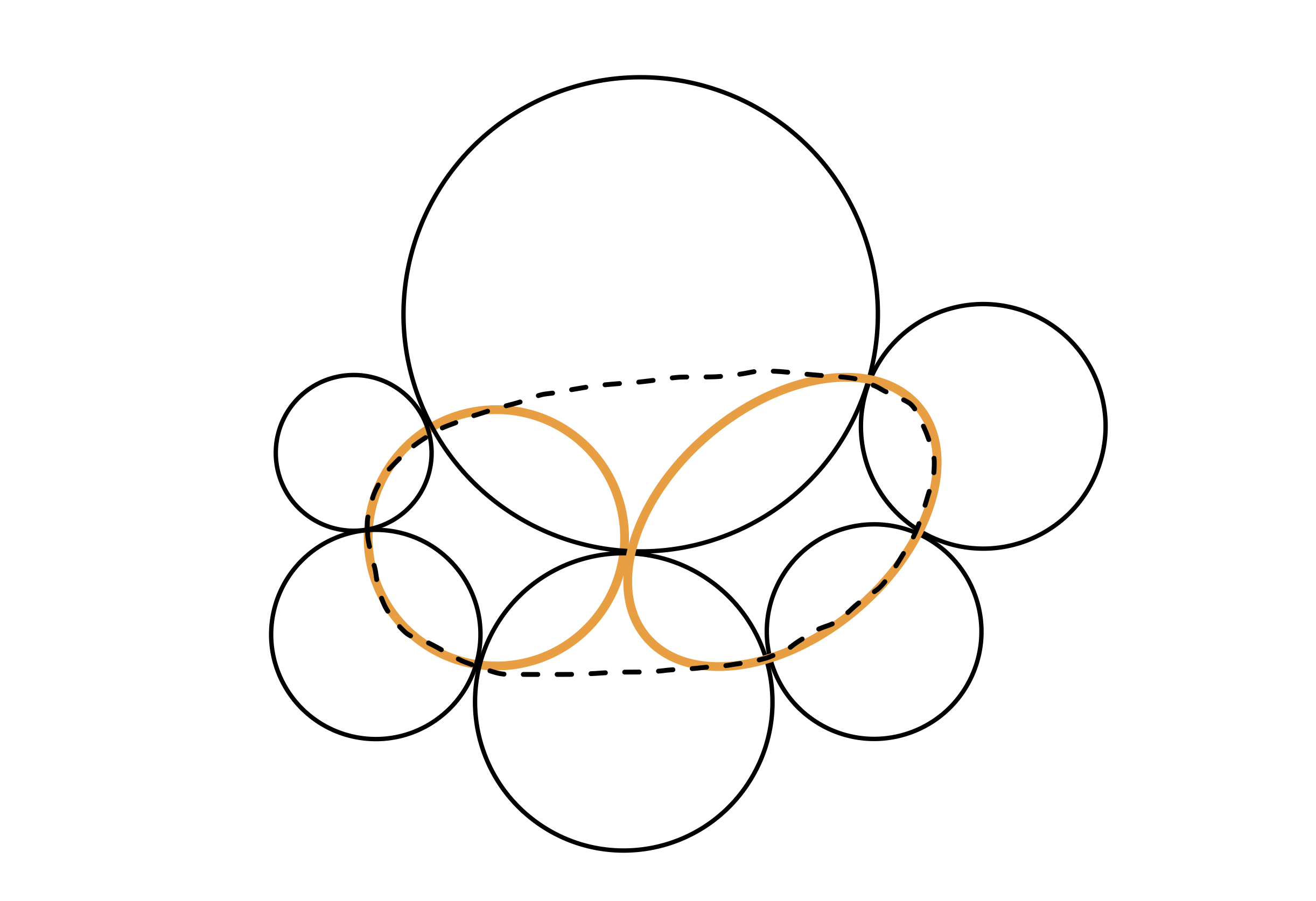}
  \caption{An illustration of a non-simple cycle in $\Lambda(G)/\sim_{\mathcal{L}}$ on the left and a non-induced simple cycle in $\Lambda(G)/\sim_{\mathcal{L}}$ on the right.
  The black dashed curve represents the image of the Jordan curve $C$ under the quotient map. The images after the modification are represented in orange.}
  \label{fig:phantomlam}
\end{figure}

\subsubsection{Ivy Buds}\label{sec:defnsw}
Let $\cT$ be the Bass-Serre tree of the incompressible decomposition of $(N,P_t)$ associated to a compatible core decomposition $\cE = \{D_1,..., D_s\}$ of compression disks.
Let $v$ be a non-{\disk} vertex. With notations as in \S\,\ref{subsec:icd}, 
we associate a lamination $\mathcal{L}_v$ and a circle lamination $\mathcal{D}_v$ as follows.
Note that $v$ corresponds to some star $\mathcal{S} \subseteq \mathcal{D}$. 
We set the circle lamination at $v$ by 
\begin{equation}\label{eqn:circlelaminationatv}
    \mathcal{D}_v = \mathcal{S}.
\end{equation}
Let $C \in \mathcal{D}_v$ and $D_C^-$ be the component of $\widehat\C \setminus C$ that is disjoint from $K_v$, so that $K_v = \widehat\C \setminus \bigcup_{C \in \mathcal{D}_v} D_C^-$.
We define 
 \begin{equation}\label{eqn:L_vlamination}
  {\mathcal{L}_v}:= \{l\cap K_v, l\in \widehat{\mathcal{L}_v}\} \text{ where } \widehat{\mathcal{L}_v}= \{l \in \mathcal{L} \text{ such that } l\cap K_v \ne \emty\}.
\end{equation}

\begin{remark}
    Let $X_v$ be the component of $N\setminus\!\setminus \bigcup_i D_i$ associated to the vertex $v$.
    Then $\mathcal{L}_v$ is simply the lift of the mixed multicurve $\mathcal{C}_v:= \{\gamma \cap X_v: \gamma \in \mathcal{C}\}$ to $\Omega(G_v)$.
\end{remark}

We summarize some simple properties of the leaves in $\mathcal{L}_v$.
\begin{lem}
    Let $l\in \mathcal{L}_v$. Then $l$ is either
\begin{enumerate}
    \item\label{lem:L_v:eq1} a geodesic segment (in $\Omega(G)$) that connects a pair of distinct Jordan curves $C_1, C_2 \in \mathcal{D}_v$; or
    \item\label{lem:L_v:eq2} a geodesic (in $\Omega(G)$) that connects a pair of points in $\Lambda_v$.
\end{enumerate}
\end{lem}
We remark that in the second case \eqref{lem:L_v:eq2}, the whole geodesic is in $K_v$ and it is possible that the two endpoints of $l$ are the same. This happens exactly when the common endpoint corresponds to a rank-two cusp of $G_v$.
We remark that by Lemma~\ref{lem:threeleaves}, the first case \eqref{lem:L_v:eq1} always occurs, but the second case \eqref{lem:L_v:eq2} might not happen (see Figure~\ref{fig:spiderwebFullBoundary} and ~\ref{fig:spiderweb}).

\begin{defn}[Ivy bud]\label{defn:spiderweb}
    Let $v$ be a non-{\disk} vertex.
    We define an {\em ivy bud} at $v$ as a subset
    $$
    \mathcal{W} \subseteq \mathcal{L}_v \cup \mathcal{D}_v
    $$
    so that its support $|\mathcal{W}|$ is a connected component of $|\mathcal{L}_v| \cup |\mathcal{D}_v|$ (see Figure~\ref{fig:spiderwebFullBoundary} and Figure~\ref{fig:spiderweb}).
    We also denote the leaves of $\mathcal{L}$ (or Jordan curves of $\mathcal{D}$) in $\mathcal{W}$ by
    \begin{equation}\label{eqn:L_W}
        \mathcal{L}_{\mathcal{W}}:= \{l\in \mathcal{L}: l \cap |\mathcal{W}| \neq\emptyset\} \text{ and }
        \mathcal{D}_{\mathcal{W}}:= \{C\in \mathcal{D}: C \cap |\mathcal{W}| \neq\emptyset\}.
    \end{equation}
    
    Note $C \in \mathcal{D}_v$ bounds a disk $D_{C}^-$ that is disjoint from $\Lambda_v$. We also denote
    \begin{equation}\label{eqn:W_full}
    \fil(|\mathcal{W}|):= |\mathcal{W}| \cup \bigcup_{C\in \mathcal{D}_{\mathcal{W}}} D_{C}^-
    \end{equation}
\end{defn}
We call an ivy bud $\mathcal{W}$ {\em elementary} if $\mathcal{W}$ consists of a single leaf and {\em non-elementary} otherwise. Note that any non-elementary ivy bud contains some $C \in \mathcal{D}_v$.

\begin{figure}[ht]
\captionsetup{width=0.96\linewidth}
  \centering
  \includegraphics[width=0.45\textwidth]{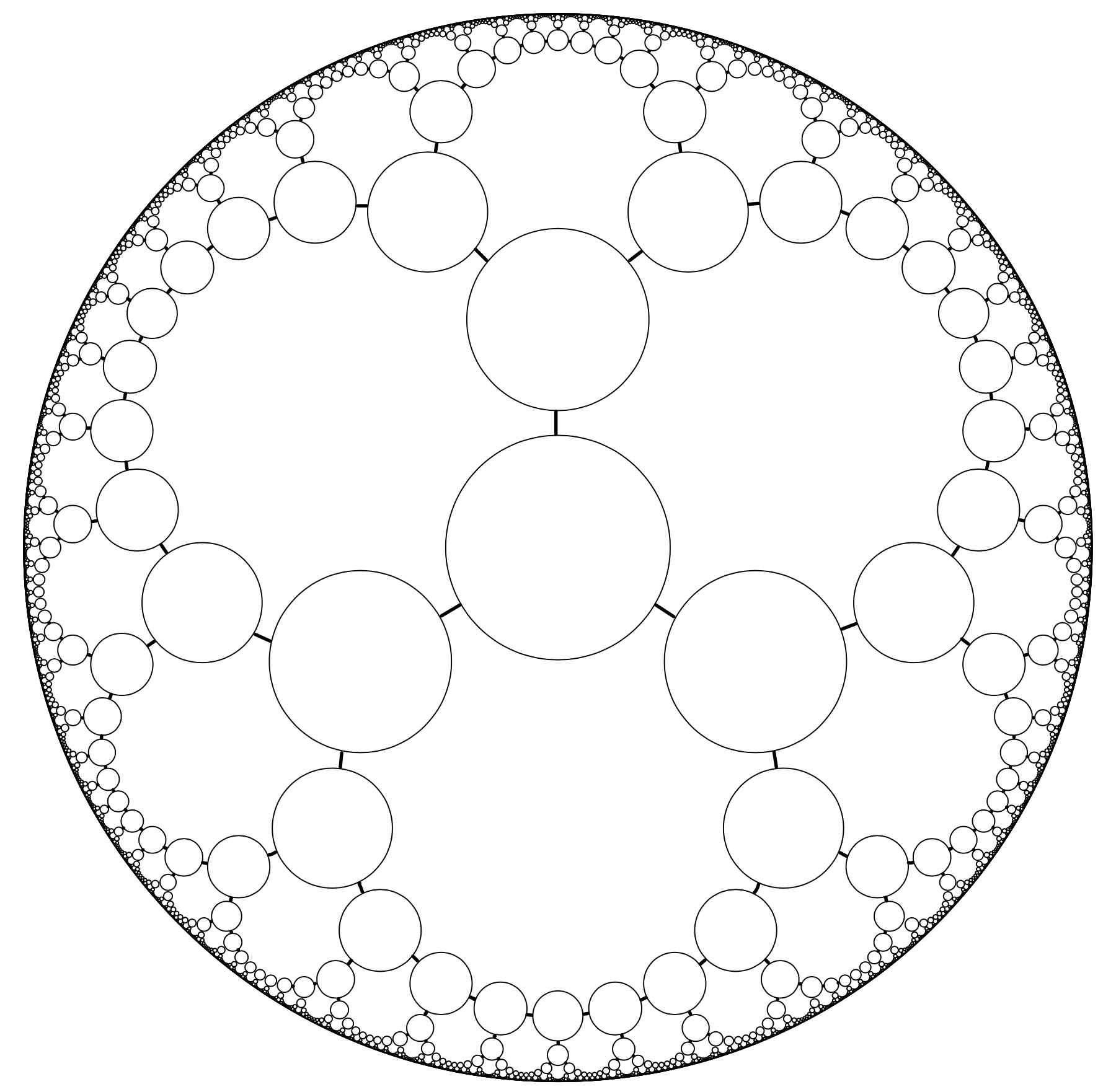}
  \includegraphics[width=0.45\textwidth]{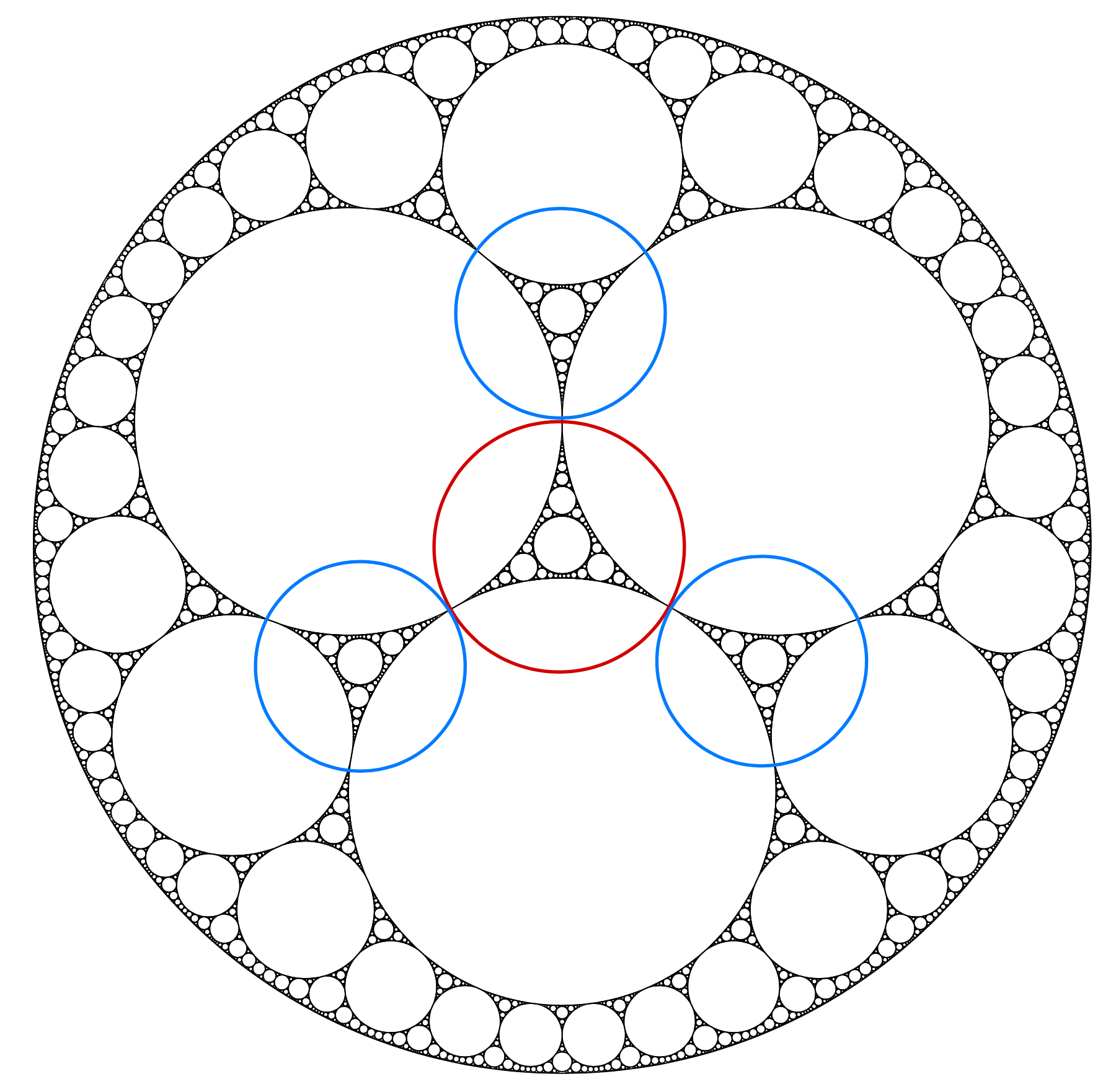}
  \caption{Left: An illustration of an ivy bud $\mathcal{W}$. Its limit set $\Lambda_{\mathcal{W}}$ is the full circle $\Lambda_v$. Each complementary component of $\fil(|\mathcal{W}|)$ gives a disk component on the right figure. Right: The Schottky limit set $\Lambda(\widehat G)$ associated with the quotient $\Lambda(G)/{\sim_{\mathcal{L}}}$. The red circle corresponds to the compression disk associated with the central circle on the left figure, while the three blue circles correspond to the compression disks associated with the three adjacent circles.}
  \label{fig:spiderwebFullBoundary}
\end{figure}

\subsection*{Limit sets of ivy buds}

We define the stabilizer and the limit set of an ivy bud $\mathcal{W}$ as follows:
\begin{align*}
    G_{\mathcal{W}} := \stab(|\mathcal{W}|) \le G_v, & \text{ and } \Lambda_\mathcal{W}:= \Lambda(G_{\mathcal{W}}).
\end{align*}
\begin{lem}\label{lem:intersectionboundary}
Let $\mathcal{W}$ be an ivy bud at $v$.
    Then $\Lambda_\mathcal{W} = \overline{|\mathcal{W}|} \cap \Lambda_v$.
\end{lem}
\begin{proof}
    Given $g\in G_v$, we have $g \in G_{\mathcal{W}}$ if and only if $gC \in \mathcal{D}_{\mathcal{W}}$ for some, and hence every, $C \in \mathcal{D}_{\mathcal{W}}$. Thus if two Jordan curves $C_1, C_2 \in \mathcal{D}_{\mathcal{W}}$ are in the same $G_v$ orbit, then they are in the same $G_\mathcal{W}$ orbit.
    Since $\mathcal{D}_{v}/G_{v}$ is finite, so is $\mathcal{D}_{\mathcal{W}}/G_{\mathcal{W}}$. Choose a fundamental domain $K\subseteq |\mathcal{W}|$ for the action of $G_\mathcal{W}$, consisting of finitely many Jordan curves in $\mathcal{D}_{\mathcal{W}}$ and finitely many segments in $\mathcal{L}_{\mathcal{W}}$. Then 
    $\Lambda_\mathcal{W}$ is the accumulation set of $G_\mathcal{W} K$. Therefore, $\Lambda_\mathcal{W} = \overline{G_\mathcal{W} K}\cap \Lambda_v = \overline{|\mathcal{W}|} \cap \Lambda_v$. 
\end{proof}

\begin{figure}[ht]
\captionsetup{width=0.96\linewidth}
  \centering
  \includegraphics[width=0.45\textwidth]{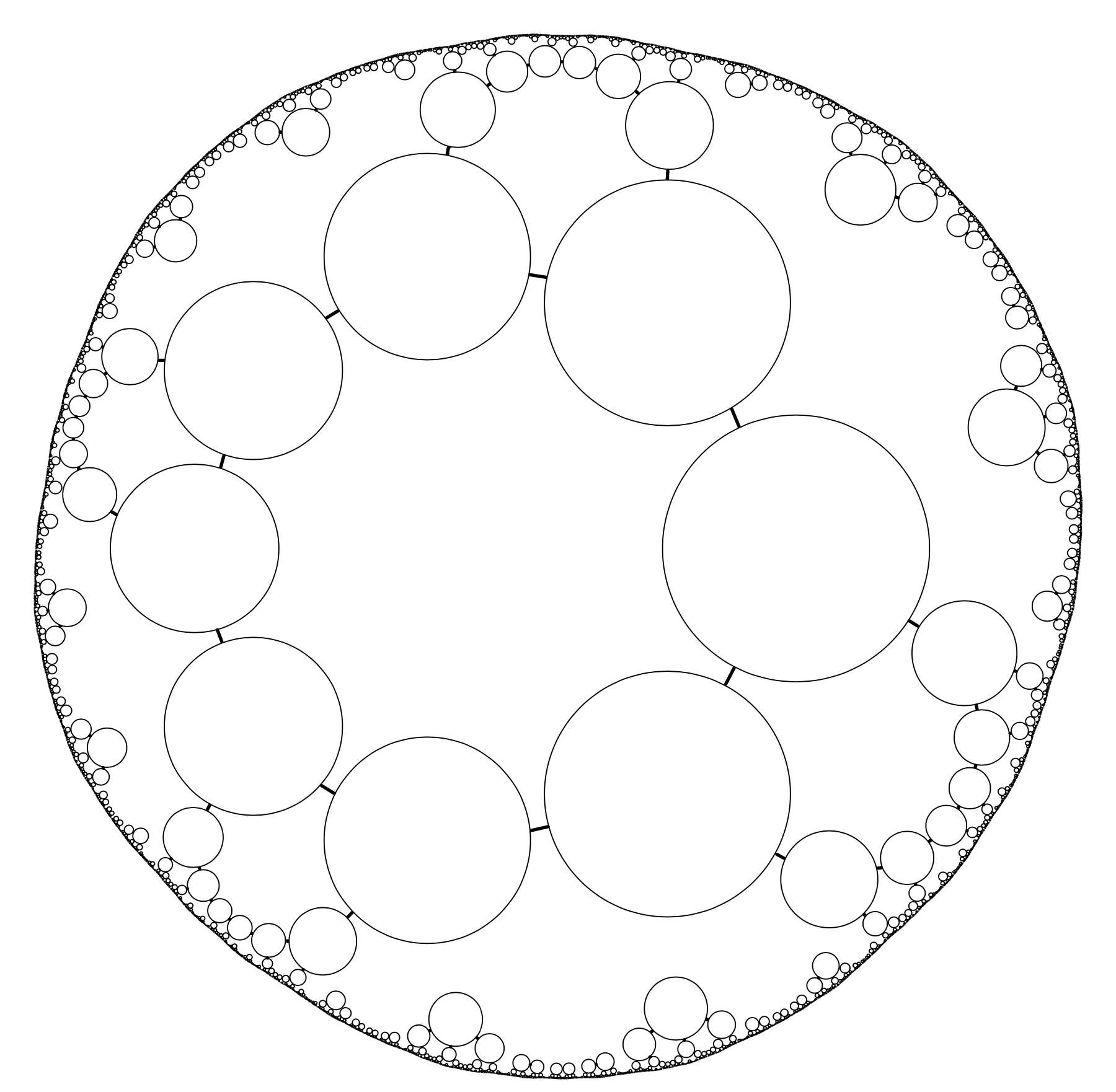}
  \includegraphics[width=0.45\textwidth]{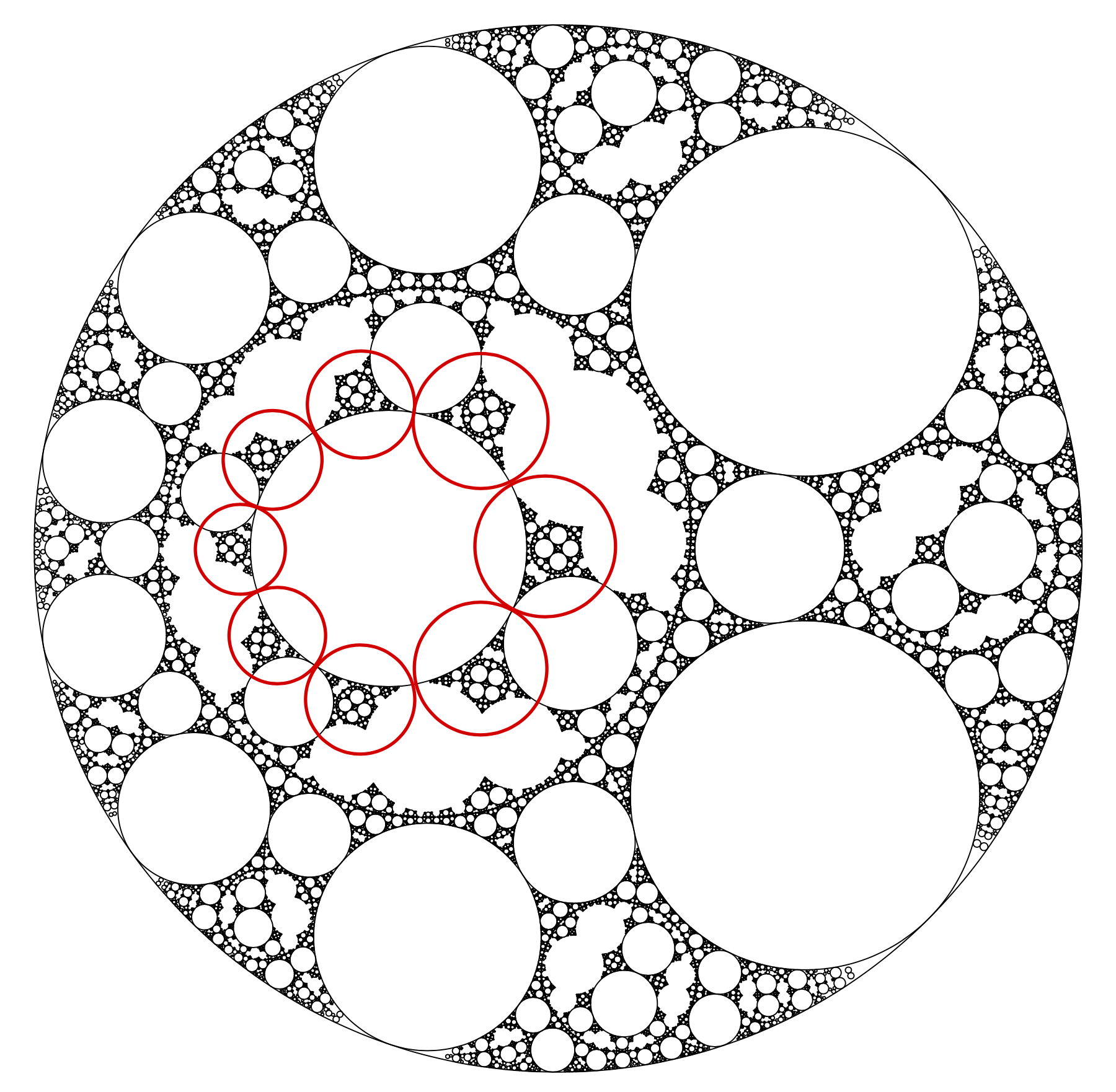}
  \caption{Left: An illustration of the union of ivy buds at a vertex $v$ from inside. Each ivy bud has limit set $\Lambda_{\mathcal{W}}$ equal to a Cantor subset of the quasicircle $\Lambda_v$. Right: The limit set $\Lambda(\widehat G) = \Lambda(G)/{\sim_{\mathcal{L}}}$, homeomorphic to a Schottky limit set. Note that infinitely many complementary components are not round disks. The quasicircle $\Lambda_v$ embeds as a subset, which intersects the boundary of some complementary components in a Cantor set. The $8$ red circles correspond to the compression disks associated with the most prominent $8$ circles in the left.}
  \label{fig:spiderweb}
\end{figure}

By the compatible condition on the collection of compression disks, for each $C \in \mathcal{D}$, distinct components of $C \setminus |\mathcal{\mathcal{L}}|$ correspond to distinct components of $\widehat\C \setminus \overline{|\mathcal{L}|}$ (see Figure~\ref{fig:spiderwebFullBoundary} and \ref{fig:spiderweb}). Therefore, we have 
\begin{lem}\label{lem:componentcorrespondence}
Let $v$ be a non-disk vertex. Then
\begin{enumerate}
    \item each component of $\widehat\C \setminus (\Lambda_v \cup \bigcup_{\mathcal{W}} \fil(|\mathcal{W}|))$ corresponds to a component of $\widehat\C \setminus \overline{|\mathcal{L}|}$; and 
    \item distinct components of $\widehat\C \setminus (\Lambda_v \cup \bigcup_{\mathcal{W}} \fil(|\mathcal{W}|))$ correspond to distinct components of $\widehat\C \setminus \overline{|\mathcal{L}|}$.
\end{enumerate}
\end{lem}

We now analyze the symmetries of ivy buds.

\begin{lem}\label{lem:unboundedspider}
    Let $\mathcal{W}$ be an ivy bud at $v$. If $v$ is {\nonel} or {\para}, then the limit set $\Lambda_{\cW}$ is non-empty.
\end{lem}

\begin{proof}
    Suppose $|\mathcal{W}|$ is bounded in $\widehat\C \setminus \Lambda_v$. Then we can construct a closed curve in $\widehat\C \setminus (\Lambda_v \cup \bigcup_{\mathcal{W}} \fil(|\mathcal{W}|))$ that separates $|\mathcal{W}|$ and $\Lambda_v$. This closed curve projects to a closed curve in the complement of the Schottky set $\Lambda(G)/\sim_\mathcal{L}$ that separates points in the limit set. This is a contradiction.
\end{proof}
\begin{cor}\label{cor:non-trival}
    If $v$ is {\nonel} or {\para}, then the stabilizer $G_{\mathcal{W}} \le G_v \le G$ of an ivy bud $\mathcal{W}$ is infinite.
\end{cor}

\subsection*{Rigidity of ivy buds}
In this subsection, we show that ivy buds satisfy a {\em two-point rigidity} for any homeomorphism of $\Lambda(G)$ that preserves $\mathcal{L}$.
One of the difficulties is that homeomorphisms in $\Homeo^+(\Lambda(G), \mathcal{L})$  may not preserve the circle lamination $\mathcal{D}$.
To overcome this, we enlarge $\cD$ to a broader collection that is invariant under such homeomorphisms as follows.
By Lemma~\ref{lem:cycle} and the compatible condition, a loop $C\in \cD$ yields an induced simple cycle $\mathcal{C}_C$ on the nerve of $\Lambda(G)/\sim_\mathcal{L}$. 
Let $\widehat{\cD}(k)$ be the set of all simple closed geodesics in $\Omega(G)$ that define induced simple cycles of length at most $k$.
We define $\widehat{\cD}(\infty) = \bigcup_k \widehat{\cD}(k)$. 
Let $N_0$ denote the largest combinatorial length of cycles $\cC_C$, $C\in \cD$. Then $\cD \subseteq \widehat{\cD}(N_0)$.

\begin{defn}\label{defn:linkedbordering_new}
Let $k \in \N \cup \{\infty\}$ and $\widehat{\mathcal{D}} = \widehat{\mathcal{D}}(k)$ or $\widehat{\mathcal{D}} = \cD$.

Let $l_1, l_2 \in \mathcal{L}$ and let $C \in \widehat{\mathcal{D}}$.
We say that $l_1$ and $l_2$ are \emph{bordering at $C$}
if the sets 
    $C \cap l_1$ and $C \cap l_2$ are adjacent in $C \cap |\mathcal{L}|$, 
    that is, if they form the two endpoints of a connected component of 
    $C \setminus |\mathcal{L}|$.

We say that $l_1$ and $l_2$ are \emph{bordering} with respect to $\widehat{\mathcal{D}}$
if there exists $C \in \widehat{\mathcal{D}}$ such that they are bordering at $C$. 

We say $l_1$ and $l_2$ are \emph{connected} with respect to $\widehat{\mathcal{D}}$ if there is a sequence $l_1 = l^0, ..., l^k = l_2$ so that $l^i, l^{i+1}$ are bordering with respect to $\widehat{\mathcal{D}}$.
\end{defn}

\begin{lem}\label{lem:homeoborderingleaves}
For each $k \in \N$, the group $\Homeo^+(\Lambda(G), \mathcal{L})$ acts on the set $X = X(k)$ of pairs of bordering leaves with respect to $\widehat{\mathcal{D}}(k)$, and the restriction of this action to $G$ has finitely many orbits.
\end{lem} 
\begin{proof}
Since $\Homeo^+(\Lambda(G), \mathcal{L})$ acts on the Schottky set $\La(G)/\sim_{\cL}$, it also acts on its nerve. In particular, its action is by simplicial isometries, so it preserves $\widehat{\cD}(k)$, and also the property for a pair of leaves to be bordering.

Note that there are only finitely many $G$-orbits of induced simple cycles of length $\leq k$ in the nerve of $\Lambda(G)/\sim_\mathcal{L}$ (see Figure~\ref{fig:rigidityspiderweb}). Since each cycle contains finitely many pairs in $X$, we have $X/G$ is finite.
\end{proof}

A leaf $l \in \mathcal{L}$ is said to be \emph{oriented} if it is equipped with a choice of an orientation, that is, a choice of one of the two possible directions along $l$.
Let $h, h' \in \Homeo^+(\Lambda(G), \mathcal{L})$ be two homeomorphisms, and let 
$l \in \mathcal{L}$ be an oriented leaf.
We say that $h$ agrees with $h'$ on $l$ if $h(l) = h'(l)$ and the map 
$h^{-1} \circ h' : l \to l$ preserves the orientation of $l$.
It is convenient to denote this by
$$
h|_{\partial l} = h'|_{\partial l}.
$$
Note that here we interpret $\partial l$ as the two ends of $l$ (as it is possible that the two ends converge to a single point on $\Lambda(G)$ if $G$ has parabolic elements).

\begin{prop}\label{prop:twopointrigiditygeneral}
    Let $l_1, l_2 \in \mathcal{L}$ be two bordering leaves with respect to $\widehat{\mathcal{D}}(\infty)$. Let $h, h' \in \Homeo^+(\Lambda(G), \mathcal{L})$.
    Suppose that
    $$
    h|_{\partial l_1\cup \partial l_2} = h'|_{\partial l_1\cup \partial l_2}.
    $$
    Then for all $l \in \mathcal{L}$ so that $l_1, l$ are connected with respect to $\widehat{\cD}(\infty)$, we have
    $h|_{\partial l} = h'|_{\partial l}$.

    In particular, if $h|_{\partial l_1 \cup \partial l_2} = \id$, then for all $l \in \mathcal{L}$ so that $l_1, l$ are connected with respect to $\widehat{\cD}(\infty)$, we have $h|_{\partial l} = \id$.
\end{prop}
\begin{proof}
    It suffices to prove the statement for each $k \in \N$ since $\{\hat{\cD}(k)\}$ is an increasing sequence.
    Let $C \in \widehat{\mathcal{D}}(k)$ so that $l_1 \cap C, l_2 \cap C \neq\emptyset$.
    
    Let $h_*$ and $h_*'$ be the induced maps on the nerve.
    Let $E_1 = [a_1,a_2], E_2=[a_2,a_3]$ be the edges of $\mathcal{C}_C$ corresponding to $l_1, l_2$. Then $h_*(a_i) = h_*'(a_i), i = 1,2,3$ by assumption.
    Let $\mathcal{S}$ be the set of all induced simple cycles of length $\leq k$ passing through $E_1$.
    We define a linear order $\prec$ on the vertices adjacent to $a_2$, other than $a_1$, by declaring that
    $$
    x \prec y 
    $$
    if the triple $a_1, x, y$ is oriented counterclockwise around $a_2$.
    This induces a discrete total preorder on $\mathcal{S}$ by
    $$
    \mathcal{C}_1 = (a_1, a_2, x_3,...,x_N, x_{N+1} = a_1) \preceq \mathcal{C}_2 = (a_1, a_2, y_3,...,y_M, y_{M+1} = a_1)
    $$
    if $x_3 \preceq y_3$. 
    We write $\mathcal{C}_1 \asymp \mathcal{C}_2$ if $\mathcal{C}_1 \preceq \mathcal{C}_2$ and $\mathcal{C}_2 \preceq \mathcal{C}_1$, i.e., $\mathcal{C}_1, \mathcal{C}_2$ have the same third vertex.
    Note that for each $N \geq 3$ and $x$ adjacent to $a_2$, there are only finitely many induced simple cycles $\mathcal{C} \in \mathcal{S}$ of length $N$ that contain $a_1, a_2, x$, i.e., there are only finitely many induced simple cycles of length $N$ in a $\asymp$-class (see Figure~\ref{fig:rigidityspiderweb}).

    \begin{figure}[ht]
\captionsetup{width=0.96\linewidth}
  \centering
  \includegraphics[width=0.6\textwidth]{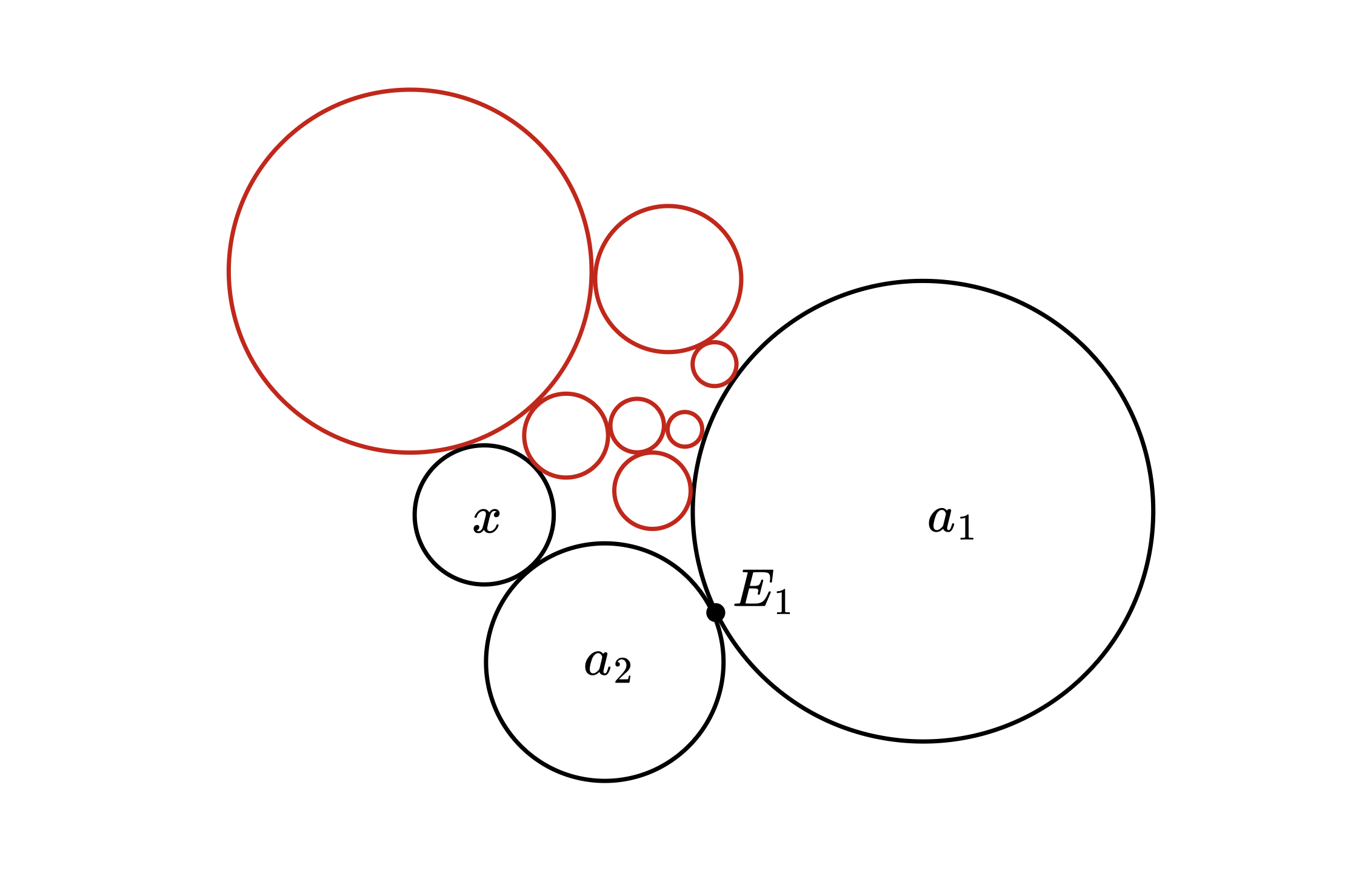}
  \caption{An illustration of finitely many induced simple cycles of length $N$ that contains $a_1, a_2, x$ in the proof of Corollary~\ref{cor:twopointrigidity}. The round circles associated to $x$ and $a_1$ have disjoint closures, so they bound an annulus in $\widehat\C$. Thus, given any $N$, there are only finitely many chains of round circles of length $N$ that connects $x$ to $a_1$.}
  \label{fig:rigidityspiderweb}
\end{figure}

    Let $\widetilde{\mathcal{S}}$ be the set of all induced simple cycles passing through $h_*(E_1)$ of length at most $k$. Then a similar preorder is defined on $\widetilde{\mathcal{S}}$.
    Since $h|_{\partial l_1\cup \partial l_2} = h'|_{\partial l_1\cup \partial l_2}$,
    $$
    h_*(\mathcal{C}_1) \preceq h_*(\mathcal{C}_2) \iff h_*'(\mathcal{C}_1) \preceq h_*'(\mathcal{C}_2) \iff \mathcal{C}_1 \preceq \mathcal{C}_2.
    $$
    Since $h_*(E_2) = h_*'(E_2)$, we conclude that $h_*(\mathcal{C}) \asymp h_*'(\mathcal{C})$ for all $\mathcal{C} \in \mathcal{S}$. 
    Thus if $l$ is a leaf such that $(l,l_1)$ or $(l,l_2)$ are bordering at some $C\in\hat{\cD}(k)$, we have $h|_{\partial l} = h'|_{\partial l}$.
    
    Let $l \in \mathcal{L}$ so that $l_1, l$ are connected. We may assume each edge is in an induced cycle of some common $\widehat{\cD}(k)$. 
    We conclude, by inductively applying the above argument, that $h|_{\partial l} = h'|_{\partial l}$.
\end{proof}

Since we chose the compression disks to give induced simple cycles (Proposition~\ref{prop:admiss}) and since any two leaves $l, l' \in \mathcal{L}_{\mathcal{W}}$ are connected with respect to $\cD \subseteq \widehat{\cD}(N_0) \subseteq \widehat{\cD}(\infty)$ by our definition (Definition~\ref{defn:spiderweb}), we immediately have the following corollaries establishing
the rigidity of ivy buds.
\begin{cor}[Two point rigidity]\label{cor:twopointrigidity}
    Let $\mathcal{W}$ be a non-elementary ivy bud and let $l_1, l_2 \in \mathcal{L}_{\mathcal{W}}$ be two bordering leaves.
    Let $h, h' \in \Homeo^+(\Lambda(G), \mathcal{L})$.
    Suppose that
    $$
    h|_{\partial l_1\cup \partial l_2} = h'|_{\partial l_1\cup \partial l_2}.
    $$
    Then for all $l \in \mathcal{L}_{\mathcal{W}}$, we have
    $h|_{\partial l} = h'|_{\partial l}$.

    In particular, if $h|_{\partial l_1 \cup \partial l_2} = \id$, then for all $l \in \mathcal{L}_{\mathcal{W}}$, we have $h|_{\partial l} = \id$. This implies that  $h|_{\Lambda_\mathcal{W}} = \id$.
\end{cor}

\begin{cor}\label{cor:finiteindexHW}
    Let $\mathcal{W}$ be a non-elementary ivy bud, and let 
    $$
    H_\mathcal{W}:=\{h \in \Homeo^+(\Lambda(G), \mathcal{L}): h(\Lambda_\mathcal{W}) = \Lambda_\mathcal{W}\}/\sim,
    $$
    where $h \sim h'$ if $h|_{\Lambda_\mathcal{W}} = h'|_{\Lambda_\mathcal{W}}$.
    Then $G_\mathcal{W}$ has finite index in $H_\mathcal{W}$.
\end{cor}
\begin{proof}  
    Let $N_0$ be so that each Jordan curve $C \in \mathcal{D}$ induces a cycle of length $\leq N_0$.
    Let $X$ be the collection of bordering pairs $(l,l')$ with respect to $\widehat{\cD}(N_0)$. Note that there is an induced action of $H_{\cW}$ on $X$.
    Let $l_1, l_2 \in \mathcal{L}_{\mathcal{W}}$ be a pair of bordering leaves at $C \in \mathcal{D}_v$ so that $(l_1, l_2) \in X$. 
    By Corollary~\ref{cor:twopointrigidity}, $h$ is uniquely determined by its image on $l_1, l_2$, i.e., the stabilizer of $H_\mathcal{W}$ for $(l_1,l_2) \in X$ is trivial.
    Let $X_0$ be the orbit of $(l_1,l_2)$ under $H_\mathcal{W}$. 

    By Lemma~\ref{lem:homeoborderingleaves}, $X/G$ is finite.
    We claim $X_0/G_\mathcal{W}$ is finite. Indeed, let $r = |X/G|$.
    We will show that $X_0/G_\mathcal{W}$ has at most $r$ elements.
    Consider $h, h'\in H_{\mathcal{W}}$ so that $h(l_1,l_2), h'(l_1,l_2)$ are in the same $G$-orbit, i.e.,
    $$
    h'(l_1,l_2) = gh(l_1,l_2),
    $$
    for some $g \in G$.
    By Corollary~\ref{cor:twopointrigidity}, we have $g h(\Lambda_\mathcal{W}) = h'(\Lambda_\mathcal{W})$. Thus, $g \in G_\mathcal{W}$.
    Therefore, $h(l_1,l_2), h'(l_1,l_2)$ are in the same $G_{\mathcal{W}}$-orbit. Since there are at most $r$ $G$-orbits, there are at most $r$ $G_\mathcal{W}$-orbits in $X_0$. This proves the claim.
    
    Therefore, by Proposition~\ref{prop:finiteindex}, $G_\mathcal{W}$ has finite index in $H_\mathcal{W}$.
\end{proof}

\subsubsection{Ivy buds for {\ball} or {\para} vertices}\label{subsec:swtypeII}
In this subsection, we show that there is a unique ivy bud at a {\ball} or {\para} vertex $v$.
\begin{prop}\label{prop:uniqueextspiderweb}
    Let $v$ be a {\ball} or a {\para} vertex. Then there is a unique ivy bud $\mathcal{W}$, i.e., $|\mathcal{L}_v| \cup |\mathcal{D}_v|$ is connected. Moreover, if $v$ is {\para}, then $G_{\mathcal{W}} = G_v = \Z^2$.
\end{prop}
\begin{proof}
    Suppose $v$ is {\ball} type. Suppose $A:=|\mathcal{L}_v| \cup |\mathcal{D}_v|$ is not connected. Then there exists a simple closed curve $\gamma \subseteq K_v$ disjoint from $A$ that separates points in $A$. Thus $\gamma$ projects to the boundary of a compression disk $(N, P\cup P_\mathcal{C})$. But since $(N, P\cup P_\mathcal{C})$ is acylindrical, in particular it has incompressible boundary, this is a contradiction.

    Suppose $v$ is {\para} type. Let $\mathcal{W}$ be an ivy bud at $v$. By Corollary~\ref{cor:non-trival}, $G_{\mathcal{W}}$ is not trivial.
    Let us normalize so that $\Lambda_v = \{\infty\}$.
    Suppose $G_{\mathcal{W}}$ has rank $1$. 
    Since $G_v$ is a rank-two abelian group, there exists $g \in G_v$ so that $g^k \fil(|\mathcal{W}|)$ are disjoint for all $k \in \Z$.
    By Lemma~\ref{lem:componentcorrespondence}, we have infinitely many components of $\widehat\C \setminus \overline{|\mathcal{L}|}$ that contain $\infty$ on their boundary. This is a contradiction.
    Therefore, $G_{\mathcal{W}}$ has rank $2$.
    Suppose $G_{\mathcal{W}} \neq G_v$. Then there exists $g \in G_v$ so that $\fil(|\mathcal{W}|)$ is disjoint from $g \fil(|\mathcal{W}|)$. But this is not possible as each component of $\C \setminus \fil(|\mathcal{W}|)$ is bounded. Therefore, $G_\mathcal{W} = G_v$. The same argument also shows that $\mathcal{W}$ is the unique ivy bud at $v$.
\end{proof}

\subsubsection{Ivy buds for {\nonel} vertices}\label{subsec:swtypeIII}
In this subsection, we show that any homeomorphism $h \in \Homeo^+(\Lambda(G), \mathcal{L})$ permutes the limit sets of ivy buds for {\nonel} vertices. Together with the rigidity results for ivy buds, we show that any homeomorphism $h \in \Homeo^+(\Lambda(G), \mathcal{L})$ that fixes a {\nonel} vertex must fix an acylindrical lamination at $v$ (that is in general not canonical).
\subsection*{Invariance under homeomorphisms}
We first recall that the {\em ideal boundary} $\partial^I U$ of a simply connected subset $U$ of $\cbar$ that omits at least three points of $\widehat\C$ is the set of its prime ends.

Let $v$ be a {\nonel} vertex and fix a complementary component $\Delta$ of $\Lambda_v$ (this is a simply connected domain with locally connected boundary).
Let $\mathcal{U}$ be a component of $\Delta \setminus \bigcup_{\mathcal{W}} \fil(|\mathcal{W}|)$ where the union is over all ivy buds $\cW$ in $\Delta$.  
This defines a subset
    \begin{equation}\label{eqn:KU}
    \partial^I \mathcal{U} :=\{ x\in \partial^I\Delta: x \text{ is on the boundary of } \mathcal{U}\}.
    \end{equation}
    We note that $\partial^I \mathcal{U} \subseteq \partial^I\Delta$ is either empty, a set consisting of two points, a Cantor set or the whole ideal boundary $\partial^I\Delta$ (see Figure~\ref{fig:spiderwebFullBoundary} where $\partial^I \mathcal{U}$ is empty, and see Figure~\ref{fig:spiderweb} where 
    we have examples of ideal boundaries $\partial^I \mathcal{U}$ that are either empty or Cantor sets).
    We associate a lamination $\mathcal{L}_\mathcal{U}$ by considering the geodesics that connect two endpoints in $\partial^I \Delta$ of each component of $\partial^I\Delta \setminus \partial^I \mathcal{U}$, and let 
    \begin{equation}\label{eqn:LDelta}
        \mathcal{L}_\Delta:= \bigcup\mathcal{L}_\mathcal{U}
    \end{equation} 
    over all components of  $\Delta \setminus \bigcup_{\mathcal{W}} \fil(|\mathcal{W}|)$. The Ahlfors finiteness theorem applied to $G_v$ ensures $\cL_{\Delta}$ is indeed a lamination.  
    \begin{lem}\label{lem:preservingLDelta}
        Let $h \in \Homeo^+(\Lambda(G), \mathcal{L})$. Let $\Delta$ be a complementary component of $\Lambda_v$ for some {\nonel} vertex. Then there exist a {\nonel} vertex $v'$ and a complementary component $\Delta'$ of $\Lambda_{v'}$ so that
        $$
        h(\mathcal{L}_\Delta) = \mathcal{L}_{\Delta'}.
        $$
    \end{lem}
    \begin{proof}
        Since any homeomorphism sends components of $\Lambda(G)$ to components of $\Lambda(G)$, there exists a {\nonel} vertex $v'$ and a complementary component $\Delta'$ of $\Lambda_{v'}$ so that $h(\Delta) = \Delta'$. Therefore, $h(\partial^I \Delta) = \partial^I \Delta'$. Since $h$ preserves $\mathcal{L}$, $h$ sends a component of $\widehat\C \setminus \overline{|\mathcal{L}|}$ to another. Thus, by Lemma~\ref{lem:componentcorrespondence}, $h$ sends a component of $\Delta \setminus \bigcup_{\mathcal{W}} \fil(|\mathcal{W}|)$ that intersects $\partial^I\Delta$ to a component of $\Delta' \setminus \bigcup_{\mathcal{W}'} \fil(|\mathcal{W}'|)$ that intersects $\partial^I\Delta'$. 
        Therefore, $h(\mathcal{L}_\Delta) = \mathcal{L}_{\Delta'}$.
    \end{proof}

\begin{lem}\label{lem:singleleaf}
    Let $v$ be a {\nonel} vertex and $\mathcal{W}$ be an ivy bud at $v$. Then $\Lambda_{\mathcal{W}}$ is finite if and only if $\mathcal{W}$ is elementary, i.e., $\mathcal{W}$ consists of a single leaf.
\end{lem}
\begin{proof}
    The `if' direction is straightforward.
    We now show the `only if' direction.
    Note that $\Lambda_{\cW}$ is finite if and only if $\partial^I |\mathcal{W}|$ consists of two points $a, b$ in the ideal boundary $\partial^I\Delta$.
    Let $\gamma_1, \gamma_2$ be the two extremal paths in $|\mathcal{W}|$ connecting $a$ to $b$. Here extremal means that $|\mathcal{W}|$ is bounded in between $\gamma_1$ and $\gamma_2$.
    Then there exist two complementary components $U_1, U_2$ of $\Delta \setminus \bigcup_{\mathcal{W}'} \fil(|\mathcal{W}'|)$ so that $\gamma_i \subsetneq \partial U_i$.
    Therefore, by Lemma~\ref{lem:componentcorrespondence}, there exist two complementary components $\widehat U_1, \widehat U_2$ of $\Lambda(G)/\sim_\mathcal{L}$ so that  $\gamma_i$ projects to a {\em proper} interval $I_i \subseteq \partial \widehat U_i$. Note that $\{a, b\}$ projects to the endpoints of both $I_1$ and $I_2$. Since $\Lambda(G)/\sim_\mathcal{L}$ is a Schottky set, the boundaries $\partial \widehat U_1$ and $\partial \widehat U_2$ intersect in exactly one point. It follows that each $I_i$ must be a singleton. Consequently, $\gamma_1$ and $\gamma_2$ correspond to the same leaf of $\mathcal{L}$. Therefore, $\mathcal{W}$ consists of a single leaf $l$ connecting $a, b$.
\end{proof}

\begin{prop}\label{prop:homeospidertypeIII}
    Let $h \in \Homeo^+(\Lambda(G), \mathcal{L})$ and let $\mathcal{W}$ be an ivy bud at some {\nonel} vertex $v$. Then there exists some ivy bud $\mathcal{W}'$ so that $h(\Lambda_\mathcal{W}) = \Lambda_{\mathcal{W}'}$.
\end{prop}
\begin{proof}
    Let $\Delta$ be the complementary component of $\Lambda_v$ so that $|\mathcal{W}| \subseteq \Delta$, and let $\Delta' = h(\Delta)$.

    If $\Lambda_\mathcal{W}$ is finite, then $\mathcal{W}$ consists of a single leaf by Lemma~\ref{lem:singleleaf}. Since $h$ preserves $\mathcal{L}$, we have $h(\Lambda_\mathcal{W}) = \Lambda_{\mathcal{W}'}$.

    If $\Lambda_\mathcal{W} = \partial \Delta$, then $h(\Lambda_\mathcal{W}) = h(\partial \Delta) = \partial \Delta' = \Lambda_{\mathcal{W}'}$.

    If $\Lambda_\mathcal{W}$ is a Cantor set, then there is a gap $A$ of $\mathcal{L}_\Delta$ so that $\Lambda_\mathcal{W} = \partial \Delta \cap \partial A$. 
    Let us first observe that $h(\mathcal{L}_\Delta) = \mathcal{L}_{\Delta'}$ (Lemma~\ref{lem:preservingLDelta}). Since the ideal boundary of each gap of $\mathcal{L}_\Delta$ is either $\partial^I \mathcal{U}$ as in \eqref{eqn:KU} or $\Lambda_{\mathcal{W}_i}$ for some ivy bud $\mathcal{W}_i$ and  that, for each $\mathcal{U}$, there exists $\mathcal{U}'$ so that $h$ sends $\partial^I \mathcal{U}$ to $\partial^I {\mathcal{U}'}$, it follows that $h(\Lambda_\mathcal{W}) = \Lambda_{\mathcal{W}'}$ for some ivy bud $\mathcal{W}'$ of $\Delta'$. 
\end{proof}

\begin{prop}\label{prop:schottkysettypeIII}
    Let $v$ be a {\nonel} vertex and let $H_v= \{h\in \Homeo^+(\Lambda(G), \mathcal{L}): h(\Lambda_v) = \Lambda_v\}$. There exists an $H_v$-invariant acylindrical lamination $\widetilde{\mathcal{L}}_v$ for $\Lambda_v$, i.e.,
    \begin{enumerate}
        \item for any $h \in H_v$, we have $h(\widetilde{\mathcal{L}}_v) = \widetilde{\mathcal{L}}_v$; and
        \item $\Lambda_v/\sim_{\widetilde{\mathcal{L}}_v}$ is homeomorphic to a Schottky set.
    \end{enumerate}
\end{prop}

\begin{proof}
    We first associate a $2$-orbifold $S_{\cW}$ to $H_{\cW}$ for each non-elementary ivy bud at $v$. 
    Let $\cW$ be an ivy bud contained in a component $\Delta$ of $\widehat\C\setminus \Lambda_v$. 
    By Corollary~\ref{cor:finiteindexHW}, $H_{\mathcal{W}}$ is a finite extension of $G_{\mathcal{W}}$ and its action on $\partial^I\Delta$ is a convergence action. 
    Since the stabilizer $G_{\Delta}$of $\Delta$ is a surface group and the action of $G_{\Delta}$ on the ivy buds contained in $\Delta$ is cofinite, the stabilizer $G_{\cW}$ on $\partial^I\Delta$ is a Fuchsian group. Thus $H_{\mathcal{W}}$ is also topologically conjugate to a Fuchsian group \cite[Theorem 6B]{tukia:fuchs}. 
    We denote by $S_{\cW}$ the corresponding quotient $2$-orbifold for $H_{\cW}$.
    
    Let $\widetilde{\mathcal{L}}_v^0:= \bigcup_\Delta \mathcal{L}_\Delta$, where $\mathcal{L}_\Delta$ is constructed in \eqref{eqn:LDelta} and where the union is over all the complementary components of $\Lambda_v$.
    We say that a lamination $\widetilde{\mathcal L}_v$ is admissible if 
    \begin{itemize}
        \item $\widetilde{\mathcal L}_v$ contains $\widetilde{\mathcal L}_v^0$;
        \item it consists of finitely many $H_v$-orbits of leaves, in particular, it is $H_v$-invariant; and
        \item each leaf corresponds to either the leaf of an elementary ivy bud $\mathcal{W}$ at $v$, or a simple closed curve on $S_{\cW}$ for some non-elementary ivy bud $\mathcal W$ at $v$.
    \end{itemize} 
    Let us check that the initial lamination $\widetilde{\mathcal L}_v^0$ is admissible. First, Lemma~\ref{lem:preservingLDelta} ensures that it is $H_v$-invariant.
    Each leaf of $\widetilde{\mathcal L}_v^0$ corresponds to either the leaf of some elementary ivy bud or to a component of $\partial S_{\cW}$ of some non-elementary ivy bud. 
    By proposition~\ref{prop:homeospidertypeIII}, $H_v$ acts on $\{\Lambda_{\mathcal{W}}: \mathcal{W} \text{ is an ivy bud at } v\}$ and has finitely many orbits since $G_v$ has only finitely many orbits.
    Since each orbifold $S_{\cW}$ has finitely many boundary components, $\widetilde{\mathcal L}_v^0$ consists of only finitely many $H_v$-orbits of leaves.
    We shall enlarge $\widetilde{\mathcal L}_v^0$ through admissible $H_v$-laminations until the lamination becomes acylindrical.

    Suppose we are given an admissible cylindrical lamination $\widetilde{\mathcal L}_v$. 
    We claim that there is a strictly larger admissible lamination that contains $\widetilde{\mathcal L}_v$. Since any $H_v$-invariant lamination is $G_v$-invariant, this process reduces the maximal number of disjoint non-homotopic essential cylinders in the quotient pared manifold for $G_v$. Thus the proposition follows from the claim.

    Let $a,b \in \Lambda_v$ project to a cut pair in the quotient $\Lambda_v/\sim_{\widetilde{\mathcal{L}}_v}$. (The case of a cut point in the quotient, or the case where $\Lambda_v/\sim_{\widetilde{\mathcal L}_v}$ is a circle can be treated similarly.)
    Then, the pair $\{a,b\}$ is on the boundary of $N \geq 2$ components $V_j, j = 1, 2,..., N$ of $\widehat\C \setminus (\Lambda_v \cup |\widetilde{\mathcal{L}}_v|)$. Let $\Lambda_j = \partial V_j \cap \Lambda_v$. 
    Since $\widetilde{\mathcal{L}}_v \supseteq \widetilde{\mathcal{L}}_v^0$, each component $V_j$ is a subset of $\widehat\C \setminus (\Lambda_v \cup |\widetilde{\mathcal{L}}_v^0|)$.
    Thus $\Lambda_j$ is a subset of either
    \begin{enumerate}
        \item\label{cor:altypeIII:1} $\Lambda_\mathcal{W}$ for a non-elementary ivy bud $\mathcal{W}$; or
        \item\label{cor:altypeIII:2} $\pi(\partial^I \mathcal{U})$ where $\partial^I \mathcal{U}$ is as in \eqref{eqn:KU} and $\pi: \partial^I \Delta \rightarrow \partial\Delta$ is the canonical projection.
    \end{enumerate}
    Since $\Lambda(G)/\sim_\mathcal{L}$ is a Schottky set, at most one set among $\{\Lambda_j\}$ is of the second kind \eqref{cor:altypeIII:2}.
    Thus, after changing indices if necessary, for $j = 1,..., N-1$, we have
    $$
    \Lambda_j \subseteq \Lambda_{\mathcal{W}_{j}}, \text{ for some non-elementary ivy bud }\mathcal{W}_{j}. 
    $$

    Let $H_{\Lambda_j} \leq H_{\cW_j}$ be the stabilizer of $\Lambda_j$ in $H_{\cW_j}$.
    Since $\widetilde{\mathcal{L}}_v$ is admissible, its leaves correspond to lifts of simple closed curves on $S_{\cW}$. Thus, we have a suborbifold $S_j \subseteq S_{\cW_j}$ corresponding to $\Lambda_j$.
    The pair $\{a,b\}$ corresponds to a simple closed curve $\gamma_j$ on $S_j$. Note that $\gamma_j$ may be homotopic to $\partial S_j$. Since $\Lambda(G)/\sim_{\mathcal{L}}$ is a Schottky set, this occurs for at most one of the indices, in which case all $\Lambda_j$ are of the first kind \eqref{cor:altypeIII:1}.
    Thus, after changing the indices if necessary, $\gamma_j$ is essential for all $j = 1,..., N-1$.
    Let $\mathcal{C}_j\subset S_j$ be  an essential simple closed curve that intersects $\gamma_j$ essentially in $S_j$. 
    Its lift is an $H_{\Lambda_j}$-invariant lamination $\mathcal{L}_j$. 
    Let us consider the union $\tilde{\cL}_v'$ of $\tilde{\cL}_v$ with $\mathcal{L}_j, j = 1,..., N-1$, and their $H_v$-translates. Since $\{\Lambda_{\cW}\}$ forms a null sequence,
    we may verify that $\tilde{\cL}_v'$ is still a lamination, admissible by construction.
    This proves the claim and thus the proposition.
    \end{proof}

\subsection*{Action on ivy buds from the boundary map}
Let $\mathcal{W}$ be an ivy bud. Recall $\mathcal{L}_{\mathcal{W}}$ is defined in \eqref{eqn:L_W}. 
Let $\Delta$ be the complementary component of $\Lambda_v$ that contains $\mathcal{W}$. Let $\partial^I|\mathcal{W}| \subseteq \partial^I\Delta$ be the ideal boundary of $|\mathcal{W}|$. Note that if $\Delta$ is a Jordan domain, then $\partial^I|\mathcal{W}| = \Lambda_\mathcal{W}$.
The following shows that the action of $\Homeo^+(\Lambda(G), \mathcal{L})$ on the ivy buds at a {\nonel} vertex is determined by its action on its boundary.
\begin{prop}\label{prop:boundaryactiondet}
    Let $v$ be a {\nonel} vertex, and $\mathcal{W}$ be an ivy bud at $v$. Let $h \in \Homeo^+(\Lambda(G), \mathcal{L})$ so that $h|_{\partial^I|\mathcal{W}|} = \id$. Then $h|_{\partial l} = \id$ for all $l \in \mathcal{L}_{\mathcal{W}}$.
\end{prop}
\begin{proof}
    We consider three cases.
    Suppose that $\partial^I |\mathcal{W}|$ consists of two points $a, b$ in the ideal boundary $\partial^I\Delta$. By Lemma~\ref{lem:singleleaf}, $\mathcal{W}$ consists of a single leaf $l$ connecting $a, b$ that are fixed by $h$ so  the proposition follows.

    Suppose that $\partial^I |\mathcal{W}|$ is a Cantor set on $\partial^I\Delta$ (see Figure~\ref{fig:spiderweb}). Let $(a,b), (c,d)$ be two complementary components of $\partial^I\Delta \setminus \partial^I |\cW|$. Let $\gamma_1, \gamma_2$ be the extremal paths in $|\mathcal{W}|$ connecting $a$ to $b$ and $c$ to $d$ respectively. Here, extremal means that $\gamma_1$ (or $\gamma_2$) separates the interval $(a, b) \subseteq \partial^I\Delta$ (or $(c,d)$) from $|\mathcal{W}|$.
    Then there exist two complementary components $\mathcal{U}_1, \mathcal{U}_2$ of $\Delta \setminus \bigcup_{\mathcal{W}'} \fil(|\mathcal{W}'|)$ so that $\gamma_i \subseteq \partial \mathcal{U}_i$.
    Note that there is a finite chain $V_i, i =0,..., l+1$ of components of $\Delta \setminus \bigcup_{\mathcal{W}'} \fil(|\mathcal{W}'|)$ so that
    \begin{enumerate}
        \item\label{condW:1} $V_0 = \mathcal{U}_1$ and $V_{l+1} = \mathcal{U}_2$;
        \item\label{condW:2} $\partial V_i \cap \partial V_{i+1} \neq \emptyset$ and $\partial V_i \cap \partial V_{i+1} \subseteq |\mathcal{W}|$.
    \end{enumerate}
    Let $\widehat {\mathcal{U}}_1, \widehat {\mathcal{U}}_2$ be the complementary components of $\Lambda(G)/\sim_\mathcal{L}$ associated to $\mathcal{U}_1, \mathcal{U}_2$ respectively. Let $I_i \subseteq \partial \widehat {\mathcal{U}}_i$ be the projection of $\gamma_i$.
    Since $h|_{\Lambda_\mathcal{W}} = \id$, we conclude that the induced map $\widehat h : \Lambda(G)/\sim_\mathcal{L} \rightarrow \Lambda(G)/\sim_\mathcal{L}$ satisfies $\widehat h(I_i) = I_i$.
    Note that $V_0,..., V_{l+1}$ induces a chain of touching circles of length $l$ in $\Lambda(G)/\sim_\mathcal{L}$ that connects $I_1 \subseteq \partial U_1$ to $I_2\subseteq \partial U_2$. For a fixed $l$, there are only finitely many such chains, and there is a linear order on the collection of such chains. 
    Therefore, $\widehat h$ fixes the chain induced by $V_0,..., V_{l+1}$. Thus, there exist two bordering leaves $l_1, l_2 \in \mathcal{L}_\mathcal{W}$ so that $h|_{\partial l_1 \cup \partial l_2} = \id$. By Corollary~\ref{cor:twopointrigidity}, $h|_{\partial l} = \id$ for all $l \in \mathcal{L}_\mathcal{W}$.

    Finally suppose that $\partial^I |\mathcal{W}| = \partial^I\Delta$. 
    Let $l_1, l_2 \in \cL_{\cW}$ be a pair of bordering leaves contained in an induced simple cycle of length some $N$. 
    Since there are only finitely many $G$-orbits of such pairs of bordering leaves, by the pigeonhole principle, there exists $k, j\in \Z_{\geq 1}$ and $g \in G$ so that $h^k(h^j(l_1)) = g(h^j(l_1)), h^k(h^j(l_2)) = g(h^j(l_2))$. By Corollary~\ref{cor:twopointrigidity}, $h^k|_{\partial l} = g|_{\partial l}$ for all $l \in h^j(\cL_{\cW})$. In particular, $h^{k} = g$ on $\partial^I|\cW| = h^j(\partial^I|\cW|)$. Since $h|_{\partial^I|\cW|} = \id$, we have $g = \id$. Thus, $h^k|_{\partial l} = \id$ for all $l \in h^j(\cL_{\cW})$. Since $h$ is invertible, $h^k|_{\partial l} = \id$ for all $l \in \cL_{\cW}$.
    Note that this means $h^k$ projects to the identity map on the corresponding subset of $\cW$ in the quotient $\Lambda(G)/\sim_{\cL}$ (see Figure~\ref{fig:spiderwebFullBoundary}).
    The proof of \cite[Proposition 3.2]{CK94} (K\'er\'ekjart\`o's theorem) that any periodic homeomorphism of the closed disk 
    which is the identity on the boundary must be the identity can be adapted to show that $h|_{\partial l} = \id$ holds for all $l \in \cL_{\cW}$
    by replacing classical (arc) crosscuts with coarse crosscuts provided by $\cW$ in their argument.
\end{proof}

\subsection{Proof of Theorem~\ref{thm:carpetfreeSchottkyimpliesrigid}}\label{sec:proofThmRigid}
We proceed by induction on $\height(G)$.
\subsection*{The base case.} 
Suppose that $\height(G) = 0$.
Let $\widehat G$ be a de-parabolization of $G$.
Then $\widehat G$ is either a cocompact Fuchsian group or a carpet group.

Suppose that $\widehat G$ is a cocompact Fuchsian group.
Let $\mathcal{L}$ be the corresponding lamination so that $\Lambda(G) = \Lambda(\widehat G)/\sim_\mathcal{L}$.
By Proposition~\ref{prop:homeomisom}, we have  $\Homeo^+(\Lambda(G)) \cong \Homeo^+(\Lambda(\widehat G), \mathcal{L})$.
By applying Corollary~\ref{cor:finiteindexfuchsian} to the case $\mathcal{L}_\partial = \emptyset$, we deduce that $\widehat G$ has finite index in $\Homeo^+(\Lambda(\widehat G), \mathcal{L})$.
Therefore, $G$ has finite index in $\Homeo^+(\Lambda(G)) = \HQ^+(\Lambda(G))$ (this case also follows from \cite[Lemma 21]{KK00}).

Suppose that $\widehat G$ is a carpet group. Then $\HQ^+(\Lambda(G)) = \QS(\Lambda(G))$ by definition. 
The conclusion then follows for instance from \cite[Theorem 6.6]{haiss:lec}.

\subsection*{The induction step} Let $n\ge 0$.
Suppose that any geometrically finite Kleinian group with Schottky limit set of height at most $n$ is a finite index subgroup of $\HQ^+(\Lambda(G))$.

Let $G$ be a geometrically finite Kleinian group with Schottky limit set of height $\height(G) = n+1$.
Let $\widehat G$ be a de-parabolization of $G$.
Let $\Lambda(\widehat G)$ be the corresponding limit set equipped with its acylindrical lamination $\mathcal{L}$ so that 
$$
\Lambda(G) = \Lambda(\widehat G)/\sim_\mathcal{L}.
$$
We consider two cases.

\medskip
\noindent{\bf Case (1):} $\Lambda(\widehat G)$ is connected.
Let $\mathcal{T}$ be the JSJ tree of $\widehat G$.
Recall that for each vertex $v\in\cT$, the JSJ splitting induces a  lamination $\mathcal{L}_v$ (see the definition in \S~\ref{subsec:separatinggeodandlam}).
The proof studies the interaction of the lamination $\cL_v$ with $\cL$. To avoid confusion, we use Greek letters for leaves in $\cL_v$ and $l$ for leaves in $\cL$.

Let $v$ be a {\twoended} vertex, and $\gamma \in \mathcal{L}_v$.
We say that two leaves $l, l' \in \cL$ are {\em bordering} at $\gamma$ if $\gamma \cap l$ and $\gamma \cap l'$ are adjacent in $\gamma \cap |\mathcal{L}|$.
Two leaves $l, l'$ are {\em bordering} if they are bordering at $\gamma$ for some {\twoended} vertex $v$ and $\gamma \in \mathcal{L}_v$.
Consider the set
\begin{equation}\label{eqn:Xbordering}
    X = \{(l, l'): l, l' \text{ are bordering}\}.
\end{equation}

The following lemma shows that $\Homeo^+(\Lambda(\widehat G), \mathcal{L})$ acts on $X$. 

\begin{lem}\label{lem:actiononborderingpair}
    Let $h \in \Homeo^+(\Lambda(\widehat G), \mathcal{L})$. Let $l, l' \in \mathcal{L}$ be a pair of bordering leaves. Then $h(l), h(l')$ are bordering.
\end{lem}
\begin{proof}
    Suppose $l, l'$ are bordering at $\gamma$.
    Let $\Delta$ be the component of $\Omega(\widehat G)$ that contains both $l$ and $l'$. Since $\Lambda(\widehat G)$ is connected, $\Delta$ is simply connected. Let $\partial^I\gamma$ be the endpoints of $\gamma$ in the ideal boundary $\partial^I\Delta$. Since the JSJ decomposition only depends on the topology of the limit set, we conclude that there exists $\gamma' \in \mathcal{L}_{v'}$ for some {\twoended} vertex $v'$ so that 
    $$
    h(\partial^I\gamma) = \partial^I\gamma' \subseteq \partial^I h(\Delta).
    $$
    The planarity implies that $h(l), h(l')$ are bordering at $\gamma'$.
\end{proof}

\begin{lem}\label{lem:infinitegammaid}
    Let $h \in \Homeo^+(\Lambda(\widehat G), \mathcal{L})$, and let $l_0, l_1 \in \mathcal{L}$ be a pair of bordering leaves at $\gamma_0$.
    Suppose that $h|_{\partial l_0 \cup \partial l_1} = \id$. Then for each leaf $l \in \mathcal{L}$ intersecting $\gamma_0$, we have $h|_{\partial l} = \id$.
\end{lem}
\begin{proof}
    Let $l \in \mathcal{L}$ be a leaf that intersects $\gamma_0$.
    Note that $\gamma_0 \cap |\mathcal{L}|$ is discrete with a linear order induced from $\gamma_0$, and any homeomorphism preserves this linear order. Since $h(l_0) = l_0$ and $h(l_1) = l_1$, we have $h(l) = l$. Thus $h(\partial l) = \partial l$.
    Note that all such curves are contained in a single component in $\Omega(G)$, and $\gamma_0$ separates the components into two parts.
    Since $h|_{\partial l_0} = \id$, we conclude that $h|_{\partial l} = \id$.
\end{proof}

 Let $\HQ^+(\Lambda(\widehat G), \mathcal{L})$ be the lift of $\HQ^+(\Lambda(G))$ in $\Homeo^+(\Lambda(\widehat G), \mathcal{L})$.
We use induction to show that the stabilizer of a point of $X$ for the action of $\HQ^+(\Lambda(\widehat G), \mathcal{L})$ is finite.
\begin{lem}\label{lem:stabbordering}
    Let $h \in \HQ^+(\Lambda(\widehat G), \mathcal{L})$ and let $l_0, l_1 \in \mathcal{L}$ be a pair of bordering leaves. Suppose $h|_{\partial l_0 \cup \partial l_1} = \id$.
    Then $h = \id$.
\end{lem}
\begin{proof}
    Let $v_0$ be the {\twoended} vertex with $\gamma_0 \in \mathcal{L}_{v_0}$ such that $l_0, l_1$ are bordering at $\gamma_0$.
    Let $w$ be a vertex of $\mathcal{T}$ adjacent to $v_0$ so that $\gamma_0 \in \mathcal{L}_w$ (\S~\ref{subsec:separatinggeodandlam} for the definition of $\mathcal{L}_w$). Since $\mathcal{T}$ is bipartite, $w$ is not {\twoended}.
    
    We claim that $h|_{\Lambda_w} = \id$. Moreover, for each $\gamma \in \mathcal{L}_w$, and each leaf $l \in \mathcal{L}$ intersecting $\gamma$, we have $h|_{\partial l} = \id$.
    \begin{proof}[Proof of claim]
    If $w$ is {\mhf}, then $G_w$ is conjugate to a convex cocompact Fuchsian group. In the setting of \S~\ref{subsec:lamination} applied to $G_w$, the
    lamination $\mathcal{L}_w$ defines the boundary lamination and $\mathcal{L}$ restricts to an acylindrical lamination. 
    Note $l_0$ restricts to a leaf $l_{0,w}$ for $G_w$ and $\gamma_0, l_{0,w}$ are linked as in Definition~\ref{defn:link}. Therefore,  Theorem~\ref{thm:toprigidityconvexcocompact} 
    implies  $h|_{\Lambda_w} = \id$, proving the moreover part.
    
    If $w$ is {\rigid}, then $\Lambda_w/\sim_{\mathcal{L}_w}$ is homeomorphic to a Schottky set by Proposition~\ref{prop:rigidvertexSchottky}. Let $\widecheck{G}_w$ be the corresponding geometrically finite group associated to $\Lambda_w/\sim_{\mathcal{L}_w}$.
    By the induction hypothesis, $\height(\widecheck{G}_w) \leq n$.
    Note that there are infinitely many leaves $\gamma_i \in \mathcal{L}_w$ such that there exists a leaf $l \in \mathcal{L}$ intersecting both $\gamma_0$ and $\gamma_i$.
    Thus, by Lemma~\ref{lem:infinitegammaid}, we conclude that $h(\gamma) = \gamma$ for infinitely many leaves $\gamma \in \mathcal{L}_w$. Therefore, the induced map $\widecheck{h}_w$ of $h$ fixes infinitely many distinct points on the quotient $\Lambda_w/\sim_{\mathcal{L}_w}$. 
    Since $h$ belongs to $\HQ^+(\Lambda(\widehat G), \mathcal{L})$, it follows that we have $\widecheck{h}_w \in \HQ^+(\Lambda(\widecheck{G}_w))$. 
    Indeed, let $\Lambda(\widecheck{H})$ be a maximal carpet-quotient in $\Lambda(\widecheck{G}_w)$ that is mapped to some $\Lambda(\widecheck{H'})$. 
    Let $\widehat{H},\widehat{H'} \leq \widehat G$ be the corresponding carpet subgroups.  
    Since $h \in \HQ^+(\Lambda(\widehat G), \mathcal{L})$, it follows from \cite[Theorem 6.6]{haiss:lec} 
    that $h:\La(\hat{H})\to \La(\hat{H'})$ is  virtually equivariant. Therefore, this is also the case for $\widecheck{h}_w$, so \cite[Theorem 3.3]{tukia:ihes85}
    implies it is quasisymmetric on $\Lambda(\widecheck{H})$.
    By the induction hypothesis and Corollary~\ref{cor:toprig:finitext}, $\widecheck{h}_w = \id$. Therefore, we obtain $h|_{\Lambda_w} = \id$. The moreover part follows from the fact $h|_{\Lambda_w} = \id$.

    If $w$ is {\ranktwo}, then $\Lambda_w$ is a point, so $h|_{\Lambda_w} = \id$. Note that $\mathcal{L}_{w}$ admits a natural linear order (see Figure~ \ref{fig:ranktwoSch}), label them by $\gamma_n, n\in \Z$.
    Since $\mathcal{L}$ is acylindrical, $\gamma_n \cap |\mathcal{L}| \neq \emptyset$ for all $n$.
    Any vertex $v$ at distance $2$ from $w$ is either {\mhf} or {\rigid}. Moreover, $\mathcal{L}_v \cap \mathcal{L}_w$ consists exactly of two adjacent leaves and each leaf $\gamma \in \mathcal{L}_w$ is contained in the lamination of exactly two vertices at distance $2$ from $w$. Thus by applying the above argument for {\mhf} and {\rigid} vertices inductively on $n$, we conclude the moreover part.
    \end{proof}
    
    Since $\mathcal{L}$ is acylindrical, at most one leaf $\gamma \in \mathcal{L}_{v_0}$ satisfies $\gamma \cap |\mathcal{L}| = \emptyset$.
    Note that $\mathcal{L}_{v_0}$ admits a natural cyclic order (see Figure~\ref{fig:I_p}).
    Since for each adjacent vertex $w$ of $v_0$, $\mathcal{L}_w \cap \mathcal{L}_{v_0}$ consists of exactly two leaves, and each leaf $\gamma\in \mathcal{L}_{v_0}$ is contained in the lamination of exactly two adjacent vertices of $v_0$, by applying the argument inductively, we conclude that $h|_{\Lambda_w} = \id$ for any adjacent vertex $w$ of $v_0$.
    Therefore, for any vertex $v$ at distance $2$ from $v_0$, there exists a leaf $\gamma \in \mathcal{L}_v$ with $\gamma \cap |\mathcal{L}| \neq \emptyset$ such that  we have $h|_{\partial l} = \id$ for each $l \in \mathcal{L}$ that intersects $\gamma$.
    
    Thus, by inductively applying the argument on the distance to $v_0$, we conclude that $h|_{\Lambda_w} = \id$ for each vertex $w$ of $\mathcal{T}$. Since the union $\bigcup_w \Lambda_w$ is dense in $\Lambda(\widehat G)$, $h|_{\Lambda(\widehat G)} = \id$ follows.
\end{proof}

To finish the proof of the virtual topological rigidity 
in this case, we note that the action of $\widehat G$ on the set $X$ defined as in \eqref{eqn:Xbordering} is cofinite.
By Lemma~\ref{lem:actiononborderingpair}, $\HQ^+(\Lambda(\widehat G), \mathcal{L})$ acts on $X$ and by Lemma~\ref{lem:stabbordering}, the stabilizer of any point in $X$ is finite.
Therefore,  Proposition~\ref{prop:finiteindex} shows that $\widehat G$ has finite index in $\HQ^+(\Lambda(\widehat G), \mathcal{L})$.
Finally,  Proposition~\ref{prop:homeomisom} asserts that $G$ has finite index in $\HQ^+(\Lambda(G))$.

\medskip
\noindent{\bf Case (2):} $\Lambda(\widehat G)$ is disconnected.

Let $\mathcal{E}=\{D_1,..., D_s\}$ be a compatible core decomposition (see \S~\ref{subsec:compatiblecompressiondisks} and Proposition~\ref{prop:admiss} for its existence), and let $\mathcal{T}$ be the corresponding tree for the incompressible decomposition associated to $\mathcal{E}$.

Following the notations in \S~\ref{subsec:al}, let $\mathcal{D}$ be the lift of $\{\partial D_i\}$ in $\Omega(\widehat G)$.
Let $N_0$ be so that each Jordan curve $C \in \mathcal{D}$ induces a simple cycle of length $\leq N_0$.
Note each leaf $l \in \mathcal{L}$ corresponds to a contact point of two complementary components of $\Lambda(\widehat G)/\sim_{\mathcal{L}}$ (see \eqref{eqn:correspodencecompcontact}), so it corresponds to an edge in the nerve.
Let $X$ be the collection of pairs $(l,l')$ so that $l, l' \in \mathcal{L}$ correspond to two adjacent edges in an induced simple cycle of length $\leq N_0$ in the nerve of $\Lambda(\widehat G)/\sim_{\mathcal{L}}$.
Note that $X/\widehat G$ is finite, and $\Homeo^+(\Lambda(\widehat G), \mathcal{L})$ acts on $X$ by Proposition~\ref{prop:homeomisom}. Thus, $\HQ^+(\Lambda(\widehat G), \mathcal{L})$ also acts on $X$.

Recall that $l_0, l_1$ are {\em bordering at $C \in \mathcal{D}$} if $C \cap l_0$ and $C \cap l_1$ form the two endpoints of a connected component of $C \setminus |\mathcal{L}|$ (Definition~\ref{defn:linkedbordering_new}).
With the above notations, we have the same conclusion as in the connected case, cf. Lemma~\ref{lem:stabbordering}.
\begin{lem}\label{lem:stabborderingdisconnected}
     Let $h \in \HQ^+(\Lambda(\widehat G), \mathcal{L})$ and let $l_0, l_1 \in \mathcal{L}$ be a pair of bordering leaves. Suppose $h|_{\partial l_0 \cup \partial l_1} = \id$. Then $h = \id$.
\end{lem}
\begin{proof}
    Let $v_0$ be the disk vertex with corresponding Jordan curve $C_0 \in \mathcal{D}$ so that $l_0, l_1$ are bordering at $C_0$. Note that there are exactly two vertices $w_1, w_2$ of $\mathcal{T}$ adjacent to $v_0$, and $C_0 \in \mathcal{D}_{w_i}$ (see \eqref{eqn:circlelaminationatv}). Neither $w_1$ nor $w_2$ is a disk vertex.
    
    We claim that for each vertex $w$ adjacent to $v_0$, $h|_{\Lambda_w} = \id$. Moreover, for any ivy bud $\mathcal{W}$ at $w$ and any leaf $l \in \mathcal{L}_\mathcal{W}$, we have $h|_{\partial l} = \id$.
    \begin{proof}[Proof of claim]
        Suppose $w$ is a ball or a parabolic vertex (i.e., type B or P) and let $\mathcal{W}$ be the unique ivy bud at $w$ (cf. Proposition~\ref{prop:uniqueextspiderweb}). By Corollary~\ref{cor:twopointrigidity}, $h|_{\partial l} = \id$ for all $l \in \mathcal{L}_{\mathcal{W}}$. 

        Suppose $w$ is non-elementary (type NE). Let $\mathcal{W}$ be the ivy bud at $w$ that contains $C_0$. Note that $\mathcal{W}$ is necessarily non-elementary. By Corollary~\ref{cor:twopointrigidity}, $h|_{\Lambda_\mathcal{W}}= \id$. Note that $h(\Lambda_w) = \Lambda_w$.
        By Proposition~\ref{prop:schottkysettypeIII}, there exists an acylindrical lamination $\widetilde{\mathcal{L}}_w$ at $\Lambda_w$ so that $h(\widetilde{\mathcal{L}}_w) = \widetilde{\mathcal{L}}_w$.
        Let $\widecheck{G}_w$ be the corresponding geometrically finite group associated to $\Lambda_w/\widetilde{\mathcal{L}}_w$.
        By the induction hypothesis, $\height(\widecheck{G}_w) \leq n$ holds. Since $\mathcal{W}$ is not a single leaf, by Lemma~\ref{lem:singleleaf}, $\Lambda_\mathcal{W} \subseteq \Lambda_w$ contains infinitely many points. Therefore, the induced map $\widecheck{h}_w$ of $h$ fixes infinitely many distinct points on the quotient $\Lambda_w/\widetilde{\mathcal{L}}_w$, and is quasisymmetric on each dynamical carpet-quotient.
        By the induction hypothesis and Corollary~\ref{cor:toprig:finitext}, $\widecheck{h}_w = \id$. Therefore, $h|_{\Lambda_w} = \id$.
        The moreover part follows from Proposition~\ref{prop:boundaryactiondet}.
    \end{proof}

    It follows from the claim that for any vertex $v$ at distance $2$ from $v_0$ (which is necessarily a disk vertex (type D)) with the corresponding Jordan curve $C_v \in \mathcal{D}$, and any $l \in \mathcal{L}$ that intersects $C_v$, we have $h|_{\partial l} = \id$. Thus, by inductively applying the above argument, we conclude that $h|_{\partial l} = \id$ for all $l \in \cL$. Since the union $\bigcup_{l\in \cL} \partial l$ is dense in $\Lambda(\widehat G)$, we have $h|_{\Lambda(\widehat G)} = \id$.
\end{proof}
To finish the proof of the virtual topological rigidity, fix a pair $(l_0, l_1) \in X$ as in Lemma~\ref{lem:stabborderingdisconnected} and let $X_0$ be the orbit of $(l_0, l_1)$ under $\HQ^+(\Lambda(\widehat G), \mathcal{L})$. By Lemma~\ref{lem:stabborderingdisconnected}, the stabilizer in $\Homeo^+(\Lambda(\widehat G), \mathcal{L})$ of any point in $X_0$ is finite. 
Since $X/\widehat G$ is finite, $X_0/\widehat G$ is also finite, and, by Proposition~\ref{prop:finiteindex}, $\widehat G$ has finite index in $\HQ^+(\Lambda(\widehat G), \mathcal{L})$.
Therefore, by Proposition~\ref{prop:homeomisom}, $G$ has finite index in $\HQ^+(\Lambda(G))$.

This proves the induction step, and the theorem follows.
\qed

\section{Quasisymmetric maps between colored Fuchsian limit
sets}\label{sec:qsfuchsian}
In this section, we construct color-preserving quasisymmetric maps between Fuchsian Cantor limit sets that are locally equivariant. This will be a key ingredient for the proof of the universality, where we glue quasisymmetries along isolated points to obtain a global quasisymmetry between limit sets.

Let $G$ be a convex cocompact Fuchsian group of the second kind that acts on $\SS^1$.
Thus, its limit set $\Lambda(G)$ is a Cantor set and the unit circle admits
a $G$-invariant partition $\SS^1= \Lambda(G) \cup \Omega(G)$. 
Let $F$ be a finite set. A {\em coloring} of $\Omega(G)$ is a $G$-invariant map $\kappa: \pi_0(\Omega(G))\to F$ that is onto. 

\begin{theorem}\label{thm:colors}
    Let $G_1, G_2$ be convex cocompact Fuchsian groups of the second kind with colorings $\kappa_i: \pi_0(\Omega(G_i))\to F, i =1,2$. 
    Let $J^i, i = 1, 2$, be a component of $\Omega(G_i)$, both of the same color.
    There exists a color-preserving quasisymmetric map $h:(\SS^1, \Lambda(G_1)) \rightarrow (\SS^1, \Lambda(G_2))$ so that
    \begin{enumerate}
        \item\label{thm:colors:eq1} $h(J^1) = J^2$;
        \item\label{thm:colors:eq2} there exists $\varepsilon > 0$ so that for each component $J$ of $\Omega(G_1)$ and $a\in \partial J$, $h$ conjugates $\stab_{G_1}(a)$ and $\stab_{G_2}(h(a))$ on $[a-\varepsilon|J|, a+ \varepsilon|J|]$.
        
        More precisely, let $g_a \in \stab_{G_1}(a)$ and $g_{h(a)} \in \stab_{G_2}(h(a))$ be the generators with repelling fixed points at $a$ and $h(a)$. Then for each $x \in [a-\varepsilon|J|, a+ \varepsilon|J|]$, we have
        $$
        h(g_a(x)) = g_{h(a)}(h(x)).
        $$
    \end{enumerate}
\end{theorem}

\begin{remark}\label{rmk:localconj}
The existence of a color-preserving quasisymmetry can be established from standard methods, cf. \cite{macmanus:cantor:qcircle}. 
 But this happens to be insufficient for our purpose since we will need to glue together quasisymmetric maps that only agree on isolated points. 
 In this respect, we need to control the behavior of the quasisymmetries in the neighborhood of these points. This difficulty is overcome with the local conjugacy obtained by the conclusion \eqref{thm:colors:eq2} of the above theorem.
\end{remark}

We equip $G_i$ with a word metric.
Let $\mathbb{P}_i = \{H_{i,1},..., H_{i,k_i}\}, i = 1,2$ be a collection of representatives of the conjugacy classes of the stabilizer of a component of $\Omega(G_i)$.
Let $\cP_i = \{g H: g \in G_i, H \in \mathbb{P}_i\}, i=1,2$ be the collection of their cosets. 
The coloring of the components of $\Omega(G_i)$ induces a coloring $\kappa_i: \cP_i \rightarrow F$.

\begin{cor}\label{cor:colors} 
Let $(G_1,\cP_1)$ and $(G_2,\cP_2)$ be two convex cocompact Fuchsian groups of the second kind, equipped with colorings $\kappa_i: \cP_i\to F, i =1,2$.  Let $P^i \in \cP_i, i = 1, 2$ be of the same color. 
Then there exists a color-preserving quasi-isometry $\phi:  (G_1,\cP_1)\to (G_2,\cP_2)$ so that
\begin{enumerate}
    \item\label{cor:colors:cond:1} $\phi(P^1) = P^2$;
    \item\label{cor:colors:cond:2} there exists $M> 0$ so that for each $P\in \cP^1$, $\phi|_P: P \rightarrow \phi(P)$ lies within distance $M$ of 
    a translation under identifications $P \cong \Z$ and  $\phi(P)\cong \Z$.
\end{enumerate}
\end{cor}

\begin{proof}  
    Note that the collection of cosets $\cP_i$ is in correspondence with the components of $\Omega(G_i)$. Let $J^i \subseteq \Omega(G_i)$ be the corresponding component for $P^i$. By Theorem~\ref{thm:colors}, we have a color-preserving quasisymmetric map $h:(\SS^1, \Lambda(G_1)) \rightarrow (\SS^1, \Lambda(G_2))$. 
   Let $H: \D \rightarrow \D$ be the Douady-Earle extension of $h$ (see \cite{DE86}). Note that $H$ is real analytic and bi-Lipschitz in the hyperbolic metric and the image of the convex hull $\cC_1$ of $\Lambda(G_1)$ is quasi-isometric to the convex hull $\cC_2$ of $\Lambda(G_2)$. By the \v{S}varc-Milnor lemma, there exist equivariant quasi-isometries $\psi_i: (G_i, P_i) \rightarrow (\cC_i, \partial \cC_i)$. Let $\phi:= \psi_2^{-1} \circ H \circ \psi_1$ be their composition. 
    Thus $\phi: (G_1, \cP_1) \rightarrow (G_2, \cP_2)$ is a quasi-isometry that is color-preserving and such that $\phi(P^1) = P^2$. 
    
    It remains to verify \eqref{cor:colors:cond:2}.
    Let $J_1 = (a_1, b_1)$ be a component of $\Omega(G_1)$ and let $\gamma_1 \subseteq \partial \cC_1$ be the hyperbolic geodesic in $\mathbb{D}$ with endpoints $\partial J_1$. Let $J_2 = h(J_1) = (a_2, b_2)$, and $\gamma_2$ be the corresponding hyperbolic geodesic with endpoints $\partial J_2$. Note that $\gamma_i$ projects to a closed geodesic in the quotient $\D/G_i$, and let $l_i$ denote its length in the quotient.
    
    We claim that $H|_{\gamma_1}$ is uniformly bounded from some affine map (with respect to arc-length in the hyperbolic metric) between $\gamma_1$ and $\gamma_2$ that stretches the hyperbolic length by $l_2/l_1$.
    Let $\alpha_1$ be the geodesic connecting $a_1 \pm \varepsilon|J_1|$, let $\alpha_2$ be the geodesic joining $h(a_1 \pm \varepsilon|J_1|)$, and let $A_i$ be the closed hyperbolic half-plane bounded by $\alpha_i$ that contains $a_i$ on its ideal boundary. 
     Let $\gamma_i^+ = \gamma_i \cap A_i$, and let $L_{J_1, +}: \gamma_1^+ \rightarrow \gamma_2^+$ be the affine map  with a stretch factor $l_2/l_1$ that sends $\gamma_1 \cap \alpha_1$ to $\gamma_2 \cap \alpha_2$.
    
    We now show that there exists $K_0$, which only depends on the quasisymmetry of $h$ and $l_1, l_2$ so that 
    $$
    \dist_\D(H(x), L_{J_1, +}(x))\leq K_0 \text{ for all } x\in \gamma_1^+.
    $$
    Indeed, since $H$ is bi-Lipschitz, the Morse lemma 
    implies that $H(\gamma_1\cap\alpha_1)$ has to be close to both geodesics $\gamma_2$ and $\alpha_2$, hence close to  $\gamma_2\cap \alpha_2$, where the distance only depends on the
    quasisymmetry of $h$.
    Thus, there exists a constant $K_0'$ that depends only on the quasisymmetry of $h$ so that
    $$
    \dist_\D(H(x_0), L_{J_1, +}(x_0))\leq K_0' \text{ where } x_0 = \gamma_1 \cap \alpha_1.
    $$
    Thus,  by compactness, there exists a constant $K_0$ that only depends on the quasisymmetry of $h$ 
    and the translation lengths $l_1, l_2$ so that 
    $$
    \dist_\D(H(y), L_{J_1, +}(y))\leq K_0 \text{ for all } y\in \gamma_1^+ \text{ with } \dist_\D(y, x_0)  \le l_1.
    $$
    Let $g_i \in G_i$ be the generator of the stabilizer of $J_i$ which is repelling at $a_i$. By the local conjugacy condition, we conclude that $g_2^n\circ h\circ g_1^{-n}(a_1\pm \epsilon|J_1|) = h(a_1\pm \epsilon|J_1|)$. 
    Applying the above argument to $g_2^n\circ h\circ g_1^{-n}$, and using the fact the Douady-Earle extension is conformally natural, we conclude that
    $$
    \dist_\D(g_2^n\circ H\circ g_1^{-n}(y), L_{J_1, +} (y))\leq K_0, \text{ for all } y\in \gamma_1^+, \text{ with } \dist_\D(y, x_0) \leq l_1.
    $$
    By construction, $L_{J_1, +} = g_2^n \circ L_{J_1, +} \circ g_1^{-n}$. Therefore, for each $x \in \gamma_1^+$, there exists $n$ so that $y = g_1^n(x) \in \gamma_1^+$ with $\dist_\D(y, x_0) \le l_1$ so that 
    \begin{eqnarray*}
        \dist_\D(H(x), L_{J_1, +}(x)) & = & \dist_\D(H(g_1^{-n}(y)), L_{J_1, +}(g_1^{-n}(y))) \\ 
        &= &\dist_\D(g_2^n\circ H\circ g_1^{-n}(y), g_2^n\circ L_{J_1, +}\circ g_1^{-n}(y)) \\
        &= &\dist_\D(g_2^n\circ H\circ g_1^{-n}(y), L_{J_1, +}(y)) \leq K_0.
  \end{eqnarray*}

    Similarly, let $\beta_1$ be the geodesic connecting $b_1 \pm \varepsilon|J_1|$, let $\beta_2$ be the geodesic connecting $h(b_1 \pm \varepsilon|J_1|)$
    and let $B_i$ be the corresponding closed hyperbolic half-plane bounded by $\beta_i$. 
    Let $\gamma_i^- = \gamma_i \cap B_i$ and $L_{J_1, -}: \gamma_1^- \rightarrow \gamma_2^-$ be the affine map with a stretch factor $l_2/l_1$ that sends $\gamma_1 \cap \beta_1$ to $\gamma_2 \cap \beta_2$. Then 
    $$
    \dist_\D(H(x), L_{J_1, -}(x))\leq K_0 \text{ for all } x\in \gamma_1^-.
    $$
    Note that we have a decomposition $\gamma_i = \gamma^+_i \cup \gamma^0_i \cup \gamma^-_i$. By construction, the hyperbolic length
    of $\g^0_i$ is bounded by some constant $K_1$ that depends on $\ep$ and the quasisymmetry of $h$, i.e., $l(\gamma^0_i) \leq K_1$. 
    Let $L_{J_1}:\gamma_1 \rightarrow \gamma_2$ be the affine map with a stretch factor $l_2/l_1$ that agrees with $L_{J_1, +}$ on $\gamma^+_1$. Then there exists $K_2$, which depends on $K_1$ and $l_2/l_1$ so that 
    $$
    \dist_\D(L_{J_1}(x), L_{J_1, -}(x))\leq K_2 \text{ for all } x\in \gamma_1^-.
    $$
    Since there are only finitely many boundary curves in $\partial \cC_i/G_i$, there are only finitely many such lengths $l_1, l_2$, so the constant $K_2$ can be chosen independent of $J_1$. Thus, there exists $K_3$, independent of the interval $J_1$, so that 
    $$
    \dist_\D(H(x), L_{J_1}(x))\leq K_3 \text{ for all } x\in \gamma_1.
    $$
    By taking compositions, we see that $\phi|_P$ is within uniformly bounded distance of a translation.
    \end{proof}

\subsection*{Strategy of the proof of Theorem~\ref{thm:colors}}
It suffices to prove Theorem~\ref{thm:colors} for torsion-free Fuchsian groups as passing to a finite index subgroup would not affect the theorem.
The proof proceeds by constructing a piecewise M\"obius map on $\SS^1$, called the {\em Bowen-Series maps} and {\em blown-up Bowen-Series maps}  which model the dynamics of the Fuchsian group $G$ (see \S~\ref{subsec:bsm}). Theorem~\ref{thm:qsmodel.n} shows that this Bowen-Series map can be quasisymmetrically modeled by a special class of piecewise linear maps, called {\em $d$-adic Markov maps} and {\em blown-up $d$-adic Markov maps}, where $d$ is determined by the number of colors (\S~\ref{subsec:dadicmodel}). 
Even though these $d$-adic models are different for different Fuchsian groups, they share the same repeller, so they provide us with a quasisymmetric homeomorphism between the colored Cantor limit sets.  Moreover, one
can arrange that these $d$-adic models contain in their full dynamics the same germs at the boundary points (Proposition~\ref{prop:qconj}). This property enables us to obtain the local conjugacy.

\subsection{Conformal Markov maps on $\SS^1$ and variants}
In this section, we summarize properties of conformal Markov maps on $\SS^1$. They are a special case of $(X, \mathscr{C})$-map studied in \cite{LMM26}.

\begin{defn}\label{XCstrs}
    A finite collection $\mathcal{A} = \{A_1,..., A_m\}$ of subsets of $\SS^1$ is called a {\em partition} if
    \begin{itemize}
        \item $A_i$ is a closed interval;
        \item $\Int A_i \cap \Int A_j = \emptyset$ if $i \neq j$;
        \item $\bigcup A_i = \SS^1$.
    \end{itemize}
    We always assume the intervals $A_i$ are ordered and oriented counterclockwise on the circle with indices in $\Z/m\Z$.
    We define the set of {\em break-points} as 
    $$
    \mathrm{bk}(\mathcal{A}):= \{x \in \SS^1: x \in A_i\cap A_{i+1}\}.
    $$

    A finite collection $\mathcal{F}=\{f_1,..., f_m\}$ is called a {\em conformal Markov map} if
    \begin{itemize}
    \item for each $i$, there exists $j_i$ and $l_i$ so that 
    $$
    f_i: A_i \longrightarrow A_{j_i} \cup A_{j_i+1} \cup \ldots \cup A_{j_i+l_i}
    $$
    is a homeomorphism;
    \item $f_i$ extends to a conformal map on a neighborhood of $A_i$.
    \end{itemize}
\end{defn}

The combinatorics of a conformal Markov map  $\mathcal{F}$  is recorded by an $m \times m$ matrix, called the {\em Markov matrix} $M = M_\mathcal{F}$ given by 
$$
M_{ij} = \begin{cases}
    1 & \text{ if }A_i \subseteq f_j(A_j)\\
    0 & \text{ otherwise.}
\end{cases}.
$$
Note that the Markov matrix requires a choice of an ordering of the elements in the partition $\mathcal{A}$. Thus, $M_{\mathcal{F}}$ is well-defined up to a cyclic permutation of the indices.

\subsection*{Markov pieces, topological expansion and primitivity}
A partition $\mathcal{B}$ of $\SS^1$ is a {\em refinement} of $\mathcal{A}$ if for each $B \in \mathcal{B}$, there exists $A \in \mathcal{A}$ so that $B \subseteq A$.
Given a conformal Markov map $\mathcal{F} = \{f_1,..,. f_m\}$, one can define the pull-back partition 
$$
\mathcal{A}^1 = \mathcal{F}^*(\mathcal{A}) = \{A_{ij} :=f_i^{-1}(A_j): A_j\subset f_i(A_i)\}.
$$
We call $\mathcal{A}^0 = \mathcal{A}$ (resp.\ $\mathcal{A}^1$) the {\em level $0$} (resp.\ {\em level $1$}) Markov partitions.
By construction, 
\begin{itemize}
    \item(inclusion) $A_{ij} \subseteq A_i$, i.e., $\mathcal{A}^1$ is a refinement of $\mathcal{A}^0$; and
    \item(dynamics) $f_i: A_{ij} \longrightarrow A_j$ is a homeomorphism.
\end{itemize}
We define the level $n$ Markov partitions inductively, and denote them by $\mathcal{A}^n = \{A_{w}\}$, where $w$ is a word of length $n$ consisting of letters in $\{1,..., m\}$.

It is convenient to write a conformal Markov map as
$$
\mathcal{F} : \mathcal{A}^1 \longrightarrow \mathcal{A}^0.
$$

We say a conformal Markov map $\mathcal{F}$ is {\em topologically expanding} if
$\bigcap_{w_n} A_{w_n}$ is a singleton set for each infinite nested sequence of level $n$ Markov pieces.
A conformal Markov map $\cF$ is said to be {\em primitive} if there exists $k$ so that for each $A \in \mathcal{A}^1$, $\cF^k(A) = \SS^1$. 
This is equivalent to the fact that the Markov matrix $M_\cF$ is {\em primitive}, i.e., there exists $k$ so that $M_\cF^k$ is a positive matrix.

\subsection*{Colored conformal Markov maps}
A coloring of a conformal Markov map $\cF$ is a surjective map $\kappa: \mathrm{bk}(\mathcal{A}) \rightarrow F$ on a finite set $F$ such that $\kappa \circ \cF = \kappa$. 
It induces a coloring $\kappa: \mathrm{bk}(\cF) = \bigcup_n \mathrm{bk}(\mathcal{A}^n) \rightarrow F$. A conformal Markov map together with a coloring is called a {\em colored conformal Markov map}. We also say $\cF$ is $k$-colored if the coloring set has cardinality $k$, i.e., $|F| = k$.

\subsection*{Blown-up Markov maps}
To model convex cocompact Fuchsian groups of the second kind, we introduce the following notion.

\begin{defn}\label{def:blowup} 
A {\em blown-up conformal Markov map} of a  conformal Markov map $\cF = \{f_1.\ldots, f_m\}$ defined on $\cA = \{A_1,\ldots, A_m\}$ is given by a collection $\hat{\cF}=\{\hat{f}_1,\ldots, \hat{f}_m\}$
of continuous injective maps defined on a finite collection of $m$ pairwise disjoint closed intervals $\hat{\cA}=\{\hat{A}_1,\ldots, \hat{A}_m\}$ of $\SS^1$ 
together with a monotone continuous surjective map $\tau:\SS^1\to\SS^1$ that satisfy the following properties.
\bit
\item For each index $i$, the map  $\hat{f}_i$  is the restriction of a conformal map defined in some neighborhood of $\hat{A}_i$ with $\hat{f}_i(\hat{A}_i) \subset \SS^1$. 
\item For each index $i$, $\tau^{-1}(\Int(A_i)) = \Int(\hat{A}_i)$ holds and we have $\tau\circ \hat{f}_i= f_i\circ \tau$ on $\hat{A}_i$.
\eit
\end{defn}

The points in $\bigcup \partial \hat{A}_i$ are {\em the boundary points} for the blown-up conformal Markov map, and the complementary intervals of $\SS^1 \setminus \bigcup_{i} \hat{A}_i$ are called {\em gap intervals}.
A blown-up conformal Markov map is a genuine locally injective holomorphic map defined in a neighborhood of $\bigcup_{i} \hat{A}_i$. 
We may then define its {\em repeller} $\mathcal{J}$ as the set of points that can be iterated infinitely many times by $\hat{\cF}$.

Blown-up conformal maps will provide us the main intermediate step. The missing notions in the following theorem will be defined in the subsequent subsections.

\begin{theorem}\label{thm:qsmodel.n}
    Let $G$ be a $(d-1)$-colored convex cocompact Fuchsian group of the second kind. For each $k$ sufficiently large, there exist
    \begin{itemize}
        \item a blown-up Bowen-Series map $\cG$ for $G$ and
        \item a blown-up $d$-adic Markov map $\cF$
    \end{itemize}
  so that the following hold. 
    \begin{itemize}
        \item There exists a color-preserving quasisymmetric map $h: (\SS^1, \Lambda(G)) \rightarrow (\SS^1, \mathcal{J}_d)$ so that  $h \circ \cG = \cF \circ h$, and where $\mathcal{J}_d$ only depends on $d$.
        \item There exists $\varepsilon>0$ such that $h$ is a conjugacy on $[a-\ep |J|,a+\varepsilon |J| ]$ where $a$ is a periodic boundary point and $J$ is the gap interval bounded by $a$. 
        \item For each periodic boundary point $x$ for $\cF$, the multiplier satisfies
        $$
        \mu_{\cF}(x)  = e^k, \text{ where } e = 2d-1.
        $$
    \end{itemize}

\end{theorem}

\subsection{$d$-adic models for colored conformal Markov maps}\label{subsec:dadicmodel}

\subsection*{$d$-adic Markov maps}
Let $\sigma_d: \SS^1\cong \R/\Z \rightarrow \SS^1$ be the multiplication by $d$ map, i.e., $\sigma_d(x) = dx$ mod $1$.
Note that $\sigma_d$ has $d-1$ fixed points, $p_j:=\frac{j}{d-1}, j =0,..., d-2$. 
They induce a partition $\mathcal{A}:=\{[\frac{j}{d-1}, \frac{j+1}{d-1}]\}$ of $\SS^1$.
Let $F = \{0,..., d-2\}$. We define the coloring $\kappa:\mathrm{bk}(\mathcal{A}) \rightarrow F$ so that $\kappa(p_j) = j$.
We call the intervals of the pullbacks of this partition {\em scaled d-adic intervals}:
\begin{defn}[Scaled $d$-adic intervals]\label{defn:dadicint}
    Let $[s, t] \subseteq \mathbb{S}^1\cong \R/\Z$ be an interval. We say it is {\em scaled d-adic} if it is of the form $[\frac{j}{(d-1)d^n}, \frac{j+1}{(d-1)d^n}]$, where we allow the numerators and denominators to have a common factor. 
    The {\em depth} of a scaled $d$-adic  $I = [\frac{j}{(d-1)d^n}, \frac{j+1}{(d-1)d^n}]$ is $\dep(I) = n$.
    We define the coloring of a boundary point of a scaled $d$-adic interval by $\kappa(\frac{j}{(d-1)d^n}) := \kappa(\sigma_d^n(\frac{j}{(d-1)d^n}))$.
\end{defn}
Note that the boundary colors of a scaled $d$-adic interval is $j$ and $j+1 \mod d-1$.
We have the following lemma by construction.
\begin{lem}\label{lem:dadicdynamics}
    Let $I,J$ be scaled d-adic intervals with the same boundary colors. Let $m = \dep(I), n = \dep(J)$. Then there exists an inverse branch so that 
    $$
    \sigma_d^{-n}\circ \sigma_d^m: \Int(I) \rightarrow \Int(J)
    $$
    is a homeomorphism.
\end{lem}
\begin{proof}
    Let $A:=[\frac{j}{d-1}, \frac{j+1}{d-1}] \in \mathcal{A}$ be the interval with the same boundary color as $I$ and $J$. Then $\sigma_d^m: \Int(I) \rightarrow \Int(A)$ and $\sigma_d^n: \Int(J) \rightarrow \Int(A)$ are homeomorphisms, and the lemma follows. 
\end{proof}

\begin{defn}
A partition $\mathcal{A}= \{A_1,..., A_m\}$ of $\SS^1$ is a {\em scaled $d$-adic partition} if each interval $A_i$ is a scaled $d$-adic interval.

A map of the form $\sigma_d^{-n} \circ \sigma_d^m$ is called a {\em $d$-adic linear map}. They are precisely maps of the form
$$
f(x) = d^ax+ b \mod 1, \text{ where } a \in \Z, b \in \Z[1/d].
$$

A conformal Markov map $\cF$ is called a {\em $d$-adic Markov map} if
\begin{itemize}
    \item $\mathcal{A}$ is a scaled $d$-adic partition; and
    \item each map is $d$-adic linear.
\end{itemize}
Note that a $d$-adic Markov map is $(d-1)$-colored.
\end{defn}

We record the following combinatorial lemma for scaled $d$-adic decompositions that builds on \cite[Lemma 4.27]{LMM26}.
\begin{lem}\label{lem:primitivedyadicdecomp}
    Let $A = [s,t]$ be a scaled $d$-adic interval. Suppose that it is decomposed into $M\geq 2$ scaled $d$-adic intervals $A = B_1 \cup B_2 \cup \ldots \cup B_M$.
    Then $M = 1+k(d-1)$ for some $k \geq 1$ and 
    $$
    |B_1|/|A| \in \left\{\frac{1}{d^{j}}: j = 1,\ldots, k\right\}.
    $$
    Moreover, for any $a \in \{\frac{1}{d^{j}}: j = 1,\ldots, k\}$, there is a scaled $d$-adic decomposition $A = B_1 \cup B_2 \cup \ldots \cup B_M$ with $|B_1|/|A| = a$.
\end{lem}
\begin{proof}
    Note that any scaled $d$-adic decomposition of $A$ is constructed inductively as follows. First, subdividing $A$ into $d$ scaled $d$-adic intervals (of depth $\dep(A)+1$). Then choose one interval in this subdivision and subdivide it into $d$ scaled $d$-adic intervals of greater depth. Iterating this procedure shows that $M = 1+k(d-1)$ for some $k \geq 1$.
    The estimate for $|B_1|/|A|$ follows immediately; see Figure~\ref{fig:dyaint}.
\end{proof}
\begin{figure}[htp]
\captionsetup{width=0.96\linewidth}
\centering
\begin{tikzpicture}[x=5cm,y=5cm]
  \draw (0,0) -- (1,0);
  \fill (0,0) circle (0.006);
  \fill (1,0) circle (0.006);
  \foreach \x/\lab in {
      0/0,
      0.5/{\tfrac12},
      0.75/{\tfrac34},
      0.875/{\tfrac78},
      0.9375/{\tfrac{15}{16}},
      1/1
  }{
    \draw (\x,0) -- (\x,-0.02);
    \node[below] at (\x,-0.02) {$\lab$};
  }
  \draw (0,0) arc [start angle=180, end angle=0, radius=0.5];
  \draw (0,0) arc [start angle=180, end angle=0, radius=0.25];
  \draw (0.5,0) arc [start angle=180, end angle=0, radius=0.125];
  \draw (0.75,0) arc [start angle=180, end angle=0, radius=0.0625];
  \draw (0.875,0) arc [start angle=180, end angle=0, radius=0.0625/2];
  \draw (0.9375,0) arc [start angle=180, end angle=0, radius=0.0625/2];  
\end{tikzpicture}
\begin{tikzpicture}[x=5cm,y=5cm]
  \draw (0,0) -- (1,0);

  \fill (0,0) circle (0.006);
  \fill (1,0) circle (0.006);

  \foreach \x/\lab in {
      0/0,
      0.0625/{\tfrac{1}{16}},
      0.125/{\tfrac18},
      0.25/{\tfrac14},
      0.5/{\tfrac12},
      1/1
  }{
    \draw (\x,0) -- (\x,-0.02);
    \node[below] at (\x,-0.02) {$\lab$};
  }

  \draw (0,0) arc [start angle=180, end angle=0, radius=0.5];

  \draw (0,0) arc [start angle=180, end angle=0, radius=0.0625/2];
  
  \draw (0.0625,0) arc [start angle=180, end angle=0, radius=0.0625/2];

  \draw (0.125,0) arc [start angle=180, end angle=0, radius=0.0625];

  \draw (0.25,0) arc [start angle=180, end angle=0, radius=0.125];

  \draw (0.5,0) arc [start angle=180, end angle=0, radius=0.25]; 
\end{tikzpicture}
\caption{An illustration of Lemma~\ref{lem:primitivedyadicdecomp} for $d = 2$. The left and right figures give the two extreme configurations of $B_1,\ldots, B_5$ for the case $M = 5 = 1+4\times (2-1)$: $|B_1|/|A| = \frac12$ on the left and $|B_1|/|A| = \frac1{2^{4}}$ on the right.}
\label{fig:dyaint}
\end{figure}
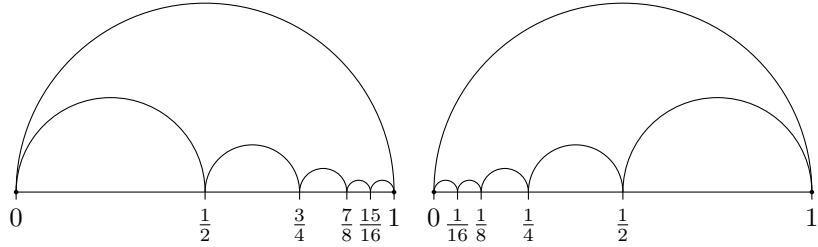

\subsection*{Cyclically colored conformal Markov maps}
Let $\cF$ be a conformal Markov map. We assume we are given a set of colors on the set of break-points $\mathrm{bk}(\mathcal{A}) = \{a_0,..., a_{m-1}\}$, i.e., a map $\kappa: \mathrm{bk}(\cA)\to \{0,1,\ldots,d-2\}$, with $d-1\ge 1$.

We say $\mathcal{\cF}$ is a {\em $(d-1)$-cyclically colored conformal Markov map}  if we can label $\mathrm{bk}(\mathcal{A})$ counterclockwise 
so that $\kappa(a_i) = i\mod d-1$. Note that this implies $d-1 \mid m$. Moreover, each refinement $\cA^n$ is also cyclically colored by pullback.

\begin{lem}\label{lem:topologicalconj}
Let $\cF:\mathcal{A}^1 \rightarrow\mathcal{A}^0$ be a topologically expanding $(d-1)$-cyclically colored conformal Markov map. 
There is a $d$-adic Markov map $\cG: \mathcal{B}^1 \rightarrow \mathcal{B}^0$ that is topologically conjugate to $\cF$.
\end{lem}

\begin{proof}
    Since $\mathcal{A}^0 = \{A_1,..., A_m\}$ is cyclically colored and has $d-1$ colors, there is a scaled $d$-adic partition $\mathcal{B}^0 = \{B_1,..., B_m\}$ so that $B_i$ has the same boundary colors as $A_i$. Since $\mathcal{A}^1$ is a refinement of $\mathcal{A}^0$ and is cyclically colored, we can subdivide each $B_i$ into scaled $d$-adic intervals to match the subdivision of $A_i$. This gives a scaled $d$-adic partition $\mathcal{B}^1$. 
    Let $A_{ij} \in \mathcal{A}^1$ and let $A_k = \mathcal{F}(A_{ij}) \in \mathcal{A}^0$.
    By Lemma~\ref{lem:dadicdynamics}, there is $m, n$ so that $\sigma_d^{-n} \circ \sigma_d^m(B_{ij}) = B_k$. This defines a $d$-adic Markov map $\cG$.
    By construction, $\cG$ is combinatorially conjugate to $\cF$. Since $\cF$ is topologically expanding, every nested Markov pieces $A_w$, with $|w| \to \infty$ contains infinitely many different intervals. Since any scaled $d$-adic interval $J$ is properly contained in a scaled $d$-adic interval $I$, we have $|J| \leq |I|/d$, we conclude that $\cG$ is topological expanding as well. It now follows from  \cite[Proposition~2.7]{LMM26} that $\cG$ is topologically conjugate to $\cF$.
\end{proof}

Every primitive colored conformal Markov map admits a cyclically colored refinement, obtained by subdividing $\cA$ and restricting $\cF$  to each new subinterval.

\begin{lem}\label{lem:primitiverefinement}
    Let $\cF:\mathcal{A}^1 \rightarrow\mathcal{A}^0$ be a primitive colored conformal Markov map. Then there is a refinement 
    $\widehat{\mathcal{F}}: \widehat{\mathcal{A}}^{1} \rightarrow \widehat{\mathcal{A}}^0$ so that $\widehat{\mathcal{F}}$ is cyclically colored.
\end{lem}
\begin{proof}
    Since $\cF$ is primitive, an interval $A \in \mathcal{A}^0$ contains points in $\mathrm{bk}(\cF)$ of any color. Thus, there is a refinement $\widehat{\mathcal{A}}^0 \subseteq \mathcal{A}^n$ of $\mathcal{A}^0$ which is cyclically colored. Let $\widehat{\mathcal{A}}^1 = \cF^*(\widehat{\mathcal{A}}^0) \subseteq \mathcal{A}^{n+1}$ be the pulled back. Then $\widehat{\mathcal{A}}^1$ is a refinement of $\mathcal{A}^1$. Define $\widehat{\mathcal{F}}|_A = \cF|_A$ for each $A \in \widehat{\mathcal{A}}^1$. Then $\widehat{\mathcal{F}}: \widehat{\mathcal{A}}^{1} \rightarrow \widehat{\mathcal{A}}^0$ is a cyclically colored refinement of $\cF$.
\end{proof}

\begin{cor}\label{cor:topprim}
    Let $\cF:\mathcal{A}^1 \rightarrow\mathcal{A}^0$ be a primitive $(d-1)$-colored conformal Markov map. Suppose $\cF$ is topologically expanding.
    Then there is a $d$-adic Markov map $\cG: \mathcal{B}^1 \rightarrow \mathcal{B}^0$ that is topologically conjugate to $\cF$ by a color-preserving
    map.
\end{cor}

\subsection*{Prescribing multipliers for $d$-adic Markov maps}
Let $\mathcal{F}$ be a topologically expanding $d$-adic Markov map. 
Let $x\in\SS^1$ be a break-point for the dynamics of $\mathcal{F}$.
Denote by $x_n^\pm$ the right and left orbits of $x$, i.e., 
$
x_n^\pm = \lim_{\ep\to 0^+} \mathcal{F}^n(x \pm\ep).$
Since the partition is finite, we note that both right and left orbits are preperiodic. Note that they may eventually map to different periodic orbits.
Given a one-sided periodic break-point $x^{\pm}$ of period $p_\pm\ge 1$, define its multiplier by
$$
\mu_{\cF}(x^\pm) =  (\cF^{p_\pm})'(x^\pm) :=\lim_{\epsilon \to 0} \frac{|\mathcal{F}^{p_\pm}(x^\pm \pm \epsilon)- x^\pm|}{\epsilon}.
$$

In the following, we show that one has some flexibility in prescribing the multipliers at periodic break-points.

\begin{prop}[Prescribing multipliers]\label{prop:qconj}
    Let $\cF: \mathcal{A}^1\rightarrow \mathcal{A}^0$ be a topologically expanding $d$-adic Markov map. Let 
    $$
    K := \max\{\log_d \left(\mu_{\cF}(x^\pm)\right): x^\pm \text{ is a periodic break-point}\}.
    $$
    Then for each $k \geq K$, there is a topologically conjugate $d$-adic Markov map $\cG: \mathcal{B}^1 \rightarrow \mathcal{B}^0$
    defined on a refinement of $\cA^0$  so that for each periodic break-point $x \in \mathrm{bk}(\mathcal{B}^0)$, the multiplier satisfies
        $$
        \mu_{\cG}(x^\pm) = d^k.
        $$
\end{prop}

\begin{proof}
    To construct the modified refinement, we first consider $n$ sufficiently large so that each interval $I \in \mathcal{A}^n$ contains at least one interval $J\in \mathcal{A}^{n+1}$ whose two boundary points $\partial J$ are strictly preperiodic. 
    This is possible as $\cF$ is topologically expanding. Let $(\mathcal{A}^n)' \subseteq \mathcal{A}^n$ be the collection of intervals that contain a periodic boundary point. Let $J \in \mathcal{A}^n \setminus (\mathcal{A}^n)'$ and $N \in \N$. We denote by $\mathcal{A}^n_N(J)$ the collection of intervals in $\mathcal{A}^{n+N}$ that are contained in $J$. We define 
    $$
    \mathcal{B}^n_N:= (\mathcal{A}^n)' \cup \bigcup_{J \in \mathcal{A}^n \setminus (\mathcal{A}^n)'}\mathcal{A}^n_N(J).
    $$
    Note that $\mathcal{B}^n_N$ is a scaled $d$-adic partition and is a refinement of $\mathcal{A}^n$. 
    Since a Markov interval with a periodic boundary point in $\mathcal{A}^{n+1}$ is mapped by $\cF$ to a Markov interval with a periodic boundary point in $\mathcal{A}^{n}$, we conclude that
    $$
    \cF: \mathcal{B}^{n+1}_N \longrightarrow \mathcal{B}^n_N
    $$
    is a refinement of $\cF: \mathcal{A}^1\rightarrow \mathcal{A}^0$.

    Let $A = [s, t] \in (\mathcal{A}^n)' \subseteq \mathcal{B}^{n}_N$ be an interval such that its left endpoint $s$ is periodic. The case where the right endpoint is periodic is handled identically.
    Let $B = [s, t'] \in (\mathcal{A}^{n+1})' \subseteq \mathcal{B}^{n+1}_N$ so that $B \subseteq A$.
    Since $\cF$ is topologically expanding, the number of intervals of $\mathcal{B}^{n+1}_N$ that are contained in $A$ goes to infinity as $N \to \infty$.
    Let $A = B_1 \cup ...\cup B_M$ be a subdivision of $A$ into $M = 1+ r(d-1)$ scaled $d$-adic intervals with $B_i \in \mathcal{B}^{n+1}_N$ and $B_1 = B$.
    By Lemma~\ref{lem:primitivedyadicdecomp}, if $N$ is chosen sufficiently large, we can construct a scaled $d$-adic decomposition $A = \widehat B_1 \cup ...\cup \widehat B_M$ so that 
    $$
    |\widehat B_1|/|B_1| = \frac{|\widehat B_1|}{|A|} \cdot \frac{|A|}{|B_1|} = \frac{\mu_{\cF}(s^+)}{d^{k}}.
    $$
    Note that, after modifying the partition $\mathcal{B}^{n+1}_N$ so that $B_i$ is replaced by $\widehat{B}_i$, and redefine the map on $\widehat{B}_i$ so that $\cG|_{\widehat{B}_i}$ is the unique $d$-adic linear map between $\widehat{B}_i$ and $\cF(B_i)$, it follows that 
    $$
    \cG'(s^+) = \frac{d^{k}}{\mu_{\cF}(s^+)} \cF'(s^+).
    $$ 
    Thus, the multiplier $\mu_{\cG}(s^+)$ at $s^+$ is $d^k$. The proposition follows from \cite[Prop.\,2.7]{LMM26} after such modifications on each cycle of periodic break-points.
\end{proof}

\subsection*{Blown-up $d$-adic Markov maps}
We give an explicit construction of a blown-up Markov map in the $d$-adic setting.

Let $e = 2d-1$. Let $\Omega^0= \bigcup_{i=0}^{d-2} (\frac{2i}{2(d-1)}, \frac{2i+1}{2(d-1)})$, $\Omega^n:= \sigma_e^{-n}(\Omega^0)$ and $\Omega^\infty:= \bigcup \Omega^n$.
Note that the components of $\Omega^\infty$ are $(d-1)$-colored. More precisely, the component $I$ of $\Omega^\infty$ has color $i$ if it eventually maps to $(\frac{2i}{2(d-1)}, \frac{2i+1}{2(d-1)})$. 
Set $\mathcal{J}_d = \SS^1 \setminus \Omega^\infty$. Note $\mathcal{J}_d$ is a Cantor set invariant under $\sigma_e$. Moreover, $\sigma_e: \mathcal{J}_d \rightarrow \mathcal{J}_d$ is semiconjugate to $\sigma_d$ on $\SS^1$. More precisely, there exists a projection map $\tau: \SS^1 \rightarrow \SS^1$ so that 
$$
\tau\circ \sigma_e = \sigma_d \circ \tau
$$
and that sends the closure of each component of $\Omega^\infty$ to a point, i.e., $$\tau(\Omega^\infty) = \left\{\frac{j}{(d-1)d^n}, j\in \{0,..., (d-1)d^n-1\}, n \in \N\right\}.$$ 
Let $I$ be a nontrivial closed interval. We define the pullback as 
$$
\tau^*(I):= \overline{\tau^{-1}(\Int(I))}.
$$
It is easy to verify that if $I$ is a scaled $d$-adic interval of depth $n$, then $\tau^*(I)$ is a scaled $e$-adic interval of depth $n$.
Thus, given a partition $\widecheck\cA$ of $\SS^1$, we let $\cA= \cA_b\sqcup \cA_g$ denote its pullback under $\tau$ where 
$$\left\{\begin{array}{ll} \cA_b:=\{\tau^*(A): A \in \widecheck\cA\} & \hbox{defines the {\em bridge type intervals}}\\
\cA_g:=\{\tau^{-1}(a): a\in \partial A, A \in \widecheck\cA\}\ & \hbox{defines the {\em gap type intervals}}.\end{array}\right.$$
Note that a gap type interval is the closure of a component of $\Omega^\infty$.

\begin{defn}
Given a $d$-adic Markov map $\widecheck\cF$ defined on $\widecheck\cA$ with a semiconjugacy as above, we define the 
{\em blown-up $d$-adic Markov map} as the pullback 
$$
\cF := \tau^*(\widecheck \cF)\hbox{ defined on }\cA_b := \tau^*\widecheck \cA
$$
where for each interval $\tau^*(A) \in \cA_b$, 
$\cF|_{\tau^*(A)}$ is the unique $e$-adic linear map that sends $\tau^*(A)$ to $\tau^*(\widecheck \cF(A))$.
Since the pullback sends a scaled $d$-adic interval of depth $n$ to a scaled $e$-adic interval of depth $n$, if $\widecheck \cF|_A = \sigma
_d^{-n} \circ \sigma_d^m$, then $\cF|_{\tau^*(A)} = \sigma
_e^{-n} \circ \sigma_e^m$, for an appropriate inverse branch.
\end{defn}
Note that a {\em blown-up $d$-adic Markov map} is a blown-up conformal Markov map in the sense of Definition~\ref{def:blowup} and that 
the repeller of any blown-up $d$-adic Markov map is $\mathcal{J}_d$.

\subsection{Bowen-Series maps}\label{subsec:bsm}
In this subsection, we associate a special conformal Markov map for a convex cocompact Fuchsian group $G$ of the second kind, and show that it is quasisymmetrically conjugate to a $d$-adic model.

We say $G$ has type $(g, n)$ if the corresponding hyperbolic surface has genus $g$ and $n$ nontrivial boundary components.
Let $\widecheck{G}$ be the geometrically finite Fuchsian group obtained from $G$ by pinching all peripheral curves to cusps. Then the corresponding surface $\widecheck S$ is a hyperbolic surface of finite type of genus $g$ with $n$ punctures. 
Fix an isomorphism $\Phi: G \rightarrow \widecheck{G}$.
There is a projection map $\pi: \SS^1 \rightarrow \SS^1$ that sends the closure of each component of $\Omega(G)$ to a point and semiconjugates the action of $G$ with $\widecheck{G}$.

There are two main objects in this section.
\begin{enumerate}
    \item We first construct a {\em Bowen-Series map} $\widecheck\cG$ for $\widecheck{G}$, which is a conformal Markov map on $\SS^1$; and
    \item we lift it via the semiconjugacy $\pi$ to a {\em blown-up Bowen-Series map} $\cG$ for $G$, which is a blown-up conformal Markov map in the sense of Definition~\ref{def:blowup}.
\end{enumerate}

\begin{figure}[ht]
\captionsetup{width=0.96\linewidth}
  \centering
  \includegraphics[width=0.49\textwidth]{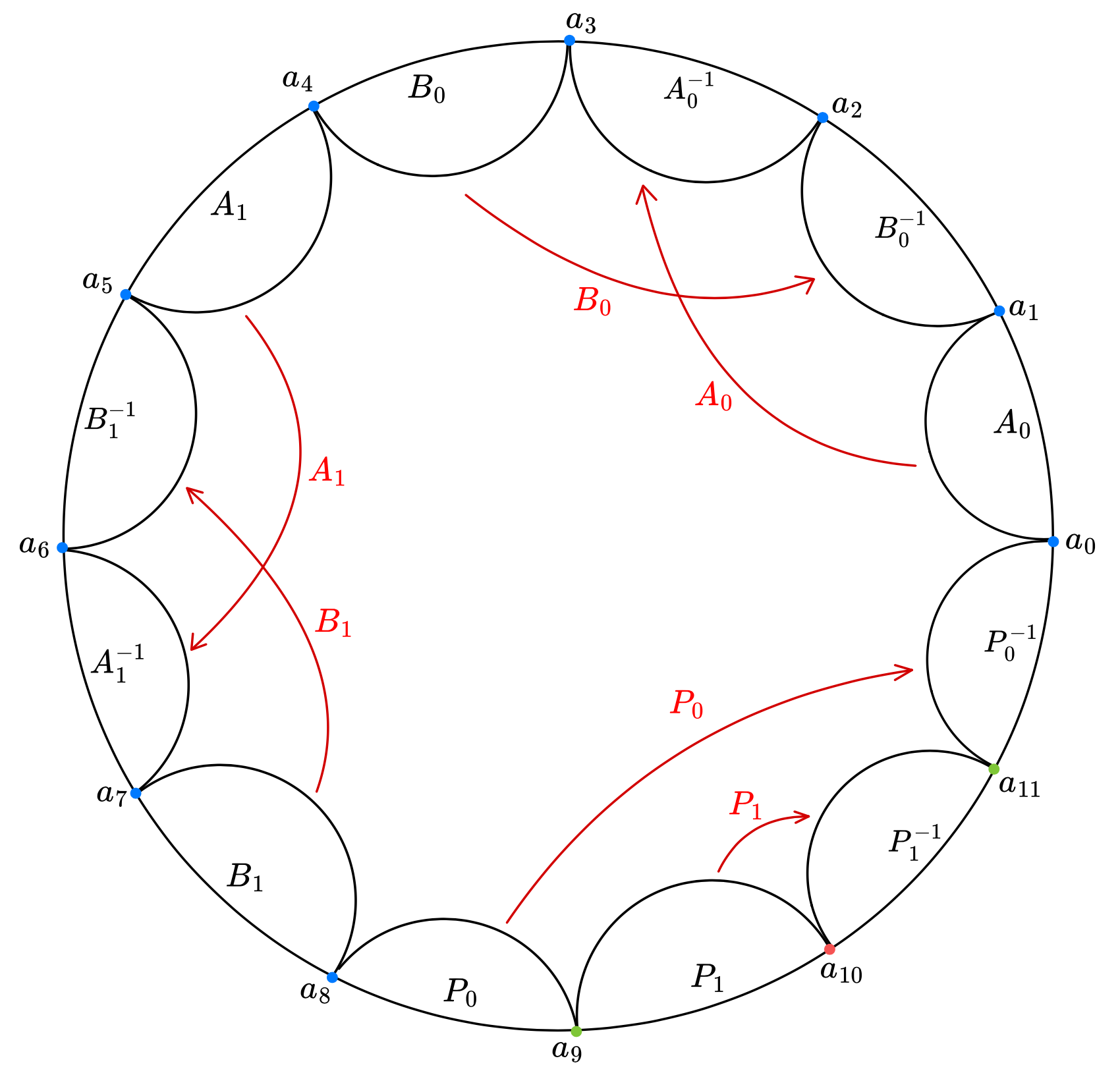}
  \includegraphics[width=0.49\textwidth]{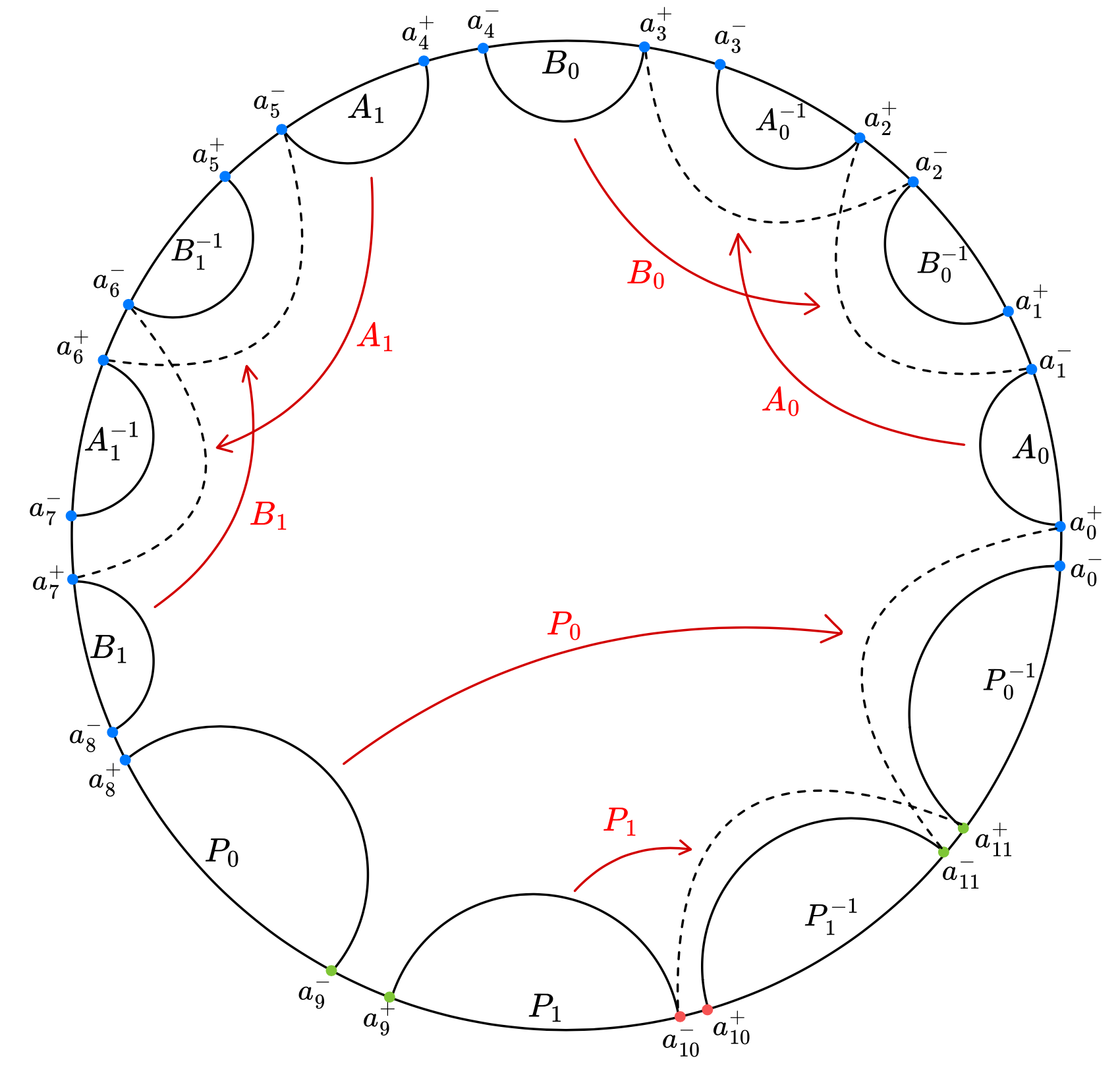}
  \caption{Left: The Bowen-Series map for the geometrically finite Fuchsian group $\widecheck G$ of the first kind with $g = 2, n = 3$. Right: The corresponding blown-up Bowen-Series map $G$ for the convex cocompact Fuchsian group of second kind with $g = 2, n = 3$. 
  Note that under $A_0$, the geodesic $\arc{a_{0}^{+},a_{1}^{-}}$ is mapped to the dashed geodesic   $\arc{a_{3}^+,a_{2}^-}$ and similarly for other generators.
  }
  \label{fig:BS}
\end{figure}

\subsection*{Bowen-Series maps for punctured surfaces} 
We follow the construction  of Bowen--Series maps  given in \cite[\S5.2]{LMM26} that we specialize to our setting.

Given a genus $g\ge 0$ closed surface $\Sigma$ with $n\ge 1$ marked points $\{p_1,\ldots, p_n\}$ subject to $2g-2+n>0$, we consider a filling 
embedded graph $\Gamma\subset \Sigma$ that consists of $g$ pairs of non-planar loops based at $p_1$ that do not separate the other
marked points to which we add segments $[p_j,p_{j+1}]$, $1\le j\le n-1$. Deleting the marked points provides us with a free presentation
$\widecheck G = \langle A_i, B_i, 0\le i\le g-1, P_j, 0\le j \le n-2\rangle$ with the standard conventions. The generators of the $n$ punctures are 
$P_jP_{j+1}^{-1}$, $0\le j\le n-3$, $P_{n-2}$ and $\prod[A_i,B_i]P_0$, up to inverses.

This defines a fundamental domain  $\widecheck \cP$ for $\widecheck G$ as follows (see Figure~\ref{fig:BS}, left). For $x,y\in\mathbb{S}^1$, we denote the hyperbolic geodesic in $\D$ connecting $x, y$ by $\arc{x,y}$.

\begin{itemize}[leftmargin=8mm]
    \item $\widecheck \cP$ is an ideal hyperbolic $N-$gon in $\D$, where $N = 4g+2(n-1)$.
    \item Denote the vertices of $\cP$ by $a_0,\ldots, a_{N-1} \in \mathbb{S}^1$ in the counter-clockwise order. Then each edge $\arc{a_i,a_{i+1}}, \ i\in\{0,\ldots, N-1\}$, is a hyperbolic geodesic in $\D$ (with indices mod $N$).
    \item ($g$ handles) There exist $2g$ loxodromic elements $A_i, B_i \in G,\  i\in\{0,\ldots, g-1\}$, such that for $i\in\{0,\cdots,g-1\}$, 
    $$
    A_i(\arc{a_{4i},a_{4i+1}}) = \arc{a_{4i+3}, a_{4i+2}}, \text{ and } B_i(\arc{a_{4i+3}, a_{4i+4}}) = \arc{a_{4i+2}, a_{4i+1}}.
    $$
    \item ($n$ punctures) There exist $n-1$ elements $P_0, \ldots, P_{n-2} \in G$ such that
    $$
    P_i(\arc{a_{4g+i}, a_{4g+i+1}}) = \arc{a_{4g+2(n-1)-i}, a_{4g+2(n-1)-i-1}}.
    $$
\end{itemize}
The vertex set induces a partition $\widecheck\cA = \{[a_i,a_{i+1}]:i\in\{0,\ldots, N-1\}\}$  of $\mathbb{S}^1$.
We define the {\em Bowen-Series map} $$\widecheck \cG= \widecheck{\cG}_{\widecheck G, \widecheck\cP}$$ associated to  $\widecheck \cP$ as the following conformal Markov map with respect to  $\widecheck\cA$:
\begin{itemize}
    \item for $i\in\{0,\ldots, g-1\}$, 
    \begin{align*}
        \widecheck{\cG}|_{[a_{4i},a_{4i+1}]} &= A_i, &
        \widecheck{\cG}|_{[a_{4i+1},a_{4i+2}]} &= B_i^{-1}, &
        \widecheck{\cG}|_{[a_{4i+2},a_{4i+3}]} &= A_i^{-1}, &
        \widecheck{\cG}|_{[a_{4i+3},a_{4i+4}]} &= B_i.
        \end{align*}
    \item For $i\in\{0,\ldots, n-2\}$,
   $$ 
        \widecheck{\cG}|_{[a_{4g+i},a_{4g+i+1}]} = P_i\quad\hbox{ and }\quad
        \widecheck{\cG}|_{[a_{4g+2(n-1)-i-1},a_{4g+2(n-1)-i}]} = P_i^{-1}.
$$ 
\end{itemize}

Note that restrictions of the parabolic generators appear as iterates of $\cG$ and they account for the $2N$ periodic one-sided break-points for $\widecheck \cG$ (two fixed points for $P^{\pm 1}_{n-2}$, $2(n-2)$  cycles of length $2$ and two  $(4g+1)$-cycles).

The coloring of the group $G$ induces a coloring of the break-points. Thus $\widecheck{\cG}: \widecheck\cA^1 \rightarrow \widecheck\cA^0 = \widecheck\cA$ is a $(d-1)$-colored conformal Markov map on $\SS^1$.
It follows from the construction that we have the following.
\begin{lem}\label{lem:symFs}
    The Bowen-Series Markov map $\widecheck{\cG} = \widecheck{\cG}_{\widecheck G, \widecheck\cP}$ is a topologically expanding, primitive, $(d-1)$-colored conformal Markov map on $\SS^1$.
\end{lem}

With Proposition~\ref{prop:qconj} and Corollary \ref{cor:topprim}, we obtain

\begin{cor}[Topological $d$-adic models for colored Bowen-Series maps]\label{prop:topconjugate}
For each $k$ sufficiently large, there exists a $d$-adic Markov map $\widecheck \cF: \widecheck{\cB}^1 \rightarrow \widecheck{\cB}^0$ so that
    \begin{itemize}
        \item $\widecheck\cF$ is topologically conjugate to $\widecheck\cG$ by a color-preserving homeomorphism;
        \item for each periodic break-point $x$, the multiplier satisfies
        $$
        \mu_{\widecheck\cF}(x^\pm) = d^k.
        $$
    \end{itemize}
  \end{cor}

\subsection*{Blown-up Bowen-Series maps} We keep the same notation.
We now lift the Bowen-Series map of $\widecheck G$ to $G$.
Let $I$ be a nontrivial closed interval of $\SS^1$. 
Recall that $\pi: \SS^1 \rightarrow \SS^1$ is the semiconjugacy between $G$ and $\widecheck{G}$.
We define the pullback of $I$ as 
$$
\pi^*(I) = \overline{\pi^{-1}(\Int(I))}.
$$
Thus, we may pullback the partition $\widecheck \cA$ and obtain a partition of $\SS^1$ for $G$:
$$
\cA = \{\pi^*(A): A \in \widecheck\cA\} \sqcup \{\pi^{-1}(a): a\in \partial A, A \in \widecheck\cA\}.
$$
We denote the vertices by $a_i^\pm, i = 0,..., N-1$, where $N = 4g+2(n-1)$, so that
$$
\pi^*([a_i, a_{i+1}]) = [a_i^+, a_{i+1}^-], \text{ and } \pi^{-1}(a_i) = [a_i^-, a_{i}^+].
$$
Similar to the case of blown-up $d$-adic intervals, there are two types of intervals:
\begin{itemize}
    \item $A = [a_i^+, a_{i+1}^-]$ is called {\em bridge type}; and 
    \item $A = [a_i^-, a_{i}^+]$ is called {\em gap type}.
\end{itemize}
Note that each gap type interval is the closure of a component of $\Omega(G) = \SS^1 \setminus \Lambda(G)$.

We define the {\em blown-up Bowen-Series map} 
$\cG:= \pi^*(\widecheck \cG)$ as the pullback of $\widecheck \cG$, which is defined on all bridge type intervals. For each $A \in \widecheck\cA$, set
$$
\cG|_{\pi^*(A)} = \Phi^{-1}(\widecheck \cG|_{A}): \pi^*(A) \rightarrow \pi^*(\widecheck \cG(A)),
$$
where $\Phi: G \rightarrow \widecheck G$ is the isomorphism.
More explicitly, this means that $\cG|_{[a_{i}^+,a_{i+1}^-]}$ has the following form (see Figure~\ref{fig:BS}, right):
\begin{itemize}
    \item For $i\in\{0,\ldots, g-1\}$, 
    \begin{align*}
        \mathcal{G}|_{[a_{4i}^+,a_{4i+1}^-]} &= A_i, &
        \mathcal{G}|_{[a_{4i+1}^+, a_{4i+2}^-]} &= B_i^{-1}, &
        \mathcal{G}|_{[a_{4i+2}^+,a_{4i+3}^-]} &= A_i^{-1}, &
        \mathcal{G}|_{[a_{4i+3}^+, a_{4i+4}^-]} &= B_i.
        \end{align*}
    \item For $i\in\{0,\ldots, n-2\}$,
    \begin{align*}
        \mathcal{G}|_{[a_{4g+i}^+, a_{4g+i+1}^-]} &= P_i, &
        \mathcal{G}|_{[a_{4g+2(n-1)-i-1}^+, a_{4g+2(n-1)-i}^-]} &= P_i^{-1}.
    \end{align*}
\end{itemize}

Observe that the repeller for $\cG$ is exactly $\Lambda(G)$ and restrictions of the generators of the gap type intervals appear as iterates of 
$\cG$ and account for the boundary periodic points as above. 

\subsection*{Quasisymmetric models for Bowen-Series maps} We may now assign a $d$-adic Markov map to any colored convex cocompact Fuchsian group of the second kind that only depends on the number of colors.

\begin{proof}[Proof of Theorem \ref{thm:qsmodel.n}]
    For $k$ sufficiently large, by Corollary~\ref{prop:topconjugate}, there exists a $d$-adic Markov map $\widecheck\cF$ topologically conjugate to $\widecheck\cG$ and the multipliers of the periodic break-points are $d^k$. Let $\cF = \tau^*(\widecheck\cF)$ be the corresponding blown-up $d$-adic Markov map.    
    Since the topological conjugacy between $\widecheck\cF$ and  $\widecheck\cG$ preserves the coloring, so does the conjugacy between $\cF$ and $\cG$. We may extend this conjugacy to a neighborhood of the bridge intervals.
 
 As both blown-up Markov maps are expanding and conformal in a neighborhood  of their repeller, the topological  conjugacy is quasisymmetric by Sullivan's  conformal elevator principle.
 One way to do this is to extend beyond a definite neighborhood of  the bridge intervals both $\cF$ and $\cG$ as smooth expanding covering maps of the same degree on $\SS^1$, 
 and apply \cite[Theorem 2.8.2]{kmp:ph:cxci}.  The existence of $\varepsilon>0$ comes from the definite neighborhood where the coverings coincide with the blown-up Markov maps.
\end{proof}

\subsection{Proof of Theorem~\ref{thm:colors}}
    We assume that $|F| = d-1$.
    We may choose the fundamental polygons $\widecheck{P}_1$, $\widecheck{P}_2$ for the pinched Bowen-Series maps 
    so that $J_1, J_2$ correspond to a gap type interval of the level $0$ partition.
    Let $\cG_1,  \cG_2$ be blown-up Bowen-Series maps for $G_1, G_2$ respectively. By Theorem~\ref{thm:qsmodel.n}, there exists $k$ such that $\cG_1, \cG_2$ are quasisymmetrically conjugate by $h_1, h_2$ to  blown-up $d$-adic Markov maps $\cF_1, \cF_2$ respectively so that the multiplier of each periodic boundary point is $e^k$, where $e = 2d-1$.
    Note that Theorem~\ref{thm:qsmodel.n} also identifies the images as $h_1(\Lambda(G_1)) = h_2(\Lambda(G_2)) = \mathcal{J}_d$.
    We may also assume that $J_1, J_2$ are mapped to the same gap type interval under $h_1, h_2$ respectively. 
    Then $h = h_2^{-1} \circ h_1$ is a color-preserving quasisymmetric map with $h(J^1) = J^2$.

We now prove the local conjugacy property. 
 Let $J$ be a component of $\Omega(G_1)$ and let $a \in \partial J$.
Then $a$ is preperiodic under $\cG_1$. Denote the preperiod and period of $a$ by $l$ and $q$. 
 Let $\widetilde a:= \cG_1^{l}(a)$ and $\widetilde{J}$ be the component of $\Omega(G_1)$ with $\widetilde{a} \in \partial \widetilde{J}$.
By construction, $\cG_1^q$ coincides with a generator $\tilde{g}_1\in G_1$ of $\stab_{G_1}(\tilde{J})$ in a neighborhood of $\tilde{a}$. 
By Theorem \ref{thm:qsmodel.n}, $h_1$ conjugates $\tilde{g}_1$ to $\cF_1^q$ in a definite neighborhood $V_{\tilde a}$ of $\tilde a$. Note that
$ \cF_1^q = \sigma_e^{-\tilde n_1}\circ \sigma_e^{k+\tilde n_1}$ on $h_1(V_{\tilde a})$, for some suitable inverse branch $\sigma_e^{-\tilde n_1}$. 
Note that $\cG_1^l$ coincides with an element of $G_1$ (possibly equal to $\id$ if $l = 0$) that conjugates $\tilde{g_1}$ to some generator $g_1$ of
$\stab_{G_1}(J)$ on the neighborhood $V_a= \cG_1^{-l}(V_{\tilde a})$. The Koebe distortion theorem implies that $V_a$ has definite size with respect to $J$.
Conjugating also $ \cF_1^q$ by $\cF_1^l$, we can conclude that $h_1|_{V_a}$ conjugates $g_1$ to some $f_a:= \sigma_e^{- n_1}\circ\sigma_e^{k+ n_1}$ on 
$h_1(V_{a})$, for some suitable inverse branch $\sigma_e^{-n_1}$ so that $f_a(h_1(a)) = h_1(a)$.

Similarly, $h(a)$ is preperiodic for $\cG_2$, so the same argument shows the existence of a definite neighborhood $V_{h(a)}$ with respect to
$h(J)$ on which
a generator $g_2$ of $\stab_{G_2}(h(J))$ is conjugate by $h_2$ to a map of the form $f_{h(a)}:=\sigma_e^{-n_2}\circ \sigma_e^{k+n_2}$ on $h_2(V_{h(a)})$ with $f_{h(a)}(h_2(h(a))) = h_2(h(a))$.
As $h_2(h(a))= h_1(a)$, it follows that $f_{h(a)}=f_a$ on the intersection, so $h$ conjugates $g_1$ to $g_2$ on $W_a :=V_a\cap h^{-1}(V_{h(a)})$. Since $h$ is quasisymmetric, 
it follows that $W_a$ also has definite size with respect to $J$. This proves the theorem.
\qed

\section{Universality}\label{sec:uni}
In this section, we give sufficient conditions for quasisymmetric and quasiconformal universality; see Theorem~\ref{thm:qsclassG} and Theorem~\ref{thm:qcext}, together with their immediate corollaries.

\subsection{Quasisymmetric universality}\label{sec:gluingqs}
In this subsection, we show that carpet-free limit sets without non-homogeneous rank-two cut points are quasisymmetrically universal.

Let $a$ be a rank-two cut point and let $\mathcal{C}_a \subseteq \Lambda(G)$ be the component that contains $a$.
The planar embedding of $\Lambda(G)$ in $\widehat\C$ gives an identification $\pi_0(\mathcal{C}^*_a) \cong \Z$, where $\mathcal{C}^*_a := \mathcal{C}_a\setminus \{a\}$.
Recall that a homeomorphism $h: (\Lambda(G), \mathfrak{P}(G)) \rightarrow (\Lambda(G'), \mathfrak{P}(G'))$ is {\em coarsely order-respecting} at $a$ if the induced bijection $h_*: \pi_0(\mathcal{C}^*_a)\rightarrow \pi_0(\mathcal{C}^*_{h(a)})$ has the form $h_*(n) = \epsilon n + c + O(1)$, where $\epsilon = \pm 1$.

\begin{theorem}\label{thm:qsclassG}
    Let $G, G'$ be geometrically finite Kleinian groups, with $\Lambda(G), \Lambda(G')\neq \widehat\C$.
    Suppose there is a homeomorphism 
    $$
    \Phi\colon (\Lambda(G), \mathfrak{P}(G)) \rightarrow (\Lambda(G'), \mathfrak{P}(G')),
    $$
    such that $\Phi$ is quasisymmetric on each dynamical carpet-quotient and is coarsely order-respecting at rank-two cut points. 
    Then $G, G'$ are quasi-isometric relative to parabolic subgroups.
    
    In particular, $(\Lambda(G), \mathfrak{P}(G))$ and $(\Lambda(G'), \mathfrak{P}(G'))$ are quasisymmetric.
\end{theorem}

Since any homeomorphism at a homogeneous rank-two cut point can be modified to be coarsely order-respecting, we have the following.
\begin{cor}\label{cor:relativeqsuniversalGF}
    Let $G$ be a geometrically finite Kleinian group. Suppose that $\Lambda(G)$ is carpet-free and every rank-two cut point is homogeneous. Then $(\Lambda(G), \mathfrak{P}(G))$ is relatively quasisymmetrically universal in $\mathfrak{X}_{\gf}$.

    In particular, if $G$ is convex cocompact and $\Lambda(G)$ is carpet free, then $\Lambda(G)$ is quasisymmetrically universal in $\mathfrak{X}_{\cc}$.
\end{cor}

We present a general framework that starts from a homeomorphism between the limit sets of two geometrically finite groups and produces a quasi-isometry relative to their parabolic groups.

Let $X_i, i = 1,2$ be $G_i$-sets. Recall that a map $h: X_1 \rightarrow X_2$ is virtually equivariant (between $G_1$ and $G_2$) if there exist two finite index subgroups $H_i \le G_i$
and an isomorphism $\phi: H_1 \rightarrow H_2$ such that
$h (g(x)) = \phi(g) h(x)$ for all $x\in X_1, g\in H_1$.

A homeomorphism $h: (\Lambda(G), \mathfrak{P}(G)) \rightarrow (\Lambda(G'), \mathfrak{P}(G'))$ between connected geometrically finite limit sets is {\em rigidly virtually equivariant} if  for each {\rigid} vertex $v$, $h|_{\Lambda_v}$ is virtually equivariant.

\begin{theorem}\label{thm:qigluing}
    Let $G, G'$ be geometrically finite Kleinian groups with connected limit sets and $\Lambda(G), \Lambda(G') \neq \widehat\C$. If there is a rigidly virtually equivariant homeomorphism 
    $$
    \Phi\colon (\Lambda(G), \mathfrak{P}(G)) \rightarrow (\Lambda(G'), \mathfrak{P}(G')),
    $$
    such that $\Phi$ is coarsely order-respecting at rank-two cut points, then $G, G'$ are quasi-isometric relative to parabolic subgroups.
    In particular, $(\Lambda(G), \mathfrak{P}(G))$ and $(\Lambda(G'), \mathfrak{P}(G'))$ are quasisymmetric.

    Moreover, if $\Phi$ extends to a homeomorphism of $\widehat\C$, then the quasisymmetry between $(\Lambda(G), \mathfrak{P}(G))$ and $(\Lambda(G'), \mathfrak{P}(G'))$ extends to a homeomorphism on $\widehat\C$.
\end{theorem}

Theorem~\ref{thm:qsclassG} follows directly from Theorem~\ref{thm:qigluing}:
\begin{proof}[Proof of Theorem~\ref{thm:qsclassG}]
   Since the homeomorphism preserves the parabolic loci and their ranks, it follows from \cite{PW02} that if nontrivial connected component of $\Lambda(G)$ and $\Lambda(G')$ are quasisymmetric, then $\Lambda(G)$ and $\Lambda(G')$ are quasisymmetric.
    It thus suffices to prove the case when $\Lambda(G)$ is connected.
    By Theorem~\ref{thm:carpetfreeSchottkyimpliesrigid}, $\Phi$ is virtually equivariant on each rigid vertex limit set.
    Thus the theorem follows from Theorem~\ref{thm:qigluing}.
\end{proof}

The proof of Theorem~\ref{thm:qigluing} gives the following more general statement for relatively hyperbolic groups.
\begin{theorem}\label{thm:enhance} 
Let $G, G'$ be relatively hyperbolic groups. 
If there is a type-preserving homeomorphism between their boundaries that is virtually equivariant at each parabolic and rigid vertex subgroups, then $G, G'$ are quasi-isometric relative to parabolic subgroups.
\end{theorem}

\subsection*{Proof of Theorem~\ref{thm:qigluing}}
Let $\cT, \cT'$ be the JSJ trees of $G, G'$ respectively. Note that any homeomorphism $\Psi: \Lambda(G) \rightarrow \Lambda(G')$ induces a type preserving isomorphism $\psi: \cT \rightarrow \cT'$.

We define colorings on the vertex sets of $\mathcal{T}$ and $\mathcal{T}'$ so that two vertices $v \in \mathcal{T}, v'\in \mathcal{T}'$ have the same color if and only if there is a coarsely order-respecting, rigidly virtually equivariant homeomorphism $\Psi: \Lambda(G) \rightarrow \Lambda(G')$ that sends the vertex $v$ to $v'$. 
Similarly, we define colorings on the edge sets of $\mathcal{T}$ and $\mathcal{T}'$ so that two edges $[v, w]$ and $[v',w']$ have the same color if and only if there is a coarsely order-respecting, rigidly virtually equivariant homeomorphism $\Psi: \Lambda(G) \rightarrow \Lambda(G')$ that sends the vertices $v, w$ to $v', w'$.
Since the quotient graphs $\cT/G$, $\cT'/G'$ are finite, the number of colors is finite.

We say an isomorphism $\psi: \cT \rightarrow\cT'$ is {\em color-preserving} if it preserves colors of both vertices and edges.
It is an {\em equivariant isomorphism} if it is induced by a coarsely order-respecting, rigidly virtually equivariant homeomorphism $\Psi: \Lambda(G) \rightarrow \Lambda(G')$. Note that any equivariant isomorphism is color-preserving.

\subsubsection*{Construction of the quasi-isometry}
Suppose that $\cT$ is trivial. Then $\Lambda(G)$ is either a quasicircle or homeomorphic to a Schottky set.
If $\Lambda(G)$ is a quasicircle, then the theorem follows by the \v{S}varc–Milnor lemma.
If $\Lambda(G)$ is homeomorphic to a Schottky set, then $\cT$ contains a unique {\rigid} vertex. Any rigidly virtually equivariant homeomorphism is a virtually equivariant homeomorphism between $\Lambda(G)$ and  $\Lambda(G')$, so the theorem follows immediately.

Thus, we assume that $\cT$ is nontrivial. 
Let $v_0$ be a non-{\twoended} vertex in $\mathcal{T}$ and $v_0':= \phi(v_0)$.
Let us consider an exhaustion of $\cT$ by finite trees 
$$\cT_0= \{v_0\}, \; \cT_n \subseteq \cT_{n+1}, \; \bigcup_{n=1}^\infty \cT_n = \cT
$$ such that each $\cT_n$  contains exactly $n+1$ non-{\twoended} vertices that span $\cT_n$.
Let $\widetilde{\cT_n} = \{x \in \cT: d(x, \cT_n) \leq 1\}$, so that the vertices in $\widetilde{\cT_n} \setminus \cT_n$ are exactly {\twoended} vertices adjacent to $\cT_n$.

Note that the JSJ decomposition gives quasiconvex decompositions of locally finite Cayley graphs $X$ and $X'$ of $G$ and $G'$ respectively \cite[\S 1]{Bow98}.
For each vertex $v \in V(\cT)$, we choose a quasiconvex subspace $X_v \subseteq X$ which is quasi-isometric to $G_v$, so that for any {\twoended} vertex $w$ adjacent to $v_1,..., v_k$, $X_w \subseteq \bigcap X_{v_i}$. 
There exists a $G$-equivariant coarse projection $\pi: X \rightarrow \mathcal{T}$ and $R< \infty$ such that $\pi(X_v)\subset N_R(v)$ for each vertex.
For each subtree $\cT_n \subseteq \cT$, we define $X_n = \bigcup_{v\in \cT_n} X_v = \bigcup_{v\in \widetilde{\cT_n}} X_v$. 
Let $\cP_n \subseteq X_n$ be the peripheral structure from the parabolic subgroups.

In the following, we find constants $C,E,K,L\ge 0$ and inductively construct finite trees $\cT'_n \subseteq \cT'$, and sequences of $(L, C)$-coarse Lipschitz maps
$$
H_n: (X_n, \cP_n) \longrightarrow (X_n', \cP_n') \text{ and } F_n: (X_n', \cP_n') \longrightarrow (X_n, \cP_n)
$$
which induce isomorphisms $\phi_n: \widetilde{\cT_n}\rightarrow \widetilde{\cT'_n}$ and $\phi_n^{-1}:\widetilde{\cT'_n} \rightarrow \widetilde{\cT_n}$ so that
\begin{enumerate}
    \item\label{eq:con1} $H_n|_{X_n} = H_m|_{X_n}, F_n|_{X_n'} = F_m|_{X_n'}$ for all $m \geq n$;
    \item\label{eq:con2} on each vertex space $X_v \subseteq X_n$, the restriction $H_n|_{X_v}$ is a $(K, C)$-quasi-isometry onto $X_{\phi_n(v)}$, and similarly for $F_n$;
    \item\label{eq:con3} on each {\twoended} vertex space $X_w \subseteq X_n$, the restrictions $H_n|_{X_w}, F_n|_{X_{\phi_n(w)}}$ are $M$-close to a translation between $X_w$ and $X_{\phi_n(w)}$, under identifications $X_w \cong X_{\phi_n(w)} \cong \Z$;
    \item\label{eq:con4} $F_n \circ H_n$ and $H_n \circ F_n$ are $E$-close to identity;
    \item\label{eq:con5} the induced map $\phi_n: \widetilde{\cT_n}\rightarrow \widetilde{\cT'_n}$ extends to an equivariant isomorphism $\phi_{n}: \cT \rightarrow \cT'$.
\end{enumerate}
The first part of the theorem follows by taking the limit. Indeed, the limits $H: (X, \cP) \rightarrow (X', \cP'), F: (X', \cP') \rightarrow (X, \cP)$ are $(L, C)$-coarse Lipschitz maps  such that $H \circ F, F \circ H$ are $E$-close to identity. 
Thus $H$ is a quasi-isometry relative to parabolic subgroups.
We then attach combinatorial horoballs to $(G, \cP)$ and $(G', \cP')$, following Groves--Manning, to obtain cusped spaces $Y$ and $Y'$ \cite{groves:manning:dehn}. Since $(G, \cP)$ and $(G', \cP')$ are quasi-isometric, by \cite[Theorem 3.18]{haiss:lec}, the spaces $Y$ and $Y'$ are quasi-isometric (see also \cite{MS24}). Hence their boundaries are quasisymmetrically homeomorphic. It follows that $(\Lambda(G), \mathfrak{P}(G))$ and $(\Lambda(G'), \mathfrak{P}(G'))$ are quasisymmetrically homeomorphic, and the first part of the theorem follows.

\medskip
We now proceed to construct the sequence of maps. It is sufficient to define them on the set of vertices.

\noindent
\textbf{Base case: $n = 0$.}
Suppose $v_0$ is {\mhf}.
Then $G_v$ is a geometrically finite Fuchsian group. Let $\widehat G_v$ be the de-parabolized Kleinian group associated to $G_v$. Then $\widehat G_v$ is a convex cocompact Fuchsian group of the second kind.
A component of $\Omega(G_v)\subseteq \mathbb S^1$ either corresponds to an edge of $\cT$ adjacent to $v$, or to a parabolic point of $G_v$ which does not correspond to an edge of $\cT$. 
In the first case, the coloring of the edge of $\cT$ induces a coloring of the corresponding component. In the second case, we assign to all such components a common color, distinct from the edge colors.

By Corollary~\ref{cor:colors}, there exists a color-preserving $(K_0,C_0)$-quasi-isometry $H_0: (X_0, \cP_0) \rightarrow (X_0', \cP_0')$, which is $M$-close to a translation on $X_v$ for each adjacent type I vertex $v$. 
We construct $F_0: X_0' \rightarrow X_0$  similarly.
Since the boundary maps of $H_0 \circ F_0$ and $F_0\circ H_0$ are identities on the corresponding limit sets, $H_0 \circ F_0$ and $F_0\circ H_0$ are $E_0$-close to the identity map.
Since $H_0$ is color-preserving, we claim that the induced map $\phi_0: \widetilde{\cT_0}\rightarrow \widetilde{\cT'_0}$ extends to an equivariant isomorphism $\phi_0: \cT \rightarrow \cT'$.
Indeed, let $v \in \partial \widetilde{\cT_0}$ and let $[v,v_0]$ be the unique edge in $\widetilde{\cT_0}$ incident at $v$. Let $\cT_v$ be the component of $\cT\setminus (v, v_0)$ that contains $v$.
Since $[v, v_0]$ and $[\phi_0(v), v_0']$ have the same color, there exists a coarsely order-respecting, rigidly virtually equivariant homeomorphism $H_v: \Lambda(G) \rightarrow \Lambda(G')$ with induced map $h_v = \phi_0$ on $[v, v_0]$. We extend $\phi_0$ by letting $\phi_0 = h_v$ on $\cT_v$ for each $v \in \partial \widetilde{\cT_0}$ and obtain an equivariant isomorphism.

Suppose $v_0$ is {\rigid}. By assumption, $\Phi|_{\Lambda_{v_0}}$ is virtually equivariant, and the existence of desired maps $H_0, F_0$ follows immediately.

Suppose $v_0$ is {\ranktwo}, with a rank-two cut point $a_0$. 
The {\twoended} vertices adjacent to $v_0$ are in one-to-one correspondence with $\pi_0(\Lambda(G)\setminus\{a_0\})$, and we label them by $w_k, k \in \Z$, according to the linear order from the embedding of $\Lambda(G)\setminus\{a_0\}$ in $\widehat\C\setminus\{a_0\}$. There exists $A\geq 1$ such that $w_i, w_j$ are in the same $G_{v_0}$-orbit if and only if $|i-j| = 0 \mod A$.
Then $X_{v_0} \cong \Z^2$, where $X_{w_k}$ is glued to $X_{v_0}$ along $\{(\lfloor k/A \rfloor,n):n \in \Z\}$. Similarly, $X_{v_0'} \cong \Z^2$ and $X_{w_k'}$ is glued to $X_{v_0'}$ along $\{(\lfloor k/A' \rfloor,n):n \in \Z\}$. Since $\cT/G$ and $\cT'/G'$ are finite graphs, both numbers $A, A'$ are uniformly bounded.

By assumption, the map $\Phi$ is coarsely order-respecting at rank-two cut points. 
Hence, after relabeling the indices possibly by a translation and an inversion, 
$\phi(w_j) = w_{\sigma(j)}'$, with $|\sigma(j)-j| = O(1)$.
We define $H_0 = \id: X_{v_0} \cong \Z^2 \rightarrow X_{v_0'}\cong \Z^2$ as the identity map in this coordinate, and similarly $F_0 = \id$. 
Note that $H_0$ is compatible with $\phi$ and we can take $\phi_0 = \phi: \widetilde{\cT_0}\rightarrow \widetilde{\cT'_0}$ as the induced isomorphism.
One can verify the five conditions are satisfied.

The constants $K_0, C_0, M, E_0$ are uniform as $\cT/G$ and $\cT'/G'$ are finite graphs.

\medskip
\noindent
\textbf{Inductive step.} 
We choose the constants $K= K_0, C = C_0+2M, L= K+C, E = E_0 + 2M$.
Suppose $H_n: (X_n, \cP_n) \rightarrow (X_n', \cP_n'), F_n: (X_n', \cP_n') \rightarrow (X_n, \cP_n)$ are constructed so that $H_n|_{X_v}$ is a $(K, C)$-quasi-isometry onto $X_{\phi_n(v)}$.

Let $v$ be the unique non-{\twoended} vertex in $\cT_{n+1} \setminus \cT_n$, and let $v':= \phi_{n}(v)$.
Let $u \in \cT_n$ and $u'\in \cT_n'$ be the non-{\twoended} vertex that are at distance $2$ from $v$ and $v'$, and let $w, w'$ be the {\twoended} vertex in between $u, v$ and $u', v'$ respectively.

Note that $X_w \cong G_w \cong \Z$ (resp. $X_{w'} \cong G_{w'} \cong \Z$) embeds in both $X_u \cong G_u$ and $X_v \cong G_v$ (resp. $X_u \cong G_u$ and $X_v \cong G_v$).
By the induction hypothesis, $H_n:(X_u, X_w) \rightarrow (X_{u'}, X_{w'})$ is within distance $M$ of a translation between $X_w$ and $X_{w'}$. By the same argument as the base case, there exists a $(K_0,C_0)$-quasi-isometry $H_v:(X_v, \cP_v, X_w) \rightarrow (X_{v'}, \cP_{v'}, X_{w'})$, which is $M$-close to a translation between $X_w$ and $X_{w'}$.
By post-composing $H_v$ with an appropriate element in the stabilizer $G_{w'}$, we may assume that 
$d(H_v(x), H_n(x)) \leq 2M \text{ for all } x\in X_w$.
Let
$$
H_{n+1}(x) = \begin{cases}
    H_n(x) &\text{ if }x\in X_n;\\
    H_v(x) &\text{ if }x\in X_v\setminus X_w.
\end{cases}
$$
Note that by construction, $H_{n+1}|_{X_v}$ is a $(K, C)$-quasi-isometry onto $X_{v'}$, where $K = K_0, C = C_0 + 2M$. Thus, by the induction hypothesis, the property~\eqref{eq:con2} is satisfied.

We claim that $H_{n+1}: X_{n+1} \rightarrow X_{n+1}'$ is $(L, C)$-coarse-Lipschitz. 
Indeed, let $x, y\in X_{n+1}$. We connect them via some geodesic $\gamma$ in the Cayley graph restricting to $X_{n+1}$. Then $\gamma$ is decomposed into $\gamma = \gamma_0 \cup \gamma_1 \cup ... \cup \gamma_r$, where each $\gamma_i$ is contained in some vertex space $X_{v_i}$. Note that $r \leq d(x,y)$.
Since the restriction of $H_{n+1}$ to each vertex space is a $(K, C)$-quasi-isometry, the triangle inequality implies 
\begin{align*}
    d(H_{n+1}(x), H_{n+1}(y)) &\leq \sum_{i=0}^r (K|\gamma_i| + C)
    = K d(x,y) + C (r+1) \\
    &\leq (K+ C)d(x,y) + C = Ld(x,y) + C.
\end{align*}

The construction of the $(L, C)$-coarse-Lipschitz map $F_{n+1}$ that satisfies the property~\eqref{eq:con2} is exactly the same.
It is easy to verify the properties~\eqref{eq:con1} and \eqref{eq:con3} by the construction and the induction hypothesis.
By restricting to each vertex, it is easy to verify property~\eqref{eq:con4}, i.e., that $F_{n+1} \circ H_{n+1}$ and $H_{n+1} \circ F_{n+1}$ are $E = E_0 + 2M$ close to the identity.
Since $H_v$ is color-preserving, by a similar argument of gluing homeomorphisms along {\twoended} vertices as in the base case, the induced map $\phi_{n+1}: \widetilde{\cT_{n+1}}\rightarrow \widetilde{\cT'_{n+1}}$ extends to an equivariant isomorphism $\phi_{n+1}: \cT \rightarrow \cT'$. Thus, property~\eqref{eq:con5} is satisfied as well.
The construction and thus the first part of the theorem follows by induction.

\medskip

For the moreover part, if $\Phi_0$ extends to a homeomorphism on $\widehat\C$, then we first modify the definition of colors accordingly and, at each stage of the construction, we can make sure the isomorphism $\phi_n: \cT \rightarrow \cT'$ induces a homeomorphism between $(\Lambda(G), \mathfrak{P}(G))$ and $(\Lambda(G'), \mathfrak{P}(G'))$ that extends to a homeomorphism on $\widehat\C$. Thus the moreover part follows by induction  by remarking that the components of $\Omega(G)$ form a null-sequence.  \qed

\subsection{Quasiconformal universality}
In this subsection, we show that carpet-free limit sets are quasiconformally universal.
Note that any ambient homeomorphism is coarsely order-respecting at rank-two cut points.
\begin{theorem}\label{thm:qcext}
    Let $G, G'$ be geometrically finite Kleinian groups, with $\Lambda(G), \Lambda(G') \neq \widehat\C$.
    If there is a homeomorphism 
    $$
    \Phi\colon (\widehat\C, \Lambda(G), \mathfrak{P}(G)) \rightarrow (\widehat\C, \Lambda(G'), \mathfrak{P}(G')),
    $$
    such that $\Phi$ is quasisymmetric on each dynamical carpet-quotient, then there is a quasiconformal map 
    $$
    h\colon (\widehat\C, \Lambda(G), \mathfrak{P}(G)) \rightarrow (\widehat\C, \Lambda(G'), \mathfrak{P}(G')).
    $$
\end{theorem}

\begin{cor}\label{cor:relativeqcuniversalGF}
    Let $G$ be a geometrically finite Kleinian group with parabolic locus $\mathfrak{P}(G)$. Suppose that $\Lambda(G)$ is carpet-free. Then $(\Lambda(G), \mathfrak{P}(G))$ is relatively quasiconformally universal in $\mathfrak{X}_{\gf}$.

    In particular, if $G$ is convex cocompact and $\Lambda(G)$ is carpet free, then $\Lambda(G)$ is quasiconformally universal in $\mathfrak{X}_{\cc}$.
\end{cor}

\subsubsection{$\Lambda(G)$ is connected}\label{subsec:connectedQCUniversal}
Let $\mathcal{T}$ and $\mathcal{T}'$ be the JSJ decomposition of $G$ and $G'$ respectively.
Let $\Phi_0: (\widehat\C, \Lambda(G)) \rightarrow (\widehat\C, \Lambda(G'))$ be a homeomorphism. Then $\Phi_0$ induces a type preserving isomorphism $\phi_0: \mathcal{T} \longrightarrow \mathcal{T}'$.

We will prove a stronger statement in the connected case (Proposition~\ref{prop:colorpreservingcompressiondisk}), which will be used in the disconnected case when we do the planar gluing.
Let $F$ be a finite set. A {\em coloring on $\Omega(G)$} is a $G$-invariant map $\kappa: \pi_0(\Omega(G)) \rightarrow F$ that is onto.
A homeomorphism $\Phi_0$ induces colorings $\kappa, \kappa'$ as follows.
Two components $U, V \in \Omega(G)$ are equivalent if there exists $g \in G$ and $g' \in G'$ so that $gU = \Phi_0^{-1}(g' \Phi_0(V))$.
Note that each equivalence class consists of finitely many $G$-orbits of components of $\Omega(G)$. Let $F$ be the set of equivalence classes, and let $\kappa$ be the projection map. Then $\kappa$ is a coloring of $\Omega(G)$.
We define $\kappa': \pi_0(\Omega(G')) \rightarrow F$ so that $\kappa'(U') = \kappa(U)$ if there exists $g \in G$ and $g' \in G'$ so that $gU = \Phi_0^{-1}(g'U')$.
By construction, $\Phi_0$ is color-preserving.

To prove the extension, we need the following gluing lemma for limit sets, which follows directly from \cite[Theorem 4.4]{ph:unifplanar}. 
\begin{lem}\label{lem:qcextension}
    Let $G_i, i = 1,2$ be Kleinian groups with connected limit sets. Let $h\colon (\widehat \C, \Lambda(G_1) ) \rightarrow (\widehat\C, \Lambda(G_2))$ be a homeomorphism and suppose that there exists $\eta: \R_+ \rightarrow \R_+$ so that
    \begin{itemize}
        \item $h|_{\Lambda(G_1)}$ is $\eta$-quasisymmetric;
        \item for each component $U \subseteq \Omega(G_1)$, $h|_{\overline{U}}$ is $\eta$-quasisymmetric.
    \end{itemize} 
    Then there exists $\eta': \R_+ \rightarrow \R_+$ so that $h$ is $\eta'$-quasisymmetric on $\widehat\C$. In particular, $h$ is quasiconformal on $\widehat\C$.
\end{lem}

\begin{prop}\label{prop:connected}
    Let $G, G'$ be geometrically finite Kleinian groups with connected limit sets. Let $\Phi_0: (\widehat\C, \Lambda(G), \mathfrak{P}(G)) \longrightarrow (\widehat\C, \Lambda(G'), \mathfrak{P}(G'))$ be a homeomorphism that is quasisymmetric on each dynamical carpet-quotient. Suppose it induces colorings $\kappa, \kappa'$ on $\Omega(G)$ and $\Omega(G')$. Then there exists a color-preserving quasiconformal map $h:(\widehat\C, \Lambda(G), \mathfrak{P}(G)) \longrightarrow (\widehat\C, \Lambda(G'), \mathfrak{P}(G'))$.
\end{prop}
\begin{proof}
    Theorem \ref{thm:qigluing} gives a color-preserving quasisymmetric map $h$ that extends to a homeomorphism on $\widehat\C$. 
    
    Let $U$ be a component of $\Omega(G)$. Since $U$ is simply connected, its stabilizer $\stab_G(U)$ is isomorphic to a surface group. Thus, we have two cases \cite{Mas70}:
    \begin{enumerate}
        \item if $\stab_G(U)$ is quasi-Fuchsian, then $U$ is a quasidisk;
        \item otherwise, $\stab_G(U)$ is a geometrically finite Kleinian group on the Bers boundary, and $\partial U$ is homeomorphic to a Basilica (see Figure~\ref{fig:ranktwoSch}, right).
    \end{enumerate}
    In the first case, the quasisymmetric map extends to a uniform quasisymmetric map on $\overline{U}$, as $h|_{\partial U}$ is quasisymmetric and $U$ is a uniform quasidisk.
    
    In the second case, we note that parabolic points on $\partial U$ are either accidental parabolic points that correspond to the cut points of $\partial U$, or {\em persistent parabolic points} that do not correspond to cut points.
    Let $\Phi:\D \rightarrow U$ (and $\Psi: \D \rightarrow h(U)$) be the uniformization map, which extends to a homeomorphism between $\SS^1$ and the ideal boundaries $\partial^IU$ (and $\partial^I h(U)$).
    The homeomorphism $h: \partial U \rightarrow \partial h(U)$ induces a homeomorphism $h_*$ between their ideal boundaries. 
    We claim that $\Psi^{-1} \circ h_* \circ \Phi$ is quasisymmetric on $\SS^1$.
    Indeed, note that the quasi-isometry between $G, G'$ restricts to a quasi-isometry 
    $$
    H: (\stab_G(U), \cP_{U}) \rightarrow (\stab_{G'}(h(U)), \cP_{h(U)}),
    $$
    where $\stab_G(U), \stab_{G'}(h(U))$ are equipped with word metrics, and $\cP_{U}, \cP_{h(U)}$ are the peripheral groups associated to the {persistent} parabolic points. Attaching horoballs in the sense of Groves--Manning to these {persistent} parabolic points, it follows that the boundary map $\Psi^{-1} \circ h_* \circ \Phi$ is quasisymmetric
    \cite[Theorem 3.18]{haiss:lec}. Therefore, $\Psi^{-1} \circ h_* \circ \Phi$ extends to a quasiconformal map on $\D$. 
    Pulling this extension back via the uniformization maps gives a quasiconformal extension of $h|_{\partial U}$ over $U$. Since the closure of each complementary component of $\overline{U}$ is a quasidisk, and $h$ is quasisymmetric on its boundary, we can extend $h$ to a homeomorphism of $\widehat\C$ which is quasiconformal on $\widehat\C \setminus \partial U$. By \cite[Corollary 1.6]{LMM26}, the Basilica limit set $\partial U$ is conformally removable, so the extension is quasisymmetric on $\overline{U}$. Since there are only finitely many orbits of such $U$ and $h$ is quasisymmetric on $\Lambda(G)$, it follows that the extension is uniformly quasisymmetric \cite[Cor.\,4.3]{ph:unifplanar}.

    The gluing lemma (Lemma~\ref{lem:qcextension}) implies that $h$ extends to a quasiconformal map $h:(\widehat\C, \Lambda(G), \mathfrak{P}(G))\rightarrow (\widehat\C, \Lambda(G'), \mathfrak{P}(G'))$, and the proposition follows.
\end{proof}

\subsubsection*{Quasiconformal extension with colored invariant sets}
Let $G$ be a non-elementary geometrically finite Kleinian group, and let $X \subseteq \Omega(G)$ be a $G$-invariant set so that $|X/G| < \infty$. Let $F$ be a finite set. A {\em coloring on $X$} is a $G$-invariant map $\kappa: X \rightarrow F$ that is onto.
For $\varepsilon > 0$, we let $\mathcal{D}_{X, \varepsilon}$ be the collection of all hyperbolic disks in $\Omega(G)$ centered at points in $X$ with radius $\varepsilon$. Note that if $\varepsilon$ is small, the disks are pairwise disjoint. The coloring on $X$ induces a coloring $\kappa: \mathcal{D}_{X, \varepsilon} \rightarrow F$.
We start with the following proposition.
\begin{prop}\label{prop:colorpreservingcompressiondisk}
    Let $G_1, G_2$ be two geometrically finite Kleinian groups with nontrivial connected limit set. 
    Suppose that $G_1, G_2$ are quasi-isometric relative to parabolic subgroups and let $\phi: (\Lambda(G_1),\mathfrak{P}(G_1)) \rightarrow (\Lambda(G_2), \mathfrak{P}(G_2))$ be an induced quasisymmetric map.
    By Proposition~\ref{prop:connected}, $\phi$ extends to a quasiconformal map $\Phi_0$ on $\widehat\C$.
    
    Let $X_i \subseteq \Omega(G_i), i =1,2$ be a $G_i$-invariant set so that $|X_i/G_i| < \infty$ with colorings $\kappa_i: X_i \rightarrow F$.
    Suppose that $\phi$ also extends to a color-preserving homeomorphism $\widetilde\phi:(\widehat\C, X_1)\rightarrow(\widehat\C, X_2)$. Then there exists a color-preserving quasiconformal extension 
    $$
    \Phi: (\widehat\C, X_1) \rightarrow (\widehat\C, X_2).
    $$
    Given a point $x_1\in X_1$, and $x_2 = \widetilde\phi(x_1)\in X_2$, the quasiconformal homeomorphism $\Phi$ can be chosen so that $\Phi(x_1) \in G_2 x_2$.
    The quasiconformal constant only depends on $G_1, G_2, X_1, X_2$ and the quasisymmetry of $\phi$ (but does not depend on $x_1$).
    
    Moreover, for all sufficiently small $\varepsilon$, $\phi$ extends to a color-preserving quasiconformal extension 
    $$
    \Phi_\varepsilon: (\widehat\C, \mathcal{D}_{X_1, \varepsilon}) \rightarrow (\widehat\C, \mathcal{D}_{X_2, \varepsilon}).
    $$
\end{prop}

\begin{proof}
    Let $\Phi_0$ be a quasiconformal extension of $\phi$, and let $\widetilde{X}_2 = \Phi_0(X_1)$.
    Let $\Delta$ be a component of $\Omega(G_2)$.
    For each color $j$, let $X_2(j, \Delta) \subseteq X_2$ and $\widetilde{X}_2(j, \Delta) \subseteq \widetilde{X}_2$ be the sets of color $j$ points in $\Delta$. 

    Since $\phi$ is induced by a relative quasi-isometry, it is compatible with the parabolic subgroups. Moreover $X_2(j, \Delta)$ is invariant under $\stab(\Delta)$, while $\widetilde{X}_2(j, \Delta)$ is the quasiconformal image of a set invariant under $\stab(\Phi_0^{-1}(\Delta))$. Thus after modifying the quasiconformal extension $\Phi_0$ in a uniform way if necessary, we may assume that $X_2(j, \Delta), \widetilde{X}_2(j, \Delta)$ are uniformly discrete and relatively dense in $\Delta^\circ = \Delta \setminus \bigcup D$, where $\bigcup D$ is a $\stab(\Delta)$-invariant family of horodisks based at the persistent cusps of $\stab(\Delta)$. Here $\Delta^\circ$ is equipped with the path metric induced by the hyperbolic metric on $\Delta$.
    They are non-amenable and have bounded geometry; see \cite[\S 2]{Why99}. Thus, by \cite[Theorem 1.1 and Theorem 5.1]{Why99} (see also \cite{Bog97}), there exists a bi-Lipschitz map
    $$
    H_{j, \Delta}: \widetilde{X}_2(j, \Delta) \rightarrow X_2(j, \Delta)
    $$ 
    such that $d_\Delta(x, H_{j, \Delta}(x))$ is uniformly bounded.
    Therefore, assembling these maps over all colors $j$ and components $\Delta$, we obtain a color-preserving bi-Lipschitz $H': \widetilde{X}_2 \rightarrow X_2$ such that $H'$ is at uniformly bounded distance from the identity. 
    Now let $\widetilde{x}_2\in G_2 x_2 \cap \Delta$ be a point closest to $\Phi_0(x_1)$ with respect to the $d_\Delta$-metric. Let $s \in \widetilde X_2$ so that $H'(s) = \widetilde{x}_2$. We define $H$ so that $H(x)= H'(x)$ if $x \in \widetilde X_2 \setminus \{\Phi_0(x_1), s\}$, $H(\Phi_0(x_1)) = \widetilde{x}_2$ and $H(s) = H'(\Phi_0(x_1))$. Then $H$ is $K$-bi-Lipschitz and $d_\Delta(x, H(x)) \leq M$ for some constant $M$, where $K, M$ do not depend on $x_1$.
    Since $H$ is at uniformly bounded distance from the identity, by locally pushing points, we can extend $H$ to a bi-Lipschitz map with respect to the hyperbolic metric on each component. This gives a quasiconformal map $\Phi_1$ so that $\Phi_1(\widetilde{X}_2) = X_2$ and $\Phi_1|_{\Lambda(G_2)} = \id$.
    The proposition follows by taking $\Phi:= \Phi_1 \circ \Phi_0$.

    For the moreover part, by choosing $\varepsilon$ sufficiently small, we can find $\delta>0$ such that for each $x \in X_1$, if $U_x$ denotes the hyperbolic disk of radius $\delta$ centered at $x$, then the collection $\{U_x: x\in X_1\}$ is pairwise disjoint, $U_x$ contains $D_{x, \varepsilon} = B(x, \varepsilon)$ and $\Phi(U_x)$ contains the disk $D_{\Phi(x), \varepsilon}$.
    Then a standard interpolation argument proves the moreover part.
\end{proof}

\subsubsection{$\Lambda(G)$ is disconnected}\label{subsec:disconnectedQCUniversal}
Note that $\phi: (\widehat\C, \Lambda(G),\mathfrak{P}(G)) \rightarrow (\widehat\C, \Lambda(G'),\mathfrak{P}(G'))$ preserves parabolic points and their ranks.
Therefore, if $\Lambda(G)$ is a Cantor set, then Theorem~\ref{thm:qcext} follows from \cite{Hid26}.
Thus, we assume from now on that $\La(G)$ is not a Cantor set. 

Moreover, we may adapt Proposition \ref{prop:colorpreservingcompressiondisk} as follows (we omit the proof).

\begin{lem}\label{lem:parabolicmatch} Let $v,v'$ be parabolic vertices (of the same rank) and let $X_v$ (resp. $X_{v'}$) be finitely many 
$G_v$-orbits (resp. $G_{v'}$-orbits). There exists an ambient quasiconformal conjugacy between finite index subgroups that maps $X_v$ to $X_{v'}$, and maps 
pairwise disjoint invariant neighborhoods $\cD_{X_v}$ of definite size to $\cD_{X_{v'}}$.
\end{lem}

We refer to \S\,\ref{subsec:icd} for the notions that are used in the sequel.
Unlike the JSJ decomposition, 
the incompressible decomposition of a hyperbolic manifold is not canonical (only the non-elementary and parabolic vertices are), and thus such homeomorphisms 
may not respect the incompressible decompositions. To overcome this, we first construct specific incompressible decompositions as follows.

Given a geometrically finite Kleinian group $G$ with disconnected limit set, we let $V_{nd}$ denote the set of nontrivial connected components
of $\La(G)$ and of those singleton components that are parabolic points. It follows that $\Homeo^+(\cbar,\La(G),\mathfrak{P}(G))$ acts on $V_{nd}$.
To $v\in V_{nd}$, we associate the corresponding component $\La_v$. Note that the set of vertices of any incompressible decomposition will contain $V_{nd}$. We define three colorings $\kappa:V_{nd}\to F$, $\kappa_{\Omega}:\pi_0(\Omega(G))\to F_{\Omega}$ and $\kappa_v: \pi_0(\Omega(G_v)) \to F_v$ where $v \in V_{nd}$ such that
\bit
\item $\kappa(v)=\kappa(w)$ if and only if there is a homeomorphism of $(\widehat\C, \Lambda(G), \mathfrak{P}(G))$ that sends $\Lambda_v$ to $\Lambda_w$;
\item $\kappa_\Omega(U)=\kappa_\Omega(V)$ if and only if there is a homeomorphism of $(\widehat\C, \Lambda(G), \mathfrak{P}(G))$ that sends $U$ to $V$;
\item $\kappa_v(U) = \kappa_{v'}(V)$ if and only if there is a homeomorphism of $(\widehat\C,\Lambda_v, \mathfrak{P}(G_v))$ to $(\widehat\C,\Lambda_{v'}, \mathfrak{P}(G_{v'}))$ that sends $U$ to $V$ for any $v'\in V_{nd}$.
\eit
Given two ambiently homeomorphic limit sets, we will choose matching colors.

\begin{prop}\label{prop:newincompressiblecoredecomp}
    Let $G, G'$ satisfy the assumptions of Theorem~\ref{thm:qcext}, whose limit sets are disconnected and are not Cantor sets.
      There exist finite index subgroups $H< G, H'< G'$ and incompressible decompositions of $H, H'$ with the following properties. Let $\cD, \cD'$ denote the  lifts of the boundaries of the compression disks and let $\cT, \cT'$ be the corresponding trees. Then their sets of vertices are $V_{nd}\cup V_D$ and $V_{nd}' \cup V_D'$ respectively so that
      \begin{enumerate}
        \item\label{prop:newincompressible:1} For any $v \in V_{P}$ (resp. $v' \in V_P'$), all vertices $w \in \cT$ (resp. $w'\in \cT'$) at distance $2$ from $v$ have the same color; and 
        \item\label{prop:newincompressible:2} for any $v \in V_{nd}, v' \in V_{nd}'$ of the same color, any $\Omega(G_v)$-color-preserving homeomorphism $h:(\widehat\C, \Lambda_v, \mathfrak{P}_v) \rightarrow (\widehat\C, \Lambda_{v'}, \mathfrak{P}_{v'})$ is isotopic relative to $\Lambda_v$ to a ${\mathcal{D}}_v$-color-preserving homeomorphism 
    $$
    h:(\widehat\C, \Lambda_{v},\mathfrak{P}_v, {\mathcal{D}}_{v})\rightarrow (\widehat\C, \Lambda_{v'}, \mathfrak{P}_{v'}, {\mathcal{D}}_{v'}),
    $$  
    which in turn is isotopic relative to $\cD_v$ to a homeomorphism
    $$
    H: (\widehat\C, \Lambda(G), \mathfrak{P}(G), \cD) \rightarrow (\widehat\C, \Lambda(G'), \mathfrak{P}(G'),\cD').
    $$
      \end{enumerate}
\end{prop}

We first derive Theorem~\ref{thm:qcext} in the disconnected case from Proposition~\ref{prop:newincompressiblecoredecomp} and then prove the latter.

\subsubsection*{Proof of Theorem~\ref{thm:qcext} in the disconnected case assuming Proposition~\ref{prop:newincompressiblecoredecomp}}
Being already dealt with, we assume that $\La(G)$ is not a Cantor set and as  previously said, we have nontrivial incompressible decompositions with at least one
\nonel\ Kleinian subgroup. Let $\Phi_0: (\widehat\C, \Lambda(G),\mathfrak{P}(G)) \longrightarrow (\widehat\C, \Lambda(G'),\mathfrak{P}(G'))$ be a homeomorphism 
and let $\cT, \cT'$ be the trees given by Proposition~\ref{prop:newincompressiblecoredecomp}.

Let us first prepare the setting. Given a non-elementary vertex $v$, we pick a $G_v$-invariant set $X_v$ consisting of finitely many orbits that intersect all the 
limit sets $\La_w$ that are at distance $2$ from $v$ in $\cT$. Proceed similarly for $G'$. By the connected case of Theorem~\ref{thm:qcext}, and since we have finitely many 
$G$ and $G'$-orbits of {\nonel} vertices by Ahlfors finiteness theorem, applying Proposition~\ref{prop:colorpreservingcompressiondisk} to each such pair $(v,X_v)$, $(v',X_{v'}')$ 
for $G$ and $G'$ provides us with a
uniform hyperbolic radius $\ep>0$ and a uniform quasiconformal constant $K_0$.
We proceed similarly for the parabolic points with Lemma \ref{lem:parabolicmatch} instead.

We consider pairwise disjoint annuli in $\Omega(H)/H = \Omega(G)/H$ and $\Omega(G')/H'$ that enclose the boundaries of the compression disks and lift them to $\Omega(G)$ so as to obtain pairwise disjoint annuli $A_C = A_v \subseteq \Omega(G)$, with core $C$ for each $v \in V_D$.
By Teichm\"uller's Round annulus theorem and since $\La(G)$ is uniformly perfect, there exists $L>0$ such that the hyperbolic diameter in $\Omega(G)$ of the core of any annulus of modulus at least $L/4$ in $\Omega(G)$ will be
smaller than $\ep/2$. Since inclusions are contractions for the Poincar\'e metrics, this also holds for the Poincar\'e metric of $\Omega(H_w) = \Omega(G_w)$ for any non-elementary vertex $w$. 
We claim that we may deform quasiconformally the complex structure of the annuli $\{A_v:\,\ v\in V_D\}$ to ensure their moduli are at least $L$ without modifying the M\"obius conjugacy classes of the vertex groups. To see this, note that $G$ can be constructed inductively from the vertex subgroups by using Maskit's combinations theorems \cite{AM77}. We may then apply this construction to hyperbolic disks of radius $\ep/2$  in $\Omega(G_v)$ centered at points from $X_v$ for the non-elementary vertices and to the domains given
by Lemma \ref{lem:parabolicmatch} for the parabolic vertices.

Let $\widetilde{A}_C = \widetilde{A}_v \subseteq A_C$ be an annulus that contains $C$ as a core curve so that $A_C \setminus \widetilde{A}_C$ consists of two annuli, each with modulus at least $ L/4$.
Let $\widetilde{C}_{\pm}$ be the two boundary components of $\widetilde{A}_C$. 

Note that non-{\disk} vertices of $\cT$ are in one-to-one correspondence with nontrivial connected components of $\widehat\C \setminus \bigcup\widetilde{A}_C$. 
Let $v$ be a non-{\disk} vertex. Denote the corresponding component by $\cK_v$. Note that $\La_v \subseteq \cK_v$.
We define
\begin{equation}\label{eq:protectionC}
    \widetilde{\mathcal{D}}_v:= \{\widetilde{C}: \widetilde{C} \subseteq \partial \cK_v\}, \quad \widetilde{\cD} = \bigcup_{v\in V_{nd}} \widetilde{\cD}_v.
\end{equation}
For each non-elementary vertex $v$, the loops $\tilde\cD_v$ are contained in the $\ep$-neighborhoods of $X_v$ in the Poincar\'e metric of $\Omega(G_v)$. 
Similar objects for $G'$ are constructed analogously with  compatible colorings.

Let $v_0$ be a {\nonel} vertex and let $v_0'$ be the corresponding {\nonel} vertex under $\Phi_0$.
Let us consider an exhaustion of $\cT$ by finite trees 
$$\cT_0= \{v_0\}, \; \cT_n \subseteq \cT_{n+1}, \; \bigcup_{n=1}^\infty \cT_n = \cT, \; \Lambda_n := \bigcup_{w\in \cT_n} \Lambda_w
$$ such that each $\cT_n$  contains exactly $n+1$ non-{\disk} vertices that span $\mathcal{T}_n$.
Let $\cK_n := \bigcup_{v \in V_{nd}\cap \cT_n} \cK_v \cup \bigcup_{w \in V_D \cap \cT_n} \widetilde{A}_w$.

In the following, we inductively construct finite trees $\cT'_n \subseteq \cT'$ and a sequence of uniformly quasiconformal maps 
$$
h_n: (\widehat\C, \cK_n,  \Lambda_n, \mathfrak{P}_n) \longrightarrow (\widehat\C, \cK_n', \Lambda_n', \mathfrak{P}_n') \text{ such that}
$$
\begin{itemize}
    \item $h_n|_{\cK_n} = h_m|_{\cK_n}$ for all $m \geq n$;
    \item For each non-{\disk} vertex $v \in \mathcal{T}_n$, $h_n(\widetilde{\mathcal{D}}_v) = \widetilde{\mathcal{D}}_{v'}$, where $v' \in \mathcal{T}'_n$ so that $h_n(\Lambda_v) = \Lambda_{v'}$;
    \item $h_n|_{\cK_n}$ extends to a homeomorphism 
    $$
    \Phi_{n+1}:(\widehat\C, \Lambda(G), \mathfrak{P}(G),\widetilde\cD) \longrightarrow (\widehat\C, \Lambda(G'), \mathfrak{P}(G'),\widetilde\cD');
    $$
\end{itemize}
The theorem follows by taking the limit.

\medskip 

\noindent
\textbf{Base case: $n = 0$.}
Let $\widetilde C \in \widetilde{\cD}_{v_0}$ and $D_{\widetilde C, v_0}$ be the topological disk bounded by $\widetilde C$ that is disjoint from $\Lambda_{v_0}$. Then $\{D_{\widetilde C, v_0}: \widetilde C \in \widetilde\cD_{v_0}\}$ and $\{D_{\widetilde C', v_0}: \widetilde C' \in \widetilde\cD_{v_0'}\}$ are invariant under $H_{v_0}$ and $H_{v_0'}$ respectively and are contained in hyperbolic disks of radius $\ep$ centered at $X_{v_0}$ and $X_{v_0'}'$.

Since $v_0$ is {\nonel},  by Proposition~\ref{prop:connected}, there exists a color-preserving quasiconformal map $\widetilde h_0: (\widehat\C, \Lambda_{v_0}, \mathfrak{P}_{v_0}) \rightarrow (\widehat\C, \Lambda_{v_0'}, \mathfrak{P}_{v_0'})$. 
We apply Proposition~\ref{prop:colorpreservingcompressiondisk} to $\{D_{\widetilde C, v_0}: \widetilde C \in \widetilde\cD_{v_0}\}$ and $\{D_{\widetilde C', v_0}: \widetilde C' \in \widetilde\cD_{v_0'}\}$, and obtain a $K_0$-quasiconformal map $h_0: (\widehat\C, \cK_0, \Lambda_0, \mathfrak{P}_0) \rightarrow (\widehat\C, \cK_0', \Lambda_0', \mathfrak{P}_0')$ extending $\widetilde h_0|_{\Lambda_0}$.
Since $h_0|_{\Lambda_{v_0}}$ is color-preserving on the star of $v_0$, by Proposition~\ref{prop:newincompressiblecoredecomp}, it extends to a homeomorphism $\Phi_1:(\widehat\C, \Lambda(G), \mathfrak{P}(G),\widetilde\cD) \rightarrow (\widehat\C, \Lambda(G'), \mathfrak{P}(G'),\widetilde\cD')$.

By our preparation, the quasiconformal constant is independent of the {\nonel} vertices.

\medskip
\noindent 
\textbf{Inductive step.} Suppose a $K$-quasiconformal map $h_n: (\widehat\C, \cK_n, \Lambda_n, \mathfrak{P}_n) \rightarrow (\widehat\C, \cK_n', \Lambda_n', \mathfrak{P}_n')$ is constructed.
Let $v$ be the unique non-{\disk} vertex in $\cT_{n+1}\setminus\cT_n$ and let $v' \in \cT'$ be the corresponding vertex by the homeomorphism $\Phi_{n+1}$.
Note that there exists a unique Jordan curve, denoted by $C_0 \in \mathcal{D}$, that separates $\Lambda_v$ and $\Lambda_n$. Similarly, let $C_0'\in \mathcal{D}$ be the Jordan curve that separates $\Lambda_{v'}$ and $\Lambda_n'$.
Let $\widetilde{C}_{0,v} \in \widetilde{\cD}_v$ and $\widetilde{C}_{0,v'} \in \widetilde{\cD}_{v'}$ be the corresponding boundary components of the annuli $\widetilde{A}_{C_0}$ and $\widetilde{A}_{C_0'}$.

We claim that there exists a color-preserving $K_0$-quasiconformal extension $h_v:(\widehat\C, \Lambda_v, \mathfrak{P}_v, \widetilde{\cD}_v) \rightarrow (\widehat\C, \Lambda_{v'}, \mathfrak{P}_{v'}, \widetilde{\cD}_{v'})$ so that $h_v(\widetilde{C}_{0,v}) = \widetilde{C}_{0,v'}$.
If $v$ is a \nonel\ vertex, the proof is identical to the base case, except that we may need to post-compose with an element of $G_{v'}$ to arrange $h_v(\widetilde{C}_{0,v}) = \widetilde{C}_{0,v'}$. If $v$ is {\para}, then, by \eqref{prop:newincompressible:1} of Proposition~\ref{prop:newincompressiblecoredecomp}, each $\widetilde{C} \in \widetilde{\cD}_{v}$ has the same color. Thus, the color-preserving quasiconformal map can be constructed directly from Lemma \ref{lem:parabolicmatch}. As in the base case, our preparation enables us to bound the quasiconformal constant independently of $n$.

We now match $h_n$ and $h_v$ along the annulus $\widetilde{A}_{C_0}$. Let $D^\pm$ be the two disk components of $\widehat\C \setminus \widetilde{A}_{C_0}$ so that $D^+$ contains $\Lambda_n$ and $D^-$ contains $\Lambda_v$.
We can glue $h_n$ and $h_v$ along the annulus $\widetilde{A}_{C_0}$ and obtain a $K$-quasiconformal map $h_{n+1}: (\widehat\C, \cK_{n+1}, \Lambda_{n+1}, \mathfrak{P}_{n+1}) \rightarrow (\widehat\C, \cK_{n+1}', \Lambda_{n+1}', \mathfrak{P}_{n+1}')$ with
$$
h_{n+1}|_{D^+}(z) = h_n|_{D^+}(z), \text{ and } h_{n+1}|_{D^-}(z) = h_v|_{D^-}(z).
$$
Since $h_{n+1}$ is color-preserving, it follows from Proposition~\ref{prop:newincompressiblecoredecomp} that it extends to a homeomorphism $\Phi_{n+2}:(\widehat\C, \Lambda(G), \mathfrak{P}(G),\widetilde\cD) \longrightarrow (\widehat\C, \Lambda(G'), \mathfrak{P}(G'),\widetilde\cD')$.
Since there are only finitely many $H$-orbits of non-{\disk} vertices, $h_v$ is uniformly quasiconformal. Since the modulus of the annulus $\widetilde{A}_{C_0}$ is bounded from below, we may choose $K$ that depends only on the modulus $L$ and the group $G$ so that the map $h_{n+1}$ is again $K$-quasiconformal.
The construction and thus the theorem now follows by induction. \qed

\subsubsection*{Construction of adapted incompressible decompositions}

For the next lemma, recall that a function group is a Kleinian group with a totally invariant component of $\Omega(G)$. This notion applies typically to the stabilizer of an ordinary component.

\begin{lem}\label{lem:incompressiblefunct}
Let $G$ be a finitely generated geometrically finite function group with a non-simply connected invariant domain $U$ and non-empty $V_{nd}$ set and let $\tilde{\kappa} : V_{nd}\to \{1,\ldots, l\}$ be a $G$-invariant coloring, $l\ge 2$. 
There exists a finite index subgroup $H$ and an incompressible decomposition of $H$  for which the vertices of the corresponding tree $\cT$ consist 
of only {\disk}, {\nonel} and {\para} types, and that satisfy the following properties.
\begin{enumerate}
    \item\label{decomposition:eqn:1} The non-{\disk} vertices are colored by $\kappa: V_{nd}(\cT) \rightarrow \{0, 1,..., l\}$ so that $\kappa(v) = \tilde\kappa(v)$ if $\tilde{\kappa}(v) \geq 2$, and $\kappa(v) \in \{0,1\}$ if $\tilde{\kappa}(v) = 1$.    
\item\label{decomposition:eqn:2a} Each {\disk} vertex $w$ is adjacent to a vertex $v$ with $\kappa(v)=0$.
\item\label{decomposition:eqn:2b} For each non-{\disk} vertex $v$ with $\kappa(v) = 0$ and for each $i \in \{0,1,..., l\}$, there is a vertex $v'$ at distance $2$ from $v$ such that $\kappa(v') = i$.
\end{enumerate}
\end{lem}
\begin{proof}
Since $V_{nd} \neq \emptyset$, the incompressible core $\Sigma$ of the Kleinian manifold $M$ of $G$ is nontrivial. The $1$-handles $\D\times [0,1]$ attached to $\Sigma$ needed to reconstruct $M$ give compression disks $\D\times\{1/2\}$. Cutting along these disks, and gluing the two components $\D \times [0,1/2]$ and $\D \times [1/2,1]$ to the incompressible core $\Sigma$, we may assume the incompressible core $\Sigma$ is the complement of the compression disks. It follows that all the vertices are either of disk type, of non-elementary type or of parabolic type. 

We decompose  $S:= U/G$ into  $S = \bigcup_{i=1}^k X_i$ where the $X_i$'s list the components of $S \cap \Sigma$.
The coloring of $V_{nd}$ for $G$ induces a coloring on the collection $\tilde\kappa:\{X_1,..., X_k\} \rightarrow \{1,..., l\}$.
After relabeling, we may assume that $\tilde\kappa(X_1) = 1$ and 
passing to a finite cover if necessary, we may further assume that $|\tilde\kappa^{-1}(i)| \geq 2$ for all $i \in \{1,..., l\}$.

For each $2 \leq i \leq k$, write $\partial X_i  = \bigcup_{j=1}^{s_i} C_{i,j}$, where $C_{i,j}$ is the boundary of some compression disk.
Choose a collection of arcs $\gamma_{i,m}$ so that $\Gamma_i:=\bigcup_{j}C_{i,j} \cup \bigcup_m \gamma_{i,m}$ is connected and projects to a finite tree after each $C_{i,j}$ is collapsed to a point. Set $\Gamma := \bigcup_{2 \leq i \leq k} \Gamma_i$, and let $N(\Gamma)$ be a small neighborhood of $\Gamma$ in $S$. 

We claim that $\partial N(\Gamma)$ consists of finitely many boundaries of compression disks. Any component  $C$ of $\partial N(\Gamma)$ is contained in some $X_i$, $2\le i\le k$. We first observe that $X_i\setminus \Gamma_i$ is connected so that $C$ retracts on $\Gamma_i$. If $C$ was null-homotopic in $X_i$, then $X_i$ would be a genus $0$ surface since $\Gamma_i$ is a tree with nodes, contrary to our assumption that $\cT$ contains no \ball\ type vertices. Moreover, $C$ bounds a disk in $M$ by concatenating the compression disks bounded by $C_{ij}$ along the boundaries of neighborhoods of the arcs $\gamma_{i,m}$.

Each connected component $A\subseteq N(\Gamma)$, together with the corresponding compression disks supported in $\partial N(\Gamma)$, bounds a handlebody in the 3-manifold $M$. By introducing additional compression disks to cut this handlebody, we may assume each component $A$ bounds a ball $B_A$ in $M$. Note that $\partial B_A$
intersects $X_1$ by construction. 
For each component $A\subseteq N(\Gamma)$, we choose a compression disk $D_A \subseteq \partial B_A$ whose boundary is contained in $X_1 \cap \partial A$, and add the ball $B_A$ to the component of $\Sigma$ that contains $X_1$ on its boundary. Its intersection $\widehat X_1$ with $S$ is obtained by gluing every component $A$ to $X_1$ along $\partial D_A$.
For $i \geq 2$, let $\widehat X_i = X_i \setminus N(\Gamma)$.
Then $\widehat X_i \cap \widehat X_j = \emptyset$ for all $i>j\geq 2$, and $\widehat X_i \cap \widehat X_1$ is the boundary of a compression disk for all $i \geq 2$.

We define a new coloring by setting $\kappa(\widehat X_i) = \tilde\kappa (X_i)$ if $i \geq 2$ and $\kappa(\widehat X_1) = 0$.
By construction, the pieces of color $j\geq 1$ are only adjacent to the color $0$ piece and the color $0$ piece is adjacent to pieces of any color $j\geq 1$. Thus, the conclusions
(\ref{decomposition:eqn:1}) and (\ref{decomposition:eqn:2a}) hold. If $X_1$ is glued to itself, then (\ref{decomposition:eqn:2b}) holds as well.
Otherwise, by choosing a non-separating essential simple closed curve $\alpha$ in $\widehat X_2$ and taking the double cover dual to $\alpha$ (obtained by cutting $S$ along $\alpha$, taking two copies of the cut surface, and regluing them crosswise), we obtain a further double cover $\pi: \widetilde{S} \to S$ so that $\pi^{-1}(\widehat X_i) = \widetilde X_i^\pm$ has two connected components for $i \neq 2$, and $\pi^{-1}(\widehat X_2) = \widetilde X_2$ is connected. Since $G$ is a function group with invariant domain $U$, this covering 
provides us with a finite index subgroup of $G$. The coloring $\kappa$ lifts to a coloring on $\{\widetilde X_2, \widetilde X_i^\pm, i \neq 2\}$ that still
satisfies (\ref{decomposition:eqn:1}) and (\ref{decomposition:eqn:2a}). 
Let $C^\pm = \widetilde X_2 \cap \widetilde X^\pm_1$. We choose an arc $\gamma$ connecting $C^+$ to $C^-$ in $\widehat X_2$ and let $N$ be a small neighborhood of $\gamma \cup C^+ \cup C^-$.
We define $\widetilde X_2^* = \widetilde X_2 \setminus N$, $(\widetilde X_1^+)^* = \widetilde X_1^+ \cup \overline{N}$ and $(\widetilde X_1^-)^* = \widetilde X_1^- \setminus N$ so that $\tilde{X}_2^*$ and $(\tilde{X}_1^-)^*$ are now
adjacent to $(\tilde{X}_1^+)^*$ along boundaries of compression disks supported in $\partial N$. Since $|\kappa^{-1}(i)| \geq 2$ holds for all $i$, the corresponding decomposition on this finite cover also satisfies the required property (\ref{decomposition:eqn:2b}).
\end{proof}

We may now address the proof of Proposition~\ref{prop:newincompressiblecoredecomp}.

\begin{proof}[Proof of Proposition~\ref{prop:newincompressiblecoredecomp}]
Let $U$ be a non-simply connected component of $\Omega(G)$. Then $\stab_G(U)$ is a function group with invariant domain the multiply connected set $U$ and with non-empty set $V_{nd}(\stab_G(U))$.
If the induced coloring $\widetilde\kappa_U$ on $V_{nd}(\stab_G(U))$ consists of only one color, then we let $\widetilde{H}_U = \stab_G(U)$ and let $\cD_U$ be the lift of the boundaries of the compression disks for any incompressible decomposition of $\stab_G(U)$.
Otherwise,  there exists a {\nonel} vertex in $V_{nd}(\stab_G(U))$, to which we assign the color $1$ and we apply Lemma~\ref{lem:incompressiblefunct}.
Let $\widetilde{H}_U$ be the finite index subgroup of $\stab_G(U)$ and let $\cD_U$ be the lift of the boundaries of the compression disks for the decomposition in Lemma~\ref{lem:incompressiblefunct}.
Since $\widetilde{H}_U$ is finitely generated and $G$ is LERF \cite[Theorem 9.2]{Ago13}, there exists a finite index subgroup $H_U$ of $G$ such that $H_U \cap \stab_G(U) = \widetilde{H}_U$.
For each $U'$ in the $G$-orbit of $U$, we choose $g_{U'} \in G$ so that $g_{U'}U = U'$ and such that $g_{U_2}^{-1}\circ g_{U_1} \in H_U$ whenever $U_1, U_2 \in G U$ with $U_1 \in H_U U_2$.
Let $S$ be the conformal boundary component for the quotient of $U$ and let
$$
\cD_S = \bigcup_{U' = g_{U'}U}g_{U'}\cD_U, \quad H_S := \bigcap_{U' = g_{U'} U} g_{U'} H_{U} g_{U'}^{-1}.
$$
Then (a) $H_S$ contains the normal core $\bigcap_g gH_Ug^{-1}$, so it has finite index in $G$, and (b) $\cD_S$ is invariant under $H_S$.

For each component of the conformal boundary, we do the same construction, with the understanding that if $S, S'$ have the same color induced by the coloring on $\Omega(G)$, then the special colors $1_S$ and $1_{S'}$ to apply Lemma~\ref{lem:incompressiblefunct} are selected for the vertices with homeomorphic vertex limit sets in both cases.
Set $\cD = \bigcup \cD_S$ and let $H = \bigcap H_S$ be the finite index subgroup of $G$. Then $\cD$ is invariant under $H$, and corresponds to an incompressible decomposition of $H$. By construction, $\cD$ satisfies the three properties in Lemma~\ref{lem:incompressiblefunct} in each component of $\Omega(G)$. 
We consider the same construction of $\cD'$ and $H'$ for $G'$, with the same understanding for the choice of the special colors.

Since the special color $1_S$ for each $S$ is always selected for {\nonel} vertices, property \eqref{prop:newincompressible:1} in the proposition is satisfied.
Let $v\in V_{nd}, v'\in V_{nd}'$ be two vertices of the same color.
By the three properties \eqref{decomposition:eqn:1}, \eqref{decomposition:eqn:2a} and \eqref{decomposition:eqn:2b} in Lemma~\ref{lem:incompressiblefunct},
any $\Omega(G_v)$-color-preserving homeomorphism 
$h:(\widehat\C, \Lambda_v, \mathfrak{P}_v) \rightarrow (\widehat\C, \Lambda_{v'}, \mathfrak{P}_{v'})$
admits an isotopy relative to $\Lambda(G_v)$ to a ${\mathcal{D}}_v$-color-preserving homeomorphism 
$$
h:(\widehat\C, \Lambda_{v},\mathfrak{P}_v, {\mathcal{D}}_{v})\rightarrow (\widehat\C, \Lambda_{v'}, \mathfrak{P}_{v'}, {\mathcal{D}}_{v'}),
$$
which in turn will produce a homeomorphism
$$
H: (\widehat\C, \Lambda(G), \mathfrak{P}(G), \cD) \rightarrow (\widehat\C, \Lambda(G'), \mathfrak{P}(G'),\cD').
$$
Thus property \eqref{prop:newincompressible:2} in the proposition is satisfied, and the proposition follows.
\end{proof}

\end{document}